\RequirePackage{ifpdf}          
\ifpdf 
\documentclass[10pt,pdftex]{article}
\usepackage{hyperref} 
\hypersetup{colorlinks, citecolor=MyDarkRed, filecolor=MyDarkBlue,
	linkcolor=MyDarkBlue, urlcolor=MyDarkBlue}
\fi 

\usepackage[in]{fullpage}       
\usepackage{amsmath,amssymb}	

\usepackage[T1]{fontenc}
\usepackage{tgtermes} 
\usepackage[usenames]{color}   
\definecolor{MyDarkBlue}{rgb}{0, 0.0, 0.45} 
\definecolor{MyDarkRed}{rgb}{0.45, 0.0, 0} 
\definecolor{MyDarkGreen}{rgb}{0, 0.45, 0} 
\definecolor{MyLightGray}{gray}{.90}
\definecolor{MyLightGreen}{rgb}{0.5, 0.99, 0.5}
\DefineNamedColor{named}{Peach}         {cmyk}{0,0.50,0.70,0}
\DefineNamedColor{named}{Cyan}          {cmyk}{1,0,0,0}

\usepackage{tikz}
\usetikzlibrary{shapes,arrows}
\usepackage{graphicx} 
\usepackage{floatflt}
\usepackage{wrapfig}
\usepackage[parfill]{parskip}    
\usepackage{amsfonts, amscd, amssymb, amsthm, amsmath}
\usepackage{mathrsfs}
\usepackage{urwchancal}
\usepackage{xcolor}
\usepackage{MnSymbol}
\usepackage{pdfsync}
\usepackage{enumitem} 
\usepackage[mathscr]{euscript}
\theoremstyle{plain}
\newtheorem{thm}{Theorem}[section]
\newtheorem*{theorem-non}{Theorem}
\newtheorem{lemma}[thm]{Lemma}

\newtheorem{proposition}[thm]{Proposition}
\newtheorem{corollary}[thm]{Corollary}

\theoremstyle{remark}
\newtheorem*{remark}{Remark}

\theoremstyle{definition}
\newtheorem{defn}{Definition}[section]
\newtheorem*{econj}{Eisenstein Conjecture}
\newtheorem*{exmp}{Example}

\usepackage{color}
\definecolor{dred}{rgb}{.65, 0, 0.15}

\def\<{\langle} \def\>{\rangle}

\def\dim{\mathrm{dim}}

\def\Tr{\mathrm{Tr}}

\newcommand{\res}[2]{\left(\frac{#1}{#2}\right)}
\def\ga{\gamma}

\newcommand{\dx}{\partial_{\scriptscriptstyle X_{1}}}
\newcommand{\dy}{\partial_{\scriptscriptstyle Y_{1}}}
\newcommand{\dz}{\partial_{\scriptscriptstyle Z_{1}}}

\usepackage{listings}           
\pagecolor{white}		

\begin{document}
	
\title{Moduli of Hyperelliptic Curves and Multiple Dirichlet Series}	
\author{Adrian Diaconu\footnote{{School of Mathematics, University of Minnesota, Minneapolis, MN 55455, Email: cad@umn.edu}} \, 
and Vicen\c{t}iu Pa\c{s}ol\footnote{{Simion Stoilow, Institute of Mathematics of the Romanian Academy, Bucharest, ROMANIA, Email: vpasol@gmail.com}}}         
\date{}	                
\maketitle

\begin{abstract} 
\noindent In this paper we provide an explicit construction of a {\it distinctive} multiple Dirichlet series associated to products of quadratic Dirichlet L-series, which we believe should be tightly connected to a generalized metaplectic Whittaker function on the double cover of a Kac-Moody group. To do so, we first impose a set of axioms, independent of any group of functional equations, which the aforementioned object should satisfy. As a consequence, we deduce that the coefficients of the $p$-parts of the multiple Dirichlet series satisfy certain recurrence relations. These relations lead to a family of identities, which turns out to be {\it encoded} in the combinatorial structure of certain moduli spaces of admissible double covers. 
Finally, via this crucial connection, we apply Deligne's theory of weights to express inductively the coefficients of the $p$-parts in terms of the eigenvalues of Frobenius acting on the $\ell$-adic \'etale cohomology of local systems on the moduli $\mathscr{H}_{g}[2]$ of hyperelliptic curves of genus $g$ with level 2 structure. 
\end{abstract}

\tableofcontents

\section{Introduction}\label{section 1} 
For a reduced root system $\Phi$ of rank $r$ and a number field $F$ containing the $2n$-th roots of unity, one associates a Weyl group multiple Dirichlet series (see \cite{BBCFH}, \cite{BBF1} and \cite{BBFH}) 
\begin{equation} \label{eq: intro0-WMDS-fin-dim-case} 
Z_{\Psi}(s_{1}, \ldots, s_{r}; \Phi) = \sum\, H (c_{1}, \ldots, c_{r}) 
\Psi(c_{1}, \ldots, c_{r}) \mathrm{N}(c_{1})^{-  s_{1}} \cdots \, \mathrm{N}(c_{r})^{-  s_{r}}
\end{equation} 
the sum being over non-zero ideals $\mathfrak{c}_{1} = (c_{1}), \ldots, \mathfrak{c}_{r} = (c_{r})$ of the ring $\mathfrak{o}_{S}$ of $S$-integers for some sufficiently large set $S$ of places; it is assumed that the finite set $S$ contains all archimedean places, 
and that $\mathfrak{o}_{S}$ is a principal ideal domain. The product $H\Psi$ 
remains unchanged if $c_{i}$ ($1\le i \le r$) is multiplied by a unit, i.e., it is a function of ideals in $\mathfrak{o}_{S}.$ The function $H$ is very important, giving the structure of the multiple Dirichlet series; it is completely determined, via a {\it twisted} multiplicativity, by the function field analog of $Z_{\Psi}(s_{1}, \ldots, s_{r}; \Phi),$ which turns out to be a rational function. Alternatively, one specifies the $p$-parts of $Z_{\Psi}$ for primes $(p),$ that is, the generating series 
\begin{equation*}
\sum_{k_{1}\!, \ldots, k_{r} \ge 0} H (p^{k_{1}}\!, 
\ldots, p^{k_{r}}) \mathrm{N}(p)^{-  k_{1} s_{1} - \cdots -  k_{r} s_{r}}.
\end{equation*} 
The factor $\Psi$ is less important, and represents a technical device chosen from a finite-dimensional vector space of functions on $F_{S} = \prod_{v\in S} F_{v},$ constant on cosets of an open subgroup, such that, together with $H,$ it gives a multiple Dirichlet series possessing meromorphic continuation to $\mathbb{C}^{r}$ and satisfying a finite group of functional equations isomorphic to the Weyl group of $\Phi.$

Over the past ten years or so, it has emerged that these multiple Dirichlet series can be 
understood in the world of Eisenstein series. More precisely, let $G$ be a simply connected 
algebraic group over $F$ whose root system is the {\it dual} of the root system $\Phi,$ 
i.e., $\Phi$ is the root system of the $L$-group ${_{}^L}G.$ Then Brubaker, Bump and Friedberg \cite{BBF1} made the following conjecture: 

\vskip10pt
\begin{econj} --- {\it The Weyl group multiple Dirichlet series 
$Z_{\Psi}(s_{1}, \ldots, s_{r}; \Phi)$ in \eqref{eq: intro0-WMDS-fin-dim-case} is a Whittaker coefficient of a minimal parabolic Eisenstein series on an $n$-fold metaplectic cover $\tilde{G}$ of $G.$}
\end{econj} 

\vskip5pt 
This conjecture has been established in the case of metaplectic covers of ${\rm GL}_{n}$ by 
Brubaker, Bump and Friedberg \cite{BBF2, BBF3}. Shortly after \cite{BBF2} and \cite{BBF3} 
have been published, Chinta and Offen \cite{CO} obtained formulas for the spherical Whittaker functions on the $m$-fold metaplectic cover of ${\rm GL}_{n}$ over a $p$-adic field. Their formulas generalize the well-known formula of Shintani for the spherical ${\rm GL}_{n}$--Whittaker function, which corresponds to the nonmetaplectic case, i.e., $m =1.$ Furthermore, the authors provide in \cite{CO} the relationship between $p$-adic metaplectic Whittaker functions and the local parts of Weyl group multiple Dirichlet series associated to root systems of type $A_{n - 1}.$ Very recently, McNamara \cite{McN} extended the results of Chinta and Offen to the more general case of covers of unramified reductive groups; see also \cite{P-P1}.

For very important applications, it is necessary, however, to establish the relevant analytic properties of Weyl group multiple Dirichlet series possessing {\it infinite} groups of functional equations. While the theory of Weyl group multiple Dirichlet series associated with classical (finite) root systems has seen great advances in the last decade, it is not at all clear how to extend this theory to the more general setting of Kac-Moody Lie algebras and their Weyl groups. Undoubtedly, such an extension will require deep foundational work. In this regard, a Casselman-Shalika type formula for Whittaker functions on metaplectic covers of Kac-Moody groups over non-archimedean local fields has been recently established by Patnaik and Pusk\'as in their beautiful paper \cite{P-P2}; their work builds on \cite{Pat} and \cite{P-P1}. However, except for some special cases of affine Kac-Moody groups, it is not at all clear, even conjecturally, how their formula is related to the local parts of Weyl group multiple Dirichlet series.

In this paper we provide an explicit construction of a {\it distinctive} multiple Dirichlet series associated to products of quadratic Dirichlet L-series, which we believe should be tightly connected to a generalized metaplectic Whittaker function on the double cover of a Kac-Moody group. To do so, we shall first impose a set of axioms, independent of any group of functional equations, which the aforementioned object should satisfy\footnote{These axioms are easily verified in the finite-dimensional case.}; see also \cite{White1} and \cite{White2}, where the same axiomatic approach was taken up to construct Weyl group multiple Dirichlet series associated to simply-laced {\it affine} root systems. As a consequence, we deduce that the coefficients of the $p$-parts of the multiple Dirichlet series we are interested in satisfy certain recurrence relations. These relations lead to a family of identities, which in turn, as shown in Sections \ref{section 7} and \ref{section: The series BarC(X, T, q)}, is {\it encoded} in the combinatorial structure of certain moduli spaces of admissible double covers. Finally, via this crucial connection, we apply Deligne's theory of weights to express inductively the coefficients of the $p$-parts of the multiple Dirichlet series in terms of certain $q$-Weil numbers (the eigenvalues of Frobenius acting on the $\ell$-adic \'etale cohomology of local systems on the moduli $\mathscr{H}_{g}[2]$ of hyperelliptic curves of genus $g$ with level 2 structure).

More concretely, let $\chi_{d}$ denote the real primitive Dirichlet character associated to 
$\mathbb{Q}\big(\sqrt{d} \big).$ For $r \ge 1,$ we take the multiple Dirichlet series attached to the $r$-th 
moment of quadratic Dirichlet L-series to be of the form (cf. \cite{CG1} in the finite-dimensional case) 
\begin{equation*}
\sum_{\substack{m_{1}, \ldots, m_{r},\, d-\text{odd} \\ d= d_{0}d_{1}^{2},\, d_{0}-\text{sq.\,free}}} \frac{\chi_{_{d_{0}}}\!(\widehat{m}_{1})\,\cdots\, 
\chi_{_{d_{0}}}\!(\widehat{m}_{r})}{m_{1}^{s_{1}}
\cdots\, m_{r}^{s_{r}}(d_{0}^{} d_{1}^2)^{s_{r+1}}}\cdot
b(m_{1}, \ldots, m_{r}, d)      \;\; \qquad \;\;  \text{(with\, $b(1, \ldots, 1) = 1$)}
\end{equation*} 
where $\widehat{m}_{i},$ for $i = 1, \ldots, r,$ denotes the part of $m_{i}$ coprime to $d_{0}.$ The coefficients 
$
b(m_{1}, \ldots, m_{r}, d)
$ 
are assumed to be multiplicative, 
$$ 
b(m_{1}, \ldots, m_{r}, d)\;\,=  \prod_{\substack{p^{k_{i}} \parallel m_{i}
\\ p^{l} \parallel d}} \hskip-3pt b(p^{k_{1}}\!, \ldots, p^{k_{r}}\!, p^{l})
$$ 
the product being over primes $p\ne 2.$ The generating series 
$$
\sum_{k_{1}\!, \ldots, k_{r}\!,\, l \, \geq \, 0} 
b(p^{k_{1}}\!, \ldots, p^{k_{r}}\!, p^{l})\, p^{-k_{1}s_{1} - \cdots - k_{r}s_{r} - l s_{r+1}}
$$ 
is usually referred to as the $p$-part of the multiple Dirichlet series.

To specify the coefficients $b(p^{k_{1}}\!, \ldots, p^{k_{r}}\!, p^{l}),$ let 
$\mathbb{F}_{\! q}$ be a finite field 
of odd characteristic, and consider a similar multiple Dirichlet series over 
$\mathbb{F}_{\! q}(x).$ In particular, for $m_{1}, \ldots, m_{r}, d$ monic polynomials in 
$\mathbb{F}_{\! q}[x],$ we have multiplicative coefficients $a(m_{1}, \ldots, m_{r}, d),$ 
$$ 
a(m_{1}, \ldots, m_{r}, d)\;=  \prod_{\substack{\pi^{k_{i}} \parallel m_{i} 
\\ \pi^{l} \parallel d}} \hskip-3pt a(\pi^{k_{1}}\!, \ldots, \pi^{k_{r}}\!, \pi^{l})
$$ 
the product being taken this time over monic irreducibles $\pi$ in $\mathbb{F}_{\! q}[x].$ 
For compatibility reasons, we shall assume that $a(\pi^{k_{1}}\!, \ldots, \pi^{k_{r}}\!, 1) = 1$
for all $k_{1}, \ldots, k_{r} \in \mathbb{N},$ and that $a(\pi^{k_{1}}\!, \ldots, \pi^{k_{r}}\!, \pi) = 0$ 
unless $k_{1} = \cdots  = k_{r} = 0,$ in which case this coefficient is $1.$ Furthermore, 
the multiple Dirichlet series can also be expressed as a power series
$$
\sum_{k_{1}\!, \ldots, k_{r}\!, \, l \, \geq \, 0} 
\lambda(k_{1}, \ldots, k_{r}, l; q)\, q^{-k_{1}s_{1} - \cdots - k_{r}s_{r} - l s_{r+1}}
$$ 
in this case. \!We require that the coefficients $\lambda(k_{1}, \ldots, k_{r}, l; \, \cdot \, )$ be defined at $1\slash \mathfrak{q}$ for every odd prime power $\mathfrak{q},$ and that 
$$
a(\pi^{k_{1}}\!, \ldots, \pi^{k_{r}}\!, \pi^{l})
= |\pi|^{k_{1} + \cdots + k_{r} + l}\lambda(k_{1}, \ldots, k_{r}, l; 1\slash |\pi|)
$$ 
for every monic irreducible $\pi$ in $\mathbb{F}_{\! q}[x].$ In addition, we require that 
$$
b(p^{k_{1}}\!, \ldots, p^{k_{r}}\!, p^{l}) 
= p^{k_{1} + \cdots + k_{r} + l}\lambda(k_{1}, \ldots, k_{r}, l; 1\slash p)
$$ 
for all primes $p \ne 2.$ In other words, the $p$-part of the multiple Dirichlet series 
over the rationals coincides with the $\pi$-part of the multiple Dirichlet series over 
$\mathbb{F}_{\! p}(x),$ for every linear monic polynomial $\pi$ in $\mathbb{F}_{\! p}[x].$ 
It may be worth remarking that the same principle applies over any number field \cite{CG1, CG2}, 
allowing one to obtain the corresponding multiple Dirichlet series in this context. Thus 
we need only discuss the rational function field case.

It turns out that any (untwisted) Weyl group multiple Dirichlet series over the rational function field and any of its own $\pi$-parts can be transformed into each other by a simple change of variables. This remarkable fact (which is also assumed in this work) was first noticed in \cite{C08} and \cite{CFH} for $\Phi = A_{2}.$ Of course, the conditions imposed so far do {\it not} suffice to determine a Weyl group multiple Dirichlet series. In the finite-dimensional case, one starts with a reduced root system $\Phi,$ and seeks to construct a multiple Dirichlet series satisfying a group of functional equations isomorphic to the Weyl group of $\Phi.$ Over a rational function field this object can be obtained by using the Chinta-Gunnells averaging 
(over the Weyl group of $\Phi$) method, introduced in \cite{CG1} and further developed in 
\cite{CFG} and \cite{CG2}. This method has been extended by Lee and Zhang \cite{LZ} to symmetrizable Kac-Moody root systems, but in general, the series constructed in this way fails to satisfy the required conditions. For the Weyl group multiple Dirichlet series associated to moments of quadratic Dirichlet L-series, for example, the underlying group of functional equations is isomorphic to the Weyl group $W_{ r}$ of the Kac-Moody 
algebra $\mathfrak{g}(A)$ with generalized Cartan matrix (cf. \cite{Kac}) 
$$
A \,= \begin{pmatrix} 2 & & & & -1 \cr  & 2 & & & -1 \cr  & & & &\vdots \cr 
& & & 2 & -1 \cr  -1 &-1 &\cdots & -1 & 2\end{pmatrix}.
$$ 
In particular, this group is {\it finite} if $r\le 3$ and {\it infinite} if $r\ge 4.$ The Chinta-Gunnells construction 
gives in this context a series whose coefficients are easily seen (see \cite{BD} and \cite{LZ}) to be polynomial 
functions of $q.$ On the other hand, as long as $r$ is large enough\footnote{For instance, one can take $r\ge 9.$} 
this property is {\it not} compatible with the above conditions imposed on this series. What happens is that 
this construction is missing the precise contribution corresponding to the so-called {\it imaginary} roots 
of $\mathfrak{g}(A),$ which grows in complexity as $r$ increases.

Instead, we proceed as follows. The conditions imposed on the multiple Dirichlet series over 
$\mathbb{F}_{\! q}(x)$ imply immediately the identity:
\begin{equation} \label{eq: intro-eq-1}
\lambda(k_{1}, \ldots, k_{r}, l; q)  
= \sum \chi_{_{d_{0}}}\!(\hat{m}_{1}) \, \cdots \,  
\chi_{_{d_{0}}}\!(\hat{m}_{r}) \, a(m_{1}, \ldots, m_{r}, d)
 \;\; \qquad \;\;  \text{(for fixed $k_{1}, \ldots, k_{r}, l \in \mathbb{N}$)}
\end{equation} 
the sum being over all monic polynomials $m_{1}, \ldots, d$ of degrees 
$k_{1}, \ldots, l,$ respectively; the coefficients $a(m_{1}, \ldots, m_{r}, d)$ 
can be expressed as
\begin{equation*}  
a(m_{1}, \ldots, m_{r}, d)\;\, =  \prod_{\substack{\pi^{k_{i}} \parallel m_{i} 
\\ \pi^{l} \parallel d}} \hskip-3pt q^{(k_{1} + \cdots + k_{r} + l)\deg \pi}
\lambda(k_{1}, \ldots, k_{r}, l;  q^{-\deg \pi}).
\end{equation*} 
These identities lead to a recurrence relation in the coefficients $\lambda(k_{1}, \ldots, k_{r}, l; q).$ (For this, we shall need to impose one additional constraint on these coefficients.) The main question arises whether, for any given $r,$ there exists a multiple Dirichlet series satisfying all our requirements; this will be answered affirmatively by providing the explicit construction of this (uniquely determined) 
series.

To construct this multiple Dirichlet series, we proceed by induction on $l.$ We have: 
\begin{equation*}
\lambda(k_{1}, \ldots, k_{r}, 0; q)  = q^{k_{1} + \cdots +k_{r}}
\end{equation*}
and  
\begin{equation*}
 \lambda(k_{1}, \ldots, k_{r}, 1; q)\, = \, 
 \begin{cases}
 q  & \text{if $k_{1} = \cdots  = k_{r} = 0$}\\ 
 0 & \text{otherwise};  
\end{cases}
\end{equation*} 
these follow directly from the compatibility/initial conditions. The coefficients 
corresponding to $l = 2$ can be read off from the identity 
\begin{equation} \label{eq: intro-l=2}
\sum_{k_{1}\!, \ldots, k_{r}\, \geq \, 0} \lambda(k_{1}, \ldots, k_{r}, 2; q)\, t_{1}^{k_{1}} \cdots \, t_{r}^{k_{r}} 
= \;\prod_{i = 1}^{r} \, (1 - q t_{i})^{-1} \, \cdot \; \underset{z = 0}{\rm{Res}}\,\bigg[\frac{q  z^{2} - q^{2}}{z(z - 1)} 
\prod_{i = 1}^{r}\, (1 - z t_{i})  \bigg(1 - \frac{q t_{i}}{z} \bigg)\bigg]
\end{equation}
\vskip-2pt
\noindent
(see Proposition \ref{Prop1-appendixC}, Appendix \ref{C}); in particular, the coefficients 
$
\lambda(k_{1}, \ldots, k_{r}, 2; q)
$ 
are polynomial functions of $q.$

The first non-trivial coefficients occur when $l = 3.$ To describe how these coefficients can be obtained, for $\kappa = (k_{1},\ldots, k_{r})\in \mathbb{N}^{r},$ put $|\kappa| =  k_{1} + \cdots + k_{r}.$ 
Let $r_{1} = r_{1}(\kappa)$ and $r_{2}  = r_{2}(\kappa)$ denote the number of components of $\kappa$ 
equal to $1$ and $2,$ respectively. Then, in this case, \eqref{eq: intro-eq-1} can be written as 
\begin{equation} \label{eq: intro-eq-2}
\lambda(\kappa, 3; q)  - q^{|\kappa| + 4}
\lambda(\kappa, 3; 1\slash q)   =   a(\kappa, 3; q)
\end{equation}
where $a(\kappa, 3; q) = 0,$ unless $r_{1}$ is even and $r_{1} + 2 r_{2} = |\kappa|,$ in which case,  
\begin{equation*} 
a(\kappa, 3; q) \, =\, 
q^{r_{2}}\mathcal{M}_{3}(r_{1}; q)\;   + \; (q - 1)\; \cdot 
\sum_{\substack{\;\;\; \kappa' \in \, \mathbb{N}^{r} \\ k_{i}^{} - k_{i}' = 0, 1}}  
(- 1)^{|\kappa - \kappa'|}\, q^{|\kappa'| + 3} \lambda(\kappa', 2; 1\slash q). 
\end{equation*}  
Here $\mathcal{M}_{3}(r; q)$ is the moment-sum defined by 
\begin{equation*}   
\mathcal{M}_{3}(r; q)\;\;\; := \hskip-3pt \sum_{\substack{\deg  d_{0} \, = \, 3 \\ d_{0}-\text{monic \& square-free}}}
\left(- \sum_{\theta \, \in \, \mathbb{F}_{\! q}}\, \chi(d_{0}(\theta))  \right)^{\! r}
\end{equation*}  
where $\chi$ is the non-trivial real character of $\mathbb{F}_{\! q}^{\times},$ extended to 
$\mathbb{F}_{\! q}$ by setting $\chi(0) := 0.$ One will notice that \eqref{eq: intro-eq-2} implies the functional equation 
\begin{equation*} 
a(\kappa, 3; q) \,= -\, q^{|\kappa| + 4} a(\kappa, 3; 1\slash q)
\end{equation*}  
which is {\it not} at all a priori clear from the above expression of $a(\kappa, 3; q),$ when $r_{1}$ is even and $r_{1} + 2 r_{2} = |\kappa|.$ The reason for this functional equation is, roughly, that $a(\kappa, 3; q)$ is a quantity attached to the one-point compactification $\bar{\mathcal{A}}_{1}$ of the moduli space $\mathcal{A}_{1}$ of elliptic curves, with $q^{r_{2}}\!\mathcal{M}_{3}(r_{1}; q)$ corresponding to $\mathcal{A}_{1}$ and the remaining 
piece of $a(\kappa, 3; q)$ corresponding to the boundary. In the general case, a similar geometric interpretation 
of \eqref{eq: intro-eq-1} is required for our purposes, and establishing it (which is quite non-trivial) represents the 
main contribution of this work.

For example, take $r = 10, k_{1} = \cdots = k_{r} = 1,$ and for simplicity, assume that $q\ge 3$ is a prime. One finds that 
\begin{equation*} 
a(1, \ldots, 1, 3; q) =  
42 q^{8} - 42 q^{6} -  q^{2}\tau(q) + q \tau(q)
\end{equation*} 
where the function $n \mapsto \tau(n)$ is Ramanujan's $\tau$-function, i.e., the $n$-th Fourier coefficient 
of the cusp form 
\begin{equation*}
\Delta(z) := e^{2\pi i z} \prod_{m = 1}^{\infty} (1 - e^{2\pi i m  z} )^{24} = 
\sum_{n = 1}^{\infty} \tau(n)  e^{2\pi i n  z}   \;\;\,  \qquad  \;\;\,  \text{($z = x + i y$ with $y > 0$)}
\end{equation*} 
of weight 12 for the group $\mathrm{SL}_{2}(\mathbb{Z}).$ This is an immediate consequence of 
\eqref{eq: intro-l=2} and a well-known identity of Bryan Birch \cite{Bi}. We shall impose an additional 
(dominance) condition to determine the coefficients $\lambda(\kappa, l; q),$ which in this case gives
\begin{equation*}  
\lambda(1, \ldots, 1, 3; q) = 42 \,q^{8}  -  q^{2}\tau(q).
\end{equation*}

More generally, we will prove: 
\vskip10pt
\newtheorem*{DP1}{Theorem}
\begin{DP1} --- {\it Let $q\ge 3$ be a prime, and let $T_{2k} (q) : = \Tr(T_{ q} \, \vert \, S_{2k})$ denote the trace of the Hecke operator $T_{ q}$ acting on the space of elliptic cusp forms of weight $2k$ on $\mathrm{SL}_{2}(\mathbb{Z}).$ For $\kappa = (k_{1},\ldots, k_{r})\in \mathbb{N}^{r},$ let $r_{1} = r_{1}(\kappa)$ and $r_{2}  = r_{2}(\kappa)$ denote the number of components 
of $\kappa$ equal to $1$ and $2,$ respectively. If 
$
r_{1} + 2 r_{2} = |\kappa|,
$ 
and 
$r_{1} = 2R$ is even, then 
\begin{equation*} 
\lambda(\kappa, 3; q)  = 
\frac{(2R)!}{R!(R + 1)!}q^{R + r_{2} + 3}\, - 
\; \sum_{j=1}^{R} \, (2j + 1)\, \frac{(2R)! \,q^{R + r_{2} - j + 2}}{(R - j)!(R + j + 1)!}\, 
T_{2j + 2} (q).
\end{equation*}
Otherwise, the coefficients $\lambda(\kappa, 3; q)$ all vanish.}
\end{DP1}

\vskip5pt
When dealing with the general case, it is more convenient to work with moments defined for partitions $\mathfrak{n} = (1^{n_{1}}\!, 2^{n_{2}}\!, \ldots)$ by  
\begin{equation*} 
A_{\mathfrak{n}}(t_{1}, \ldots, t_{r}, q) \; := \sum_{d \, \in \, \mathscr{P}(\mathfrak{n}, q)} 
\bigg(\prod_{k = 1}^{r} P_{\scriptscriptstyle C_{ d}}(t_{k})  \bigg); 
\end{equation*} 
here $\mathscr{P}(\mathfrak{n}, q) = \mathscr{P}(\mathfrak{n}, \mathbb{F}_{\! q}) \subset \mathbb{F}_{\! q}[x]$ is the set of all monic square-free polynomials $d$ with factorization type $\mathfrak{n},$ and $P_{\scriptscriptstyle C_{ d}}$ ($d \in \mathscr{P}(\mathfrak{n}, q)$) is the numerator of the zeta function of the (hyper)elliptic curve 
$C_{ d}\slash \mathbb{F}_{\! q}$ defined by the affine model $y^{2} = d(x).$ The connection with the cohomology of local systems on moduli spaces of hyperelliptic curves comes by expressing the symmetric polynomial $A_{\mathfrak{n}}(t_{1}, \ldots, t_{r}, q)$ in terms of symplectic Schur functions \cite[Appendix A, \S A.3., A.45]{FH}, and by applying Behrend's Lefschetz trace formula \cite{Beh} to each of the corresponding coefficients. This alternative expression and the Poincar\'e-Verdier duality will be used to define $A_{\mathfrak{n}}(t_{1}, \ldots, t_{r}, 1\slash q).$

We note here that an identity of independent interest is obtained by taking a suitable linear combination of the 
quantities $A_{\mathfrak{n}}(q^{-1\slash 2}, \ldots, q^{-1\slash 2}, q)$ (see \eqref{eq: euler-characteristic-w1}). 
Concretely, for a partition $\lambda = (\lambda_{1} \ge \cdots  \ge \lambda_{g} \ge 0)$ of weight
$|\lambda|,$ let $V_{\lambda}$ denote the corresponding irreducible representation of 
$\mathrm{Sp}_{2g}(\mathbb{C}).$ For $\lambda \subseteq (r^{g})$ (i.e., $\lambda_{i} \le r$ for all $1\le i \le g$), 
define $\lambda^{\dag} :=  (g - \lambda_{r}' \ge \cdots  \ge g - \lambda_{1}'),$ 
where the non-zero integers among $\lambda_{j}'$ ($j = 1, \ldots, r$) are the parts of the conjugate partition 
of $\lambda.$ As usual, write 
\begin{equation*} 
\prod_{k = 1}^{r}\frac{k !}{(2 k)!} =
\frac{g_{r}}{(r (r + 1)\slash 2)!}
\end{equation*} 
for the constant appearing in the moment-conjectures of Conrey-Farmer \cite{ConrF} and 
Keating-Snaith \cite{KeS} for the leading-order asymptotics of the moments of L-functions within 
symplectic families, and let 
\begin{equation*}
P_{r}(x) = \prod_{i = 1}^{r}(x + 2 i) \prod_{1\le i < j \le r} (x + i + j).
\end{equation*} 
The polynomial $P_{r}(x - 1)$ appears in the main term of the polynomial 
$Q_{r}(x)$ conjectured by Conrey, Farmer, Keating, Rubinstein and Snaith \cite{CFKRS}; 
see also the recent work of Andrade and Keating \cite[Conjecture 5]{AnK} in the function-field setting. With this notation, we have
\begin{equation*}
\frac{1}{q^{2g}(q - 1)}  \!\!\sum_{\substack{\deg d \, = \, 2g + 1 \\ \text{$d$ \rm{monic \& square-free}}}}  
L({\scriptstyle \frac{1}{2}}, \chi_{_{d}})^{r}  =\, 
\frac{g_{r} P_{r}(2 g)}{\big(r (r + 1)\slash 2 \big)!}  \, + \, 
\sum\, (\dim \, V_{\lambda^{\dag}}) \Tr(F^{*} \vert \, e_{c}^{\lambda}(\mathscr{H}_{g}(w^{1})  \otimes_{_{\mathbb{F}_{\! q}}} \! \overline{\mathbb{F}}_{\! q}))\,
q^{1 - 2g - \frac{|\lambda|}{2}}
\end{equation*} 
the sum in the right-hand side being over all {\it non-trivial} partitions $\lambda \subseteq (r^{g})$ of even weight, and the trace of the geometric Frobenius $F^{*}$ on the Euler characteristic 
$
e_{c}^{\lambda}(\mathscr{H}_{g}(w^{1})  \otimes_{_{\mathbb{F}_{\! q}}} \! \overline{\mathbb{F}}_{\! q})
$ 
is given by 
\begin{equation*}  
\Tr(F^{*} \vert \,
e_{c}^{\lambda}(\mathscr{H}_{g}(w^{1})  \otimes_{_{\mathbb{F}_{\! q}}} \! \overline{\mathbb{F}}_{\! q}))\,
= \frac{q^{\frac{|\lambda|}{2} - 1}}{q - 1} \!\!\sum_{\substack{\deg d \, = \, 2g + 1 \\ \text{$d$ \rm{monic \& square-free}}}}
s_{\scriptscriptstyle \langle \lambda \rangle}(\omega_{\scriptscriptstyle 1}(C_{d})^{\pm 1}, \ldots, \omega_{\scriptscriptstyle g}(C_{d})^{\pm 1});
\end{equation*} 
here $s_{\scriptscriptstyle \langle \lambda \rangle}$ is the symplectic Schur function associated to $V_{\lambda},$ and for a hyperelliptic curve $C_{d},$ 
$\omega_{\scriptscriptstyle 1}(C_{d}), \ldots, \omega_{\scriptscriptstyle g}(C_{d}),$ 
$\omega_{\scriptscriptstyle 1}(C_{d})^{-1}, \ldots, \omega_{\scriptscriptstyle g}(C_{d})^{-1}$ 
are the {\it normalized} (i.e., unitary) eigenvalues of the endomorphism $F^{*}$ of 
$H_{\text{\'et}}^{1}(\bar{C}_{d},\mathbb{Q}_{\ell})$ (with $\ell \nmid q$). The sum in the right-hand side of the above moment-identity can be estimated using the following well-known result \cite[Theorem 10.8.2]{KaS} of Katz and Sarnak: 
\vskip10pt
\newtheorem*{Ka-Sa}{Theorem (Katz-Sarnak)}
\begin{Ka-Sa} --- {\it There exist positive constants $A(g)$ and $C(g)$ 
such that, for every partition $\lambda \subseteq (r^{g}),$ $\lambda \ne \bf{0},$ of even weight, we have the estimate 
\begin{equation*} 
\frac{1}{q^{2g}(q - 1)}\Bigg|\sum_{\substack{\deg d \, = \, 2g + 1 \\ \text{$d$ \rm{monic \& square-free}}}} s_{\scriptscriptstyle \langle \lambda \rangle}(\omega_{\scriptscriptstyle 1}(C_{d})^{\pm 1}, \ldots, \omega_{\scriptscriptstyle g}(C_{d})^{\pm 1}) 
\Bigg| \, \le  \, \frac{2 C(g)(\dim \, V_{\lambda})}{\sqrt{q}}
\end{equation*} 
as long as $q = |\mathbb{F}_{\! q}|  \ge  A(g).$}
\end{Ka-Sa}

\vskip5pt
Returning now to our main goal, we shall see that \eqref{eq: intro-eq-1} (suitably transformed) reflects in fact certain relations among slightly more general moments of characteristic polynomials $A_{\mathfrak{n}, \mathfrak{i}, \mathfrak{j}}(t_{1}, \ldots, t_{r}, q),$ defined for partitions $\mathfrak{n}, \mathfrak{i}$ and $\mathfrak{j}.$ 
These moments can be obtained by applying certain differential operators to $A_{\mathfrak{n}}(t_{1}, \ldots, q);$ in particular, we can define $A_{\mathfrak{n}, \mathfrak{i}, \mathfrak{j}}(t_{1}, \ldots, t_{r}, 1\slash q).$ We reiterate 
that the identity \eqref{eq: intro-eq-1} (and implicitly the relations among $A_{\mathfrak{n}, \mathfrak{i}, \mathfrak{j}}(t_{1}, \ldots, t_{r}, q)$) is just {\it hypothetical} at this point.

For notational convenience, we state here our main result only in the {\it split} case, i.e., for partitions $\mathfrak{n}$ of the form $(1^{n}).$ Although quite technical, the relations among the moments $A_{\mathfrak{n}, \mathfrak{i}, \mathfrak{j}}(t_{1}, \ldots, t_{r}, q)$ and their moduli interpretation in the general case, discussed in detail in Sections \ref{section 7} and \ref{section: The series BarC(X, T, q)}, follow the same principle as in the split case.

For $n \ge 1,$ let $\mathscr{P}_{\! n}\subset \mathbb{F}_{\! q}[x]$ denote the set 
of all monic square-free polynomials of degree $n$ splitting in $\mathbb{F}_{\! q}.$ 
Define $N_{i, j}(d, q)$ (for $d \in \mathscr{P}_{\! n}$ and $i, j \ge 0$) by 
\begin{equation*}
N_{i, j}(d, q)  \,=\,  \prod_{k = 0}^{i - 1} \bigg(\frac{q - n + a_{1}(C_{d}) + \epsilon - 2k}{2} \bigg)\, 
\prod_{l = 0}^{j - 1}  \bigg(\frac{q - n - a_{1}(C_{d}) -\epsilon  - 2l}{2} \bigg) 
\end{equation*} 
with $\epsilon = 0$ or $1$ according as $n$ is odd or even, and 
$
a_{1}(C_{d}) := \Tr(F^{*} \vert \, H_{\text{\'et}}^{1}(\bar{C}_{d}, \mathbb{Q}_{\ell})).
$ 
(If $i = 0$ or $j = 0,$ we take the corresponding product to be $1.$) Put 
\begin{equation*}
A_{n, i, j}(t_{1}, \ldots, t_{r}, q) \, := \sum_{d \, \in \, 
\mathscr{P}_{\! \scriptscriptstyle n}} 
\bigg(N_{i, j}(d, q) \prod_{k = 1}^{r} P_{\scriptscriptstyle C_{d}}(t_{k}) \bigg).
\end{equation*} 
Note that if we denote 
$^{D}\!\!A_{n, 0, 0}(t_{1}, \ldots, t_{r}, q) = \frac{\partial A_{n, 0, 0}}{\partial t_{r+1}} (t_{1}, \ldots, t_{r}, 0, q),$ then 
\begin{equation*}
^{D}\!\!A_{n, 0, 0}(t_{1}, \ldots, t_{r}, q) \, = \, - \sum_{d \, \in \, \mathscr{P}_{\! \scriptscriptstyle n}} 
\bigg(a_{1}(C_{d}) \prod_{k = 1}^{r} P_{\scriptscriptstyle C_{d}}(t_{k}) \bigg).
\end{equation*} 
Define $^{D^{k}}\!\!\!A_{n, 0, 0}(t_{1}, \ldots, t_{r}, q)$ ($k\in \mathbb{N}$) by iterating. Then 
\begin{equation*}
A_{n, i, j}(t_{1}, \ldots, t_{r}, q)  \, = \,  
^{\scriptscriptstyle i! j! \binom{(q - n + \epsilon - D)\slash 2}{i}\binom{(q - n -\epsilon + D)\slash 2}{j}}\!A_{n, 0, 0}(t_{1}, \ldots, t_{r}, q)
\end{equation*} 
where the two binomial symbols are viewed as differential polynomials in $D.$ The right-hand side of this identity makes sense if we replace $q$ by $1\slash q,$ allowing us to define 
$A_{n, i, j}(t_{1}, \ldots, t_{r}, 1\slash q).$

Now, for $x, y$ and $z$ algebraically independent variables, consider the (exponential) generating functions:
\begin{equation*}
c_{\text{odd}}(x, y, z; t_{1}, \ldots, t_{r}, q) \; =
\sum_{n, i, j \, \ge \, 0}\, \frac{A_{2n + 1, i, j}(t_{1}, \ldots, t_{r}, q) \, x^{2n + 1} y^{i}  z^{j}}{i! j! \tilde{E}(- t_{1}, \ldots, - t_{r})^{i} \tilde{E}(t_{1}, \ldots, t_{r})^{j}}
\end{equation*} 
and 
\begin{equation*} 
\begin{split}
c_{\text{even}}(x, y, z; t_{1}, \ldots, t_{r}, q)  =\, & \tilde{E}(t_{1}, \ldots, t_{r}) 
\big[\big(1 + z\slash \tilde{E}(t_{1}, \ldots, t_{r}) \big)^{q} - 1\big] \\  
& \hskip40pt  + \, \sum_{n \, \ge \, 1}\, \sum_{i, j \, \ge \, 0}\, \frac{A_{2n, i, j}(t_{1}, \ldots, t_{r}, q) \, x^{2n} y^{i}  z^{j}}{i! j! \tilde{E}(- t_{1}, \ldots, - t_{r})^{i}   \tilde{E}(t_{1}, \ldots, t_{r})^{j}}
\end{split}
\end{equation*} 
where 
\begin{equation*} 
\tilde{E}(t_{1}, \ldots, t_{r}) = \prod_{k = 1}^{r} \frac{1}{(1 - t_{k}) (1 -  q  t_{k})};
\end{equation*} 
we have  
\begin{equation*}
c_{\text{odd}}(x, z, y; - t_{1}, \ldots, - t_{r}, q) = c_{\text{odd}}(x, y, z; t_{1}, \ldots, t_{r}, q).
\end{equation*}
Finally, define 
\begin{equation*}
c(\underline{x}, t_{1}, \ldots, t_{r}, q) = (c_{\text{odd}}(x, y, z; t_{1}, \ldots, t_{r}, q), 
c_{\text{even}}(x, z, y; - t_{1}, \ldots, - t_{r}, q), c_{\text{even}}(x, y, z; t_{1}, \ldots, t_{r}, q))
\end{equation*} 
where $\underline{x}: = (x, y, z).$ Then the relations required in this case are precisely the ones encoded in the following theorem, which itself is of independent interest. 
\vskip10pt
\newtheorem*{DP2}{Theorem (The Split Case)}
\begin{DP2} --- {\it We have 
\begin{equation*}
c(c(\underline{x}, T, q), q  \, T, 1\slash q) = \underline{x}.
\end{equation*}
In other words, $c(\underline{x}, T, q)$ is the formal compositional inverse of 
$c(\underline{x}, q\, T, 1\slash q).$}
\end{DP2}

\vskip5pt
To prove this, we shall make essential use of the functional relation satisfied by the power series $\bar{c}(x, T, q)$ constructed later in Section \ref{section: The series BarC(X, T, q)}. The corresponding functional relation in the general case is a simple consequence of Theorem \ref{ess-main-thm}, whose proof in turn rests on a careful analysis of the combinatorial structure of certain moduli spaces of admissible double covers. 

\vskip3pt
The above result is in the same spirit as the following beautiful theorem of Getzler: 
\vskip10pt
\newtheorem*{Getzler0}{Theorem (Getzler)}
\begin{Getzler0} --- {\it Let $\mathcal{M}_{0, n}$ denote the moduli space of Riemann surfaces of genus zero with $n$ ordered marked points, and let $\overline{\mathcal{M}}_{0, n}$ denote 
its Deligne-Mumford compactification. Then the generating series 
\begin{equation*}
f(x) = x - \sum_{n = 2}^{\infty} \chi(\mathcal{M}_{0, n + 1})\frac{x^{n}}{n!} \,=\, 2  x - (1 + x)  \log (1 + x) \;\;\,  \text{and} \;\;\;
g(x) = x + \sum_{n = 2}^{\infty} \chi(\overline{\mathcal{M}}_{0, n + 1})\frac{x^{n}}{n!}
\end{equation*} 
of Euler characteristics are compositional inverses of one another.}
\end{Getzler0}

\vskip5pt
In fact, Getzler's theorem is easily seen to be the limiting case of our theorem as 
$$
t_{1} = \cdots = t_{r} = 0, \;\; y = z \;\;\text{and}\;\; q\to 1
$$ 
see Section \ref{section: The series BarC(X, T, q)} for details. Closely related results can be found in the work of Kisin and Lehrer \cite{KiLe}, where it is shown how, in certain circumstances, an equivariant comparison theorem in $\ell$-adic cohomology may be used to convert the computation of the graded character of the induced action of a finite group (acting as a group of automorphisms of a smooth complex algebraic variety defined over a number field) on cohomology into questions about numbers of rational points of varieties over finite fields.

Having established a similar result in the general case, we shall apply Deligne's theory of weights to determine inductively (essentially as in the case when $l = 3$) the coefficients 
$\lambda(k_{1}, \ldots, k_{r}, l; q);$ see Section \ref{section 9} for details. The upshot is that we obtain inductively an expression for the differences 
\begin{equation*}
\lambda(k_{1}, \ldots, k_{r}, l; q)  \, - \, q^{k_{1} + \cdots + k_{r} + l + 1} 
\lambda(k_{1}, \ldots, k_{r}, l; 1\slash q)
\end{equation*} 
in terms of $q$-Weil numbers, from which the coefficients $\lambda(k_{1}, \ldots, k_{r}, l; q)$ are determined by imposing the dominance condition. Notice that, by applying a well-known result of Deligne \cite[Proposition 4.8]{Del1} which, in the notation of \ref{Loc-sys Ag}, gives that
\begin{equation*} 
T_{\lambda + 2}(q) = \Tr(T_{q} \, \vert \, S_{\lambda + 2}) 
=  \Tr(F^{*} \vert \, H_{!}^{1}(\mathscr{A}_{1}\otimes_{_{\mathbb{F}_{\! q}}} 
\! \overline{\mathbb{F}}_{\! q}, \mathbb{V}(\lambda))) 
\end{equation*} 
the above formula for the coefficients $\lambda(k_{1}, \ldots, k_{r}, 3; q)$ is, indeed, of this type.

The multiple Dirichlet series over $\mathbb{Q}$ constructed in this paper may be written symbolically as 
\begin{equation} \label{eq: into-MDS-symb}
\sum_{d-\text{odd}} \frac{L^{\scriptscriptstyle (2)}(s_{1}, \chi_{d}) \, \cdots \,  
L^{\scriptscriptstyle (2)}(s_{r}, \chi_{d})}{d^{s_{r+1}}} 
\end{equation} 
although, strictly speaking, the numerator is a product of L-series {\it only} when $d$ is square-free. When $r\le 3,$ this multiple Dirichlet series and its analytic properties have been studied in \cite{Sie}, \cite{GH}, \cite{BFH1}, \cite{DGH} and \cite{Di-Wh}.

On the other hand, it was noticed by Bump, Friedberg and Hoffstein in \cite{BFH1} 
that multiple Dirichlet series satisfying {\it infinite} groups of functional equations 
(e.g., \eqref{eq: into-MDS-symb} when $r = 4$) {\it cannot} be continued in 
$s_{1}, s_{2}, \ldots$ everywhere, i.e., they must have a wall of singularities; see the conjectures in loc. cit., p. 167 and p. 170. We should say that obtaining the continuation of 
\eqref{eq: into-MDS-symb}, when $r \ge 4,$ to a domain in $\mathbb{C}^{r+1}$ containing the point ${\scriptstyle \left(\frac{1}{2}, \ldots, \frac{1}{2}, 1\right)}$ is an {\it extremely difficult} problem; the analogous problem for {\it affine} Weyl group multiple Dirichlet series over rational function fields of odd characteristic has been investigated in \cite{BD}, \cite{White1} and \cite{White2} (see also \cite{Florea}, where an asymptotic formula for the fourth moment of quadratic Dirichlet L-functions in the rational function field case, summed over monic square-free polynomials, 
has been established). As far as we can see, when $r \ge 5,$ this is a very difficult problem even in the rational function field case. We should point out here that a similar ``natural boundary'' phenomenon occurs when dealing with the meromorphic continuation of negative (i.e., lower triangular) parabolic Eisenstein series on loop groups, introduced in \cite{BK} in the function field case; see also \cite{GMP1}, and the introduction of \cite{GMP2}. 
As already alluded to at the beginning of this introduction, the meromorphic continuation of similar metaplectic loop group Eisenstein series could potentially yield the relevant analytic properties of \eqref{eq: into-MDS-symb} when $r = 4.$

Finally, let us briefly describe the structure of the paper. In Section \ref{section 2}, we collect some generalities about quadratic Dirichlet L-functions, which will be used later in the paper. We then discuss the method introduced by Bump, Friedberg and Hoffstein in \cite{BFH2} to construct multiple Dirichlet series attached to moments of quadratic L-series. In particular, this method gives the ``correct'' Weyl group multiple Dirichlet series {\it only} when $r\le 3.$ In Section \ref{section 3}, we first list the conditions which the Weyl group multiple Dirichlet series (to be constructed) has to satisfy; these conditions imply immediately the identity \eqref{eq: intro-eq-1}. Next, we wish to express the left-hand side of this identity more conveniently. Probably the most efficient way to proceed is to reformulate \eqref{eq: intro-eq-1} in terms of generating functions; after some preliminary considerations, we do so in Section \ref{section 4}. It is this form of the original identity where the moments $A_{\mathfrak{n}, \mathfrak{i}, \mathfrak{j}}(t_{1}, \ldots, t_{r}, q)$ arise. In Section \ref{section 5}, we combine a well-known formula for the above moment-sums $\mathcal{M}_{3}(r; q)$ (see \cite{Bi} and \cite{Ih}) with Proposition \ref{Prop2-appendixB}, Appendix \ref{B}, Proposition \ref{Prop-appendixD}, Appendix \ref{D}, and the formula given in Proposition \ref{Prop2-appendixC}, Appendix \ref{C} to illustrate how the coefficients $\lambda(k_{1}, \ldots, k_{r}, l; q)$ ($l = 3, 4$) are constructed. In Section \ref{section 6}, we review some facts about the cohomology of symplectic local systems on moduli of hyperelliptic curves, and discuss the cohomological interpretation of the above moment-sums $A_{\mathfrak{n}}(t_{1}, \ldots, t_{r}, q).$ We then use this interpretation to define $A_{\mathfrak{n}}(t_{1}, \ldots, t_{r}, 1\slash q)$ by duality. In Section \ref{section 7}, we prove our main result, first in the split case, and then in general, subject to a functional relation satisfied by a certain generating function, whose proof is postponed to Section \ref{section: The series BarC(X, T, q)}. In particular, this ensures the compatibility of the relations generated by \eqref{eq: intro-eq-1} among the coefficients of the multiple Dirichlet series. In Section \ref{section: The series BarC(X, T, q)}, we study various generating series attached to moduli spaces of admissible double covers, and then show how these series are connected to the generating series introduced in the previous sections. We then use this connection to establish the functional relation needed to finish the proof of Theorem \ref{princ-theorem}. This section, which may be of some interest in its own right, is essentially independent from the rest of the paper. In Section \ref{section 9}, we recall some facts from Deligne's theory \cite{Del2}, which will then be used to determine inductively the coefficients $\lambda(k_{1}, \ldots, k_{r}, l; q).$ It should be pointed out that the arguments can be reversed to obtain new information about the cohomology of local systems on moduli of hyperelliptic curves, assuming full knowledge\footnote{When $r\le 3,$ these coefficients can be computed from the explicit expressions of $Z(t_{1}, \ldots, t_{r + 1}; q)$ ($r \le 3$) recorded in Appendix \ref{A}.} of the coefficients $\lambda(k_{1}, \ldots, k_{r}, l; q).$ We illustrate this in Section \ref{section 10} by computing explicitly the trace of Frobenius on the (motivic) Euler characteristic corresponding to the Eisenstein cohomology of local systems on moduli spaces of principally polarized abelian surfaces, and thus recovering a result of 
Bergstr\"om, Faber and van der Geer (see \cite[Corollary 4.6]{BFvdG1} and 
\cite[Theorem 9.1]{vdG}). 

Here is a roadmap summarizing the key steps of our construction:

\vskip5pt
\tikzstyle{block} = [rectangle, draw, fill=black!10, text width=7em, text badly centered, minimum height=5em, font=\bfseries\footnotesize\sffamily]
\tikzstyle{line} = [draw, -latex']

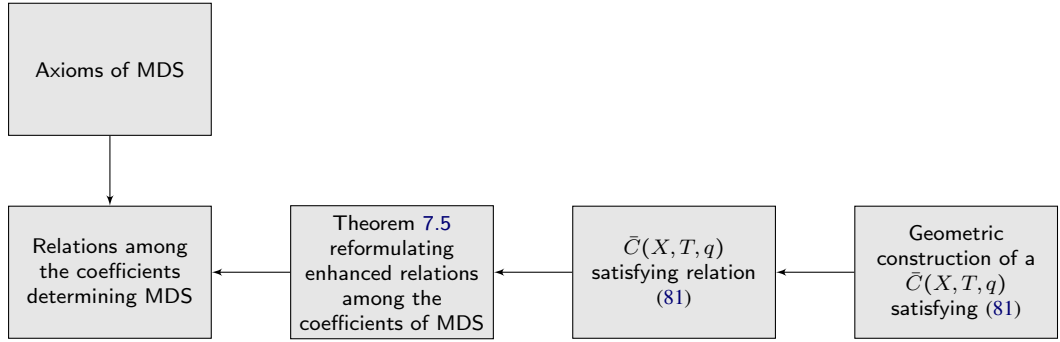
\begin{tikzpicture}[node distance = 3.75cm, auto]
\node [block] (init) {Relations among the coefficients determining MDS};
\node [block, above of=init, node distance=2.70cm] (axioms) {Axioms of MDS};
\node [block, right of=init, node distance=3.75cm] (enhance) {Theorem \ref{princ-theorem} reformulating enhanced relations among the coefficients of MDS};
\node [block, right of=enhance, node distance=3.75cm] (functionalrel) {$\bar{C}(X, T, q)$ satisfying relation \eqref{eq: func-rel-section8}};
\node [block, right of=functionalrel, node distance=3.75cm] (geometric) {Geometric construction of a $\bar{C}(X, T, q)$ satisfying \eqref{eq: func-rel-section8}};
\path [line] (geometric) -- (functionalrel);
\path [line] (functionalrel) -- (enhance);
\path [line] (axioms) -- (init);
\path [line] (enhance) -- (init);
\end{tikzpicture} 

\vskip5pt
For the reader's convenience, we have also included four appendices. In the first one, we provide the (previously known) simple description of the Weyl group multiple Dirichlet series discussed in this paper, for each $r\le 4.$ In Appendix \ref{C}, we compute explicitly the generating series of the coefficients $\lambda(k_{1}, \ldots, k_{r}, 2; q),$ and in Appendices \ref{B} and \ref{D}, we discuss some elementary relations among moments of character sums used in previous sections.

\vskip5pt
{\bf Acknowledgements.} We are very grateful to Jeff Hoffstein, Will Sawin and Ian Whitehead for their comments and suggestions. We would also like to thank Max Planck Institute for Mathematics, where parts of this work have been completed, for support and for providing a wonderful working environment. The second author (VP) would like to thank the School of Mathematics of the University of Minnesota for its hospitality. He was supported by UEFISCDI grant PN-II-RU-TE-2014-4-2077.

\section{Preliminaries}\label{section 2}

We recall that for a smooth, projective, and geometrically connected 
curve $C$ of genus $g \ge 1$ over a finite field $\mathbb{F} = \mathbb{F}_{\! q}$ with $q$ elements, 
one defines the zeta function $Z_{C}(t)$ of $C$ by 
$$
Z_{C}(t)   =  \exp  \Bigg(\sum_{n \ge 1}  |C(\mathbb{F}_{\! q^{n}}  )| \frac{t^n}{n}  \Bigg) \qquad 
(\text{for $| t | < 1/q$}).
$$ 
It is well-known that $Z_{C}(t)$ is a rational function
$$
Z_{C}(t) = \frac{P_{C}(t)}{(1 - t)(1 - q t)}
$$ 
with numerator $P_{C}(t) \in \mathbb{Z}[t]$ of degree $2g,$ and 
constant term $P_{C}(0) = 1.$ Moreover, it satisfies the functional 
equation 
\begin{equation} \label{eq: tag 2.1}
Z_{C}(t) = \big(q t^2 \big)^{ g - 1}  Z_{C}(1\slash qt). 
\end{equation} 
Although we shall not need it here, we recall for completeness that 
the analogue of the Riemann hypothesis, proved by Andr\'e Weil \cite{W}, 
states that $P_{C}(t)$ has all its zeros on the circle $|t| = q^{-1/2}.$

From now on, we shall assume that $q$ is odd, and fix once for all 
an algebraic closure $\overline{\mathbb{F}}$ of $\mathbb{F}.$ For a non-zero 
polynomial $m  \in  \mathbb{F}[x],$ we define its norm by $|m| = q^{\deg m}.$ 
Let $\pi \in \mathbb{F}[x]$ be monic and irreducible. For a polynomial 
$d\in \mathbb{F}[x]$ coprime to $\pi,$ we define the Legendre symbol 
$(d/\pi) = 1$ or $-1$ according as $d$ is a square modulo $\pi$ or not; 
if $d\in \mathbb{F}[x]$ and $\pi$ divides $d,$ we set $(d/\pi) = 0.$ 
Finally, define $(d/m)$ for arbitrary $d\in \mathbb{F}[x]$ 
and monic $m\in \mathbb{F}[x]$ by setting $(d/m) = 1$ if $m = 1,$ 
and by multiplicativity for general $m.$ 

The quadratic residue symbol $(d/m)$ we just defined is completely multiplicative 
in both $d$ and $m,$ and satisfies the following reciprocity law \cite{R}:

{\it Reciprocity.} For $d, m \in \mathbb{F}[x]$ monic polynomials, we have
$$
\res{d}{m} = (-1)^{\frac{|d| - 1}{2}\frac{|m| - 1}{2}} \res{m}{d}.
$$ 
In addition, we have the supplement
$$
\res{d}{m} = \mathrm{sgn}^{\deg m}(d) \;\;\; \qquad 
(\text{for $d\in \mathbb{F}^{\times}$})
$$ 
where $\mathrm{sgn}(d) = 1$ or $-1$ according as $d$ is a square in 
$\mathbb{F}$ or not.

{\it L-functions.} For $d\in \mathbb{F}[x]$ square-free and a monic polynomial 
$m \in \mathbb{F}[x],$ let $\chi_{d}(m) := (d/m).$ 
We define the quadratic Dirichlet $L$-function attached to $\chi_{d}$ by 
\begin{equation*}
L(s, \chi_{d}) \;\,
= \sum_{\substack{m \, \in \, \mathbb{F}[x] \\ m-\text{monic}}}
\chi_{d}(m)|m|^{ - s} =\, \prod_{\pi} (1 - \chi_{d}(\pi)|\pi|^{-s})^{-1} \;\qquad (\text{for complex $s$ with $\Re(s) > 1$});
\end{equation*}
see Artin \cite{Art}. Here the product is over all monic irreducible polynomials $\pi \in \mathbb{F}[x].$ 
In particular, when $d = 1$ (i.e., $\chi_{d}$ is trivial) we obtain 
the zeta function of $\mathbb{F}(x),$
\begin{equation*}
\zeta_{\mathbb{F}(x)}(s)\;\, 
= \sum_{m-\text{monic}} |m|^{ - s} \,
=\, \frac{1}{1 - q^{1 - s}}
\end{equation*} 
and for 
$d\in \mathbb{F}^{\times}  \setminus  (\mathbb{F}^{\times})^{2}$
\begin{equation*}
L(s, \chi_{d}) \,=\, \frac{1}{1 + q^{1 - s}}.
\end{equation*}
If $\deg d = k \ge  1,$ the $L$-function $L(s, \chi_{d})$ is a polynomial (in $q^{-s}$) 
of degree $k - 1.$ To be more precise, let $C_{d}$ denote 
the hyperelliptic curve defined in affine form by $y^2 = d(x),$ and 
consider the numerator $P_{\scriptscriptstyle C_{d}}(t)$ of the zeta function 
$Z_{\scriptscriptstyle C_{d}}(t);$ if $\deg d = 1, 2,$ we take 
$P_{\scriptscriptstyle C_{d}}(t) = 1.$ 
Then
\begin{equation*}
L(s, \chi_{d})\, = \, \sum_{i = 0}^{k-1} \, q ^{ - i s} 
\sum_{\substack{\deg m = i \\  m-\text{monic}}} 
\chi_{d} (m) \,=\, (1 \pm q^{- s})^{\epsilon_{\scriptscriptstyle k}} 
P_{\scriptscriptstyle C_{ d}}(q^{-s})
\end{equation*}
where the $+$ or $-$ sign is determined according to whether 
the leading coefficient of $d$ is a square in $\mathbb{F}$ or not, 
and $\epsilon_{\scriptscriptstyle k} = 0$ or $1$ according as $k$ is odd or even. 
From \eqref{eq: tag 2.1}, we deduce that $L(s, \chi_{d})$ satisfies 
the functional equation 
\begin{equation} \label{eq: tag 2.2}
(1 \pm q^{- s})^{-\epsilon_{\scriptscriptstyle k}}L(s, \chi_{d})\, = \, 
q^{(\epsilon_{\scriptscriptstyle k} + 1)(s - \frac{1}{2})}|d|^{\frac{1}{2} - s}
(1 \pm q^{s - 1})^{-\epsilon_{\scriptscriptstyle k}} L(1 - s, \chi_{d}).
\end{equation}

{\it Multiple Dirichlet series attached to moments of quadratic L-functions.} 
Adapting the discussion from \cite{BFH2}, this is a function of several complex 
variables admitting series representations of the form 
\begin{equation} \label{eq: tag 2.3}
Z(s_{1}, \ldots, s_{r + 1})\;\;
= \sum_{\substack{d \, \in \, \mathbb{F}[x] \\ d-\text{monic}}} 
\frac{A(s_{1}, \ldots, s_{r}; d)}{|d|^{s_{r+1}}}  \;\;\;\;\;=
\sum_{\substack{m_{1}\!, \ldots, m_{r} \, \in \, \mathbb{F}[x] \\ m_{1}\!, \ldots, m_{r}-\text{monic}}} \frac{B(s_{r+1}; m_{1}, \ldots, m_{r})}{|m_{1}|^{s_{1}}
\cdots \, |m_{r}|^{s_{r}}} 
\end{equation}
for $s_{1}, \ldots, s_{r + 1}$ with sufficiently large real parts, and whose coefficients are Euler products 
\begin{equation*}
A(s_{1}, \ldots, s_{r}; d) \,= \prod_{\substack{\pi \\ \pi^{l} \parallel d}}
A_{\pi,\,  l}(s_{1}, \ldots, s_{r}, d_{0}) 
\;\;\; \text{and}\;\;\; B(s_{r+1}; m_{1}, \ldots, m_{r})\; = 
\prod_{\substack{\pi \\ \pi^{k_{i}} \parallel m_{i}}} B_{\pi,\, \kappa}(s_{r+1}, M_{0})
\end{equation*}
over all monic irreducibles $\pi \in \mathbb{F}[x].$ Here we set 
$\kappa = (k_{1}, \ldots, k_{r}),$ and for $d, m_{1}, \ldots, m_{r}$ monic polynomials, 
we denoted by $d_{0}$ and $M_{0}$ the monic square-free parts of $d$ and 
$M = m_{1} \cdots \, m_{r},$ respectively. Letting $\widetilde{\chi}_{\scriptscriptstyle M_{0}} (\pi) = (\pi/M_{0})$ and $|\kappa| = k_{1} + \cdots + k_{r},$ we can present the local factors as
$$
A_{\pi,\,  l}(s_{1}, \ldots, s_{r}, d_{0}) 
\,=\,
 \begin{cases} 
 \prod_{i = 1}^{r}(1 - \chi_{_{d_{0}}}\!(\pi)|\pi|^{-s_{i}})^{-1} 
 \cdot \, P_{ \pi,\, l}(s_{1}, \ldots, s_{r}, \chi_{_{d_{0}}}\!(\pi)) 
 & \text{if $l\ge 0$ is even}\\ 
P_{\pi, \, l}(s_{1}, \ldots, s_{r})  
 & \text{if $l$ is odd} 
\end{cases}
$$
and 
$$
B_{\pi,\, \kappa}(s_{r+1}, M_{0}) 
\,=\,
 \begin{cases} 
 (1 - \widetilde{\chi}_{\scriptscriptstyle M_{0}}(\pi)|\pi|^{-s_{r+1}})^{ -1} 
 \cdot \, Q_{\pi,\,  \kappa}(s_{r+1}, \widetilde{\chi}_{\scriptscriptstyle M_{0}}(\pi)) 
 & \text{if $|\kappa|\ge 0$ is even}\\ 
Q_{\pi, \,  \kappa}(s_{r+1})  
 & \text{if $|\kappa|$ is odd}  
\end{cases}
$$ 
for a certain class of polynomials $P_{\pi,\, l}$ (respectively $Q_{\pi,\,  \kappa}$) in 
$|\pi|^{ - s_{1}}\!, \ldots, |\pi|^{ - s_{r}}$ (respectively $|\pi|^{ - s_{r + 1}}$) characterized as follows.

For arbitrary odd prime power $\mathfrak{q},$ $\kappa = (k_{1}, \ldots, k_{r})\in \mathbb{N}^{r}$ and $l\in \mathbb{N},$ there exist (see \cite{BFH2}) polynomials $P_{\scriptscriptstyle l}(t_{1}, \ldots, t_{r}; \mathfrak{q})$ and 
$Q_{\kappa}(t_{r+1}; \mathfrak{q})$ in $t_{1}, \ldots, t_{r}$ (respectively $t_{r + 1}$) 
with coefficients depending on $l$ and $\mathfrak{q}$ (respectively $\kappa$ and $\mathfrak{q}$)
such that:

$(1).$ For sufficiently small $|t_{1}|, \ldots, |t_{r+1}|,$ we have 
\begin{equation*}
\begin{split}
& (1 - t_{1})^{ -1}\cdots \, (1 - t_{r})^{ -1} \!\sum_{l-\text{even}} P_{\scriptscriptstyle l}(t_{1}, \ldots, t_{r}; \mathfrak{q})t_{r+1}^{l} \; 
+ \, 
\sum_{l-\text{odd}} P_{\scriptscriptstyle l}(t_{1}, \ldots, t_{r}; \mathfrak{q})t_{r+1}^{l}\\ 
& = \; (1 - t_{r + 1})^{-1} \! \sum_{|\kappa|-\text{even}} Q_{\kappa}(t_{r+1}; \mathfrak{q})t_{1}^{k_{1}}\cdots \, t_{r}^{k_{r}} \; + \, 
\sum_{|\kappa|-\text{odd}} Q_{\kappa}(t_{r+1}; \mathfrak{q}) t_{1}^{k_{1}}\cdots \, t_{r}^{k_{r}}
\end{split}
\end{equation*} 
i.e., the generating series of the polynomials $P_{\scriptscriptstyle l}(t_{1}, \ldots, t_{r}; \mathfrak{q})$ and 
$Q_{\kappa}(t_{r+1}; \mathfrak{q})$ coincide. We take this series normalized by assuming that $P_{\scriptscriptstyle 0}(0, \ldots, 0; \mathfrak{q}) = 1.$ 

$(2).$ The power series in $t_{1}, \ldots, t_{r+1}$ obtained by expanding 
$$
\sum_{l \, \ge \, 0} P_{\scriptscriptstyle l}(t_{1}, \ldots, t_{r}; \mathfrak{q})t_{r+1}^{l}
$$ 
is absolutely convergent for arbitrary $t_{1}, \ldots, t_{r} \in \mathbb{C},$ provided 
$|t_{r+1}|$ is sufficiently small, and for arbitrary $t_{r+1} \in \mathbb{C},$ provided 
all $|t_{1}|, \ldots, |t_{r}|$ are sufficiently small.

$(3).$ The polynomials $P_{\scriptscriptstyle l}(t_{1}, \ldots, t_{r}; \mathfrak{q})$ are symmetric in the variables $t_{1}, \ldots, t_{r},$ and they satisfy the functional equation 
\begin{equation*}
P_{\scriptscriptstyle l}(t_{1}, t_{2}, \ldots, t_{r}; \mathfrak{q}) \, =\, 
 \begin{cases} 
(\mathfrak{q} t_{1}^2)^{l\slash 2} P_{\scriptscriptstyle l}(1\slash \mathfrak{q}t_{1}, t_{2}, \ldots, t_{r}; \mathfrak{q})
 & \text{if $l$ is even}\\ 
(\mathfrak{q} t_{1}^2)^{(l - 1)\slash 2} P_{\scriptscriptstyle l}(1 \slash \mathfrak{q}t_{1}, t_{2}, \ldots, t_{r}; \mathfrak{q})  
 & \text{if $l$ is odd}.  
\end{cases}
\end{equation*} 

$(4).$ The polynomials $Q_{\kappa}(t_{r+1}; \mathfrak{q})$ satisfy the functional equation 
\begin{equation*}
Q_{\kappa}(t_{r+1}; \mathfrak{q}) \, =\, 
 \begin{cases} 
(\mathfrak{q} t_{r + 1}^2)^{|\kappa| \slash 2} Q_{\kappa}(1\slash \mathfrak{q}t_{r+1}; \mathfrak{q})
& \text{if $|\kappa|$ is even}\\ 
(\mathfrak{q} t_{r + 1}^2)^{(|\kappa| - 1)\slash 2} Q_{\kappa}(1 \slash \mathfrak{q}t_{r+1}; \mathfrak{q})
& \text{if $|\kappa|$ is odd}.  
\end{cases}
\end{equation*} 

$(5).$ If $l$ is odd, $P_{\scriptscriptstyle l}(t_{1}, \ldots, t_{r}; \mathfrak{q}) = P_{\scriptscriptstyle l}(- t_{1}, \ldots, - t_{r}; \mathfrak{q}).$

Defining 
\begin{equation*}
\begin{split}
&P_{\pi,\, l}(s_{1}, \ldots, s_{r}, \chi_{_{d_{0}}}\!(\pi)) = 
P_{\scriptscriptstyle l}(\chi_{_{d_{0}}}\!(\pi) |\pi|^{- s_{1}}\!, \ldots, \chi_{_{d_{0}}}\!(\pi) |\pi|^{- s_{r}}; |\pi|)\;\;\;\; \text{if $l\ge 0$ is even}\\  
& P_{\pi, \, l}(s_{1}, \ldots, s_{r})  = P_{\scriptscriptstyle l}(|\pi|^{- s_{1}}\!, \ldots, |\pi|^{- s_{r}}; |\pi|) \qquad \qquad \qquad \qquad \quad\hskip9.1pt \text{if $l$ is odd}
\end{split}
\end{equation*}
and 
\begin{equation*}
\begin{split}
& Q_{\pi,\,  \kappa}(s_{r+1}, \widetilde{\chi}_{\scriptscriptstyle M_{0}}(\pi)) = 
Q_{\kappa}(\widetilde{\chi}_{\scriptscriptstyle M_{0}}(\pi) |\pi|^{- s_{r+1}}; |\pi|) \;\;\;\; \text{if $|\kappa|\ge 0$ is even} \\
& Q_{\pi, \,  \kappa}(s_{r+1})  = 
Q_{\kappa}(|\pi|^{- s_{r+1}}; |\pi|) \qquad \qquad \qquad \;\;\;\;\;\;\;\hskip0.9pt \text{if $|\kappa|$ is odd}
\end{split}
\end{equation*} 
one checks that, indeed, the identity \eqref{eq: tag 2.3} holds. Note that 
$P_{\scriptscriptstyle l}(t_{1}, \ldots, t_{r}; \mathfrak{q}) = 1$ if $l = 0, 1,$ and hence 
$
A(s_{1}, \ldots, s_{r}; 1) = \zeta_{\mathbb{F}(x)}(s_{1}) \, \cdots \,  \zeta_{\mathbb{F}(x)}(s_{r}),
$ 
and 
$
A(s_{1}, \ldots, s_{r}; d_{0}) = L(s_{1}, \chi_{\scriptscriptstyle d_{0}})\, \cdots \, L(s_{r}, \chi_{\scriptscriptstyle d_{0}})
$ 
if $d = d_{0}$ is monic and square-free; we also have 
$
B(s_{r+1}; m_{1}, \ldots, m_{r}) = L(s_{r+1},  \widetilde{\chi}_{m_{1}\cdots \, m_{r}})
$ 
if the product $m_{1} \cdots \, m_{r}$ is monic and square-free. \!Accordingly, it makes sense to denote the coefficients $A(s_{1}, \ldots, s_{r}; d)$ and $B(s_{r+1}; m_{1}, \ldots, m_{r}),$ 
for arbitrary monic polynomials $d, m_{1}, \ldots, m_{r},$ by 
$
L(s_{1}, \chi_{d}) \, \cdots \, L(s_{r}, \chi_{d})
$ 
and
$
L(s_{r+1},  \widetilde{\chi}_{m_{1}\cdots \, m_{r}}),
$ 
respectively. We hope that our notation is not causing any confusion, especially with 
a product of $r$ {\it imprimitive} $L$-functions in case $d$ is {\it not} a square-free polynomial.

It can be observed that the multiple Dirichlet series \eqref{eq: tag 2.3} can be expressed 
as a power series in $q^{- s_{i}},$ $i = 1, \ldots, r+1,$ for $s_{1}, \ldots, s_{r + 1}$ 
with sufficiently large real parts. \!In what follows, we shall substitute $t_{i} = q^{- s_{i}}$ 
for $i = 1, \ldots, r + 1,$ set $T = (t_{1}, \ldots, t_{r}),$ and denote the power series 
of \eqref{eq: tag 2.3} by $Z(T, t_{r + 1}; q).$ This function satisfies a group of functional 
equations generated by the $r+1$ involutions 
$\alpha_{i}: (T, t_{r + 1}) \to (t_{1}, \ldots, 1\slash q t_{i}, \ldots, t_{r}, \sqrt{q}\, t_{i} t_{r + 1})$ for $1\le i \le r,$ and $\alpha_{r + 1}: (T, t_{r + 1}) \to (\sqrt{q}\, t_{r + 1}T, 1\slash q t_{r+1}).$ Indeed, from the definition of $Z(T, t_{r + 1}; q)$ and \eqref{eq: tag 2.2}, one finds that: 
\begin{equation} \label{eq: tag 2.4}
\begin{split}
& Z(T, t_{r + 1}; q) \,=\,  
\frac{1}{\sqrt{q}\, t_{i}}\bigg(\frac{Z(\alpha_{i}(T, t_{r + 1}); q) - Z(\alpha_{i}(T, - t_{r + 1}); q)}{2}\bigg) \\ 
& \hskip56pt - \, \frac{1 - t_{i}}{t_{i}(1 - qt_{i})}
\bigg(\frac{Z(\alpha_{i}(T, t_{r + 1}); q) + Z(\alpha_{i}(T, - t_{r + 1}); q)}{2} \bigg)
\end{split}
\end{equation}
for $1\le i \le r,$ and 
\begin{equation} \label{eq: tag 2.5}
\begin{split}
& Z(T, t_{r + 1}; q) \,=\,  
\frac{1}{\sqrt{q}\, t_{r + 1}}\bigg( \frac{Z(\alpha_{r + 1}(T, t_{r + 1}); q) - Z(\alpha_{r + 1}(- T, t_{r + 1}); q)}{2}  \bigg)  \\
& \hskip55pt 
- \, \frac{1 - t_{r + 1}}{t_{r + 1}(1 - qt_{r + 1})}\bigg( \frac{Z(\alpha_{r + 1}(T, t_{r + 1}); q) 
+ Z(\alpha_{r + 1}(- T, t_{r + 1}); q)}{2}  \bigg)
\end{split} 
\end{equation} 
where for the latter, we use in addition the quadratic reciprocity. 
\!The group of functional equations is isomorphic to the Coxeter group $W_{r}$ mentioned in the introduction. \!In particular, it is finite if $r\le 3,$ and infinite if $r\ge 4.$ It is straightforward to check that the functional equations \eqref{eq: tag 2.4} and \eqref{eq: tag 2.5} are equivalent to certain recurrence relations among the 
coefficients of the power series $Z(T, t_{r + 1}; q).$ Fortunately, when $r\le 3,$ one can apply these recurrence 
relations as in the proof of Theorem 3.7 in \cite {BD} to determine all the coefficients of $Z(T, t_{r + 1}; q);$ 
we refer the reader to Appendix \ref{A} for an expression of 
$Z(T, t_{r+1}; q)$ as a rational function in each case. When $r \ge 4,$ we just remark for the moment that additional
information is {\it required} to completely characterize the power series $Z(T, t_{r+1}; q);$ the case $r = 4$ 
discussed in \cite{BD} is still very special in many respects.

On the other hand, letting $P(T, t_{r + 1}; \mathfrak{q})$ denote the generating series in $(1),$ the properties 
$(3)$ and $(4)$ satisfied by the polynomials $P_{\scriptscriptstyle l}$ and $Q_{\kappa},$ respectively, translate into the functional equations: 
\begin{equation*}
\begin{split}
& P(T, t_{r + 1}; \mathfrak{q}) \,= 
\frac{1}{\sqrt{\mathfrak{q}}\, t_{i}}\bigg( \frac{P(\alpha_{i}(T, t_{r + 1}); \mathfrak{q}) 
- P(\alpha_{i}(T, - t_{r + 1}); \mathfrak{q})}{2}  \bigg)\\
& \hskip60pt - \, \frac{1 - \mathfrak{q}t_{i}}{\mathfrak{q}t_{i}(1 - t_{i})}\bigg( \frac{P(\alpha_{i}(T, t_{r + 1}); \mathfrak{q}) + P(\alpha_{i}(T, - t_{r + 1}); \mathfrak{q})}{2}  \bigg)
\end{split}
\end{equation*} 
for $i = 1, \ldots, r,$ and 
\begin{equation*}
\begin{split}
& P(T, t_{r + 1}; \mathfrak{q}) \,=\,  
\frac{1}{\sqrt{\mathfrak{q}}\, t_{r + 1}}\bigg( \frac{P(\alpha_{r + 1}(T, t_{r + 1}); \mathfrak{q})
- P(\alpha_{r + 1}(- T, t_{r + 1}); \mathfrak{q})}{2}  \bigg)\\
& \hskip60pt - \, \frac{1 - \mathfrak{q}t_{r + 1}}{\mathfrak{q}t_{r + 1}(1 - t_{r + 1})}
\bigg(  \frac{P(\alpha_{r + 1}(T, t_{r + 1}); \mathfrak{q}) + P(\alpha_{r + 1}(- T, t_{r + 1}); \mathfrak{q})}{2}  \bigg).
\end{split}
\end{equation*} 
As in \cite{CG1}, we call $P(T, t_{r + 1}; \mathfrak{q}),$ for $\mathfrak{q} = |\pi|,$ the $\pi$-part of $Z(T, t_{r + 1}; q).$ We note that there is a slight difference, as explained in \cite[Remark 4.3]{CG1}, between this definition of the $\pi$-part of our multiple Dirichlet series, and the generating series of the $p$-part-coefficient 
$H(p^{k_{1}}\!,\ldots, p^{k_{r}})$ in \cite{BBCFH}, \cite{BBF1}, \cite{BBFH} and \cite{CG2}.

Just as for $Z(T, t_{r + 1}; q),$ the power series $P(T, t_{r + 1}; \mathfrak{q})$ is also uniquely determined 
by its properties, when $r\le 3.$ Then, we must have that 
\begin{equation} \label{eq: tag 2.6}
P(T, t_{r + 1}; \mathfrak{q}) = Z(\mathfrak{q}T, \mathfrak{q}t_{r + 1}; 1\slash \mathfrak{q}) \qquad \;\; (\text{for $r\le 3$})
\end{equation} 
as the power series in the right-hand side satisfies the required properties. The reader might want 
to verify the relevant facts about $Z(\mathfrak{q}T, \mathfrak{q}t_{r + 1}; 1\slash \mathfrak{q}),$ when $r\le 3,$ 
using the expression of $Z(T, t_{r+1}; q)$ given in Appendix \ref{A}. Unfortunately, when $r\ge 4$ 
the properties $(1)$ -- $(5)$ are insufficient to determine $P(T, t_{r + 1}; \mathfrak{q}),$ 
the obvious reason being that in this case there exist non-trivial power series in $t_{1}, \ldots, t_{r + 1}$ 
invariant under the group of functional equations. For example, if $r = 4$ a power series in the 
variable $t_{1}^{}\cdots \, t_{4}^{}t_{5}^{2}$ is, indeed, invariant under this group. Very shortly, 
we shall propose a way to remedy this deficiency, but before doing so, let us 
conclude this section with the following remark.

\vskip10pt
\begin{remark} \!As indicated in the introduction, every choice of the generating series $(1)$ gives rise, as above (with just some standard adjustments), to the corresponding 
number field version of \eqref{eq: tag 2.3}. When $r \le 3,$ this Weyl group multiple Dirichlet series is a Whittaker coefficient of a minimal parabolic Eisenstein series on the double cover of a split, semisimple, simply-connected algebraic group. To investigate the possibility of extending this result, it is important (as previous experience shows) to have the potential candidates for the generating series $P(T, t_{r + 1}; \mathfrak{q})$ ($r\ge 4$) available beforehand. 
\end{remark}

\section{The series $Z(T, t_{r + 1}; q)$ and moments of character sums}\label{section 3}

To correct the deficiency indicated at the end of the previous section, 
we begin by imposing more stringent conditions on the multiple Dirichlet series \eqref{eq: tag 2.3}. It is interesting to note that our assumptions make {\it no} reference to the group $W_{r}$ of functional equations.

As in \cite{CFH}, we first require $Z(s_{1}, \ldots, s_{r + 1})$ to be of the form
\begin{equation} \label{eq: tag 3.1}
Z(s_{1}, \ldots, s_{r + 1})\;\;\,   = 
\sum_{\substack{m_{1}\!, \ldots, m_{r}, \, d \, \in \, \mathbb{F}[x] \\  d = d_{0}^{}d_{1}^2, \; d_{0}^{}\; \text{square-free}}}\,  \frac{\chi_{_{d_{0}}}\!(\hat{m}_{1}) \, \cdots \,  
\chi_{_{d_{0}}}\!(\hat{m}_{r})}{|m_{1}|^{s_{1}}
\cdots \, |m_{r}|^{s_{r}} |d|^{s_{r + 1}}}
\cdot a(m_{1}, \ldots, m_{r}, d) 
\end{equation}
the sum taken over all $(r+1)$-tuples of monic polynomials, 
with $\hat{m}_{i}$ the part of $m_{i}$ coprime to $d_{0}$ for $i = 1, \ldots, r.$ 
The coefficients $a(m_{1}, \ldots, m_{r}, d)$ are assumed to be multiplicative, 
i.e., 
\begin{equation} \label{eq: tag 3.2}
a(m_{1}, \ldots, m_{r}, d)\,\,= 
\prod_{\substack{\pi^{k_{i}}\parallel m_{i} \\ \pi^{l} \parallel d}} a(\pi^{k_{1}}\!, \ldots, \pi^{k_{r}}\!,\pi^{l})   
\end{equation} 
the product being taken over monic irreducibles $\pi  \in  \mathbb{F}[x].$
Now, express the series \eqref{eq: tag 3.1} as  
\begin{equation} \label{eq: tag 3.3}
Z(T, t_{r + 1}; q)\;\;\,  = \sum_{k_{1}\!, \ldots,  k_{r}\!,\, l \, \geq \, 0} \,
\lambda(k_{1}, \ldots, k_{r}, l; q)\, t_{1}^{k_{1}} \cdots \, t_{r}^{k_{r}} t_{r+1}^{l}
\end{equation} 
where, as before, we substituted $t_{i} = q^{- s_{i}},$ for $i = 1, \ldots, r + 1,$ and 
set $T = (t_{1}, \ldots, t_{r}).$ We require 
$a(\pi^{k_{1}}\!, \ldots, \pi^{k_{r}}\!,\pi^{l})$ and 
$\lambda(k_{1}, \ldots, k_{r}, l; q)$ be such that:
\begin{enumerate}[label=(\roman*)]
\item {\it Initial Conditions.} The subseries
$$
\sum_{k_{1}\!, \ldots, k_{r} \, \ge \, 0}
\lambda(k_{1}, \ldots, k_{r}, 0; q)\, t_{1}^{k_{1}} \cdots \, t_{r}^{k_{r}}
=\; \prod_{i = 1}^{r} \frac{1}{1 - qt_{i}}
$$ 
i.e., is a product of $r$ zeta functions. In addition,
$$
\sum_{k_{1}\!, \ldots, k_{r} \, \ge \, 0}
\lambda(k_{1}, \ldots, k_{r}, 1; q)\, t_{1}^{k_{1}}\cdots \, t_{r}^{k_{r}} \;\, = 
\sum_{\substack{\deg d = 1 \\ d-\text{monic}}}
L(s_{1}, \chi_{d}) \, \cdots \, L(s_{r}, \chi_{d}) \,=\, q 
$$ 
and 
$$
\sum_{l \, \ge \, 0}\,  \lambda(0, \ldots, 0, l; q)\, t_{r + 1}^{l}
= \frac{1}{1 - qt_{r + 1}}.
$$ 
In particular, $a(1, \ldots, 1, 1) = \lambda(0, \ldots, 0, 0; q) = 1.$

\item {\it Extension of \eqref{eq: tag 2.6}.} Fix nonnegative integers $k_{1}, \ldots, k_{r}, l.$ Then, for each odd prime $p,$ there exists a finite number of complex numbers $c_{j}$ and $w_{j} = p^{a_{j} + \, i b_{j}}$ ($a_{j}, b_{j}\in \mathbb{R}$) such that 
\begin{equation*}
\lambda(k_{1}, \ldots, k_{r}, l; p^{n}) = \sum_{j} c_{j}(w_{j}^{n} + \overline{w}_{j}^{n}) 
\qquad \text{(for all $n\in \mathbb{Z}$)}
\end{equation*} 
all data in the right-hand side depending on $k_{1}, \ldots, k_{r}, l$ and $p.$ Moreover, 
\begin{equation*}
a(\pi^{k_{1}}\!, \ldots, \pi^{k_{r}}\!, \pi^{l})
\,=\, |\pi|^{k_{1} + \cdots + k_{r} + l}\lambda(k_{1}, \ldots, k_{r}, l; 1\slash |\pi|)
\end{equation*}
for every monic irreducible $\pi \in \mathbb{F}[x].$

\item {\it Dominance.} For every $(r+1)$-tuple 
$(k_{1}, \ldots, k_{r}, l) \in \mathbb{N}^{r + 1}$ 
with $k_{1} + \cdots + k_{r} + l \ge 2,$ 
the coefficient $\lambda(k_{1}, \ldots, k_{r}, l; q)$ represents 
the dominant half (as in the example below) of      
\begin{equation*}
\lambda(k_{1}, \ldots, k_{r}, l; q)\, -\, 
q^{k_{1} + \cdots + k_{r} + l + 1}\lambda(k_{1}, \ldots, k_{r}, l; 1\slash q) 
\end{equation*} 
that is, $\mathrm{min}\{a_{j}\} > (k_{1} + \cdots + k_{r} + l + 1)\slash 2.$ 
\end{enumerate} 

\vskip10pt
\begin{exmp}
When $r = 10, \,k_{1} = \cdots = k_{r} = 1,\,l = 3,$ and $q$ 
is an odd prime, it turns out (see Section \ref{section 5}) that 
\begin{equation*}
\lambda(1, \ldots, 1, 3; q)\, -\, 
q^{14}\lambda(1, \ldots, 1, 3; 1\slash q) \, =\,  
42 q^{8} - 42 q^{6} -  q^{2}\tau(q) + q\tau(q)  
\end{equation*}
where $\tau(q)$ is Ramanujan's tau function; we recall, see \cite{Del1}, that $\tau(q) = \alpha_{q} + \overline{\alpha}_{q}$ with $|\alpha_{q}| = q^{11\slash 2}.$ By (iii), it follows that
$
\lambda(1, \ldots, 1, 3; q) =  
42 q^{8}  -  q^{2}\tau(q).
$ 
(Here we set $\tau(1\slash q) = q^{-11} \tau(q).$) 
\end{exmp}

To obtain concrete information about the coefficients of 
the multiple Dirichlet series \eqref{eq: tag 3.1}, fix $k_{1},\ldots, k_{r}, l \in \mathbb{N}.$ 
Comparing the coefficients $(k_{1}, \ldots, k_{r}, l)$ of \eqref{eq: tag 3.1} and \eqref{eq: tag 3.3}, we see that 
\begin{equation} \label{eq: tag 3.4}
\lambda(k_{1}, \ldots, k_{r}, l; q) \, 
= \,\sum\, \chi_{_{d_{0}}}\!(\hat{m}_{1}) \, \cdots \,  
\chi_{_{d_{0}}}\!(\hat{m}_{r}) \, a(m_{1}, \ldots, m_{r}, d)
\end{equation} 
with the sum over all monic polynomials $m_{1}, \ldots, m_{r}, d$ 
of degrees $k_{1}, \ldots, k_{r}, l,$ respectively. \!By \eqref{eq: tag 3.2} and (ii), 
the dependence of the coefficients 
$a(m_{1}, \ldots, m_{r}, d)$ on monic irreducible polynomials 
$\pi$ is only upon their degrees $\mu_{\pi},$ and the exact powers 
$\nu_{1, \pi}, \ldots, \nu_{r, \pi},\, \nu_{\pi}$ 
to which $\pi$ is dividing $m_{1}, \ldots, m_{r}, d.$ 
This simple observation will be used to express the right-hand side of \eqref{eq: tag 3.4} in terms of classical moments of (projective) character sums defined for partitions 
$\mu = (\mu_{_1}\ge \cdots \ge \mu_{_n} \ge 1)$ and 
$\gamma = (1^{\gamma_{_{1}}}\!, \ldots, m^{\gamma_{_{m}}} \!)$ by   
\begin{equation*}
M_{\mu, \gamma}(q) \;\, =  
\sum_{d_{0} \, \in \, \mathscr{P}(\mu)} \, \prod_{j = 1}^{m}  \left(-\, \frac{1 + (- 1)^{\deg d_{0}}}{2}  - \sum_{\theta \, \in  \, \mathbb{F}_{\! j}} \chi_{_{j}}(d_{0}(\theta))\right)^{\!\gamma_{_{ j}}}\!.
\end{equation*}
Here the sum is over the subset $\mathscr{P}(\mu)\subset \mathbb{F}[x]$ of all monic square-free polynomials $d_{0}$ having factorization type (over $\mathbb{F}$) into irreducibles
$
d_{0} = \prod_{i = 1}^{n} \pi_{_{ i}}
$ 
with $\deg \pi_{_{ i}} = \mu_{_i};$ 
for $j \ge 1,$ we put $\mathbb{F}_{\! j} = \mathbb{F}_{\! q^{j}},$ 
and $\chi_{_{j}} : = \chi \circ \mathrm{N}_{_{\mathbb{F}_{\! j}\slash \mathbb{F}}}$ 
is the non-trivial real character of $\mathbb{F}_{\! j}^{\times},$ extended to $\mathbb{F}_{\! j}$ 
by setting $\chi_{_{j}}(0) := 0.$ The connection to moments of character sums 
will play a major role in this work. It will become apparent that 
the conditions (i) -- (iii) together with the identity \eqref{eq: tag 3.4} 
allow a recursive computation of the coefficients 
$
\lambda(k_{1}, \ldots, k_{r}, l; q),
$ 
provided certain precise information about the moments 
$M_{\mu, \gamma}(q)$ could be obtained.

We begin by first summing over all $m_{1}, \ldots, m_{r}, d$ 
with prescribed factorizations into irreducibles by giving a partition  
$\mu = (\mu_{1}\ge \cdots \ge \mu_{n} \ge 1)$ together with a collection of 
$(r+1)$-tuples $(\nu_{1j}, \ldots, \nu_{rj},\, \nu_{j})\in \mathbb{N}^{r+1},$ 
with $1\le j\le n,$ such that
\begin{equation} \label{eq: tag 3.5}
\sum_{j = 1}^{n}\,  \nu_{ij}\mu_{j} \,=\,  k_{i} \;\;\; \text{and} \;\;\;
\sum_{j = 1}^{n}\, \nu_{j}\mu_{j}  \,=\, l \qquad (\text{for $i = 1, \ldots, r$}).
\end{equation} 
To $\mu$ and $\mathcal{N} : = (\nu_{1j}, \ldots, \nu_{rj},\, \nu_{j})_{1\le j\le n}$ as above, 
we associate the coefficients 
\begin{equation} \label{eq: tag 3.6}
\lambda_{\mu}(\mathcal{N}; q) := \prod_{j=1}^{n}\,
q^{\mu_{j}(\nu_{1j} + \cdots + \nu_{rj} + \nu_{j})}
\lambda(\nu_{1j}, \ldots, \nu_{rj},\, \nu_{j}; 1\slash q^{\mu_{j}})
\end{equation} 
and certain character sums $S_{\mu, \, \mathcal{N}}$ 
defined as follows. Split the set of indices $1\le j \le n$ 
into two parts $J_{0}$ and $J_{1}$ according as $\nu_{j}$ is 
odd or even. Let $\nu := (\nu_{1},\ldots,\nu_{n}),$ and let 
$\mathscr{P}_{\! \nu}(\mu)\subset \mathbb{F}[x]$ denote the set of all 
monic square-free polynomials $d_{0}$ having factorization type 
into irreducibles, 
$
\prod_{j\in J_{0}} \pi_{j}
$ 
with $\deg \pi_{j} = \mu_{j}.$ 
With this notation, we define 
\begin{equation*}
S_{\mu,\, \mathcal{N}} \, :=\,  C_{\mu, \nu} \, \cdot
\sum_{d_{0} \, \in \, \mathscr{P}_{\!\nu}(\mu)} \;\,
\sideset{}{^*}\sum_{(\theta_{j})_{j\in J_{1}}}\, \prod_{j\, \in \, J_{1}} 
\chi_{\mu_{_{ j}}}\!(d_{0}(\theta_{j}))^{\nu_{1j} + \cdots + \nu_{rj}}
\end{equation*}
where ${\sum}^{*}$ indicates that we are summing over all tuples 
$(\theta_{j})_{j\in J_{1}}$ such that $\mathbb{F}(\theta_{j}) = \mathbb{F}_{\! \mu_{_{j}}}$ 
for every $j\in J_{1},$ and $\sigma(\theta_{j}) \ne \theta_{j'}$ for all 
$\sigma \in \mathrm{Gal}(\overline{\mathbb{F}}\slash \mathbb{F}),$ and all $j\ne j'.$ 
The normalizing constant is given by 
$$
C_{\mu, \nu} =\, \prod_{i = 1}^{m}\frac{1}{n_{i}!\, \mu_{j_{_{i}}}^{n_{i}}} 
$$ 
where $\mu_{j_{_{1}}}\!, \ldots, \mu_{j_{_{m}}}$ are the distinct components of 
$
(\mu_{j})_{j\in J_{1}},
$ 
and $n_{i},$ for $i = 1, \ldots, m,$ is the multiplicity of $\mu_{j_{_{i}}}$ in $J_{1},$ 
i.e., the cardinality of the set 
$
J_{1}^{(i)}  =\{j\in J_{1}  : \mu_{j} = \mu_{j_{_{i}}} \}.
$

Recalling that $\chi_{_{d_{0}}}\!(\pi) = \chi_{_{j}}(d_{0}(\theta))$ 
(see \cite{R}) for monic $d_{0}, \pi \in \mathbb{F}[x],$ with $d_{0}$ 
square-free and $\pi$ irreducible of degree $j,$ and 
$\theta \in \overline{\mathbb{F}}$ any root of $\pi,$ 
we see that the identity \eqref{eq: tag 3.4} can be presented 
in the equivalent form
\begin{equation} \label{eq: tag 3.7}
\lambda(\kappa, l; q) \, = \,\sum \lambda_{\mu}(\mathcal{N}; q) S_{\mu,\, \mathcal{N}}
\qquad \;\;\; ((\kappa, l):= (k_{1}, \ldots, k_{r}, l))
\end{equation} 
with the sum over all pairs $(\mu, \mathcal{N})$ satisfying 
\eqref{eq: tag 3.5} and the following simple condition: \!for all 
$j\in J_{0}$ and $j' \in J_{1}$ such that $\mu_{j} = \mu_{j'},$ we shall 
always assume that $j < j'.$ 

From the initial conditions (i), we compute 
\begin{equation} \label{eq: tag 3.8}
\lambda(k_{1}, \ldots, k_{r}, 0; q) \, =\, q^{k_{1} + \cdots + k_{r}}
\end{equation}
for all $k_{1}, \ldots, k_{r} \in \mathbb{N},$ and  
\begin{equation} \label{eq: tag 3.9}
\lambda(k_{1}, \ldots, k_{r}, 1; q)\; = \; 
 \begin{cases}
 q  & \text{if $k_{1} = \cdots  = k_{r} = 0$}\\ 
 0 & \text{otherwise}.  
\end{cases}
\end{equation}
Thus, it suffices to focus our investigation on the 
coefficients $\lambda(\kappa, l; q)$ with $l \ge 2.$

\vskip7pt 
The following proposition computes the part of the sum in \eqref{eq: tag 3.7} corresponding 
to coefficients $\lambda_{\mu}(\mathcal{N}; q)$ containing a factor 
$\lambda(\kappa', l'; 1\slash q)$ with $l' = l.$ 

\vskip10pt 
\begin{proposition} \label{Proposition 3.10} --- For 
$\kappa = (k_{1},\ldots, k_{r})\in \mathbb{N}^{r}$ and $l \ge 2,$ let 
$a(\kappa, l; q)$ denote the part of the sum in \eqref{eq: tag 3.7} over all pairs 
$(\mu, \mathcal{N})$ for which $\lambda_{\mu}(\mathcal{N}; q)$ does not contain 
a factor of the form $\lambda(\kappa', l; 1\slash q).$ Then, 
\begin{equation*}
\lambda(\kappa, l; q) \, -\, a(\kappa, l; q)
\;=\; 
 \begin{cases}
q^{|\kappa| + l + 1}
\lambda(\kappa, l; 1\slash q)    & \text{if $l$ is odd}\\
{\displaystyle \sum_{\kappa' \, \le \, \kappa}} \, 
q^{|\kappa| + l - \mathrm{r}(\kappa - \kappa') + 1}\, 
(q-1)^{\mathrm{r}(\kappa - \kappa')}\lambda(\kappa', l; 1\slash q) & \text{if $l$ is even}  
\end{cases}
\end{equation*}
where $|\kappa| = k_{1} + \cdots + k_{r},$ and $\mathrm{r}(\kappa - \kappa')$ 
is the number of non-zero components of $\kappa - \kappa'.$ 
\end{proposition}

\begin{proof} Let $\lambda_{\mu}(\mathcal{N}; q)$ be a coefficient 
containing a factor of the form $\lambda(\kappa', l; 1\slash q).$ Since 
$$
\lambda_{\mu}(\mathcal{N}; q)
\,=\, \prod_{j=1}^{n}\,
q^{\mu_{j}(\nu_{1j} + \cdots + \nu_{rj} + \nu_{j})}
\lambda(\nu_{1j}, \ldots, \nu_{rj},\, \nu_{j}; 1\slash q^{\mu_{j}})
 \;\; \text{with}\;\; 
\sum_{j = 1}^{n}\, \nu_{j}\mu_{j}  \,=\, l
$$ 
it follows that $\nu_{j_{_{0}}} \hskip-4pt = l$ and $\mu_{j_{_{0}}} \hskip-4pt = 1$ 
for some $1\le j_{_{0}}\le n,$ and $\nu_{j}  =  0$ for $j\ne j_{_{0}}.$ 
By \eqref{eq: tag 3.8}, we see that
$
\lambda_{\mu}(\mathcal{N}; q)
\,=\,
q^{|\kappa'| + l}\lambda(\kappa', l; 1\slash q).
$ 
The contribution of the part in \eqref{eq: tag 3.7} attached to 
this coefficient comes then from $(r+1)$-tuples of 
monic polynomials 
$$
((x - \theta)^{k_{1}'} m_{1}, \ldots, (x - \theta)^{k_{r}'} m_{r}, (x - \theta)^{l})   
\qquad \;\;\; (\text{with $\theta \in \mathbb{F}$ and $m_{i}(\theta)\ne 0$ 
for $i = 1, \ldots, r$})  
$$ 
and is obtained immediately from the coefficient 
$\kappa - \kappa' = (k_{1}^{} - k_{1}', \ldots, k_{r}^{} - k_{r}')$ of 
$$
\sum_{\theta \, \in \, \mathbb{F}} \, \sum_{m_{1}\!, \ldots, m_{r}}\frac{\chi_{x - \theta}(m_{1})\, \cdots \, \chi_{x - \theta}(m_{r})}{|m_{1}|^{s_{1}}\cdots \, |m_{r}|^{s_{r}}}  \, = \,  
\sum_{\theta \, \in \, \mathbb{F}} \; \prod_{i=1}^{r}L(s_{i}, \chi_{x - \theta}) = q \;\;\; (\text{for $l$ odd})
$$ 
and 
$$
q\cdot \prod_{i=1}^{r} \frac{1 - q^{-s_{i}}}{1 - q^{1-s_{i}}}
\;\;\; (\text{for $l$ even}).
$$
The coefficient 
$(k_{1}^{} - k_{1}', \ldots, k_{r}^{} - k_{r}')$ of the last product is 
$
q^{|\kappa| + l - \mathrm{r}(\kappa - \kappa') + 1} 
(q-1)^{\mathrm{r}(\kappa - \kappa')}.
$ 

The proposition now follows by summing over all $\kappa'  \le \kappa.$  
\end{proof}

\begin{remark}\label{rem2} \hskip-2ptOne can use a similar argument as in the proof of 
Proposition \ref{Proposition 3.10} to obtain the part of the sum in 
$
a(\kappa, l; q)
$ 
corresponding to an arbitrary coefficient 
$
\lambda_{\mu}(\mathcal{N}; q).
$ 
Indeed, recall that to a given pair $(\mu, \mathcal{N}),$ we associated a 
set of polynomials $\mathscr{P}_{\! \nu}(\mu),$ and we split
the set of indices $1\le j \le n$ into two parts $J_{0}$ and $J_{1}.$ 
By \eqref{eq: tag 3.6} and \eqref{eq: tag 3.8}, we have  
$$
\lambda_{\mu}(\mathcal{N}; q)\, = \prod_{j \, \notin \, J_{1}^{0}}\,
q^{\mu_{j}(\nu_{1j} + \cdots + \nu_{rj} + \nu_{j})}
\lambda(\nu_{1j}, \ldots, \nu_{rj},\, \nu_{j}; 1\slash q^{\mu_{j}})
$$ 
where $J_{1}^{0} := \{j \in J_{1}: \nu_{j} = 0 \}.$ 
Set 
$
\mu_{_{0}}  := (\mu_{j})_{j \in \{1, \ldots, n \} \backslash J_{1}^{0}}
$ 
and 
$
\mathcal{N}_{0}  : = (\nu_{1j}, \ldots, \nu_{rj},\, \nu_{j})_{j \in \{1, \ldots, n \} 
\backslash J_{1}^{0}},
$ 
i.e., $(\mu_{_{0}}, \mathcal{N}_{0})$ is the pair obtained by removing the components of $\mu$ and $\mathcal{N}$ corresponding to all elements in $J_{1}^{0}.$ For convenience, 
let us refer to $(\mu_{_{0}}, \mathcal{N}_{0})$ as a {\it primitive} pair,
and to $(\mu, \mathcal{N})$ as a pair {\it derived} from $(\mu_{_{0}}, \mathcal{N}_{0}).$ Since
$
\lambda_{\mu}(\mathcal{N}; q) = \lambda_{\mu_{_{0}}}\!(\mathcal{N}_{0}; q),
$ 
we can express the identity \eqref{eq: tag 3.7} as 
\begin{equation} \label{eq: tag 3.12}
\lambda(\kappa, l; q) \;\; = \sum_{(\mu_{_{0}},\, \mathcal{N}_{0})-\text{primitive}}
\lambda_{\mu_{_{0}}}\!(\mathcal{N}_{0}; q) 
\sum_{(\mu,\, \mathcal{N})-\text{derived from $(\mu_{_{0}},\, \mathcal{N}_{0})$}} S_{\mu,\, \mathcal{N}}.
\end{equation}
To achieve our goal, consider 
$
\mu_{_{0}}  = (\mu_{1},\ldots, \mu_{n_{_{0}}})
$ 
and 
$
\mathcal{N}_{0} = (\nu_{1j}, \ldots, \nu_{rj},\, \nu_{j})_{1 \le j \le n_{_{0}}}
$ 
giving a primitive pair $(\mu_{_{0}},\, \mathcal{N}_{0}).$ Let 
$\nu_{_{0}} \! = (\nu_{1}, \ldots, \nu_{n_{_{0}}}),$ and let 
$J_{0}$ and $J_{1}$ be, as before, the sets of indices 
$1\le j \le n_{_{0}}$ associated to $\nu_{_{ 0}}.$ If 
$\kappa' = (k_{1}', \ldots, k_{r}'),$ with 
$k_{i}' = \sum_{1\le j  \le  n_{_{0}}}  \nu_{ij}\mu_{j}$ for $i = 1, \ldots, r,$ 
we see that the inner sum in \eqref{eq: tag 3.12} (corresponding to 
$(\mu_{_{0}},\, \mathcal{N}_{0})$) is  
\begin{equation} \label{eq: tag 3.13}
C_{\mu_{_{0}},\,  \nu_{_{0}}}  \, \cdot 
\sum_{d_{0} \, \in \, \mathscr{P}_{\! \nu_{_{0}}}\!(\mu_{_{0}})} \;\;
\sideset{}{^*}\sum_{\substack{(\theta_{j})_{j\in J_{1}}}}
\left(\, \prod_{j \, \in \, J_{1}} 
\chi_{\mu_{_{j}}}\!(d_{0}(\theta_{j}))^{\nu_{1j} + \cdots + \nu_{rj}}\right)  
b_{d_{0}}^{(\theta_{j})}(\kappa - \kappa').
\end{equation} 
Here,
\begin{equation} \label{eq: tag 3.14}
b_{d_{0}}^{(\theta_{j})}(\kappa - \kappa') \,  = \,
\text{Coefficient}_{(k_{1}^{} - k_{1}', \ldots, k_{r}^{} - k_{r}')} \Bigg[ 
\prod_{i = 1}^{r}\; (1 - q t_{i})^{-1}  \prod_{j \, \in \, J_{1}} (1 -  t_{i}^{\mu_{j}}) \Bigg] 
\qquad (\text{if $d_{0} = 1$}).
\end{equation}
When $d_{0}$ is non-constant, 
\begin{equation} \label{eq: tag 3.15}
b_{d_{0}}^{(\theta_{j})}(\kappa - \kappa') \,  = \,
\text{Coefficient}_{(k_{1}^{} - k_{1}', \ldots, k_{r}^{} - k_{r}')} 
\Bigg[\prod_{i = 1}^{r}\; (1 - t_{i})^{\epsilon_{\scriptscriptstyle d_{0}}} P_{\scriptscriptstyle C_{d_{0}}}\!(t_{i}) \prod_{j \, \in \, J_{1}} 
(1 - \chi_{\mu_{_{ j}}}\!(d_{0}(\theta_{ j}))\, t_{i}^{\mu_{j}})
\Bigg]
\end{equation}
where $P_{\scriptscriptstyle C_{d_{0}}}\!(t)$ denotes, as in Section \ref{section 2}, the numerator of the zeta function $Z_{\scriptscriptstyle C_{d_{0}}}\!(t),$ and 
$\epsilon_{\scriptscriptstyle d_{0}} \! = 0$ or $1$ according as $\deg d_{0}$ 
is odd or even. As $\deg d_{0}\equiv l \pmod 2,$ the exponent 
$\epsilon_{\scriptscriptstyle d_{0}} \! = \epsilon_{\scriptscriptstyle l}$ depends, in fact, only on the parity of $l.$ 
\end{remark}

\vskip10pt
\begin{exmp}\label{exmpl=2} Let us take $l = 2.$ By Proposition \ref{Proposition 3.10}, we know that 
\begin{equation} \label{eq: tag 3.16}
\lambda(\kappa, 2; q)\;=\; a(\kappa, 2; q)\, + 
\sum_{\substack{\kappa' \, \le \, \kappa}}\, 
q^{|\kappa| - \mathrm{r}(\kappa - \kappa') + 3} 
(q-1)^{\mathrm{r}(\kappa - \kappa')}\lambda(\kappa', 2; 1\slash q).
\end{equation} 
Note that every coefficient $\lambda_{\mu}(\mathcal{N}; q)$ 
appearing in $a(\kappa, 2; q)$ is either $0$ or $1;$ 
the product \eqref{eq: tag 3.6} involves only coefficients 
$
\lambda(\nu_{1j}, \ldots, \nu_{rj},\, \nu_{j}; 1\slash q^{\mu_{j}})
$ 
with $\nu_{j} \in \{0, 1 \},$ from which, by \eqref{eq: tag 3.8} and \eqref{eq: tag 3.9}, one deduces that 
$\lambda_{\mu}(\mathcal{N}; q) \in \{0 , 1\}.$ Moreover, the above remark implies that 
$$ 
a(\kappa, 2; q)\;\;\;\, = \sum_{\substack{\deg d_{0} = 2  \\   d_{0}\; \text{monic and square-free}}}
b_{d_{0}}\!(\kappa)
$$ 
where $b_{d_{0}}\!(\kappa)$ is the coefficient 
$(k_{1},\ldots, k_{r})$ of the product 
$$
\prod_{i = 1}^{r}\, (1 - t_{i}) P_{\scriptscriptstyle C_{d_{0}}}\!(t_{i}) \, 
=\, \prod_{i = 1}^{r}\, (1 - t_{i}).
$$ 
Hence,
$$
a(\kappa, 2; q)
\;=\; 
 \begin{cases} 
(- 1)^{|\kappa|}\, q (q-1) 
    & \text{if $k_{i} = 0$ or $1$ for all $i = 1, \ldots, r$}\\ 
0 & \text{otherwise}.  
\end{cases}
$$ 
The identity \eqref{eq: tag 3.16} allows a recursive computation of the coefficients 
$\lambda(\kappa, 2; q).$ To see this, it is convenient to simplify 
notation by setting 
$
\lambda_{2}(r; q) := \lambda(1,\ldots,1, 2; q).
$ 
Since
$
\lambda(\kappa, 2; q) = \lambda_{2}(j; q)
$ 
when $\kappa = (k_{1}, \ldots, k_{r})$ with $k_{i} = 0$ or $1$ for all $i = 1, \ldots, r,$ 
and $|\kappa| = j,$ we deduce that
$$
\lambda_{2}(r;q) \, - \, q^{r+3} \lambda_{2}(r;1\slash q)\,=\, (-1)^{r} q(q-1) \, +\,
\sum_{j = 0}^{r-1}\, \binom{r}{j}\,q^{j+3} (q-1)^{r-j} \lambda_{2}(j;1\slash q).
$$ 
To put things in perspective, one should think of the quantity 
$(-1)^{r} q(q-1),$ representing $a(\kappa, 2; q)$ for 
$\kappa = (1, \ldots,1),$ as obtained from the simple identity 
\begin{equation} \label{eq: tag 3.17}
\sum_{\substack{\deg d_{0} = 2   \\   d_{0}\; \text{monic and square-free} }}  
\left(\sum_{\theta \, \in \, \mathbb{F}}\,  \chi(d_{0}(\theta))  \right)^{r} 
\, =\,\,  (-1)^{r} q(q-1) \qquad \;  (\text{for $r \in \mathbb{N}$}). 
\end{equation} 
For example, take $r = 6.$ By the method we are just describing, 
one finds inductively that: 
$$ 
\lambda_{2}(0; q) = q^2 \;\;\; \lambda_{2}(1; q) = 0 \;\;\;
\lambda_{2}(2; q) = q^3 \;\;\;  \lambda_{2}(3; q)  = q^4  \;\;\;
\lambda_{2}(4; q)  = 2 q^4 +\, q^5 \;\;\;
\lambda_{2}(5; q)  = 5 q^5 +\,  q^6.
$$ 
(An alternative way of finding these coefficients, but {\it only} 
when $r\le 3,$ comes by expressing the rational function 
$
Z(T, t_{4}; q)
$ 
into a power series.) Replacing $q$ by $1\slash q$ in these values of $\lambda_{2}(j; q)$ 
($0 \le j\le 5$), we obtain:  
$$
\lambda_{2}(6;q) \, - \, q^{9} \lambda_{2}(6;1\slash q)\,=\,  
q^7 +\,  9 q^6 + \,  5 q^5 - \, 5 q^4 - \,  9 q^3 - \, q^2.
$$ 
By condition (iii), we identify $\lambda_{2}(6; q)$ 
with the {\it dominant half} of the right-hand side, obtaining 
$$
\lambda_{2}(6; q)  =  q^7 +  9 q^6 +  5 q^5. 
$$ 
In Appendix \ref{C} we shall see that
$$
\lambda_{2}(r;q) \,=\, \sum_{j = 1}^{[\frac{r}{2}]}\, 
\frac{1}{(r - j + 1)(r - j)}\,
\frac{r!}{j!(j-1)!(r - 2j)!}\,q^{r + 2 -  j} \qquad \;
(\text{for $r \ge 1$}) 
$$
where $[x]$ denotes the integer part of a real number $x.$
\end{exmp} 

Presumably, all coefficients 
$
\lambda(\kappa, l; q)
$ 
can be determined by a similar procedure using 
an induction over $l$ and $r,$ 
should one be able to get some information about the moment-sums
$
M_{\mu, \gamma}(q)
$ 
defined in this section. More precisely, we need a different expression for 
$
M_{\mu, \gamma}(q)
$ 
playing the same role as \eqref{eq: tag 3.17} did 
in the above example. To gain more insight into the 
general problem, consider a coefficient 
$
\lambda(\kappa, l; q)
$ 
with $l \ge 3.$ 
For simplicity, let us assume that $l$ is odd. \hskip-1ptBy Proposition \ref{Proposition 3.10}, 
we have 
$$
\lambda(\kappa, l; q) \, -\, q^{|\kappa| + l + 1}
\lambda(\kappa, l; 1\slash q) \,  = \, a(\kappa, l; q).
$$ 
The right-hand side involves only coefficients $\lambda(\kappa, l'; q)$ with $l' < l$ 
(which are supposed to be known by the induction hypothesis), 
and hence, $a(\kappa, l; q)$ is an explicit combination of 
moments of character sums. \hskip-2ptMoreover, the 
left-hand side of the identity suggests that 
$
a(\kappa, l; q)
$ 
should satisfy the functional equation
$$
a(\kappa, l; q) \,= -\, q^{|\kappa| + l + 1}
a(\kappa, l; 1\slash q). 
$$ 
Accordingly, one needs to find an {\it alternative} expression of $a(\kappa, l; q),$ 
allowing to identify (by condition (iii)) the two parts corresponding to  
$
\lambda(\kappa, l; q)
$ 
and  
$
- q^{|\kappa| + l + 1}\lambda(\kappa, l; 1\slash q). 
$ 
This will be completely clarified in Sections \ref{section 7} and 
\ref{section: The series BarC(X, T, q)}.

\section{The generating series of $a(\kappa, l; q)$}\label{section 4}

For $l \in \mathbb{N}$ and algebraically independent variables $t_{1},\ldots, t_{r},$ consider the generating series 
$$
\Lambda_{l}(T, q) \, = \sum_{\kappa \, \in \, \mathbb{N}^{r}} \lambda(\kappa, l; q)\, T^{\kappa}\;\;\; \text{and}\;\;\; 
A_{l}(T, q) \,= \sum_{\kappa \, \in \, \mathbb{N}^{r}} a(\kappa, l; q) \, T^{\kappa} 
$$ 
where we set $T^{\kappa} := t_{1}^{k_{1}} \cdots \, t_{r}^{k_{r}}.$ By Proposition \ref{Proposition 3.10}, $\Lambda_{l}(T, q)$ should satisfy: 
\begin{equation} \label{eq: tag 4.1}
\Lambda_{l}(T, q) \, = \, A_{l}(T, q) \, 
+ \, q^{l + 1} \frac{E(q \, T)^{\epsilon_{\scriptscriptstyle l}}}
{E(T)^{\epsilon_{\scriptscriptstyle l}}}\Lambda_{l}(q \, T, 1\slash q) \qquad (\text{for $l\ge 2$}) 
\end{equation}
where 
$$
E(T)\,=\, \prod_{i = 1}^{r}\, (1 - t_{i})^{-1}
$$ 
and $\epsilon_{\scriptscriptstyle l} \! = 0$ or $1$ according as $l$ is odd or even. The product $E(T)$ should be interpreted as the gamma factor of $\Lambda_{l}(T, q)$ when $l$ is even. 
The identity \eqref{eq: tag 4.1} implies that $A_{l}(T, q)$ should satisfy the functional equation
\begin{equation}\label{eq: tag 4.1'}
E(T)^{\epsilon_{\scriptscriptstyle l}} A_{l}(T, q) \, 
=\, - \, q^{l + 1}E(q \, T)^{\epsilon_{\scriptscriptstyle l}} A_{l}(q \, T, 1\slash q).
\end{equation} 
Note that $\Lambda_{l}(T, q)$ is just what 
\begin{equation*} 
\sum_{\substack{\deg d = l \\ d-\text{monic}}} \, L(s_{1}, \chi_{d}) \, \cdots \,  
L(s_{r}, \chi_{d})
\end{equation*} 
(with $q^{-s_{i}}$ replaced by $t_{i}$ for $i = 1, \ldots, r$) should be.

To gain a better understanding of what these generating functions should be, fix 
a partition $\mu = (\mu_{1} \ge \cdots \ge \mu_{n} \ge 1)$ and an $n$-tuple 
of positive integers $\nu = (\nu_{1}, \ldots, \nu_{n})$ 
such that
$$
\sum_{j = 1}^{n} \nu_{j} \mu_{j}  = \, l.
$$ 
We shall assume throughout the section that $l \ge 2.$ Let $A_{\mu, \nu}(T, q)$ 
denote the generating series whose $\kappa$-th coefficient is 
the part of the sum in \eqref{eq: tag 3.12} corresponding to all primitive pairs 
$(\mu_{_{0}}, \mathcal{N}_{0})$ for which $\mu_{_{0}} = \mu$ and 
$\nu_{_{ 0}} = \nu.$ For instance, if $\mu = (1)$ and $\nu = (l)$ 
we have 
$$
A_{\mu, \nu}(T, q)\, 
=\, q^{l + 1} \frac{E(q \, T)^{\epsilon_{\scriptscriptstyle l}}}
{E(T)^{\epsilon_{\scriptscriptstyle l}}}\Lambda_{l}(q \, T, 1\slash q).
$$ 
Note that 
\begin{equation} \label{eq: tag 4.2}
A_{l}(T, q) \, + \, q^{l + 1} \frac{E(q \, T)^{\epsilon_{\scriptscriptstyle l}}}
{E(T)^{\epsilon_{\scriptscriptstyle l}}}\Lambda_{l}(q \,T, 1\slash q)\; 
= \sum_{(\mu, \nu)} A_{\mu, \nu}(T, q).  
\end{equation} 
This reduces the study of $A_{l}(T, q)$ to the study of the 
generating series $A_{\mu, \nu}(T, q),$ for arbitrary $\mu$ and $\nu.$

\vskip10pt 
The following special case is an immediate consequence of \eqref{eq: tag 3.12}, \eqref{eq: tag 3.13} and \eqref{eq: tag 3.14}. 

\vskip10pt 
\begin{proposition}\label{Proposition 4.3} --- Let $\mu_{j_{_{1}}}, \ldots, \mu_{j_{_{m}}}$ 
be the distinct components of $\mu,$ each $\mu_{j_{_{i}}}$ occurring with multiplicity $n_{i}.$ 
If all components $\nu_{j},$ $j = 1, \ldots, n,$ of $\nu$ are even, then 
\begin{equation*}
A_{\mu, \nu}(T, q) \, =\, q^{l}\prod_{i = 1}^{m} \frac{\big({\mathrm{Irr}}_{q}(\mu_{j_{_{i}}} \!)\big)!}{n_{i}!\, \big({\mathrm{Irr}}_{q}(\mu_{j_{_{i}}}\!) - n_{i} \big)!}
\cdot E(q \, T) \prod_{j = 1}^{n} 
\frac{\Lambda_{\nu_{j}}(q^{\mu_{j}}T^{\mu_{j}}\!, 1\slash q^{\mu_{j}})}
{E(T^{\mu_{j}})} 
\end{equation*}
where ${\mathrm{Irr}}_{q}\big(\mu_{j_{_{i}}} \!\big)$ denotes the number 
of irreducible polynomials of degree $\mu_{j_{_{i}}}$ over $\mathbb{F}.$
\end{proposition}

For general $\mu$ and $\nu,$ split, as before, the set of indices $1\le j \le n$ 
into two parts $J_{0}$ and $J_{1}$ according as $\nu_{j}$ is 
odd or even. For all $j\in J_{0}$ and $j' \in J_{1}$ such that 
$\mu_{j} = \mu_{j'},$ we shall assume that $j < j'.$ 
For $\varepsilon = (\varepsilon_{j})_{j\in J_{1}}$ with $\varepsilon_{j} \in \{1, 2\},$ let
\begin{equation*}
\Lambda_{\mu, \nu}^{\varepsilon}(T, q)  = 
q^{l} \!\prod_{j \, \in \, J_{0}} 
\Lambda_{\nu_{j}}(q^{\mu_{j}}T^{\mu_{j}}\!, 1\slash q^{\mu_{j}}) \; \cdot  
\prod_{j \, \in \, J_{1}} 
\frac{\Lambda_{\nu_{j}}(q^{\mu_{j}}T^{\mu_{j}}\!, 1\slash q^{\mu_{j}}) 
+ (-1)^{\varepsilon_{j}} \!\Lambda_{\nu_{j}}(-q^{\mu_{j}}T^{\mu_{j}}\!, 1\slash q^{\mu_{j}})}{2}
\end{equation*} 
and 
\begin{equation*}
U_{\mu, \nu}^{\varepsilon}(T, q)  = 
\frac{C_{\mu, \nu}}{E(T)^{^{\epsilon_{\scriptscriptstyle l}}}}\!\sum_{d_{0} \, \in \, \mathscr{P}_{\! \nu}(\mu)} \; \sideset{}{^*}\sum_{(\theta_{j})} \left(\prod_{j \, \in \, J_{1}} 
\chi_{\mu_{j}}(d_{0}(\theta_{j}))^{\varepsilon_{j}}  \cdot 
\prod_{i = 1}^{r} P_{\scriptscriptstyle C_{d_{0}}}\!(t_{i})\prod_{j \, \in \, J_{1}}(1 - \chi_{\mu_{j}}(d_{0}(\theta_{j}))\, t_{i}^{\mu_{j}}) \right).
\end{equation*} 
Here $C_{\mu, \nu}$ is the normalizing constant defined in the previous section. 
By \eqref{eq: tag 3.12}, \eqref{eq: tag 3.13} and \eqref{eq: tag 3.15}, it is not hard to see that 
the contribution to $A_{l}(T, q)$ corresponding to $(\mu, \nu)$ is  
$$
A_{\mu, \nu}(T, q)  \,= \, \sum_{\varepsilon}\, 
U_{\mu, \nu}^{\varepsilon}(T, q) \Lambda_{\mu, \nu}^{\varepsilon}(T, q)
$$  
the sum being taken over all $\varepsilon$ as above.

To allow more flexibility in these expressions, introduce a sum in 
$U_{\mu, \nu}^{\varepsilon}(T, q)$ 
over all tuples $\delta = (\delta_{j})_{j\in J_{1}}$ with $\delta_{j} \in \{-1, 1\},$ 
and sum over $d_{0}\in \mathscr{P}_{\! \nu}(\mu)$ and  
$(\theta_{j})_{j\in J_{1}}$ such that 
$\chi_{\mu_{j}}(d_{0}(\theta_{j})) = \delta_{j}.$ Note that for 
every $d_{0}\in \mathscr{P}_{\! \nu}(\mu)$ and $\delta,$ the number 
$N_{\mu, \nu}(d_{0}, \delta)$ of such tuples $(\theta_{j})_{j\in J_{1}}$ can be expressed as 
\begin{equation*}
N_{\mu, \nu}(d_{0}, \delta) \;\, = 
\sideset{}{^*}\sum_{\substack{\;\;(\theta_{j})_{_{_{\phantom{J}}}} \\ d_{0}(\theta_{j})\ne 0}}  
\left(\prod_{j \, \in \, J_{1}} 
\frac{\delta_{j}\chi_{\mu_{j}}(d_{0}(\theta_{j})) + 1}{2}\right)\, = \, 
\frac{1}{2^{^{| J_{1}|}}}\!\sum_{S\subseteq J_{1}} \delta_{S} b_{\mu}^{S}(d_{0})
\end{equation*} 
with
$$
\delta_{S} \, :=\,  \prod_{j \, \in \, S} \delta_{j} \;\;\; \text{and}\;\;\;\, 
b_{\mu}^{S}(d_{0}) \, :=  
\sideset{}{^*}\sum_{\substack{(\theta_{j})_{j\in J_{1}} \\  d_{0}(\theta_{j})\ne 0}}
\left(\prod_{j \, \in \, S} \chi_{\mu_{j}}(d_{0}(\theta_{j}))\right).
$$ 
Moreover, for every $\delta = (\delta_{j})_{j\in J_{1}},$ we have 
\begin{equation*}
\sum_{\varepsilon} \, \left(\prod_{j \, \in \, J_{1}} \delta_{j}^{\varepsilon_{j}}\right)
\Lambda_{\mu, \nu}^{\varepsilon}(T, q) \,=\, 
q^{l}\!\prod_{j \, \in \, J_{0}} 
\Lambda_{\nu_{j}}(q^{\mu_{j}}T^{\mu_{j}}\!, 1\slash q^{\mu_{j}}) \; \cdot  
\prod_{j \, \in \, J_{1}} \Lambda_{\nu_{j}}(\delta_{j} q^{\mu_{j}}T^{\mu_{j}}\!, 1\slash q^{\mu_{j}}). 
\end{equation*}
From these two observations it follows that 
\begin{equation} \label{eq: tag 4.4}
A_{\mu, \nu}(T, q)  \,= \, \frac{C_{\mu, \nu}}{E(T)^{^{\epsilon_{\scriptscriptstyle l}}}}
\sum_{\delta}\left(\sum_{d_{0} \, \in \, \mathscr{P}_{\! \nu}(\mu)}
N_{\mu, \nu}(d_{0}, \delta) \prod_{i = 1}^{r} P_{\scriptscriptstyle C_{d_{0}}}\!(t_{i})\right)
\cdot \Lambda_{\delta, \mu, \nu}(T, q) 
\end{equation}
with
\begin{equation} \label{eq: tag 4.5}
\Lambda_{\delta, \mu, \nu}(T, q)  :=  q^{l}\!\prod_{j \, \in \, J_{0}} \Lambda_{\nu_{j}}  (q^{\mu_{j}}T^{\mu_{j}}\!, 1\slash q^{\mu_{j}}) \; \cdot  
\prod_{j \, \in \, J_{1}} \frac{\Lambda_{\nu_{j}}(\delta_{j} q^{\mu_{j}}T^{\mu_{j}}\!, 1\slash q^{\mu_{j}})}{E(\delta_{j} T^{\mu_{j}})}. 
\end{equation} 
Our next objective is to obtain an explicit expression for the 
character sum $b_{\mu}^{S}(d_{0})$ in $N_{\mu, \nu}(d_{0}, \delta)$ in terms of the coefficients of 
the characteristic polynomial $P_{\scriptscriptstyle C_{d_{0}}}\!(t).$ For this purpose, 
let us first fix some notation. 

For every $\omega\ge 1,$ let $\mathbb{F}_{\! \omega}'$ denote 
the set of all elements in $\mathbb{F}_{\! \omega} = \mathbb{F}_{\! q^{\omega}}$ 
of degree $\omega$ over $\mathbb{F}.$ For $\mathrm{s}, \mathrm{t}\ge 0$ with 
$\mathrm{s} + \mathrm{t} \ge 1,$ and a monic square-free polynomial $d_{0}$ 
over $\mathbb{F},$
we set
\begin{equation*}
b_{\omega}^{(\mathrm{s}, \mathrm{t})}(d_{0})\; :=  
\sideset{}{^*}\sum_{\substack{(\theta_{1}\!, \ldots, \theta_{\mathrm{s} + \mathrm{t}})}}
\; \left(\prod_{j = 1}^{\mathrm{s}} 
\chi_{\omega}(d_{0}(\theta_{j})) \; \cdot \prod_{j = \mathrm{s} + 1}^{\mathrm{s} + \mathrm{t}} 
\chi_{\omega}(d_{0}(\theta_{j}))^{2} \right)
\end{equation*}
where, as before, the sum is over all tuples 
$(\theta_{j})_{\scriptstyle{j = 1}}^{\scriptstyle{\mathrm{s}+\mathrm{t}}}$ 
with $\theta_{j}\in \mathbb{F}_{\! \omega}'$ 
for every $1\le j\le \mathrm{s} + \mathrm{t},$ and $\sigma(\theta_{j}) \ne \theta_{j'}$ for all 
$\sigma \in \mathrm{Gal}(\overline{\mathbb{F}}\slash \mathbb{F}),$ and all $j\ne j'.$ Finally, let 
\begin{equation*}
a_{\omega}^{*}(d_{0}) \,=\, \sum_{\theta \, \in \, \mathbb{F}_{\! \omega}'} 
\chi_{\omega}(d_{0}(\theta)) \;\;\; \text{and}\;\;\;
c_{\omega}(d_{0}) \,=\, \sum_{\theta \, \in \, \mathbb{F}_{\! \omega}'} 
\chi_{\omega}(d_{0}(\theta))^{2}.
\end{equation*} 
Note that $c_{\omega}(d_{0})$ is constant on $\mathscr{P}_{\! \nu}(\mu).$

Now write $b_{\mu}^{S}(d_{0})$ in $N_{\mu, \nu}(d_{0}, \delta)$ as 
\begin{equation*}
b_{\mu}^{S}(d_{0}) \, =  
\sideset{}{^*}\sum_{(\theta_{j})_{j\in J_{1}}} 
\left(\prod_{j \, \in \, S} \chi_{\mu_{j}}(d_{0}(\theta_{j})) \, \cdot 
\prod_{j \, \in \, S'} \chi_{\mu_{j}}(d_{0}(\theta_{j}))^{2} \right)
\end{equation*} 
where $S'$ denotes the complement of the subset $S$ in $J_{1}.$ 
Let $\mu_{j_{_{1}}}, \ldots, \mu_{j_{_{m}}}$ be, as before, 
the distinct components of 
$
(\mu_{j})_{j\in J_{1}},
$ 
and for $i = 1, \ldots, m,$ set 
$
J_{1}^{(i)}  :\, =\{j\in J_{1}  : \mu_{j} = \mu_{j_{_{i}}} \}.
$ 
With this notation, we can further write 
\begin{equation*}
b_{\mu}^{S}(d_{0}) \; = \; 
\prod_{i = 1}^{m}\;\; \sideset{}{^*}\sum_{(\theta_{j})_{\!\! j\in J_{1}^{(i)}}} 
\left(\prod_{j \, \in \, S\cap J_{1}^{(i)}} \chi_{\mu_{j_{_{i}}}}\!(d_{0}(\theta_{j})) \;\; \cdot 
\prod_{j \, \in \, S'\cap J_{1}^{(i)}} \chi_{\mu_{j_{_{i}}}}\!(d_{0}(\theta_{j}))^{2} \right)\, = \; 
\prod_{i = 1}^{m}\,
b_{\mu_{j_{_{i}}}}^{(\mathrm{s}_{_{i}},\, \mathrm{t}_{_{i}})}(d_{0})
\end{equation*}  
where $\mathrm{s}_{_{i}}  \! = \big|S\cap J_{1}^{(i)} \big|$ and 
$\mathrm{t}_{_{i}} \! = \big|S'\cap J_{1}^{(i)} \big|$ 
(so that $\big|J_{1}^{(i)}\big| = \mathrm{s}_{_{i}}  + \mathrm{t}_{_{i}}$) 
for $i = 1, \ldots, m.$ Moreover, it is easy to see that 
\begin{equation*}
b_{\omega}^{(\mathrm{s}, \mathrm{t})}(d_{0}) \,=\, 
b_{\omega}^{(\mathrm{s}, 0)}(d_{0}) \cdot
\prod_{j = 0}^{\mathrm{t} - 1} (c_{\omega}(d_{0}) - \omega \mathrm{s} - \omega j) \qquad \;\;\;
(\text{if $\mathrm{t} \ge 1$})
\end{equation*} 
and hence, we can restrict ourselves to the character sums  
\begin{equation*}
b_{\omega}^{\mathrm{s}}(d_{0})\, =\, b_{\omega}^{(\mathrm{s}, 0)}(d_{0}) \; =  
\sideset{}{^*}\sum_{\substack{(\theta_{1}\!, \ldots, \theta_{\mathrm{s}})}}
\left(\prod_{j = 1}^{\mathrm{s}} 
\chi_{\omega}(d_{0}(\theta_{j})) \right).
\end{equation*} 
The following lemma expresses $b_{\omega}^{\mathrm{s}}(d_{0})$ 
in terms of $a_{\omega}^{*}(d_{0})$ and $c_{\omega}(d_{0}).$

\vskip10pt
\begin{lemma} --- For $\omega, \rm{s} \ge 1,$ and $d_{0}$
a monic square-free polynomial over $\mathbb{F},$ we have
\begin{equation*}
b_{\omega}^{\rm{s}}(d_{0})\, = \, 
\rm{s}!\sum \, \frac{(-\omega)^{\rm{s} - |\mathfrak{i}|}}{z(\mathfrak{i})} \cdot 
a_{\omega}^{*}(d_{0})^{|\mathfrak{i}_{\rm{odd}}|}
c_{\omega}^{}(d_{0})^{|\mathfrak{i}_{\rm{even}}|}
\end{equation*} 
where the sum is over all $\rm{s}$-tuples 
$\mathfrak{i} = (i_{1}, \ldots, i_{\rm{s}})$ 
of nonnegative integers with 
$i_{1} + 2i_{2} + \cdots + {\rm{s}} i_{{\rm{s}}} = {\rm{s}},$ 
$z(\mathfrak{i}) = i_{1}! 1^{i_{1}}\cdot i_{2}! 2^{i_{2}}\cdots \, i_{{\rm{s}}}! {\rm{s}}^{i_{{\rm{s}}}}\!,$ 
and $|\mathfrak{i}| = |\mathfrak{i}_{\rm{odd}}| + |\mathfrak{i}_{\rm{even}}|;$ 
here $|\mathfrak{i}_{\rm{odd}}| = i_{1} + i_{3} + \cdots$ {\rm{(}}respectively 
$|\mathfrak{i}_{\rm{even}}| =  i_{2} + i_{4} + \cdots${\rm{)}} is the sum of the odd 
{\rm{(}}respectively even{\rm{)}} components of $\mathfrak{i}.$ 
Equivalently, we have the identity:
\begin{equation*}
1\, +\, \sum_{\rm{s} \, \ge \, 1}\, \frac{b_{\omega}^{\rm{s}}(d_{0})}{\rm{s}!}X^{\rm{s}}\, =\, 
(1\, +\, \omega X)^{\frac{c_{\omega}^{}(d_{0}) \, + \, a_{\omega}^{*}(d_{0})}{2\omega}}  
(1\, - \, \omega X)^{\frac{c_{\omega}^{}(d_{0}) \, - \, a_{\omega}^{*}(d_{0})}{2\omega}}.  
\end{equation*}
\end{lemma}

\begin{proof} Consider the $j$-th elementary symmetric and power sum polynomials in $k$ variables, 
$$
e_{j}(x_{1}, \ldots, x_{k}) \;\;\;\;= \sum_{1\le i_{1} < \cdots < i_{j} \le k} x_{i_{1}}\cdots \,  x_{i_{j}} \;\;\, \text{and}\;\;\;\, 
p_{j}(x_{1}, \ldots, x_{k}) = x_{1}^{j} + \cdots + x_{k}^{j}
$$ 
respectively. For $\mathfrak{i} = (i_{1}, \ldots, i_{\mathrm{s}})$ an $\mathrm{s}$-tuple 
of nonnegative integers with 
$i_{1} + 2i_{2} + \cdots + \mathrm{s}i_{\mathrm{s}} = \mathrm{s},$ set 
$p^{(\mathfrak{i})} : = p_{1}^{i_{1}} p_{2}^{i_{2}} \cdots \, p_{\mathrm{s}}^{i_{\mathrm{s}}}.$ 
By \cite[Appendix A, \S A.1., Exercise $\text{A.32}^{*}$ (vi)]{FH}, we have the identity
\begin{equation} \label{eq: tag 4.7}
e_{\mathrm{s}} \,=\, \sum \frac{(-1)^{|\mathfrak{i}| - \mathrm{s}}}{z(\mathfrak{i})} \cdot p^{(\mathfrak{i})} 
\end{equation}
where the sum is over all $\mathfrak{i} = (i_{1}, \ldots, i_{\mathrm{s}})$ as above, 
$|\mathfrak{i}| = i_{1} + i_{2} + \cdots + i_{\mathrm{s}},$ and 
$z(\mathfrak{i})$ is the product in the statement of this lemma. 
This identity is encoded in the well-known identity of formal power series: 
\begin{equation} \label{eq: tag 4.8}
1\,+\, \sum_{\mathrm{s} = 1}^{k} \,
(-1)^{\mathrm{s}}e_{\mathrm{s}} t^{\mathrm{s}}\,=\, 
\exp  \Bigg( - \sum_{j \, \ge \, 1} \, p_{_{ j}} \frac{t^{j}}{j}  \Bigg).
\end{equation}

We apply \eqref{eq: tag 4.7} and \eqref{eq: tag 4.8} in our context as follows. Let 
$
\text{O}_{ \omega}^{\mathrm{Gal}} = \{\mathcal{O}_{1}, \ldots, \mathcal{O}_{k_{\omega}}\}
$ 
denote the set of Galois orbits in $\mathbb{F}_{\! \omega}'$, and set 
$
\chi_{\omega}(d_{0}(\mathcal{O})) := \chi_{\omega}(d_{0}(\theta))
$ 
for $\mathcal{O} \in \text{O}_{ \omega}^{\mathrm{Gal}}$ and 
$\theta \in \mathcal{O}.$ Clearly the value of 
$\chi_{\omega}(d_{0}(\mathcal{O}))$ is independent of the choice of 
$\theta \in \mathcal{O}.$ 
Since 
\begin{equation*}
\frac{1}{\mathrm{s}!} \frac{1}{\omega^{\mathrm{s}}} b_{\omega}^{\mathrm{s}}(d_{0})\;\;\;
= \sum_{1\le i_{1} < \cdots < i_{\mathrm{s}}\le k_{\omega}} 
 \chi_{\omega}(d_{0}(\mathcal{O}_{i_{1}}))\, \cdots \,  \chi_{\omega}(d_{0}(\mathcal{O}_{i_{\mathrm{s}}})) \, =\,
e_{\mathrm{s}}(\chi_{\omega}(d_{0}(\mathcal{O}_{1})), \ldots, \chi_{\omega}(d_{0}(\mathcal{O}_{k_{\omega}})))
\end{equation*}  
as it can be easily checked, the lemma follows by applying \eqref{eq: tag 4.7} and \eqref{eq: tag 4.8}. \end{proof}

Using this lemma, we can write 
\begin{equation} \label{eq: tag 4.9}
N_{\mu, \nu}(d_{0}, \delta) \, =\, 
\frac{1}{2^{^{| J_{1}|}}} 
\!\sum_{S\subseteq J_{1}} \delta_{S} \cdot 
\prod_{i = 1}^{m} b_{\mu_{j_{i}}}^{\mathrm{s}_{_{i}}}\!(d_{0}) \cdot 
\prod_{j = 0}^{\mathrm{t}_{i} - 1} (c_{\mu_{j_{i}}}\!(d_{0}) - \mu_{j_{i}}\mathrm{s}_{_{i}}  - \mu_{j_{i}} j)
\end{equation} 
with 
$$
\delta_{S} \, =\,  \prod_{j \, \in \, S} \delta_{j} \;\;\; \text{and} \;\;\;
b_{\mu_{j_{i}}}^{\mathrm{s}_{_{i}}}\!(d_{0})\, = \, 
(\mathrm{s}_{i})! \sum 
\frac{(-\mu_{j_{i}})^{\mathrm{s}_{_{i}}   -\,  |\mathfrak{i}|}}
{z(\mathfrak{i})} \cdot 
a_{\mu_{j_{i}}}^{*}\!(d_{0})^{|\mathfrak{i}_{\text{odd}}|}
c_{\mu_{j_{i}}}^{}\!(d_{0})^{|\mathfrak{i}_{\text{even}}|}.
$$ 
The sum in the expression of 
$b_{\mu_{j_{i}}}^{\mathrm{s}_{_{i}}}\!(d_{0})$ 
is over all $\mathrm{s}_{_{i}}$-tuples 
$\mathfrak{i} = (\iota_{_{1}}, \ldots, \iota_{\mathrm{s}_{_{i}}})$ 
of nonnegative integers with 
$\iota_{_{1}} + 2\iota_{_{2}} + \cdots + \mathrm{s}_{_{i}}  \iota_{\mathrm{s}_{_{i}}}  =\, \mathrm{s}_{_{i}}.$

The following proposition gives a simpler expression for $N_{ \mu, \nu}(d_{0}, \delta).$

\vskip10pt
\begin{proposition}\label{Proposition 4.9} --- For $i = 1,\ldots, m,$ let $n_{i}$ denote the cardinality of the set $J_{1}^{(i)}.$ If $\delta = (\delta_{j})_{j\in J_{1}}$ with $\delta_{j} \in \{-1, 1\}$ and $d_{0}\in \mathscr{P}_{\! \nu}(\mu),$ we have
\begin{equation*}
N_{\mu, \nu}(d_{0}, \delta) \, =\,  
\prod_{i = 1}^{m}
\frac{\Big(\frac{c_{\mu_{j_{i}}}^{}\!\!(d_{0}) \,+\, a_{\mu_{j_{i}}}^{*}\!\!(d_{0})}{2\mu_{j_{i}}}\Big)! \,
\Big(\frac{c_{\mu_{j_{i}}}^{}\!\!(d_{0}) \,-\, a_{\mu_{j_{i}}}^{*}\!\!(d_{0})}{2\mu_{j_{i}}} \Big)!\, \mu_{j_{i}}^{n_{i}}}
{\Big(\frac{c_{\mu_{j_{i}}}^{}\!\!(d_{0}) \,+\, a_{\mu_{j_{i}}}^{*}\!\!(d_{0})}{2\mu_{j_{i}}} \,- \, \mathrm{n}_{\delta^{(i)}}^{+} \Big)! \, 
\Big(\frac{c_{\mu_{j_{i}}}^{}\!\!(d_{0}) \,-\, a_{\mu_{j_{i}}}^{*}\!\!(d_{0})}{2\mu_{j_{i}}} \,-\, \mathrm{n}_{\delta^{(i)}}^{-} \Big)!}
\end{equation*} 
where $\delta^{(i)} = (\delta_{j})_{j\in J_{1}^{(i)}},$ and 
$\mathrm{n}_{ \delta^{(i)}}^{ \pm}  = (n_{i} \pm |\delta^{(i)}|)\slash 2.$   
\end{proposition}

\begin{proof} We observe from \eqref{eq: tag 4.9} that $N_{\mu, \nu}(d_{0}, \delta)$ factors as  
$$
N_{\mu, \nu}(d_{0}, \delta) \, =\,  
\prod_{i = 1}^{m} N_{ \mu, \nu}^{({\scriptstyle i})}(d_{0}, \delta^{(i)})
$$ 
where 
\begin{equation*}
N_{\mu, \nu}^{({\scriptstyle i})}(d_{0}, \delta^{(i)}) \, = \, 
\frac{1}{2^{n_{i}}}\!\sum_{\mathrm{s}_{_{i}}  \! =\, 0}^{n_{ i}}
b_{\mu_{j_{i}}}^{(\mathrm{s}_{_{i}},\, n_{_{ i}} \! - \mathrm{s}_{_{i}})}(d_{0}) \; \cdot 
\sum_{\substack{S_{i}\subseteq J_{1}^{(i)} \\  |S_{i}| \,=\, \mathrm{s}_{_{i}}}}
\delta_{S_{i}}. 
\end{equation*}
The inner sum is 
\begin{equation*}
e_{\mathrm{s}_{_{i}}}\!(\delta^{(i)})  = 
\text{Coefficient}_{t^{\mathrm{s}_{_{i}}}}\!(1 \,+\,  t)^{(n_{ i} + |\delta^{(i)}|)\slash 2} \,
(1 \,-\,  t)^{(n_{ i} - |\delta^{(i)}|)\slash 2}
\end{equation*} 
and hence, 
\begin{equation*} 
N_{\mu, \nu}^{({\scriptstyle i})}(d_{0}, \delta^{(i)}) \, =\, 
\frac{1}{2^{n_{ i}}}\!\sum_{\mathrm{s}_{_{i}} \! =\, 0}^{n_{ i}}
b_{\mu_{j_{i}}}^{(\mathrm{s}_{_{i}},\, n_{_{ i}} \! - \mathrm{s}_{_{i}})} (d_{0})  \cdot 
\sum_{h = 0}^{\mathrm{s}_{_{i}}}\, (-1)^{\mathrm{s}_{_{i}} \! - h}\binom{(n_{i} + |\delta^{(i)}|)\slash 2}{h}
\binom{(n_{i} - |\delta^{(i)}|)\slash 2}{\mathrm{s}_{_{i}} \!- h}.
\end{equation*} 
On the other hand, by the second identity in the lemma, we deduce easily that the 
generating polynomial $F_{\omega}(X, Y) = F_{\omega, a_{\omega}^{*}(d_{0}),
c_{\omega}^{}(d_{0})}(X, Y)$ of $b_{\omega}^{(\mathrm{s}, \, \mathrm{t})}(d_{0})\slash (\mathrm{s}! \mathrm{t}!)$ is 
\begin{equation*}
F_{\omega}(X, Y)\, = \, \sum_{\mathrm{s}, \mathrm{t} \, \ge \, 0} b_{\omega}^{(\mathrm{s}, \, \mathrm{t})}(d_{0}) \frac{X^{\mathrm{s}}}{\mathrm{s}!} \frac{Y^{\mathrm{t}}}{\mathrm{t}!} = 
(1  + \omega X  + \omega Y)^{\!\frac{c_{\omega}^{}(d_{0}) \, + \, a_{\omega}^{*}(d_{0})}{2\omega}}  
(1 - \omega X + \omega Y)^{\!\frac{c_{\omega}^{}(d_{0}) \, - \, a_{\omega}^{*}(d_{0})}{2\omega}}.  
\end{equation*}
One verifies that
\begin{equation*}
N_{\mu, \nu}^{({\scriptstyle i})}(d_{0}, \delta^{(i)}) = 
\text{Coefficient}_{X^{n_{\!\! \delta^{(i)}}^{+}}} \bigg[\frac{\big(n_{\delta^{(i)}}^{+}\big)!}{2^{n_{i}}} 
\bigg(\frac{\partial}{\partial Y}  -  \frac{\partial}{\partial X} \bigg)^{\! n_{ \delta^{(i)}}^{-}} 
\!\!\! F_{\mu_{j_{i}}}\!(X, Y) \bigg|_{X = Y}\bigg].
\end{equation*}
It is also easy to see that for a power series in two variables 
$
F(X, Y) = \sum_{\mathrm{s},  \mathrm{t} \ge 0} b_{\mathrm{s}, \mathrm{t}}
\frac{X^{\mathrm{s}}}{\mathrm{s}!} \frac{Y^{\mathrm{t}}}{\mathrm{t}!},
$ 
we have the formal series identity 
\begin{equation*}
\sum_{n \, \ge \, 0}\,  \bigg[\bigg(\frac{\partial}{\partial Y}  -  \frac{\partial}{\partial X}  \bigg)^{\! n} \! F(X, Y) \bigg|_{X = Y = V}\bigg] \frac{U^{n}}{n!}  = F(V - U, V + U).
\end{equation*} 
From these simple observations, it follows (as stated in the proposition) that for 
$d_{0}\in \mathscr{P}_{\! \nu}(\mu)$ and $\delta = (\delta_{j})_{j\in J_{1}}$ with $\delta_{j} \in \{-1, 1\},$ we have 
\begin{equation*}
N_{\mu, \nu}(d_{0},\delta) \, =\,  \prod_{i = 1}^{m} \frac{\Big(\frac{c_{\mu_{j_{i}}}^{}\!\!(d_{0}) \,+\, a_{\mu_{j_{i}}}^{*}\!\!(d_{0})}{2\mu_{j_{i}}}\Big)! \,
\Big(\frac{c_{\mu_{j_{i}}}^{}\!\!(d_{0}) \,-\, a_{\mu_{j_{i}}}^{*}\!\!(d_{0})}{2\mu_{j_{i}}} \Big)!\, \mu_{j_{i}}^{n_{i}}}{\Big(\frac{c_{\mu_{j_{i}}}^{}\!\!(d_{0}) \,+\, a_{\mu_{j_{i}}}^{*}\!\!(d_{0})}{2\mu_{j_{i}}} \,- \, n_{\delta^{(i)}}^{+} \Big)! \, 
\Big(\frac{c_{\mu_{j_{i}}}^{}\!\!(d_{0}) \,-\, a_{\mu_{j_{i}}}^{*}\!\!(d_{0})}{2\mu_{j_{i}}} \,-\, n_{\delta^{(i)}}^{-} \Big)!}.
\end{equation*}
\end{proof}

Now, for a hyperelliptic curve $C_{d_{0}}$ of genus $g$ corresponding to a 
monic square-free polynomial $d_{0}\in \mathbb{F}[x]$ of degree $2 g + 1$ or $2 g + 2,$
and a prime $\ell$ different from the characteristic $p$ of $\mathbb{F},$ we recall that
\begin{equation*}
P_{\scriptscriptstyle C_{d_{0}}}\!(t)  = \det(I -   F^{*} t \, \vert \, H_{\text{\'et}}^{1}(\bar{C}_{ d_{0}},\mathbb{Q}_{\ell}))
\end{equation*} 
where $\bar{C}_{d_{0}}$ is obtained from $C_{d_{0}}$ by extending the scalars from $\mathbb{F}$ to 
$\overline{\mathbb{F}},$ and $F^{*}$ is the endomorphism of the 
$\ell$-adic \'etale cohomology induced by the Frobenius morphism 
$F : \bar{C}_{d_{0}} \rightarrow \bar{C}_{d_{0}}.$ 
Let $\alpha_{1}(C_{d_{0}}), \ldots, \alpha_{2 g}(C_{d_{0}})$ 
denote the eigenvalues of $F^{*}\!,$ ordered in such a way that 
$\alpha_{k}(C_{d_{0}})\alpha_{k + g}(C_{d_{0}}) = q,$ 
and for every positive integer $j,$ put 
\begin{equation}\label{eq: tag-special}
a_{j}(C_{d_{0}}) :\, =\, - \, \frac{1 + (-1)^{\deg d_{0}}}{2}
\, - \, \sum_{\theta \, \in \, \mathbb{F}_{\! j}} 
\chi_{j}(d_{0}(\theta)) \, =\, \Tr(F^{*j} \vert \, H_{\text{\'et}}^{1}(\bar{C}_{d_{0}},\mathbb{Q}_{\ell})).
\end{equation}
Writing $P_{\scriptscriptstyle C_{d_{0}}}\!(t)$ as 
$$
P_{\scriptscriptstyle C_{d_{0}}}\!(t) = 1 + \sum_{k = 1}^{2g}
(-1)^{k}\lambda_{k}(C_{d_{0}})t^{k}
$$ 
it follows from \eqref{eq: tag 4.7} that 
$$
\lambda_{k}(C_{d_{0}}) = e_{k}(\alpha_{1}(C_{d_{0}}), \ldots, \alpha_{2g}(C_{d_{0}})) 
= \sum \frac{(-1)^{|\mathfrak{i}_{k}| - k}}{z(\mathfrak{i}_{k})} \cdot a^{(\mathfrak{i}_{k})}(C_{d_{0}}) 
$$ 
the sum being over all $k$-tuples $\mathfrak{i}_{k} = (i_{1}, \ldots, i_{k})$ of nonnegative integers such that $i_{1} + \cdots + ki_{k} = k.$ Here we set 
$a^{(\mathfrak{i}_{k})}(C_{d_{0}}) : = a_{1} (C_{d_{0}})^{i_{1}}\cdots \,
a_{k} (C_{d_{0}})^{i_{k}}$ (with the understanding that $a_{j}(C_{d_{0}})^{i_{j}} = 1$ if $i_{j} = 0$ for some $j$).

One can easily express $a_{\omega}^{*}(d_{0})$ ($\omega \ge 1$) in terms of the classical 
character sums \eqref{eq: tag-special} and the elementary quantities 
$
c_{\omega \slash  2^{u}}(d_{0}),
$ 
with $u = 1, 2, \ldots,$ as follows. We can write 
$$
a_{k}(C_{d_{0}})  = - \, \frac{1 + (-1)^{\deg d_{0}}}{2}\, 
- \, \sum_{\omega | k}\, \sum_{\theta \, \in \, \mathbb{F}_{\! \omega}'} 
\chi_{k}(d_{0}(\theta)) \qquad (\text{for $k \ge 1$}).
$$ 
Recalling that for $\theta \in \mathbb{F}_{\! \omega},$ we have 
$\chi_{k}(d_{0}(\theta)) = \chi_{\omega}(d_{0}(\theta))$ 
or $\chi_{\omega}(d_{0}(\theta))^{2}$ according as $k\slash \omega$ is odd or even, 
we can rewrite the above equality as  
$$
- \, \frac{1 + (-1)^{\deg d_{0}}}{2} -  a_{k}(C_{d_{0}}) \;\;
= \sum_{k\slash \omega-\text{odd}} a_{\omega}^{*}(d_{0}) \; +
\sum_{k\slash \omega-\text{even}} c_{\omega}^{}(d_{0}).
$$ 
By setting $\rho_{\omega} = 1$ or $0$ according as $\omega$ is a power of two or not,  
and $\epsilon_{d_{0}} : = (1 + (-1)^{\deg d_{0}})\slash 2,$
it follows using the M\"obius inversion that
$$
a_{\omega}^{*}(d_{0}) = - \, \rho_{\omega}\epsilon_{d_{0}} \;\,  -
\sum_{\substack{k \mid \omega \\ \omega\slash k-\text{odd}}} 
\mu(\omega\slash k) a_{k}(C_{d_{0}}) \, -  
\sum_{u = 1}^{a}\, c_{\omega \slash  2^{u}}(d_{0}) 
\qquad (\text{for $2^{a} \parallel  \omega$}).
$$ 
Note that we can express $P_{\scriptscriptstyle C_{d_{0}}}\!(t)$ as 
\begin{equation}\label{eq: pol-char-in-terms-of-a*}
P_{\scriptscriptstyle C_{ d_{0}}}\!(t)  = (1 - t)^{- \epsilon_{d_{0}}} \cdot 
\prod_{m = 1}^{\infty} (1 -  t^{m})^{-\frac{c_{m}^{}(d_{0}) \,+\, a_{m}^{*}(d_{0})}{2m}}  
(1 + t^{m})^{-\frac{c_{m}^{}(d_{0})\,-\, a_{m}^{*}(d_{0})}{2m}}.
\end{equation}

\section{The coefficients $\lambda(\kappa, l; q)$ for $l = 3, 4$}\label{section 5} 
To determine the coefficients $\lambda(\kappa, 3; q),$ we follow the 
strategy outlined at the end of Section \ref{section 3}. \!By Proposition \ref{Proposition 3.10}, 
$$
\lambda(\kappa, 3; q) \, -\, q^{|\kappa| + 4}
\lambda(\kappa, 3; 1\slash q)   \, =  \, a(\kappa, 3; q)
$$ 
with $a(\kappa, 3; q)$ computed using {\it Remark} 2 in Section \ref{section 3}. 
Explicitly, 
\begin{equation*}  
a(\kappa, 3; q)\;\; = \sum_{\substack{\deg  d_{0} = 3 \\  d_{0}-\text{monic \& square-free}}}
\hskip-6pt b_{d_{0}}(\kappa)\;   +  \,
\sum_{\kappa' \, \le \, \kappa}  \left(\, \sum_{\substack{\deg  d_{0}  =  1 \\ d_{0}-\text{monic}}}\,  \sum_{\theta \, \in \, \mathbb{F}} 
\chi(d_{0}(\theta))^{|\kappa'|} \, b_{d_{0}}^{(\theta)}(\kappa - \kappa') \right) 
\cdot q^{|\kappa'| + 2}\lambda(\kappa', 2; 1\slash q) 
\end{equation*} 
where $b_{d_{0}}(\kappa)$ is the coefficient $(k_{1},\ldots,k_{r})$ of  
$$
\prod_{i = 1}^{r} P_{\scriptscriptstyle C_{d_{0}}}\!(t_{i})\,
=\, \prod_{i = 1}^{r} (1  -  a(C_{ d_{0}}) t_{i}  +  q t_{i}^2) 
$$ 
and 
$$
b_{d_{0}}^{(\theta)}(\kappa - \kappa')
\;=\; 
 \begin{cases} 
(- \chi(d_{0}(\theta)))^{|\kappa - \kappa'|}
    & \text{if $k_{i}^{} - k_{i}' = 0$ or $1$ for all $i = 1, \ldots, r$}\\ 
0 & \text{otherwise}.  
\end{cases}
$$ 
For every $\kappa' = (k_{1}',\ldots, k_{r}')\in \mathbb{N}^{r}$ whose 
components satisfy $k_{i}^{} - k_{i}' = 0$ or $1$ for all $i = 1, \ldots, r,$ 
we have 
$$  
\sum_{\substack{\deg d_{0} = 1 \\ d_{0}-\text{monic}}} \, \sum_{\theta \, \in \, \mathbb{F}} 
\chi(d_{0}(\theta))^{|\kappa'|} \,  b_{d_{0}}^{(\theta)}(\kappa - \kappa') 
\;=\; 
 \begin{cases} 
(- 1)^{|\kappa - \kappa'|}\, q (q-1) 
    & \text{if $|\kappa|$ is even}\\ 
0 & \text{if $|\kappa|$ is odd}.  
\end{cases}
$$ 
It follows that the triple sum in the expression of $a(\kappa, 3; q)$ is 
$$  
q(q - 1)\;\;\;\, \cdot  \sum_{\substack{\;\;\; \kappa' \, \le \, \kappa_{_{\phantom{X}}} \\  k_{i}^{} - k_{i}' = 0, 1}} (- 1)^{|\kappa - \kappa'|}\, q^{|\kappa'| + 2} 
\lambda(\kappa', 2; 1\slash q) 
$$ 
or zero according as $|\kappa|$ is even or odd.

To deal with the remaining part of $a(\kappa, 3; q),$ let  
\begin{equation*}  
\mathcal{M}_{3}(r; q)\;\;\; := \sum_{\substack{\deg d_{0} = 3 \\ d_{0}-\text{monic \& square-free}}}
\hskip-6pt a(C_{d_{0}})^{r} 
\;\;\;= \sum_{\substack{\deg d_{0} = 3 \\ d_{0}-\text{monic \& square-free}}}
\left(- \sum_{\theta \, \in \, \mathbb{F}}\, \chi(d_{0}(\theta))\right)^{\! r}.
\end{equation*}
Notice that a simple substitution $\theta \to \theta_{0}\theta,$ with $\chi(\theta_{0}) = -1,$ 
implies immediately that $\mathcal{M}_{3}(r; q) = 0$ for $r$ odd. \!We normalize 
$\mathcal{M}_{3}(r; q)$ by setting 
$
\mathcal{M}_{3}^{*}(r; q) := (q(q - 1))^{-1}\mathcal{M}_{3}(r; q).
$ 
Using the Eichler-Selberg trace formula \cite{Sel} 
(see also \cite[Appendix by D. Zagier, pp. 44--54]{Lang}), Birch \cite{Bi} and Ihara \cite{Ih} proved\footnote{Strictly speaking, in \cite{Bi} this is proved, for simplicity, in the case 
when $q$ is an odd prime, and it is implicit, for an arbitrary finite field of odd characteristic, 
in \cite{Ih}.} independently the following beautiful theorem:

\vskip10pt
\newtheorem*{Birch-Ihara}{Theorem (Birch-Ihara)}
\begin{Birch-Ihara} --- {\it If $\mathbb{F} = \mathbb{F}_{\! q}$ is a finite field of odd characteristic $p,$ then 
\begin{equation*}
\sum_{r = 0}^{k - 1}\, \binom{k + r - 1}{2r} (-q)^{k - 1 - r} \mathcal{M}_{3}^{*}(2r; q) \,
=  - \, T_{2k}(q) \, - \,  1 \quad  (\text{for $k \ge 1$})
\end{equation*} 
with $T_{2k}(q) = \Tr(T_{q} \, \vert \, S_{2k})$ 
if $q = p,$ and $T_{2k}(q) = \Tr(T_{ q} \, \vert \, S_{2k}) 
- p^{2k - 1} \Tr(T_{ q\slash p^{2}} \vert\, S_{2k})$ if $q \ne p,$ 
where $\Tr(T_{n} \, \vert \, S_{2k})$ is the trace of the Hecke operator
$T_{n}$ acting on the space of elliptic cusp forms of weight $2k$ on the full 
modular group. Equivalently, for $r\ge 0$ we have the identity: 
\begin{equation*}
\mathcal{M}_{3}^{*}(2r; q) \,=\,  \frac{(2r)!}{r!(r + 1)!}q^{r + 1} - 
\, \sum_{k=1}^{r} (2k + 1)\frac{(2r)! \, q^{r - k}}{(r - k)!(r + k + 1)!} 
(T_{2k + 2}(q)  + 1).
\end{equation*}}
\end{Birch-Ihara}

The above result is the main ingredient in the determination of the coefficients 
$\lambda(\kappa, 3; q)$ which is given in the following 

\vskip10pt
\begin{thm}\label{Theorem 5.2} --- For $\kappa = (k_{1},\ldots, k_{r})\in \mathbb{N}^{r},$ 
let $r_{1} = r_{1}(\kappa)$ and $r_{2}  = r_{2}(\kappa)$ denote the number of components 
of $\kappa$ equal to $1$ and $2,$ respectively. If 
$
r_{1} + 2 r_{2} =|\kappa|,
$ 
and 
$r_{1} = 2R$ is even, then 
\begin{equation*}
\lambda(\kappa, 3; q)\,  = \,
\frac{(2R)!}{R!(R + 1)!}q^{R + r_{2} + 3}\; - 
\;\, \sum_{j=1}^{R} \, (2j + 1)\, \frac{(2R)! \,q^{R + r_{2} - j + 2}}{(R - j)!(R + j + 1)!}\, 
T_{2j + 2}(q).
\end{equation*} 
Otherwise, the coefficients $\lambda(\kappa, 3; q)$ all vanish.
\end{thm}

\begin{proof} From the above considerations and the discussion in 
Appendix \ref{C}, it is easy to see that $a(\kappa, 3; q) = 0,$ 
unless $r_{1}$ is even, and $r_{1} + 2 r_{2} = |\kappa|.$ 
Moreover,
\begin{equation*}
(q(q - 1))^{-1} a(\kappa, 3; q) \, =\, 
q^{r_{2}}\mathcal{M}_{3}^{*}(r_{1}; q)\;\;\,   + \sum_{\substack{\;\;\; \kappa' \, \le \, \kappa_{_{\phantom{X}}} \\ k_{i}^{} - k_{i}' = 0, 1}}  
(- 1)^{|\kappa - \kappa'|}\, q^{|\kappa'| + 2} \lambda(\kappa', 2; 1\slash q) 
\end{equation*} 
in the remaining case. The sum in the right is 
\begin{equation*}
\sum_{j = 0}^{R} (2j + 1)\frac{(2R)! \,q^{R + r_{2} - j}}{(R - j)!(R + j + 1)!}  
\end{equation*} 
see Appendix \ref{C}, which combined with Birch-Ihara identity gives     
\begin{equation*}
\begin{split}
a(\kappa, 3; q) \, =\, & \frac{(2R)!}{R!(R + 1)!}q^{R + r_{2} + 3}\, - 
\, \sum_{j=1}^{R} (2j + 1)\frac{(2R)! \,q^{R + r_{2} - j + 2}}{(R - j)!(R + j + 1)!}\, 
T_{2j + 2}(q)\\
& - \frac{(2R)!}{R!(R + 1)!}q^{R + r_{2} + 1}\, + 
\, \sum_{j=1}^{R} (2j + 1)\frac{(2R)! \,q^{R + r_{2} - j + 1}}{(R - j)!(R + j + 1)!}\, 
T_{2j + 2}(q).
\end{split}
\end{equation*}
(Notice the cancellation that occurs in the process of obtaining 
the last expression of $a(\kappa, 3; q).$) Letting 
$$
T_{2j + 2}(1\slash q): = q^{-2j - 1}T_{2j + 2}(q)
$$  
we can present the second sum as 
$$
\sum_{j=1}^{R} (2j + 1)\frac{(2R)! \,q^{R + r_{2} + j + 2}}{(R - j)!(R + j + 1)!}\, 
T_{2j + 2}(1\slash q).
$$ 
To finish the proof, we notice that the coefficient $\lambda(\kappa, 3; q)$ in the statement of the theorem is just the dominant half of $a(\kappa, 3; q).$
\end{proof}

{\it The case $l = 4.$} To find the coefficients $\lambda(\kappa, 4; q),$ we apply Proposition \ref{Proposition 3.10} with $l = 4,$ and thus 
\begin{equation*}
\lambda(\kappa, 4; q) = a(\kappa, 4; q)\, + 
\sum_{\kappa'  \le \, \kappa} 
q^{|\kappa| - \mathrm{r}(\kappa - \kappa') + 5} 
(q-1)^{\mathrm{r}(\kappa - \kappa')}\lambda(\kappa', 4; 1\slash q).  
\end{equation*} 
The right-hand side of the identity consists of distinct 
contributions corresponding to pairs 
$\mathrm{x} = (x_{1} \ge \cdots \ge x_{m})$ and 
$\mathrm{y} = (y_{1}, \ldots,  y_{m}),$ 
with $x_{i},  y_{i}$ positive integers for all 
$i = 1, \ldots, m,$ 
such that 
\begin{equation} \label{eq: tag 5.3}
\sum_{i = 1}^{m} y_{i} x_{i}  = 4.
\end{equation}
(Notice that the contribution corresponding to $4\cdot 1$ 
is precisely the sum in the right-hand side of the identity.)

To compute $a(\kappa, 4; q),$ recall the notation introduced in Section \ref{section 3}. 
By \eqref{eq: tag 3.6}, \eqref{eq: tag 3.8} and \eqref{eq: tag 3.9}, for every partition
$
\mu = (\mu_{1}\ge \cdots \ge \mu_{n} \ge 1)
$ 
and 
$
\mathcal{N} = (\nu_{1j}, \ldots, \nu_{rj},\, \nu_{j})_{1\le j\le n},
$ 
subject to $\nu_{j} = 0$ or $1$ for all 
$j = 1, \ldots, n,$ and  
$$
\sum_{j = 1}^{n} \nu_{ij}\mu_{j} =  k_{i}, \;\;\;\;\;\;
\sum_{j \, \notin \, J_{1}^{0}}  \mu_{j}  = 4 \qquad (\text{for $i = 1, \ldots, r$})
$$ 
we have 
\begin{equation*}
\lambda_{\mu}(\mathcal{N}; q)\, = \prod_{j \, \notin \, J_{1}^{0}}
q^{\mu_{j}(\nu_{1j} + \cdots + \nu_{rj} + 1)}
\lambda(\nu_{1j}, \ldots, \nu_{rj}, 1; 1\slash q^{\mu_{j}}) \, =\,  1
\end{equation*} 
if $\nu_{1j} = \cdots = \nu_{rj} = 0$ for all $j\notin J_{1}^{0},$ and zero otherwise. \!By Remark \ref{rem2} in Section \ref{section 3}, we see that for every fixed $\kappa = (k_{1}, \ldots, k_{r})\in \mathbb{N}^{r},$ the total contribution to $a(\kappa, 4; q)$ 
corresponding to all such pairs $(\mu, \mathcal{N})$ is 
$$
\sum_{\substack{\deg d_{0} = 4 \\ d_{0}-\text{monic \& square-free}}} 
\hskip-6pt b_{d_{0}}(\kappa)
$$ 
where $b_{d_{0}}(\kappa)$ is the coefficient $(k_{1},\ldots, k_{r})$ of  
\begin{equation*}
\prod_{i = 1}^{r} (1 - t_{i})(1  -  a(C_{d_{0}}) t_{i}  +  q t_{i}^2)\;\;\; \text{with}\;\;\; 
a(C_{d_{0}}) = - \, 1  \,- \, \sum_{\theta \, \in \, \mathbb{F}} \chi(d_{0}(\theta)).
\end{equation*}
Note that this contribution corresponds to the five partitions $\mathrm{x} = (x_{1} \ge \cdots \ge x_{m}), 1 \le m\le 4,$ of $4.$ (For every partition $\mathrm{x},$ we 
choose $y_{1} =  \cdots = y_{m} = 1$ in \eqref{eq: tag 5.3}.) In what follows, we shall refer to the expression 
\begin{equation} \label{eq: tag 5.4}
\sum_{\substack{\deg d_{0} = 4 \\ d_{0}-\text{monic \& square-free}}} 
\left(\, \prod_{i = 1}^{r} (1 - t_{i})(1  -  a(C_{ d_{0}}) t_{i}  +  q t_{i}^2) \right)
\end{equation} 
as the {\it non-degenerate} part of $A_{4}(T, q).$

There are four additional (degenerate) contributions to $a(\kappa, 4; q)$ which can be computed similarly using Remark \ref{rem2} in Section \ref{section 3}. They are: 
\begin{equation*}  
\frac{1}{2}(q - 1) \;\, \cdot  \sum_{2\kappa' \, \le \, \kappa} 
q^{|\kappa| + 5} (1 - q^{-2})^{r(\kappa - 2\kappa') - r_{1}(\kappa - 2\kappa')}
\lambda(\kappa', 2; 1\slash q^{2})
\end{equation*}
with generating series
\begin{equation} \label{eq: tag 5.5}
\frac{q^{5}(q - 1)}{2} \frac{E(q\, T)\Lambda_{2}(q^2 T^2\!, 1\slash q^2)}{E(T)E(-T)}
\end{equation} 
corresponding to $2\cdot 2$ in \eqref{eq: tag 5.3},
\begin{equation*}  
\frac{1}{2} q^{|\kappa| + 5} (q - 1)\; \cdot  
\hskip-6pt \sum_{\substack{\kappa'\!,\, \kappa'' \\ \kappa' + \kappa'' \le \kappa}} 
\left(\frac{q - 1}{q} \right)^{\! 2r(\kappa - \kappa' - \kappa'')}
\left(\frac{q(q - 2)}{(q - 1)^{2}}\right)^{\! r_{1}(\kappa - \kappa' - \kappa'')}
 \lambda(\kappa', 2; 1\slash q)
\lambda(\kappa'', 2; 1\slash q)
\end{equation*} 
with generating series 
\begin{equation} \label{eq: tag 5.6}
\frac{q^{5}(q - 1)}{2} \frac{E(q\, T)\Lambda_{2}(q \, T, 1\slash q)^{2}}{E(T)^{2}}
\end{equation} 
corresponding to $2\cdot 1 \, +\,  2\cdot 1,$
\begin{equation*} 
\begin{split}
& \sum_{\substack{\;\;\; \kappa' \, \le \, \kappa_{_{\phantom{X}}}   \\  k_{i}^{} - k_{i}' \, \le \, 2 \\  r_{1}\!(\kappa - \kappa') = 0}} 
\begin{cases}
 q(q - 1)^{2}  & \text{if $|\kappa'|  + r(\kappa - \kappa')$ even}\\ 
 -q(q  - 1) & \text{if $|\kappa'|  + r(\kappa - \kappa')$ odd}  
\end{cases}  \cdot \; q^{|\kappa'| + 2}\,
\lambda(\kappa', 2; 1\slash q)\\
&+ \;
\frac{q(q - 1)(q - 2)}{2}\;\;\;\, \cdot  
\sum_{\substack{\;\;\; \kappa' \, \le \, \kappa_{_{\phantom{X}}}  \\ k_{i}^{} - k_{i}' \, \le \, 2 \\  r_{1}\!(\kappa - \kappa') \ne 0}} 
(-2)^{r_{1}\!(\kappa - \kappa')}\, q^{|\kappa'| + 2}\lambda(\kappa', 2; 1\slash q)
\end{split}
\end{equation*} 
with generating series 
\begin{equation} \label{eq: tag 5.7}
\frac{q^{3}(q - 1)(q - 2)}{2} \frac{\Lambda_{2}(q \, T, 1\slash q)}{E(T)^{2}}\,+\,
\frac{q^{4}(q - 1)}{2}\frac{\Lambda_{2}(- q \, T, 1\slash q)}{E(T)E(-  T)}
\end{equation} 
corresponding to $1\cdot 2 + 2\cdot 1$ or $2\cdot 1 + 1\cdot 1 + 1\cdot 1,$ 
and finally 
\begin{equation*}  
(q - 1)\;\;\;\, \cdot  \sum_{\substack{\;\;\; \kappa' \, \le \, \kappa_{_{\phantom{X}}} \\ k_{i}^{} - k_{i}' = 0, 1}} (- 1)^{|\kappa - \kappa'|}\, q^{|\kappa'| + 4} \lambda(\kappa', 3; 1\slash q) 
\end{equation*} 
with generating series
\begin{equation} \label{eq: tag 5.8}
q^{4}(q - 1)\frac{\Lambda_{3}(q \, T, 1\slash q)}{E(T)} 
\end{equation}
corresponding to $3\cdot 1  +  1\cdot 1.$ 
Notice that in the above degenerate contributions we have evaluated the elementary character sums 
that occur.

Now set  
\begin{equation*} 
\mathcal{M}_{4}(r; q)\;\;\; := \hskip-3pt \sum_{\substack{\deg d_{0} = 4 \\ d_{0}-\text{monic \& square-free}}}\hskip-6pt a(C_{d_{0}})^{r} 
\;\;\;= \sum_{\substack{\deg d_{0} = 4 \\ d_{0}-\text{monic \& square-free}}}
\left(-\, 1\, - \, \sum_{\theta \, \in \, \mathbb{F}} \chi(d_{0}(\theta)) \right)^{\! r}\!.
\end{equation*} 
We have: 
\begin{equation*}
\mathcal{M}_{4}(r; q)
\;=\; 
 \begin{cases} 
 -\, \mathcal{M}_{3}(r + 1; q)   & \text{if $r$ is odd}\\ 
q \mathcal{M}_{3}(r; q) & \text{if $r$ is even}  
\end{cases}
\end{equation*} 
see Appendix \ref{D} for details. Accordingly, we can express the non-degenerate contribution \eqref{eq: tag 5.4} as 
\begin{equation*} 
\frac{1}{E(T)}\sum_{\substack{\kappa \\ r_{1}(\kappa) + 2r_{2}(\kappa) = |\kappa|}} 
\mathcal{M}_{3}(r_{1}(\kappa); q)q^{r_{2}(\kappa) + 1}T^{\kappa} \; +\;
\frac{1}{E(T)}\sum_{\substack{\kappa \\ r_{1}(\kappa) + 2r_{2}(\kappa) = |\kappa|}} 
\mathcal{M}_{3}(r_{1}(\kappa) + 1; q)q^{r_{2}(\kappa)}T^{\kappa} 
\end{equation*} 
where, as before, $r_{i}(\kappa) =   \#\{j : k_{j} = i\}$ for $i = 1, 2.$ 
With this last piece of information, we are now in the position to justify the functional equation satisfied by $A_{4}(T, q).$ 

\vskip10pt
\begin{proposition}\label{functional equation A4} --- The generating series $A_{4}(T, q)$ satisfies the 
functional equation
\begin{equation*}
E(T)A_{4}(T, q) \,=\, -\, q^{5}E(q \, T)A_{4}(q \, T, 1\slash q).
\end{equation*}
\end{proposition}

\begin{proof} The idea of the proof is to identify the parts of $E(T)A_{4}(T, q)$ satisfying the functional equation in the statement of the proposition. Indeed, from the $l = 3$ case we know that 
\begin{equation*}
\Lambda_{3}(T, q) - q^{4}\Lambda_{3}(q \, T, 1\slash q) \, =\, A_{3}^{(0)}(T, q) +  
\frac{q^{3}(q - 1)}{2}\Big(\frac{\Lambda_{2}(q\, T, 1\slash q)}{E(T)} + 
\frac{\Lambda_{2}(- q\, T, 1\slash q)}{E(- T)}\Big)
\end{equation*} 
where 
\begin{equation*}
A_{3}^{(0)}(T, q) \;\;\;\;\;  := 
\sum_{\substack{\kappa \\ r_{1}(\kappa) + 2r_{2}(\kappa) = |\kappa|}} 
\mathcal{M}_{3}(r_{1}(\kappa); q)q^{r_{2}(\kappa)}T^{\kappa}
\end{equation*} 
is the non-degenerate part of $A_{3}(T, q).$ Using this identity together with 
\eqref{eq: tag 5.7}, \eqref{eq: tag 5.8}, and the above expression of \eqref{eq: tag 5.4}, 
compute 
\begin{equation*}
E(T)A_{4}(T, q) - q\Lambda_{3}(T, q) + q^{4}\Lambda_{3}(q\,T, 1\slash q)
\end{equation*} 
which should clearly satisfy the correct functional equation. Moreover, employing the identity 
\begin{equation*}
- q^{2} + E(T)\Lambda_{2}(T, q) = - q  + q^{3}E(q\, T)\Lambda_{2}(q\, T, 1\slash q) 
\end{equation*}
corresponding to the $l = 2$ case, it follows that both expressions:  
\begin{equation*}
\frac{q(q - 1)}{2E(T) E(q\, T)}\, -\, q^{3}(q - 1)\frac{\Lambda_{2}(q\, T, 1\slash q)}{E(T)} \, +\, 
\frac{q^{5}(q - 1)}{2} \frac{E(q\, T)\Lambda_{2}(q\, T, 1\slash q)^{2}}{E(T)}
\end{equation*} 
and 
\begin{equation*} 
-\, \frac{q(q - 1)}{2 E(-T) E(-q\, T)} \, +\, \frac{q^{5}(q - 1)}{2}
\frac{E(q\, T) \Lambda_{2}(q^{2} T^{2}\!, 1\slash q^{2})}{E(-T)}
\end{equation*} 
satisfy the desired functional equation. A simple calculation reduces then the problem 
to showing the identity 
\begin{equation} \label{eq: tag 5.10}
A_{3}^{(1)}(T, q) \,+\, q^{5}A_{3}^{(1)}(q\, T, 1\slash q) \, 
=\,\frac{q(q^{2} - 1)(q - 1)}{2}\Big(\frac{1}{E(-T) E(-q\, T)} \, -\, 
\frac{1}{E(T) E(q\, T)} \Big)
\end{equation}
where we set
\begin{equation*}
A_{3}^{(1)}(T, q) \;\;\;\;\;  := \sum_{\substack{\kappa \\ r_{1}(\kappa) + 2r_{2}(\kappa) = |\kappa|}} \mathcal{M}_{3}(r_{1}(\kappa) + 1; q)q^{r_{2}(\kappa)}T^{\kappa}.
\end{equation*}

To prove this, we add an additional variable $t_{r+1}$ (i.e., pass from $r$ to $r+1$) 
and split the sum defining 
$A_{3}^{(0)}(t_{1}, \ldots, t_{r}, t_{r+1}, q) = A_{3}^{(0)}(T, t_{r+1}, q)$ 
into three parts according to whether $k_{r + 1} = 0,$ $1$ or $2$ to obtain: 
\begin{equation} \label{eq: tag 5.11}
A_{3}^{(0)}(T, t_{r+1}, q) \,=\, 
(1 + qt_{r+1}^{2})A_{3}^{(0)}(T, q) \, +\, 
t_{r+1}A_{3}^{(1)}(T, q). 
\end{equation}
Replacing $t_{i}$ by $qt_{i}$ for all $i=1, \ldots, r+1,$ $q$ by $1\slash q,$ 
and then multiplying by $q^{4},$ we obtain 
\begin{equation} \label{eq: tag 5.12}
q^{4}A_{3}^{(0)}(q\, T, q t_{r+1}, 1\slash q) = 
q^{4}(1 + q t_{r+1}^{2})A_{3}^{(0)}(q\, T, 1\slash q) \, +\, 
q^{5}t_{r+1}A_{3}^{(1)}(q\, T, 1\slash q).  
\end{equation} 
Add \eqref{eq: tag 5.11} and \eqref{eq: tag 5.12}, and apply the identity 
\begin{equation*} 
A_{3}^{(0)}(T, q)  = \Lambda_{3}(T, q)  -  q^{4}\Lambda_{3}(q\, T, 1\slash q)  -  
\frac{q^{3}(q - 1)}{2}\Big(\frac{\Lambda_{2}(q\, T, 1\slash q)}{E(T)} + 
\frac{\Lambda_{2}(- q\, T, 1\slash q)}{E(- T)}\Big)
\end{equation*}
to express $A_{3}^{(0)}(T, t_{r+1}, q),$ 
$A_{3}^{(0)}(q\, T, q t_{r+1}, 1\slash q),$ $A_{3}^{(0)}(T, q)$ and $A_{3}^{(0)}(q\, T, 1\slash q)$ 
in terms of $\Lambda_{2}.$ Now \eqref{eq: tag 5.10} follows by applying the identity 
corresponding to the $l = 2$ case in the form 
$$
\frac{\Lambda_{2}(T, q)}{E(q\, T)}  - q^{3}\frac{\Lambda_{2}(q\, T, 1\slash q)}{E(T)} 
= \frac{q(q - 1)}{E(T)E(q\, T)}
$$ 
and completes the proof of the proposition. 
\end{proof}

Using the relation in Proposition \ref{Proposition 3.10} rewritten in the equivalent form 
\begin{equation} \label{eq: l=4-recurrence}
\lambda(\kappa, 4; q)  -  q^{|\kappa| +  5} \lambda(\kappa, 4; 1\slash q)   
\,=\, a(\kappa, 4; q)\, + \sum_{\kappa'  \, < \, \kappa} 
q^{|\kappa| - \mathrm{r}(\kappa - \kappa') + 5} 
(q-1)^{\mathrm{r}(\kappa - \kappa')}\lambda(\kappa', 4; 1\slash q) 
\end{equation} 
we can find recursively the coefficients $\lambda(\kappa, 4; q);$ the sum in 
\eqref{eq: l=4-recurrence} can be expressed as 
\begin{equation*}
\sum_{\kappa' \,  < \, \kappa} \big(a(\kappa', 4; q)  - \lambda(\kappa', 4; q)  +   
q^{|\kappa| +  5}\lambda(\kappa', 4; 1\slash q) \big)
\end{equation*} 
and hence, the right-hand side of \eqref{eq: l=4-recurrence} satisfies the same symmetry, as 
$q\to 1\slash q,$ as the $\kappa$-coefficient 
\begin{equation*}
\sum_{\kappa'  \, \le \, \kappa} a(\kappa', 4; q) 
\end{equation*}
of $E(T)A_{4}(T, q)$ does (cf. Proposition \ref{functional equation A4}). The coefficients of $A_{4}(T, q)$ are determined by expressing the moments $\mathcal{M}_{4}(r; q)$ in \eqref{eq: tag 5.4} using Birch-Ihara identity (see also Appendix \ref{D}), and by \eqref{eq: tag 5.5}, \eqref{eq: tag 5.6}, \eqref{eq: tag 5.7} (combined with the computations in Appendix \ref{C}), \eqref{eq: tag 5.8} and Theorem \ref{Theorem 5.2}.

Now starting from $\lambda(0, \ldots, 0, 4; q) = q^{4}$ and using the above information, one obtains the coefficients $\lambda(\kappa, 4; q)$ in the form 
\begin{equation*}
\lambda(\kappa, 4; q)\,  = \, P_{0}(q)\; + 
\sum_{j = 5}^{[(r_{1}(\kappa) + 1)\slash 2]} P_{j}(q)\, T_{2j + 2}(q)
\end{equation*} 
where $P_{0}(q) = P_{0}(\kappa, q), P_{5}(q) = P_{5}(\kappa, q), \ldots$ are polynomials 
with integer coefficients, independent of $q,$ such that:
\begin{equation*}
\deg  P_{0} \le |\kappa| + 4  \;\;\;\qquad \;\;\; 
\deg  P_{j} \le |\kappa| - 2j + 3 \qquad
\text{(for $j = 5, \ldots, [(r_{1}(\kappa) + 1)\slash 2]$)}
\end{equation*} 
and 
\begin{equation*}
q^{[(|\kappa| + 1)\slash 2] + 3} \mid  P_{0}(q)   \;\;\;\qquad \;\;\; 
q^{[|\kappa|\slash  2] - j + 3} \mid P_{j}(q) \qquad \text{(for $j = 5, \ldots, [(r_{1}(\kappa) + 1)\slash 2]$)}.
\end{equation*}

In the general case we shall use the same idea to determine the coefficients
$\lambda(\kappa, l; q),$ but for this we need \eqref{eq: tag 4.1'} for arbitrary $l,$ and a suitable substitute for Birch-Ihara identity; such an identity is available, unfortunately, only in very special cases.

\section{Cohomological interpretation of the moment-sums}\label{section 6} 
For the purposes of this section, it is more convenient to work with the 
slightly modified moment-sums $\tilde{M}_{\nu, \gamma}(q)$ defined 
for partitions 
$
\nu= (1^{\nu_{1}}\!, \ldots, (2g + 2)^{\nu_{2g + 2}})
$ 
of $2g + 2,$ and $\gamma= (1^{\gamma_{1}}\!, \ldots, g^{\gamma_{g}}),$ $g\ge 2,$ 
by 
\begin{equation*}
\tilde{M}_{\nu, \gamma}(q)  \,=\, |\mathrm{GL}_{2}(\mathbb{F})|^{-1} \cdot
\sum_{d_{0} \, \in \, \mathscr{P}_{\! g}(\nu,\,  \mathbb{F})} \, 
\prod_{j = 1}^{g} a_{j}(C_{d_{0}})^{\gamma_{j}}
\end{equation*} 
where $\mathscr{P}_{\! g}(\nu,\,  \mathbb{F})\subset \mathbb{F}[x]$ 
is the subset of {\it all} square-free polynomials of degree $2g + 1$ or 
$2g + 2$ defining hyperelliptic curves whose ramification points have fields 
of definition given by $\nu,$ and where 
$
a_{j}(C_{d_{0}}) =  \Tr(F^{*j} \vert \, H_{\text{\'et}}^{1}(\bar{C}_{d_{0}},\mathbb{Q}_{\ell})).
$ 
The normalizing factor $|\mathrm{GL}_{2}(\mathbb{F})|$ represents the number 
of $\mathbb{F}$-isomorphisms between curves 
$C_{d_{0}}$ ($d_{0}\in \mathscr{P}_{\! g}(\nu,\,  \mathbb{F})$), see \cite[Section 3]{Getz2} and \cite[Section 3]{Berg}. When $g = 2$ these sums appear in 
\cite[eqn. (5.1)]{BFvdG1}.

Notice that $\tilde{M}_{\nu, \gamma}(q)$ vanishes if its weight 
$|\gamma| : = \gamma_{1} + 2\gamma_{2} + \cdots + g \gamma_{g}$ is odd. 
Note also that $\tilde{M}_{\nu, \gamma}(q)$ can be expressed in terms of the sums 
$M_{\mu, \gamma}(q)$ introduced in Section \ref{section 3} and vice versa. (To see it the other way around, apply the functional equation (2.1) of $Z_{\scriptscriptstyle C_{d_{0}}}\!(t),$ $d_{0}\in \mathscr{P}(\mu),$ and Propositions \ref{Prop1-appendixB}, \ref{Prop2-appendixB} in Appendix \ref{B}.) 

It turns out that the moment-sums $\tilde{M}_{\nu, \gamma}(q)$ arise naturally when studying 
the cohomology of local systems on certain classical moduli spaces. To make this 
precise, let $\mathscr{A}_{g}$ denote the moduli stack of principally polarized 
abelian schemes of relative dimension $g;$ this is a smooth (but not proper) Deligne-Mumford stack defined over $\mathrm{Spec}(\mathbb{Z}).$ We shall need to work with the moduli stack $\mathscr{A}_{g, 2}$ over $\mathrm{Spec}(\mathbb{Z}[1\slash 2])$ classifying principally polarized abelian schemes of relative dimension $g$ with principal symplectic level-$2$ structure. (The standard reference on these moduli and their compactification is \cite{FC}.)

\subsection{Local systems on $\mathscr{A}_{g}$} \label{Loc-sys Ag} 
Let $\mathrm{GSp}_{2g}$ denote the Chevalley group scheme over $\mathbb{Z}$ of symplectic 
similitudes on the symplectic space $\mathbb{Z}^{2g}$ with its standard non-degenerate 
alternating form defined by 
$
(u, v) \times  (u', v') \mapsto u \cdot {^t}v' - v \cdot {^t}u'
$ 
for $u, v, u', v' \in \mathbb{Z}^{g}.$ We can write 
$$
\mathrm{GSp}_{2g}(R) = \left\{
\delta = \begin{pmatrix}
A & B \cr 
C & D 
\end{pmatrix}  \bigg\vert     
\begin{matrix}
A \cdot {^t}  B - B \cdot {^t}\!  A \, =\,  C \cdot {^t}  D - D \cdot {^t}  C \,=\,  0, \cr 
\text{$A\cdot {^t} D - B \cdot {^t} C \,=\, \eta(\delta) I_{g}$ for $\eta(\delta) \in R^{\times}$}  
\end{matrix}
\right\}
$$ 
for any commutative ring $R$ with identity. The homomorphism 
$\eta : \mathrm{GSp}_{2g} \to \mathbb{G}_{m}$ is called the 
{\it multiplier} representation of $\mathrm{GSp}_{2g}.$ The kernel of $\eta$ 
is by definition the group scheme $\mathrm{Sp}_{2g}.$

For $\lambda = (\lambda_{1} \ge \cdots  \ge \lambda_{g} \ge 0),$ 
we have natural $\mathbb{Q}_{\ell}$-adic smooth \'etale sheaves 
$\mathbb{V}(\lambda)$ on $\mathscr{A}_{g} \otimes \mathbb{Z}[1\slash \ell]$ corresponding 
to irreducible algebraic representations $V_{\lambda}$ of $\mathrm{GSp}_{2g}(\mathbb{Q}).$ 
Specifically, the $\mathbb{Q}_{\ell}$-sheaf $\mathbb{V} = \mathbb{V}(1, 0, \ldots, 0) = 
R^{1}\pi_{*}\mathbb{Q}_{\ell},$ defined using the universal family 
$\pi: \mathcal{X} \to \mathscr{A}_{g},$ corresponds to the contragredient of the standard 
representation $V$ of $\mathrm{GSp}_{2g}(\mathbb{Q}).$ We have a non-degenerate 
alternating pairing 
$$
\mathbb{V} \times \mathbb{V} \to \mathbb{Q}_{\ell}(-1).
$$
For general $\lambda,$ the $\mathbb{Q}_{\ell}$-sheaf $\mathbb{V}(\lambda)$ occurs in the decomposition of 
$$
\mathrm{Sym}^{\lambda_{1} - \lambda_{2}}(\mathbb{V}) \otimes \cdots \otimes 
\mathrm{Sym}^{\lambda_{g - 1} -\lambda_{g}}(\wedge^{g -1} \mathbb{V}) 
\otimes  \mathrm{Sym}^{\lambda_{g}}(\wedge^{g} \mathbb{V})
$$ 
into irreducibles; it is the $\ell$-adic coefficient system of weight 
$|\lambda| = \lambda_{1} + \cdots + \lambda_{g}$ corresponding to 
the irreducible representation $V_{\lambda}$ with dominant weight 
$(\lambda_{1} - \lambda_{2})\omega_{1} + \cdots + (\lambda_{g - 1} - \lambda_{g})\omega_{g - 1} 
+ \lambda_{g} \omega_{g} - |\lambda|\eta.$ 
Here $\omega_{1}, \ldots, \omega_{g}$ is the set of fundamental weights. 
If $\lambda_{1} > \cdots > \lambda_{g} > 0,$ the local system $\mathbb{V}(\lambda)$ 
will be called {\it regular}. Finally, the Tate sheaf $\mathbb{Q}_{\ell}(1)$ corresponds 
to $\eta.$ 

In a series of papers, Bergstr\"om, Faber, and van der Geer 
investigated the motivic Euler characteristic
\begin{equation*}
e_{c}(\mathscr{A}, \mathbb{V}(\lambda)) \; =
\sum_{i = 0}^{g(g + 1)} (-1)^{i} [H_{c}^{i}(\mathscr{A}, \mathbb{V}(\lambda))]
\end{equation*}
when $\mathscr{A} = \mathscr{A}_{g},$ $g = 2, 3$ (see \cite{FvdG}, \cite{BFvdG2}), and 
$\mathscr{A} = \mathscr{A}_{2, 2}$ (see \cite{BFvdG1}). This expression is taken in 
the Grothendieck group $K_{0}(\mathrm{Gal}_{\mathbb{Q}})$ of $\ell$-adic 
$\mathrm{Gal}(\overline{\mathbb{Q}}\slash \mathbb{Q})$-representations, 
or in $K_{0}(\mathrm{MHS}),$ the Grothendieck group of the category of 
mixed Hodge structures\footnote{In this instance, we are taking the 
compactly supported cohomology of the Betti version of 
$\mathbb{V}(\lambda),$ constructed from $\mathbb{V} = R^{1}\pi_{*}\mathbb{Q},$ 
on $\mathscr{A} \otimes \mathbb{C}.$}. (Here and throughout the section, 
we shall use the same notation $\mathbb{V}(\lambda)$ for the corresponding local 
systems obtained by pullback to $\mathscr{A}_{g, 2}.$) 

\vskip10pt
\begin{remark} \hskip-2.5pt As the compactly supported cohomology of $\mathbb{V}(\lambda)$ on 
$\mathscr{A} = \mathscr{A}_{g}$ (or $\mathscr{A}_{g, 2}$) always vanishes if $|\lambda|$ is odd, 
one can only consider local systems $\mathbb{V}(\lambda)$ of even weights.
\end{remark}

\vskip5pt
In \cite[Theorem 5.5, p. 233]{FC}, Faltings and Chai provided 
Hodge filtrations of mixed Hodge structures on the cohomology groups 
$H_{c}^{i} (\mathscr{A}_{g} \otimes \mathbb{C}, \mathbb{V}(\lambda))$ and  
$H^{i}(\mathscr{A}_{g} \otimes \mathbb{C}, \mathbb{V}(\lambda))$ (and also on 
$H_{c}^{i}(\mathscr{A}_{g, 2} \otimes \mathbb{C}, \mathbb{V}(\lambda))$ and  
$H^{i}(\mathscr{A}_{g, 2} \otimes \mathbb{C}, \mathbb{V}(\lambda))$) 
of weights $\le |\lambda|  + i$ and $\ge |\lambda| + i,$ respectively. 
(There is also an analogue of this for the $\ell$-adic cohomology.) 
A main ingredient in their theory is the construction in \cite{Fa}, 
using some ideas of Bernstein, Gelfand and Gelfand, 
of a complex $\mathcal{K}_{\lambda}^{\bullet}$ of vector bundles, 
called the {\it dual BGG-complex} for $\lambda,$ 
which is a direct summand in the de Rham complex of 
$\mathbb{V}(\lambda)^{\vee};$ this complex is a filtered resolution of 
$\mathbb{V}(\lambda)^{\vee}.$ The steps in the Hodge filtrations are 
given (see \cite{BFvdG2}) by the sums of the elements 
of any of the $2^{g}$ subsets of $\{\lambda_{g} + 1, \lambda_{g - 1} + 2, \ldots, \lambda_{1} +  g \}.$ 
It was also proved by Faltings \cite{Fa} that, for $\mathbb{V}(\lambda)$ regular, 
the inner cohomology groups $H_{!}^{i}(\mathscr{A}, \mathbb{V}(\lambda))$ 
($\mathscr{A} = \mathscr{A}_{g},$ or $\mathscr{A}_{g, 2}$),
i.e., the image of the natural map 
$H_{c}^{i}(\mathscr{A}, \mathbb{V}(\lambda)) \to H^{i}(\mathscr{A}, \mathbb{V}(\lambda)),$ vanish for 
$i \ne g(g + 1)\slash 2.$ The middle inner cohomology group $H_{!}^{g(g + 1)\slash 2}(\mathscr{A}, \mathbb{V}(\lambda))$ 
is pure of weight $|\lambda| + g (g + 1)\slash 2,$ and by \cite[Section 6, p. 237]{FC}, the step of the Hodge filtration of $H_{!}^{g(g + 1)\slash 2}(\mathscr{A} \otimes \mathbb{C}, \mathbb{V}(\lambda))$ 
corresponding to the full set $\{\lambda_{g} + 1, \lambda_{g - 1} + 2, \ldots, \lambda_{1} +  g \}$ 
can be identified with the complex vector space $S_{n(\lambda)}(\Gamma(\mathscr{A}))$ 
of Siegel modular cusp forms of weight 
$n(\lambda) = (\lambda_{1} - \lambda_{2}, \lambda_{2} - \lambda_{3}, \ldots, 
\lambda_{g - 1} - \lambda_{g}, \lambda_{g} + g + 1)$ on $\Gamma(\mathscr{A}),$ 
with $\Gamma(\mathscr{A}) = \mathrm{Sp}_{2g}(\mathbb{Z})$ if $\mathscr{A} = \mathscr{A}_{g},$ 
and $\Gamma(\mathscr{A}) = \Gamma_{ g}[2]$ is the kernel of 
$\mathrm{Sp}_{2g}(\mathbb{Z}) \to \mathrm{Sp}_{2g}(\mathbb{Z}\slash 2)$ if $\mathscr{A} = \mathscr{A}_{g, 2}.$ 
For basic facts about the theory of Siegel modular forms, we refer the reader to 
\cite{AZ}, \cite{FC} and \cite{Frei}. It is expected (see \cite{BFvdG2} for more details) that 
\begin{equation} \label{eq: moteuler-conj}
e_{c}(\mathscr{A}, \mathbb{V}(\lambda)) \, =\, (- 1)^{\frac{g(g + 1)}{2}} \widehat{S}[\Gamma(\mathscr{A}), n(\lambda)] \, +\,
e_{g, \text{endo}}(\mathscr{A}, \lambda) \, +\, e_{g, \text{Eis}}(\mathscr{A}, \lambda)
\end{equation} 
for $\lambda \ne (0, \ldots, 0)$ with $|\lambda| \equiv 0 \pmod 2,$ 
where $\widehat{S}[\Gamma(\mathscr{A}), n(\lambda)]$ is the conjectural motive in \cite{BFvdG2} 
associated to $S_{n(\lambda)}(\Gamma(\mathscr{A})),$ $e_{g, \text{endo}}(\mathscr{A}, \lambda)$ is 
the Euler characteristic of the remaining part of the inner cohomology, presumably consisting of 
contributions connected to the endoscopic groups, and $e_{g, \text{Eis}}(\mathscr{A}, \lambda)$ 
denotes the Euler characteristic corresponding to the kernel of the natural map 
$H_{c}^{i}(\mathscr{A}, \mathbb{V}(\lambda)) \to H^{i}(\mathscr{A}, \mathbb{V}(\lambda)),$ 
called the {\it Eisenstein cohomology}. The most successful method used to investigate variants of 
\eqref{eq: moteuler-conj} is that developed by Ihara, Langlands and Kottwitz, and it is based on a comparison between the Grothendieck-Lefschetz and Arthur-Selberg trace formulas, see \cite{Mo1}, \cite{Mo2} and \cite{Kott}. When $g = 2,$ see \cite[Theorem 2.1]{Peter}, \cite{Har}, \cite{Lau}, \cite{Weiss1}, \cite{Weiss2}, \cite{Tay} and \cite{Teh}; when $g = 1,$ see \cite{Sch}, \cite{ConsF}, and the classical work \cite{Del1}.

The connection to the sums $\tilde{M}_{\nu,  \gamma}(q)$ comes by considering 
the $\ell$-adic Euler characteristic $e_{c}(\mathscr{H}_{g}[2]$ $\otimes\, \overline{\mathbb{Q}}, \mathbb{V}'(\lambda))$ 
in $K_{0}(\mathrm{Gal}_{\mathbb{Q}})$ of the hyperelliptic Jacobian locus $\mathscr{H}_{g}[2]$ in $\mathscr{A}_{g, 2}.$ 
Here $\mathbb{V}'(\lambda)$ denotes the restriction of $\mathbb{V}(\lambda)$ to $\mathscr{H}_{g}[2].$ 
Let $[\mathscr{H}_{g}[2](\mathbb{F})]$ denote the set of isomorphism classes of the category 
$\mathscr{H}_{g}[2](\mathbb{F}).$ By Torelli's theorem 
\cite[Appendice, Th\'eor\`eme 1]{Laut}, we can identify the elements of $[\mathscr{H}_{g}[2](\mathbb{F})]$ by 
isomorphism classes of tuples $(C, w_{1}, \ldots, w_{2g + 2})$ of hyperelliptic curves of genus $g$ 
together with their $2g + 2$ marked Weierstrass points; we shall also identify the local system 
$\mathbb{V}'(\lambda)$ and its pullback under the Torelli morphism.

The natural action of the symmetric group $\mathbb{S}_{2g + 2}$ on $\mathscr{H}_{g}[2]$ induces a decomposition 
of the representation $H_{c}^{i}(\mathscr{H}_{g}[2] \otimes \overline{\mathbb{F}}, \mathbb{V}'(\lambda))$ into a direct sum of factors $H_{c,  \mu}^{i}(\mathscr{H}_{g}[2] \otimes \overline{\mathbb{F}}, \mathbb{V}'(\lambda))$ indexed 
by the partitions $\mu$ of $2g + 2,$ and by \cite[eqn. (2.32)]{FH}, the weighted sum  
\begin{equation*} 
\frac{\chi_{\mu}(id)}{(2g + 2)!}\sum_{\sigma \, \in \, \mathbb{S}_{2g + 2}} \chi_{\mu}(\sigma^{-1})\cdot \sigma^{*}
: H_{c}^{i}(\mathscr{H}_{g}[2] \otimes_{_{\mathbb{F}}} \! \overline{\mathbb{F}}, \mathbb{V}'(\lambda)) 
\to H_{c}^{i}(\mathscr{H}_{g}[2] \otimes_{_{\mathbb{F}}} \! \overline{\mathbb{F}}, \mathbb{V}'(\lambda)) 
\end{equation*} 
is the projection of $H_{c}^{i}(\mathscr{H}_{g}[2] \otimes \overline{\mathbb{F}}, \mathbb{V}'(\lambda))$ onto 
$H_{c,  \mu}^{i}(\mathscr{H}_{g}[2] \otimes \overline{\mathbb{F}}, \mathbb{V}'(\lambda)).$ Here $\chi_{\mu}$ is the character of the irreducible representation $R_{\mu}$ of $\mathbb{S}_{2g + 2}$ corresponding to $\mu.$ For a conjugacy class $[\nu]$ of $\mathbb{S}_{2g + 2}$ determined by a partition $\nu= (1^{\nu_{1}}\!, \ldots, (2g + 2)^{\nu_{2g + 2}})$ of $2g + 2,$ we shall denote the value $\chi_{\mu}(\sigma)$ (for $\sigma \in [\nu]$) by $\chi_{\mu}(\nu).$

Now for any partition $\lambda$ of length at most $g,$ let $s_{\scriptscriptstyle \langle \lambda \rangle}$ denote the symplectic Schur function 
(see \cite[Appendix A, \S A.3., A.45]{FH}) associated to the irreducible representation $V_{\lambda}.$ If we write 
$
s_{\scriptscriptstyle \langle \lambda \rangle} = 
\sum_{|\gamma| \le |\lambda|}  r_{\gamma}^{\lambda} p^{(\gamma)}
$ 
for some $r_{\gamma}^{\lambda} \in \mathbb{Q},$ where the sum is over 
$\gamma = (1^{\gamma_{1}}\!, \ldots, g^{\gamma_{g}}),$ and 
$p^{(\gamma)} = p_{1}^{\gamma_{1}} \cdots \, p_{g}^{\gamma_{g}}$ 
is the power sum polynomial associated to $\gamma,$ then 
by Behrend's Lefschetz trace formula (see \cite{Beh} or \cite[Thm. 19.3.4]{L-MB}),
the trace of the geometric Frobenius on 
$e_{c,  \mu}(\mathscr{H}_{g}[2] \otimes \overline{\mathbb{F}}, \mathbb{V}'(\lambda))$ 
is given by 
\begin{equation} \label{eq: eulercharschur}
\Tr(F^{*} \vert \, e_{c,  \mu}(\mathscr{H}_{g}[2] \otimes_{_{\mathbb{F}}} \! \overline{\mathbb{F}}, \mathbb{V}'(\lambda))) =
q^{\frac{|\lambda|}{2}}\frac{\chi_{\mu}(id)}{|\mathrm{GL}_{2}(\mathbb{F})|}\sum_{\nu} \chi_{\mu}(\nu) 
\sum_{d_{0} \, \in \, \mathscr{P}_{\! g}(\nu,\,  \mathbb{F})}
s_{\scriptscriptstyle \langle \lambda \rangle}(\omega_{1}(C_{ d_{0}})^{\pm 1}\!, \ldots, \omega_{g}(C_{ d_{0}})^{\pm 1}). 
\end{equation}
Here for a hyperelliptic curve $C_{d_{0}}$ ($d_{0}  \in \mathscr{P}_{\! g}(\nu,\,  \mathbb{F})$), 
$
\omega_{1}(C_{d_{0}}), \ldots, \omega_{g}(C_{d_{0}}), 
\omega_{1}(C_{d_{0}})^{-1}, \ldots, \omega_{g}(C_{d_{0}})^{-1}
$ 
denote the {\it normalized} (i.e., unitary) eigenvalues of $F^{*};$ one can also see this by 
applying the Lefschetz trace formula to the correspondence 
on $(H_{g}[2], p_{*}\!\mathbb{V}'(\lambda))$ defined by composing the action of an element 
$\sigma^{-1} \in [\nu]$ with Frobenius (i.e., the relative Frobenius with respect to the 
$\mathbb{F}$-structure induced by $\sigma^{-1}$), where $H_{g}[2]$ denotes the 
coarse moduli space of 
$
\mathscr{H}_{g}[2] \otimes_{_{\mathbb{F}}} \! \overline{\mathbb{F}},
$ 
and 
$
p :  \mathscr{H}_{g}[2] \otimes_{_{\mathbb{F}}} \! \overline{\mathbb{F}} \to H_{g}[2]
$ 
is the natural map.

In terms of the moment-sums $\tilde{M}_{\nu, \gamma}(q),$ 
\begin{equation*} \label{eq: eulercharpower}
\Tr(F^{*} \vert \, e_{c,  \mu}(\mathscr{H}_{g}[2] \otimes_{_{\mathbb{F}}} \! \overline{\mathbb{F}}, \mathbb{V}'(\lambda))) =
\chi_{\mu}(id) \sum_{\nu} \chi_{\mu}(\nu) 
\sum_{|\gamma| \, \le \, |\lambda|} r_{\gamma}^{\lambda} q^{\frac{|\lambda| - |\gamma|}{2}} 
\tilde{M}_{\nu, \gamma}(q)
\end{equation*} 
see also \cite{Berg}, \cite{BFvdG1} and \cite{BFvdG2}. (Notice that $|\gamma|$ stands here 
for the sum $\gamma_{1} + 2\gamma_{2} + \cdots + g \gamma_{g}.$) The last equation can be 
inverted to express $\tilde{M}_{\nu, \gamma}(q),$ and hence $M_{\nu, \gamma}(q),$ 
in terms of $\Tr(F^{*} \vert \, e_{c,  \mu}(\mathscr{H}_{g}[2] \otimes \overline{\mathbb{F}}, \mathbb{V}'(\lambda))).$ 
It would be very interesting to identify 
the contribution of the Euler characteristic 
$e_{c}(\mathscr{H}_{g}[2], \mathbb{V}'(\lambda))$ to \eqref{eq: moteuler-conj} when $\mathscr{A} = \mathscr{A}_{g, 2}.$

\subsection{Moments of characteristic polynomials} 
If $\lambda = (\lambda_{1} \ge \cdots  \ge \lambda_{g} \ge 0)$ is a partition, we shall write 
$\lambda \subseteq (r^{g})$ to indicate that $\lambda_{i} \le r$ for all $1\le i \le g.$ For 
$\lambda \subseteq (r^{g}),$ let $\lambda^{ \dag} =  (g - \lambda_{r}' \ge \cdots  \ge g - \lambda_{1}'),$ where the non-zero integers among $\lambda_{j}',$ $j = 1, \ldots, r,$ are the parts of the conjugate partition $\lambda'$ of $\lambda.$

\vskip10pt
\begin{lemma}
\label{BG} --- For independent variables $t_{1}, \ldots, t_{r}, z_{1}, \ldots, z_{g},$ we have the identity
\begin{equation*}
\prod_{i = 1}^{g} \prod_{j = 1}^{r}\, (z_{i}^{} + z_{i}^{-1} + t_{j}^{} + t_{j}^{-1}) \; = 
\sum_{\lambda \subseteq (r^{g})} 
s_{\scriptscriptstyle \langle \lambda \rangle}(z_{1}^{\pm 1}\!, \ldots, z_{g}^{\pm 1}) 
s_{\scriptscriptstyle \langle \lambda^{ \dag} \rangle}(t_{1}^{\pm 1}\!, \ldots, t_{r}^{\pm 1}).
\end{equation*}
\end{lemma} 

\begin{proof} See \cite[Section 5]{BG}. 
\end{proof}

\vskip10pt
\begin{lemma}
\label{odd lambda} --- With notations as above, we have 
\begin{equation*}
\sum_{d_{0} \, \in \, \mathscr{P}_{\! g}(\nu,\,  \mathbb{F})} s_{\scriptscriptstyle \langle \lambda \rangle}(\omega_{1}(C_{d_{0}})^{\pm 1}\!, \ldots, \omega_{g}(C_{d_{0}})^{\pm 1})  = 0 \qquad (\text{if $|\lambda|$ is odd}). 
\end{equation*}
\end{lemma}

\begin{proof} \hskip-1ptWrite as above 
$$
s_{\scriptscriptstyle \langle \lambda \rangle}(z_{1}^{\pm 1}\!, \ldots, z_{g}^{\pm 1}) \; = 
\sum_{|\gamma| \, \le \, |\lambda|}  r_{\gamma}^{\lambda} p^{(\gamma)}(z_{1}^{\pm 1}\!, \ldots, z_{g}^{\pm 1}).
$$ 
Replacing $z_{i}$ by $-z_{i}$ for $i = 1, \ldots, g,$ we deduce that $r_{\gamma}^{\lambda} = 0$ 
unless $|\lambda|$ and $|\gamma|$ have the same parity. In particular, if $|\lambda|$ is odd 
the only terms that contribute are those corresponding to partitions $\gamma$ of odd weights. 

Now take $z_{i} = \omega_{i}(C_{d_{0}})$ ($i = 1, \ldots, g$) for $d_{0}\in \mathscr{P}_{\! g}(\nu,\,  \mathbb{F}),$ and recall that if $C_{d_{0}}$ is the hyperelliptic curve corresponding to $d_{0},$ then
\begin{equation*}
a_{k}(C_{d_{0}}) =  \Tr(F^{*k} \vert \, H_{\text{\'et}}^{1}(\bar{C}_{d_{0}}, \mathbb{Q}_{\ell}))
= - \sum_{\theta \, \in  \, {\bf{P^{1}}}(\mathbb{F}_{\! q^{k}})} 
\chi_{k}(d_{0}(\theta)).
\end{equation*} 
Note that 
\begin{equation*}
\sum_{d_{0} \, \in \, \mathscr{P}_{\! g}(\nu,\,  \mathbb{F})}
p^{(\gamma)}(\omega_{1}(C_{d_{0}})^{\pm 1}\!, \ldots, \omega_{g}(C_{d_{0}})^{\pm 1})\;\;
= \sum_{d_{0} \, \in \, \mathscr{P}_{\! g}(\nu,\,  \mathbb{F})} 
q^{-\frac{|\gamma|}{2}}\prod_{k = 1}^{g} a_{k}(C_{d_{0}})^{\gamma_{k}}  =\, 0
\qquad \text{(if $|\gamma|$ is odd)}
\end{equation*} 
and the lemma follows. 
\end{proof}

\vskip10pt
\begin{thm}\label{q->1/q} --- For any partition $\mu$ of $2 g + 2$ and independent variables $t_{1}, \ldots, t_{r},$ let $S_{\mu}(T, q) = S_{\mu}(t_{1},\ldots, t_{r}, q)$ 
be defined by 
\begin{equation*}
S_{\mu}(T, q) = \frac{\chi_{\mu}(id)}{|\rm{GL}_{2}(\mathbb{F})|}\sum_{\nu} \chi_{\mu}(\nu) 
\sum_{d_{0} \, \in \, \mathscr{P}_{\! g}(\nu,\,  \mathbb{F})}
\left(\prod_{k = 1}^{r}  P_{\scriptscriptstyle C_{d_{0}}}\!(t_{k}) \right).
\end{equation*} 
Then 
\begin{equation*}
S_{\mu}(T, q)  = (t_{1}\cdots \, t_{r})^{g}
\!\!\sum_{\substack{\lambda \subseteq (r^{g}) \\ \text{$|\lambda|$-\rm{even}}}}
\Tr(F^{*} \vert \, e_{c,  \mu}(\mathscr{H}_{g}[2] \otimes \overline{\mathbb{F}}, \mathbb{V}'(\lambda) \otimes \mathbb{Q}_{\ell}(|\lambda| \slash 2)))\,
s_{\scriptscriptstyle \langle \lambda^{\dag} \rangle}(q^{\pm \frac{1}{2}} t_{1}^{\pm 1}\!, \ldots, q^{\pm \frac{1}{2}} t_{r}^{\pm 1})\, q^{\frac{gr}{2}}.
\end{equation*}
\end{thm}

\begin{proof} \hskip-1ptFor any $d_{0}\in \mathscr{P}_{\! g}(\nu,\,  \mathbb{F}),$
we apply Lemma \ref{BG} with $t_{k} \to - \sqrt{q}\,t_{k}$ 
for $k = 1,\ldots, r,$ and $z_{1}^{\pm 1}\!, \ldots, z_{g}^{\pm 1}$ chosen to be 
the {\it normalized} (by $\sqrt{q}$) eigenvalues of the Frobenius $F^{*}$ on
$
H_{\text{\'et}}^{1}(\bar{C}_{d_{0}},\mathbb{Q}_{\ell}),
$ 
with $\ell$ a prime different from the characteristic of $\mathbb{F}.$ Multiplying both sides of the identity by $(-1)^{g r}(\sqrt{q} t_{1}\, \cdots \, \sqrt{q} t_{r})^{g}\chi_{\mu}(\nu)$ 
and summing over $d_{0}$ and $\nu,$ the theorem follows at once from \eqref{eq: eulercharschur} and Lemma \ref{odd lambda}. 
\end{proof}

\vskip10pt
\begin{remark} \hskip-2.5pt By applying Lemma \ref{BG} and the identity 
\begin{equation*}
\prod_{i = 1}^{m} \prod_{j = 1}^{g} (1 - t_{i}^{} z_{j}^{})^{-1}  (1 - t_{i}^{} z_{j}^{-1})^{-1} \; = 
\prod_{1 \, \le \, i \, < \, j \, \le \, m}(1 - t_{i}^{} t_{j}^{})^{-1} \cdot 
\sum_{l(\lambda) \, \le \, g} s_{\scriptscriptstyle \langle \lambda \rangle}(Z^{\pm 1}) 
s_{\scriptscriptstyle \lambda}(T)
\end{equation*} 
(see \cite[Section 5, Proof of Theorem 4]{BG}) one can obtain in a similar fashion 
the variant of Theorem \ref{q->1/q} for sums of ratios of characteristic polynomials; here $s_{\scriptscriptstyle \lambda}$ is the ordinary Schur polynomial, 
$Z^{\pm 1}  = (z_{1}^{}, z_{1}^{-1}\!, \ldots, z_{g}^{}, z_{g}^{-1}),$ and $T = (t_{1}, \ldots, t_{m}).$
\end{remark}

\vskip10pt
Let 
$
\mathcal{S} = (\chi_{\mu}(\nu))_{\mu, \nu}, 
$ 
where $\mu, \nu$ indicate respectively the rows and columns of $\mathcal{S}.$ This is 
the transpose of the transition matrix from the basis of the $\mathbb{Q}$-vector space of homogeneous symmetric polynomials of degree $2g + 2$ in $2g + 2$ variables consisting of the power sum polynomials to the basis consisting of the ordinary Schur polynomials. If 
$\nu = (1^{\nu_{1}}\!, 2^{\nu_{2}}\!, \ldots, (2g + 2)^{\nu_{2g + 2}}),$ put, as before, 
$z(\nu) = \nu_{1}! 1^{\nu_{1}} \nu_{2}! 2^{\nu_{2}} \cdots \, \nu_{2g + 2}! (2g + 2)^{\nu_{2g + 2}}.$ Then by \cite[Appendix A, \S A.1, Exercise A.29]{FH}, we have 
$
\mathcal{S}^{-1} = (\chi_{\mu}(\nu)\slash z(\nu))_{\nu, \mu}.
$ 
Accordingly, if we denote the components of 
$$
^{t}\!\mathcal{S} \cdot \big(\chi_{\mu}(id)^{-1} e_{c, \mu}(\mathscr{H}_{g}[2], \mathbb{V}'(\lambda))\big)_{\mu}
$$ 
by $\tilde{e}_{c, \nu}(\mathscr{H}_{g}[2], \mathbb{V}'(\lambda)),$ then Theorem \ref{q->1/q} implies that, for every $\nu,$ we have\footnote{This can probably be better formulated using 
$\mathbb{S}_{2g + 2}$-equivariant Euler characteristics.} 
\begin{equation}\label{eq: basic-moment-id} 
\frac{z(\nu)}{|\mathrm{GL}_{2}(\mathbb{F})|}
\sum_{d_{0}\, \in \, \mathscr{P}_{\! g}(\nu,\,  \mathbb{F})}
\left(\prod_{k = 1}^{r}  P_{\scriptscriptstyle C_{d_{0}}}\!(t_{k})  \right)
=\, (t_{1}\cdots \, t_{r})^{g}
\!\!\sum_{\substack{\lambda \subseteq (r^{g}) \\ \text{$|\lambda|$-even}}}
\Tr(F^{*} \vert \, \tilde{e}_{c,  \nu}(\mathscr{H}_{g}[2] \otimes_{_{\mathbb{F}}} 
\overline{\mathbb{F}}, \mathbb{V}'(\lambda)))\,
s_{\scriptscriptstyle \langle \lambda^{ \dag} \rangle}(q^{\pm \frac{1}{2}} T^{\pm 1})\,
q^{ \frac{gr - |\lambda|}{2}}.
\end{equation} 
Recalling that $\chi_{\mu}$ is the character of the irreducible representation $R_{\mu}$ of $\mathbb{S}_{2g + 2}$ corresponding to $\mu,$ note that 
$$
\chi_{\mu}(id) = \dim \, R_{\mu} = \frac{(2g + 2)!}{\prod_{i = 1}^{k}(\mu_{i} + k - i)!}
\prod_{1 \, \le \, i \,  <  \, j \, \le \, k}(\mu_{i} - \mu_{j} + j - i)  \qquad 
\text{(if $\mu = (\mu_{1} \ge \cdots \ge \mu_{k} \ge 0)$)}
$$ 
(see, for instance, \cite[\S 4.1, eqn. (4.11)]{FH}). 

By letting 
\begin{equation}\label{eq: euler-characteristic-w1} 
e_{c}^{\lambda}(\mathscr{H}_{g}(w^{1})) : \;\;= \sum_{\nu = (1^{\nu_{1}}\!, \ldots, (2g + 2)^{\nu_{2g + 2}})}\frac{\nu_{1}}{z(\nu)} \tilde{e}_{c, \nu}(\mathscr{H}_{g}[2], \mathbb{V}'(\lambda))
\end{equation} 
one obtains the following 

\vskip10pt
\begin{thm}\label{mom-for-odd_poly} --- For $g \ge 1,$ we have 
\begin{equation}\label{eq: moments-odd-sqfree} 
\frac{1}{q (q - 1)} \sum_{\substack{\deg d = 2g + 1 \\ \text{$d$ \rm{square-free}}}}
\left(\prod_{k = 1}^{r} P_{\scriptscriptstyle C_{d}}(t_{k})  \right)  = \,
(t_{1}\cdots \, t_{r})^{g}\!\!\sum_{\substack{\lambda \subseteq (r^{g}) \\ \text{$|\lambda|$-\rm{even}}}}
\Tr(F^{*} \vert \, e_{c}^{\lambda}(\mathscr{H}_{g}(w^{1})\otimes_{_{\mathbb{F}}} \!\overline{\mathbb{F}}))\,
s_{\scriptscriptstyle \langle \lambda^{ \dag} \rangle}(q^{\pm \frac{1}{2}} T^{\pm 1})\,
q^{\frac{gr - |\lambda|}{2}}
\end{equation} 
the first sum being over monic polynomials in $\mathbb{F}[x].$ When $g = 1,$ we 
define $\mathscr{H}_{1}(w^{1}) : = \mathscr{A}_{1}.$
\end{thm}

\begin{proof} \hskip-1ptWhen $g = 1$ this is clear, and so we can assume that $g \ge 2.$ 
By \eqref{eq: basic-moment-id} and \eqref{eq: euler-characteristic-w1}, 
the right-hand side clearly equals
\begin{equation*}
\frac{1}{|\mathrm{GL}_{2}(\mathbb{F})|}\sum_{\nu = (1^{\nu_{1}}\!, \ldots, (2g + 2)^{\nu_{2g + 2}})}\nu_{1} \, \cdot \sum_{d_{0} \, \in \, \mathscr{P}_{\! g}(\nu,\,  \mathbb{F})}
\left(\prod_{k = 1}^{r} P_{\scriptscriptstyle C_{d_{0}}}\!(t_{k}) \right).
\end{equation*} 
Note that $|\mathrm{GL}_{2}(\mathbb{F})| =  q (q + 1) (q - 1)^{2}.$ We split the sum in the left-hand side of \eqref{eq: moments-odd-sqfree} according to the factorization types 
$\mu = (1^{\mu_{1}}\!, \ldots, (2g + 1)^{\mu_{2g + 1}}),$ with $|\mu| = 2g + 1,$ of $d.$ 

Now for a hyperelliptic curve $C_{d_{0}},$ express the product of 
$P_{\scriptscriptstyle C_{d_{0}}}\!(t_{k})$ over $k = 1, \ldots, r$ in the form 
\begin{equation*} 
\prod_{k = 1}^{r} P_{\scriptscriptstyle C_{d_{0}}}\!(t_{k}) \,=\,  
\exp \bigg( - \sum_{j = 1}^{\infty} \frac{a_{j}(C_{d_{0}})}{j}
\cdot\sum_{k = 1}^{r}t_{k}^{j} \bigg)
\end{equation*}
with
\begin{equation*}
a_{j}(C_{d_{0}}) = 
\Tr(F^{*j}  \vert \, H_{\text{\'et}}^{1}(\bar{C}_{d_{0}}, \mathbb{Q}_{\ell}))
= - \sum_{\theta \, \in \,  {\bf{P^{1}}}(\mathbb{F}_{\! q^{j}})} 
\chi_{j}(d_{0}(\theta)).
\end{equation*} 
Applying this, we see that it suffices to show that 
\begin{equation*}
\sum_{\mu = (1^{\mu_{1}}\!, \ldots, (2g + 1)^{\mu_{2g + 1}})}\;
\sum_{d \, \in \, \mathscr{P}(\mu)}\, 
\prod_{j = 1}^{m} a_{j}(C_{d})^{\gamma_{j}}
=\, \frac{1}{q^{2} - 1}\sum_{\nu = (1^{\nu_{1}}\!, \ldots, (2g + 2)^{\nu_{2g + 2}})}
\nu_{1} \sum_{d_{0} \, \in \, \mathscr{P}_{\! g}(\nu,\,  \mathbb{F})}\, 
\prod_{j = 1}^{m} a_{j}(C_{d_{0}})^{\gamma_{j}}\!.
\end{equation*} 
for any partition $\gamma = (1^{\gamma_{1}}\!, \ldots, m^{\gamma_{m}}).$ Here $\mathscr{P}(\mu) \subset \mathbb{F}[x]$ denotes, as in Section \ref{section 3}, the set of all monic square-free polynomials $d$ with factorization type $\mu.$ It is easy to see that both sides vanish when $\gamma$ has odd weight (see also the beginning of the proof of Lemma \ref{evenness-A}). When $|\gamma|$ is even, the identity is easily reduced to 
\begin{equation}\label{eq: mom-id-even-weight} 
\sum_{\mu = (1^{\mu_{1}}\!, \ldots, (2g + 1)^{\mu_{2g + 1}})}\;
\sum_{d \, \in \, \mathscr{P}(\mu)}\, 
\prod_{j = 1}^{m} a_{j}(C_{d})^{\gamma_{j}} 
=\, \frac{1}{q + 1}\sum_{\nu = (1^{\nu_{1}}\!, \ldots, (2g + 2)^{\nu_{2g + 2}})}
\nu_{1} \sum_{\substack{d_{0} \, \in \, \mathscr{P}_{\! g}(\nu,\,  \mathbb{F}) \\ \text{$d_{0}$-monic}}}\, \prod_{j = 1}^{m} a_{j}(C_{d_{0}})^{\gamma_{j}}\!.
\end{equation} 
By Proposition \ref{Prop1-appendixB} in Appendix \ref{B}, for any partition 
$\mu = (1^{\mu_{1}}\!, \ldots, (2g + 1)^{\mu_{2g + 1}})$ of $2g + 1,$ we have 
\begin{equation*}
\sum_{d \, \in \, \mathscr{P}(\mu)}\, 
\prod_{j = 1}^{m} a_{j}(C_{d})^{\gamma_{j}} 
=\, \frac{\mu_{1} + 1}{q + 1}\Bigg(\sum_{d_{0} \, \in \, \mathscr{P}(\mu)}\, 
\prod_{j = 1}^{m} a_{j}(C_{d_{0}})^{\gamma_{j}} \; + 
\sum_{d_{0} \, \in \, \mathscr{P}(\mu')}\, 
\prod_{j = 1}^{m} a_{j}(C_{d_{0}})^{\gamma_{j}} \Bigg)
\end{equation*} 
with $\mu' : = (1^{\mu_{1} + 1}\!, \ldots, (2g + 1)^{\mu_{2g + 1}}).$ Clearly this 
implies \eqref{eq: mom-id-even-weight}, and completes the proof.
\end{proof}

\vskip10pt
\begin{corollary}\label{special case mom-for-odd_poly} --- Notation being as in Theorem \rm{\ref{mom-for-odd_poly}}, we have 
\begin{equation*}
\frac{1}{q (q - 1)}  \sum_{\substack{\deg d = 2g + 1 \\ \text{$d$ \rm{square-free}}}}  
L({\scriptstyle \frac{1}{2}}, \chi_{d})^{r}  \;= 
\sum_{\substack{\lambda \subseteq (r^{g}) \\ \text{$|\lambda|$-\rm{even}}}}
(\dim \, V_{\lambda^{ \dag}}) \Tr(F^{*} \vert \,
e_{c}^{\lambda}(\mathscr{H}_{g}(w^{1})  \otimes_{_{\mathbb{F}}} \! \overline{\mathbb{F}}))\,
q^{-\frac{|\lambda|}{2}}.
\end{equation*}
\end{corollary} 

\begin{proof} \hskip-1ptWe recall from Section \ref{section 2} that 
$
L(s, \chi_{d}) = P_{\scriptscriptstyle C_{d}}(q^{-s}).
$ 
The identity stated is just the limiting case of the identity in Theorem \ref{mom-for-odd_poly} as $t_{i} \to 1\slash \sqrt{q},$ for $i = 1, \ldots, r.$ \end{proof}

When $g = 1, 2,$ one can use Theorem (Birch-Ihara), and the results in \cite{Har}, \cite{Peter} 
to express the traces of Frobenius on $e_{c}(\mathscr{A}_{g} \otimes \mathbb{F}, \mathbb{V}(\lambda))$ in terms of traces of Hecke operators on spaces of Siegel (vector-valued) modular forms. We give here the simplest version of this trace comparison in the following 
corollary.

\vskip10pt
\begin{corollary} --- Notation being as in Section \ref{section 5}, if $g = 1$ we have 
\begin{equation*}
\frac{1}{q (q - 1)} \sum_{\substack{\deg d = 3\\ 
\text{$d$ \rm{square-free}}}}   
L({\scriptstyle \frac{1}{2}}, \chi_{d})^{r} 
\, =  \sum_{\substack{\lambda  = 0\\ 
\text{$\lambda$-\rm{even}}}}^{r} (\dim\, V_{ \lambda^{\dag}}) 
(- 1 - \text{T}_{\lambda + 2}(q)) \, q^{-\frac{\lambda}{2}}
\end{equation*}
where we take $\text{T}_{\lambda + 2}(q) = - q - 1$ if $\lambda = 0.$
\end{corollary} 

\begin{proof} \hskip-1ptExpressing the traces of Frobenius in 
\eqref{eq: moments-odd-sqfree} in terms of moment-sums, our assertion 
follows at once by combining Corollary \ref{special case mom-for-odd_poly} 
and Theorem (Birch-Ihara). \end{proof} 

\vskip10pt
\begin{remark} By \cite[\S 24.2, Exercise 24.20]{FH}, the dimension of the irreducible 
representation $V_{ \lambda^{\dag}}$ is
\begin{equation*}
\dim \, V_{\lambda^{\dag}} 
= \frac{\prod_{i = 1}^{r}(g + r - \lambda_{i}' - i + 1) \cdot  
\prod_{1\le i < j \le r}(\lambda_{j}' - \lambda_{i}' + j - i)(2g + 2r + 2 - \lambda_{i}' - \lambda_{j}' - i - j)} {1! \, \cdots \, (2r - 3)! (2r - 1)!}.
\end{equation*} 
In particular, if $\lambda = {\bf{0}}$ we can write 
\begin{equation*}
\dim \, V_{{\bf{0}}^{\dag}} 
= \prod_{k = 1}^{r}\frac{k!}{(2 k)!}\cdot \prod_{i = 1}^{r}(2 g + 2 i)
\prod_{1 \, \le \, i \, < \, j \, \le \, r} (2g + i + j).
\end{equation*} 
Note that the first product is precisely the constant $g_{r}\slash \big(r (r + 1)\slash 2 \big)!$ appearing in the moment-conjectures of Conrey-Farmer \cite{ConrF} and Keating-Snaith \cite{KeS} for the leading-order asymptotics of the moments of L-functions within symplectic families. Moreover, if we let 
\begin{equation*}
P_{r}(x) = \prod_{i = 1}^{r}(x + 2 i) \prod_{1 \, \le \, i \, < \, j \, \le \, r} (x + i + j)
\end{equation*} 
then the polynomial $P_{r}(x - 1)$ appears in the main term of the polynomial 
$Q_{r}(x)$ conjectured by Conrey, Farmer, Keating, Rubinstein and Snaith \cite{CFKRS}; 
see also the recent work of Andrade-Keating \cite[Conjecture 5]{AnK} in the function-field setting. In particular, 
$P_{1}(x - 1) = 1 + x,$ 
$P_{2}(x - 1) = 6 + 11 x + 6 x^{2} + x^{3}$ 
and 
$$
P_{3}(x - 1) = 360 + 942 x + 949 x^{2} + 480 x^{3} + 130 x^{4} + 18 x^{5} + x^{6};  
$$ 
compare this to eqns. \!(5.16), (5.21) and (5.26), respectively, in \cite{AnK}. 

By counting points over $\mathbb{F},$ one finds easily that  
$
\Tr(F^{*} \vert \,
e_{c}^{{\bf{0}}}(\mathscr{H}_{g}(w^{1})  \otimes_{_{\mathbb{F}}}  \! \overline{\mathbb{F}}))\,
= q^{2g - 1}, 
$ 
and therefore
\begin{equation*}
(\dim \, V_{{\bf{0}}^{\dag}})\Tr(F^{*} \vert \,
e_{c}^{{\bf{0}}}(\mathscr{H}_{g}(w^{1})  \otimes_{_{\mathbb{F}}} \! \overline{\mathbb{F}}))\,
= \frac{g_{r} P_{r}(2 g)}{\big(r (r + 1)\slash 2 \big)!} \, q^{2g - 1}.
\end{equation*} 
If $|\lambda| \ne  0,$ we have the well-known result \cite[Theorem 10.8.2]{KaS} of Katz and Sarnak: there exist positive constants $A(g)$ and $C(g)$ such that we have the estimate 
\begin{equation*}
\frac{1}{q (q - 1)} \Bigg| \sum_{\substack{\deg d = 2g + 1 \\ \text{$d$ \rm{monic \& square-free}}}} s_{\scriptscriptstyle \langle \lambda \rangle}(\omega_{1}(C_{d})^{\pm 1}\!, \ldots, \omega_{g}(C_{d})^{\pm 1}) \Bigg| \, \le \,  2 C(g)(\dim \, V_{\lambda})q^{2g - \frac{3}{2}} \qquad \text{(if $q  \ge  A(g)$)}.
\end{equation*} 
The (normalized) sum in the left-hand side is precisely 
$
q^{- |\lambda|\slash 2}
\Tr(F^{*} \vert \, e_{c}^{\lambda}(\mathscr{H}_{g}(w^{1}) \otimes  \overline{\mathbb{F}}));
$ 
this follows from \eqref{eq: moments-odd-sqfree} and Lemma \ref{BG}, or by combining the 
Grothendieck trace formula \cite{Groth1}, \cite{Groth2} (see also the end of Section \ref{section 9}) with \cite[Theorem 10.1.18.3]{KaS}. 
\end{remark}

Of great importance for us is that Theorem \ref{q->1/q} allows us to make the following definition:

\vskip10pt
\begin{defn} \label{Definition q->1/q}\hskip-1ptFor any partition $\mu$ of $2g + 2,$ we define 
$
q^{\dim \, \mathscr{H}_{g}[2]}S_{\mu}(q\, T, 1\slash q) = q^{2 g - 1}S_{\mu}(q\, T, 1\slash q)
$ 
by simply replacing the compactly supported cohomology
in Theorem \ref{q->1/q} with $H_{\mu}^{i}(\mathscr{H}_{g}[2] \otimes \overline{\mathbb{F}}, \mathbb{V}'(\lambda)\otimes \mathbb{Q}_{\ell}(|\lambda| \slash 2)),$ that is
\begin{equation*}
q^{2 g - 1}S_{\mu}(q\, T, 1\slash q) = 
(t_{1} \, \cdots \, t_{r})^{g}
\!\!\sum_{\substack{\lambda \subseteq (r^{g}) \\ \text{$|\lambda|$-even}}}
\Tr(\Phi_{q}^{-1}\,  \vert \, e_{\mu}(\mathscr{H}_{g}[2] \otimes \overline{\mathbb{F}}, \mathbb{V}'(\lambda) \otimes \mathbb{Q}_{\ell}(|\lambda| \slash 2)))\,
s_{\scriptscriptstyle \langle \lambda^{\dag} \rangle}(q^{\pm \frac{1}{2}} t_{1}^{\pm 1}\!, \ldots, q^{\pm \frac{1}{2}} t_{r}^{\pm 1})\, q^{\frac{gr}{2}}.
\end{equation*} 
Here $\Phi_{q}$ is the arithmetic Frobenius endomorphism relative to $\mathbb{F}$ acting, by {\it transport of structures}, on the $\ell$-adic cohomology.
\end{defn}

Note that this definition is made according to the duality between cohomology with compact support and ordinary cohomology. We recall that for a Deligne-Mumford stack $\mathscr{X},$ assumed to be smooth over some base $S$ and purely $d$-dimensional, a map $f : \mathscr{X} \to S$ and a $\mathbb{Q}_{\ell}$-local system $\mathcal{F}$ on $\mathscr{X}$ (with $\ell$ invertible on our base), the Poincar\'e-Verdier duality gives a natural isomorphism (see \cite{BvdG}):  
\begin{equation*}
R{\mathcal Hom}_{S}(R f_{!} \mathcal{F}, \mathbb{Q}_{\ell}) \cong
Rf_{*}(\mathcal{F}^{\vee}(-d)[-2d]).
\end{equation*} 
In the next section, we shall need to work with the sums of products of characteristic 
polynomials $A_{\nu}(T, q)$ defined for partitions $\nu = (1^{\nu_{1}}\!, 2^{\nu_{2}}\!, \ldots)$ of $2g + 1$ or $2g + 2$ by
\begin{equation*}
A_{\nu}(T, q) \; = \sum_{d \, \in \, \mathscr{P}(\nu)} 
\bigg(\prod_{k = 1}^{r} P_{\scriptscriptstyle C_{d}}(t_{k}) \bigg). 
\end{equation*}
Here 
$
\mathscr{P}(\nu) =  \mathscr{P}(\nu, q) \subset \mathbb{F}[x]
$ 
stands, as before, for the set of all monic square-free polynomials with factorization type $\nu.$ Since the moment-sums $M_{\mu, \gamma}(q)$ can be expressed in terms of $\tilde{M}_{\nu, \gamma}(q)$ (by Propositions \ref{Prop1-appendixB} and \ref{Prop2-appendixB} in Appendix \ref{B}), we can express $A_{\nu}(T, q),$ analogously, as a sum of traces of Euler characteristics. (This will be made precise in \ref{Decomp Al(T, q) in q-Weil numbers}.) 
Accordingly, we can define $A_{\nu}(q \, T, 1\slash q)$ similarly.

\section{The main theorem}\label{section 7} 
We recall that our main goal is to construct inductively a sequence 
of generating series $(\Lambda_{l}(T, q))_{l\ge 2}$ satisfying \eqref{eq: tag 4.1}. 
In particular, for every fixed $l\ge 2,$ the generating series 
\begin{equation*}
A_{l}(T, q) \;\;= \sum_{(\mu, \nu) \ne ((1), (l))} A_{\mu, \nu}(T, q)
\end{equation*}
with $A_{\mu, \nu}(T, q)$ as in Proposition \ref{Proposition 4.3} and \eqref{eq: tag 4.4}, 
has to satisfy the functional equation
\begin{equation}\label{eq: func-eq A_{l}(T, q)} 
E(T)^{\epsilon_{l}} A_{l}(T, q) 
+ q^{l + 1}E(q\, T)^{\epsilon_{l}} A_{l}(q\, T, 1\slash q) = 0
\end{equation} 
with $\epsilon_{l} = 0$ or $1$ according as $l$ is odd or even. 

To see how \eqref{eq: func-eq A_{l}(T, q)} can be explained, consider 
\begin{equation*} 
q^{l + 1}E(q\, T)^{\epsilon_{l}} A_{l}(q\, T, 1\slash q) 
\,=\, q^{l + 1}E(q\, T)^{\epsilon_{l}}\; \cdot
\sum_{(\mu, \nu) \ne ((1), (l))} A_{\mu, \nu}(q\, T, 1\slash q)
\end{equation*} 
and express $A_{\mu, \nu}(q\, T, 1\slash q),$ for each 
$(\mu, \nu),$ by \eqref{eq: tag 4.4} (or by the formula in Proposition \ref{Proposition 4.3}). 
The other term in \eqref{eq: func-eq A_{l}(T, q)} has a 
similar expression, and it is related to that of $E(q\, T)^{\epsilon_{l}} A_{l}(q\, T, 1\slash q)$ by replacing $q \to 1\slash q$ and $T\to q\, T.$
Assuming for the moment that the identity \eqref{eq: tag 4.1} holds for all $2 \le l' < l,$ 
and applying it to transform each factor $\Lambda_{\nu_{j}}\!(\pm T^{\mu_{j}}\!, q^{\mu_{j}})$ 
of each term in $q^{l + 1}E(q\, T)^{\epsilon_{l}} A_{l}(q\, T, 1\slash q)$ into 
$\Lambda_{\nu_{j}}\!(\pm q^{\mu_{j}}T^{\mu_{j}}\!, 1\slash q^{\mu_{j}}),$ one finds that 
\begin{equation*}
E(T)^{\epsilon_{l}} A_{l}(T, q) 
+ q^{l + 1}E(q\, T)^{\epsilon_{l}} A_{l}(q\, T, 1\slash q)
\end{equation*} 
can be expressed as a sum over monomials of the form     
\begin{equation} \label{eq: monomials1}
\prod_{\substack{j \, \in \, J_{0} \\ \nu_{j} \ne 1}} 
\Lambda_{\nu_{j}}\!(q^{\mu_{j}}T^{\mu_{j}}\!, 1\slash q^{\mu_{j}}) \; \cdot  
\prod_{j \, \in \, J_{1}} \frac{\Lambda_{\nu_{j}}\!(\delta_{j} q^{\mu_{j}}T^{\mu_{j}}\!, 1\slash q^{\mu_{j}})}{E(\delta_{j} T^{\mu_{j}})}.  
\end{equation} 
(We are omitting here the factors corresponding to all $j\in J_{0}$ for which $\nu_{j} = 1$ as 
$q \Lambda_{1}(q\, T, 1\slash q) = 1$ by \eqref{eq: tag 3.9}.) 
Henceforth, we shall slightly abuse notation and continue to use $\Lambda_{\delta, \mu, \nu}(T, q)$ to denote the monomial \eqref{eq: monomials1}. If we put 
\begin{equation*}
\Sigma_{\delta, \mu, \nu}(T, q; l): = \text{Coefficient}_{\Lambda_{\delta, \mu, \nu}(T, q)} 
\left[E(T)^{\epsilon_{l}} A_{l}(T, q) +  q^{l + 1}E(q\, T)^{\epsilon_{l}} A_{l}(q\, T, 1\slash q) \right] 
\end{equation*} 
then we will show that
\begin{equation} \label{eq: MDS-relations}
\Sigma_{\delta, \mu, \nu}(T, q; l) =  0
\end{equation} 
for all $\delta, \mu, \nu$ and $l \ge 2.$ It is clear that this strengthening of \eqref{eq: func-eq A_{l}(T, q)} makes no reference whatsoever to any of the generating series $\Lambda_{l}(T, q)$ ($l\ge 2$). Simply put, showing \eqref{eq: MDS-relations} is a completely independent problem. 
In this section and the next, we shall see how the family of identities 
\eqref{eq: MDS-relations} is encoded in the combinatorial structure of moduli spaces of admissible double covers. Once this is established, we shall apply Deligne's theory of weights 
to construct the sequence $(\Lambda_{l}(T, q))_{l\ge 2}.$

\subsection{A Special Case} \label{sub-section-7.1-split-case}
For greater clarity, let us begin by first investigating the identities \eqref{eq: MDS-relations} corresponding to partitions of the form $(1^{n}),$ that is, corresponding to 
\begin{equation} \label{eq: sigma-MDS-special}
\Sigma_{n}(T, q): = \text{Coefficient}_{\Lambda_{3}(q T, 1\slash q)^{n}} 
\left[E(T)^{\epsilon_{n}} A_{3 n}(T, q) +  q^{3 n + 1}E(q\, T)^{\epsilon_{n}} A_{3 n}(q\, T, 1\slash q) \right] \qquad \text{(for $n \ge 2$)}. 
\end{equation} 
Although the main ideas involved in the general case are essentially the same, the overall discussion of it may seem quite technical.

For $n \ge 1,$ let $\mathscr{P}_{\! n}\subset \mathbb{F}_{\! q}[x]$ denote the set 
of all monic square-free polynomials of degree $n$ splitting in $\mathbb{F}_{\! q}.$ 
For $i, j \ge 0$ and $d \in \mathscr{P}_{\! n},$ define 
\begin{equation*}
N_{i, j}(d, q)  \,=\,  \prod_{k = 0}^{i - 1} \bigg(\frac{q - n + a_{1}(C_{d}) + \epsilon - 2k}{2} \bigg) \prod_{l = 0}^{j - 1}  \bigg(\frac{q - n - a_{1}(C_{d}) -\epsilon  - 2l}{2} \bigg) 
\end{equation*}
where $\epsilon = 0$ or $1$ according as $n$ is odd or even. Here if $i = 0$ or $j = 0,$ 
we take the corresponding product to be $1.$ (Recall that 
$
a_{1}(C_{d}) = \Tr(F^{*} \vert \, H_{\text{\'et}}^{1}(\bar{C}_{d},\mathbb{Q}_{\ell}))
$ 
with $\ell$ a prime different from the characteristic of $\mathbb{F}_{\! q}.$) Set 
\begin{equation*}
A_{n, i, j}(T, q) \; := \sum_{d \, \in \, \mathscr{P}_{\! n}} 
\bigg(N_{i, j}(d, q) \prod_{k = 1}^{r} P_{\scriptscriptstyle C_{d}}(t_{k}) \bigg).
\end{equation*}
Note that if we denote $^{D}\!\!A_{n, 0, 0}(T, q) = \frac{\partial A_{n, 0, 0}}{\partial t_{r+1}} (t_{1}, \ldots, t_{r}, 0 , q),$ we can write 
\begin{equation*}
^{D}\!\!A_{n, 0, 0}(T, q) \, =  - \sum_{d \, \in \, \mathscr{P}_{\! n}} 
\bigg(a_{1}(C_{d}) \prod_{k = 1}^{r} P_{\scriptscriptstyle C_{d}}(t_{k}) \bigg).
\end{equation*}
Define $^{D^{k}}\!\!\!A_{n, 0, 0}(T, q)$ for $k\in \mathbb{N}$ by iterating. Then, it is clear that 
\begin{equation*}
A_{n, i, j}(T, q)  \, =\,  ^{\scriptscriptstyle i! j! \binom{(q - n + \epsilon - D)\slash 2}{i}\binom{(q - n -\epsilon + D)\slash 2}{j}}\!A_{n, 0, 0}(T, q) \qquad (\text{for $n\ge 1$}).
\end{equation*}
Here the two binomial symbols are viewed as differential polynomials in $D.$ Note that the right-hand side makes sense if we replace $T$ by $q\, T$ and $q$ by $1\slash q,$ allowing us to define $A_{n, i, j}(q  \, T, 1\slash q).$

For $x, y$ and $z$ algebraically independent variables, consider the (exponential) generating functions 
\begin{equation*}
c_{\text{odd}}(x, y, z; T, q) \; =
\sum_{n, i, j \, \ge \, 0} \frac{A_{2n + 1, i, j}(T, q) \, x^{2n + 1} y^{i}  z^{j}}{i! j! \tilde{E}(-T)^{i} \tilde{E}(T)^{j}}
\end{equation*}
and
\begin{equation*}
c_{\text{even}}(x, y, z; T, q) \, =\, 
\tilde{E}(T) \big[\big(1 + z\slash \tilde{E}(T) \big)^{q} - 1\big]  \,+\,
\sum_{n \, \ge \, 1} \sum_{i, j \, \ge \, 0} \frac{A_{2n, i, j}(T, q) \, x^{2n} y^{i}  z^{j}}{i! j! \tilde{E}(-T)^{i} \tilde{E}(T)^{j}}
\end{equation*}
where $\tilde{E}(T): =  E(T) E(q\, T).$ Since $A_{2n + 1, i, j}(- T, q) = A_{2n + 1, j, i}(T, q)$ (see Lemma \ref{evenness-A}), we see that  
\begin{equation*}
c_{\text{odd}}(x, z, y; -T, q) = c_{\text{odd}}(x, y, z; T, q).
\end{equation*}
Let 
$c(\underline{x}, T, q) = (c_{\text{odd}}(x, y, z; T, q), c_{\text{even}}(x, z, y; -T, q), c_{\text{even}}(x, y, z; T, q)),$ where we set $\underline{x}: = (x, y, z).$

\vskip10pt
\begin{thm}\label{split case} --- With notation as above, we have 
\begin{equation*}
c(c(\underline{x}, T, q), q  \, T, 1\slash q) = \underline{x}.
\end{equation*}
In other words, $c(\underline{x}, T, q)$ is the formal compositional inverse of 
$c(\underline{x}, q\, T, 1\slash q).$
\end{thm} 

\begin{proof} Expand 
\begin{equation*}
c_{\text{odd}}(c(\underline{x}, T, q); q\, T, 1\slash q) \; = \sum_{n, i, j \, \ge \, 0} 
C_{2n + 1, i, j}(T, q) \, x^{2n + 1} y^{i}  z^{j}
\end{equation*} 
and 
\begin{equation*}
c_{\text{even}}(c(\underline{x}, T, q); q\, T, 1\slash q) \; = \sum_{n, i, j \, \ge \, 0} 
C_{2n, i, j}(T, q) \, x^{2n} y^{i}  z^{j}.
\end{equation*}
Note that $C_{1, 0, 0}(T, q) = A_{1, 0, 0}(q\, T, 1\slash q) A_{1, 0, 0}(T, q) = 1,$ and 
$C_{1, i, j}(T, q) = 0$ if either $i \ne 0$ or $j \ne 0.$ Similarly, $C_{0, 0, 1}(T, q) = 1,$ and $C_{0, i, j}(T, q) = 0$ otherwise. Moreover, one checks that 
\begin{equation*}
\Sigma_{N}(T, q) \, = \, q \, C_{N, 0, 0}(T, q)   \qquad \;\;\; (\text{for $N \ge 2$}) 
\end{equation*}
with $\Sigma_{N}(T, q)$ given by \eqref{eq: sigma-MDS-special}.

We shall first prove that if $C_{N, 0, 0}(T, q) = 0$ for an $N\ge 2,$ then 
$C_{N, I, J}(T, q) = 0$ for all $I, J \ge 0.$ Indeed, note that for any $N,$ we can write 
\begin{equation*}
C_{N, 0, 0}(T, q) = \sum_{\underline{\alpha}} c_{\underline{\alpha}} 
M_{\underline{\alpha}}(T, q) 
\end{equation*} 
where the sum is over tuples 
$
\underline{\alpha} = (\alpha_{1}^{}, \ldots, \alpha_{n}^{}, \alpha_{n+1}^{-}, \ldots, \alpha_{n+i}^{-},\alpha_{n+i+1}^{+}, \ldots, \alpha_{n+i+j}^{+})\in \mathbb{N}^{n+i+j}
$ 
($|\underline{\alpha}| = \alpha_{1}^{} + \cdots + \alpha_{n+i+j}^{+} = N$), 
with $\alpha_{1}^{}, \ldots, \alpha_{n}^{}$ all odd, and $\alpha_{k}^{\pm}$'s all even, and where 
\begin{equation*}
M_{\underline{\alpha}}(T, q) : = \bigg(\prod_{m = 1}^{n} A_{\alpha_{m}^{}, 0, 0}(T, q) \; \cdot 
\prod_{k = n+1}^{n+i} A_{\alpha_{k}^{-}, 0, 0}(-T, q) \; \cdot 
\prod_{l = n + i + 1}^{n+i+j} A_{\alpha_{l}^{+}, 0, 0}(T, q)\bigg) \, \frac{A_{n, i, j}(q\, T, 1\slash q)}{\tilde{E}(-T)^{i}   \tilde{E}(T)^{j}}.
\end{equation*} 
Here we allow $n = 0$ if we put 
\begin{equation*}
A_{0, i, j}(T, q) 
=\begin{cases}
\frac{q!}{(q - j)!} \tilde{E}(T) &\mbox{if } i = 0\; \text{and}\; j\ge 1 \\
0 & \mbox{if } \text{$i = j = 0$ or $i\ne 0$}. 
\end{cases}  
\end{equation*}
(Of course, $\underline{\alpha} = (\alpha_{1}^{+}, \ldots, \alpha_{j}^{+})$ in 
this case.) To express the coefficient $c_{\underline{\alpha}},$ it is rather  
convenient to represent  
$
\underline{\alpha} = (1^{\kappa_{1}}\!, 3^{\kappa_{3}}\!, 5^{\kappa_{5}}\!, \ldots; 
2^{\kappa_{2}^{-}}\!, 4^{\kappa_{4}^{-}}\!, 6^{\kappa_{6}^{-}}\!, \ldots; 
2^{\kappa_{2}^{+}}\!, 4^{\kappa_{4}^{+}}\!, 6^{\kappa_{6}^{+}}\!, \ldots).$ With this notation, 
one checks that 
\begin{equation*}
c_{\underline{\alpha}} =  \frac{\binom{n}{\kappa_1, \kappa_3, \kappa_5, \ldots}}
{\kappa_{2}^{-}!\, \kappa_{2}^{+}!\, \kappa_{4}^{-}!\, \kappa_{4}^{+}! \, \cdots}.
\end{equation*} 
Now assume that $C_{N, 0, 0}(T, q) = 0$ for some $N\ge 2.$ Pick arbitrary $I, J \ge 0,$ and express $C_{N, I, J}(T, q),$ as above, in terms of the coefficients $A_{n, i, j}.$ We can relate the expression of $C_{N, I, J}(T, q)$ to the corresponding expression of $C_{N, 0, 0}(T, q)$ in the following way. Every term, say, $c_{\underline{\tilde{\alpha}}} M_{\underline{\tilde{\alpha}}}(T, q),$ occurring in $C_{N, I, J}(T, q)$ is attached to 
a tuple 
$
\underline{\tilde{\alpha}} = (\alpha_{1}^{}, i_{1}^{}, j_{1}^{}; \ldots; \alpha_{n}^{}, i_{n}^{}, j_{n}^{}; \alpha_{n+1}^{-}, i_{n+1}^{}, j_{n+1}^{}; \ldots)
$ 
by  
\begin{equation*}
M_{\underline{\tilde{\alpha}}}(T, q) = \frac{A_{\alpha_{1}^{}, i_{1}^{}, j_{1}^{}}(T, q)}{i_{1}! \, j_{1}!\, \tilde{E}(-T)^{i_{1}} \tilde{E}(T)^{j_{1}}}  \cdots 
\frac{A_{\alpha_{n+1}^{-}, i_{n+1}^{}, j_{n+1}^{}}(-T, q)}{i_{n + 1}! \,  j_{n + 1}! \,  \tilde{E}(T)^{i_{n+1}} \tilde{E}(-T)^{j_{n+1}}} \cdots 
\frac{A_{n, k, l}(q\, T, 1\slash q)}{\tilde{E}(-T)^{k} \tilde{E}(T)^{l}}.
\end{equation*} 
Here $\alpha_{1}^{}, \ldots, \alpha_{n}^{}, \alpha_{n+1}^{-}, \ldots$ are exactly as before. Let $\underline{\alpha} = (\alpha_{1}, \alpha_{2}, \ldots, \alpha_{n}, \ldots)$ be obtained by taking the corresponding {\it non-zero} components of $\underline{\tilde{\alpha}}.$ We refer to 
$\underline{\tilde{\alpha}}$ as {\it derived} from $\underline{\alpha},$ and denote this by 
$\underline{\tilde{\alpha}}\slash \underline{\alpha}.$ It is clear that 
$c_{\underline{\alpha}} M_{\underline{\alpha}}(T, q)$ occurs in the expression of 
$C_{N, 0, 0}(T, q).$ Furthermore, one checks that the contribution of all 
terms $c_{\underline{\tilde{\alpha}}} M_{\underline{\tilde{\alpha}}}(T, q),$ 
with $\underline{\tilde{\alpha}}\slash \underline{\alpha}$ for a fixed $\underline{\alpha},$ 
in $C_{N, I, J}(T, q)$ is 
\begin{equation*}
\sum_{\underline{\tilde{\alpha}}\slash \underline{\alpha}} c_{\underline{\tilde{\alpha}}} 
M_{\underline{\tilde{\alpha}}}(T, q) \, = \,
\frac{c_{\underline{\alpha}} \cdot \, 
^{^{\scriptscriptstyle \binom{(1 + \epsilon q - N - D)\slash 2}{I} \binom{(1-\epsilon q - N + D)\slash 2}{J}}}\!M_{\underline{\alpha}}(T, q)}{\tilde{E}(-T)^{I} \tilde{E}(T)^{J}}
\end{equation*} 
with $\epsilon$ as above. The last identity holds even if $\underline{\alpha}$ is such that $n = 0.$ Summing now over all $\underline{\alpha},$ it follows that 
\begin{equation*}
C_{N, I, J}(T, q) \, = \,
\frac{^{^{\scriptscriptstyle \binom{(1 + \epsilon q - N - D)\slash 2}{I}\binom{(1-\epsilon q - N + D)\slash 2}{J}}}\!C_{N, 0, 0}(T, q)}{\tilde{E}(-T)^{I} \tilde{E}(T)^{J}} \,=\, 0.
\end{equation*} 
Now, suppose that $C_{n, i, j}(T, q) = 0$ for $2\le n \le N-1,$ and all $i, j\ge 0,$ and let us show that $C_{N, 0, 0}(T, q) = 0.$ (It is easy to check that $C_{2, 0, 0}(T, q) = 0,$ hence 
$C_{2, i, j}(T, q) = 0$ for $i, j \ge 0$ by what we just proved.) We shall assume that $N$ is odd since the argument is completely analogous if $N$ is even. Accordingly, we can write 
\begin{equation*}
c_{\text{odd}}(c(\underline{x}, T, q); q\, T, 1\slash q) \, = \, x \; +  \sum_{n \, \ge \, (N - 1)\slash 2} \; \sum_{i, j \, \ge \, 0} C_{2n + 1, i, j}(T, q) \, x^{2n + 1} y^{i}  z^{j}.
\end{equation*}
To this we apply $\bar{c}(x, T, q) : = (\bar{A}_{\text{odd}}(x, T, q), \bar{A}_{\text{even}}(x, - T, q), \bar{A}_{\text{even}}(x, T, q))$ given in Section \ref{section: The series BarC(X, T, q)}, and satisfying 
\begin{equation} \label{eq: functional-rel-special-split-case}
c(c(\bar{c}(x, T, q), T, q), q  \, T, 1\slash q) = \bar{c}(x, T, q).
\end{equation} 
This implies that 
\begin{equation*}
\sum_{n \, \ge \, (N - 1)\slash 2} \; \sum_{i, j \, \ge \, 0} 
C_{2n + 1, i, j}(T, q) \, \bar{A}_{\text{odd}}(x, T, q)^{2n + 1} \bar{A}_{\text{even}}(x, - T, q)^{i} \bar{A}_{\text{even}}(x, T, q)^{j} = 0. 
\end{equation*}
Taking the coefficient of $x^{N}\!,$ it follows that $C_{N, 0, 0}(T, q) = 0.$ 

This combined with the previous step completes an induction process and the proof of the theorem. \end{proof} 

In particular, we have 
\begin{equation*}
\Sigma_{N}(T, q) = 0 \qquad \;\; (\text{for $N \ge 2$}) 
\end{equation*}
which is just \eqref{eq: MDS-relations} corresponding to $(1^{N}).$ 

\subsection{The General Case} \label{Gen-Case} 
In what follows, we shall extend the above considerations in a way that will allow us to treat 
all partitions at once. Before doing so, let us fix some notation. 

{\it Notation.} Throughout this subsection $\mathfrak{n}, \mathfrak{i}, \mathfrak{j}, \ldots$ will denote {\it partitions} written as $\mathfrak{n} = (1^{n_{1}}\!, 2^{n_{2}}\!, \ldots),$ where for all but finitely many $j \ge 1,$ $n_{j} = 0.$ As usual, 
$|\mathfrak{n}| = n_{1} + 2 n_{2} + \cdots$ stands for the {\it weight} of $\mathfrak{n}.$ For every $\mathfrak{n} = (1^{n_{1}}\!, 2^{n_{2}}\!, \ldots),$ let $\mathscr{P}(\mathfrak{n}, q) = \mathscr{P}(\mathfrak{n}, \mathbb{F}_{\! q}) \subset \mathbb{F}_{\! q}[x]$ 
denote the set of all monic square-free polynomials $d$ with factorization type $\mathfrak{n}.$ Let 
\begin{equation} \label{eq: A_{n}(T, q)}
A_{\mathfrak{n}}(T, q) \; := \sum_{d \, \in \, \mathscr{P}(\mathfrak{n}, q)} 
\bigg(\prod_{k = 1}^{r} P_{\scriptscriptstyle C_{d}}(t_{k}) \bigg). 
\end{equation}
(If $n_{1}, n_{2}, \ldots$ are all zero, we denote the corresponding partition by ${\bf 0},$ and take $A_{\bf 0}(T, q)$ to be zero.) Note that $A_{\mathfrak{n}}(T, q)$ corresponds to the sum $A_{n, 0, 0}(T, q)$ introduced at the beginning of the previous subsection. In addition, 
for $\mathfrak{n} \ne {\bf{0}}$ as above, and $\mathfrak{i} = (1^{i_{1}}\!, 2^{i_{2}}\!, \ldots),$ $\mathfrak{j} = (1^{j_{1}}\!, 2^{j_{2}}\!, \ldots),$ we define $A_{\mathfrak{n}, \mathfrak{i}, \mathfrak{j}}(T, q)$ by
\begin{equation} \label{eq: A_{n, i, j}(T, q)}
A_{\mathfrak{n}, \mathfrak{i}, \mathfrak{j}}(T, q)\;\, := 
\sum_{d \, \in \, \mathscr{P}(\mathfrak{n}, q)}
\left(\prod_{m \, \ge \, 1}
\frac{\big(\frac{c_{m}^{}\!(d, q) \,+\, a_{m}^{*}\!(d, q)}{2 m} \big)! 
\big(\frac{c_{m}^{}\!(d, q) \,-\, a_{m}^{*}\!(d, q)}{2 m} \big)!}
{\big(\frac{c_{m}^{}\!(d, q) \,+\, a_{m}^{*}\!(d, q)}{2 m} -  j_{m} \big)!  
\big(\frac{c_{m}^{}\!(d, q) \,-\, a_{m}^{*}\!(d, q)}{2 m} - i_{m} \big)!}\,\cdot \,
\prod_{k = 1}^{r} P_{\scriptscriptstyle C_{d}}(t_{k}) \right)
\end{equation} 
where $c_{m}(d, q) = c_{m}(d)$ and $a_{m}^{*}(d, q) = a_{m}^{*}(d)$ are as 
defined before. It is understood that $A_{\mathfrak{n}, {\bf 0}, {\bf 0}}(T, q) = A_{\mathfrak{n}}(T, q).$ Notice that the quantity   
\begin{equation} \label{eq: c_{m}(d, q)}
c_{m}(d, q) \;  =  \sum_{\theta \, \in \, \mathbb{F}_{\! q^{m}}'} \chi_{m}(d(\theta))^{2} 
= m (\mathrm{Irr}_{q}(m) - n_{m})  \;\;\qquad \;\; (\text{for $m\ge 1$})
\end{equation}
does not depend on $d\in \mathscr{P}(\mathfrak{n}, q).$ We shall also need the following transform of $A_{\mathfrak{n}, \mathfrak{i}, \mathfrak{j}}(T, q)$ defined by 
\begin{equation*}
^{\iota}\!\!A_{\mathfrak{n}, \mathfrak{i}, \mathfrak{j}}(T, q)  \;\, = 
\sum_{d \, \in \, \mathscr{P}(\mathfrak{n}, q)}
\left(\prod_{m \, \ge \, 1}
\frac{\big(\frac{c_{m}^{}\!(d, q) \,+\, (-1)^{m}\!a_{m}^{*}\!(d, q)}{2 m}\big)! 
\big(\frac{c_{m}^{}\!(d, q) \, - \, (-1)^{m}\!a_{m}^{*}\!(d, q)}{2 m} \big)!}
{\big(\frac{c_{m}^{}\!(d, q) \,+\, (-1)^{m}\!a_{m}^{*}\!(d, q)}{2 m} -  j_{m} \big)!  
\big(\frac{c_{m}^{}\!(d, q) \, - \, (-1)^{m}\!a_{m}^{*}\!(d, q)}{2 m} - i_{m} \big)!} \cdot 
\prod_{k = 1}^{r} P_{\scriptscriptstyle C_{d}}(- t_{k}) \right).
\end{equation*}

\begin{lemma}
\label{evenness-A} --- If $|\mathfrak{n}|$ is odd, then
$
^{\iota}\!\!A_{\mathfrak{n}, \mathfrak{i}, \mathfrak{j}}(T, q) = A_{\mathfrak{n}, \mathfrak{i}, \mathfrak{j}}(T, q) 
$ 
for arbitrary partitions $\mathfrak{i}$ and $\mathfrak{j}.$
\end{lemma}

\begin{proof} For $\mathfrak{i} = (1^{i_{1}}\!, 2^{i_{2}}\!, \ldots),$ 
consider the moment-sum 
\begin{equation*}
M_{\mathfrak{n}, \mathfrak{i}}(q) \;\, 
= \sum_{d \, \in \, \mathscr{P}(\mathfrak{n}, q)} a^{(\mathfrak{i})}(C_{d})
\;\, = \sum_{d \, \in \, \mathscr{P}(\mathfrak{n}, q)} \prod_{j} \bigg( - 
\sum_{\theta_{j} \, \in \,  \mathbb{F}_{\! q^{j}}} 
\chi_{j}(d(\theta_{j})) \bigg)^{\! i_{j}}\!.
\end{equation*}
As before, by changing $\theta_{j}  \to \xi\theta_{j},$ with $\chi(\xi) = -1$ 
($\xi\in \mathbb{F}_{\! q}$), one sees that $M_{\mathfrak{n}, \mathfrak{i}}(q) = 0$ whenever 
$
|\mathfrak{i}| 
$ 
is odd. Expressing $A_{\mathfrak{n}}(T, q)$ in the form 
\begin{equation*}
A_{\mathfrak{n}}(T, q) \;\,  = \sum_{d \, \in \, \mathscr{P}(\mathfrak{n}, q)} 
\bigg(\prod_{k = 1}^{r} P_{\scriptscriptstyle C_{d}}(t_{k}) \bigg) \;\, =
\sum_{d \, \in \, \mathscr{P}(\mathfrak{n}, q)} 
\exp \bigg( - \sum_{j = 1}^{\infty} \frac{a_{j}(C_{d})}{j}
\cdot\sum_{k = 1}^{r}t_{k}^{j}  \bigg)
\end{equation*} 
it follows that 
$
A_{\mathfrak{n}}( - T, q) = A_{\mathfrak{n}}(T, q).
$ 
This identity holds for arbitrary $r.$ Accordingly, by adding more variables 
$t_{r + 1}, t_{r + 2}, \ldots, t_{r + \mathrm{s}}$ to $t_{1}, \ldots, t_{r},$ we see that \begin{equation} \label{eq: derivative}
(-1)^{\alpha_{1} + \cdots + \alpha_{\mathrm{s}}}\frac{\partial^{\alpha_{1} + \cdots + \alpha_{\mathrm{s}}}\!A_{\mathfrak{n}}}{\partial t_{r + 1}^{\alpha_{1}} \cdots \, 
\partial t_{r + \mathrm{s}}^{\alpha_{\mathrm{s}}}}( - t_{1},\ldots, - t_{r}, 0, \ldots, 0, q) 
= \frac{\partial^{\alpha_{1} + \cdots + \alpha_{\mathrm{s}}}\!A_{\mathfrak{n}}}{\partial t_{r + 1}^{\alpha_{1}} \cdots \, \partial t_{r + \mathrm{s}}^{\alpha_{\mathrm{s}}}}(t_{1},\ldots, t_{r}, 0,\ldots, 0, q) 
\end{equation}
for arbitrary $\alpha_{1}, \ldots, \alpha_{\mathrm{s}} \in \mathbb{N}.$ Dividing both sides by 
$\alpha_{1}! \, \cdots \, \alpha_{\mathrm{s}}!,$ and expressing 
\begin{equation*}
P_{\scriptscriptstyle C_{d}}(t) = 1 + \sum_{\alpha = 1}^{2g} (-1)^{\alpha}\lambda_{\alpha}(C_{d})t^{\alpha}
\end{equation*} 
we deduce that 
\begin{equation*}
(- 1)^{\alpha_{1} + \cdots + \alpha_{\mathrm{s}}} \; \cdot  \sum_{d \, \in \,  \mathscr{P}(\mathfrak{n}, q)} \bigg(\lambda_{\alpha_{1}}\!(C_{d}) \, \cdots \,  \lambda_{\alpha_{\mathrm{s}}}\!(C_{d})\prod_{k = 1}^{r} P_{\scriptscriptstyle C_{d}}( - t_{k})  \bigg) \;\; =  \sum_{d \, \in \, \mathscr{P}(\mathfrak{n}, q)} 
\bigg(\lambda_{\alpha_{1}}\!(C_{d}) \, \cdots \, \lambda_{\alpha_{\mathrm{s}}}\!(C_{d}) \prod_{k = 1}^{r} P_{\scriptscriptstyle C_{d}}(t_{k}) \bigg).
\end{equation*} 
Now recall that we expressed the quantities $a_{m}^{*}(d, q)$ by 
\begin{equation} \label{eq: a_{m}^{*}}
a_{m}^{*}(d, q) \,= \,  - \sum_{\substack{k \, \mid \, m \\ m\slash k-\text{odd}}} 
\mu(m\slash k) a_{k}(C_{d}) \; -  
\sum_{u = 1}^{\nu_{2}(m)} c_{m \slash 2^{u}}(d, q)  \qquad (\text{for $m\ge 1$ and $d\in \mathscr{P}(\mathfrak{n}, q)$})
\end{equation}
where $\nu_{2}$ denotes the $2$-adic order. By applying Girard's formula (more commonly known as Girard-Waring formula),
\begin{equation} \label{eq: a_{k}}
a_{k}(C_{d}) = k \sum \frac{(-1)^{l_{1} + \cdots + l_{k} +  k} 
(l_{1} + \cdots + l_{k} - 1)!}{l_{1}! \, \cdots \, l_{k}!}\lambda_{1}^{l_{1}}\!(C_{d}) \, \cdots \, \lambda_{k}^{l_{k}}\!(C_{d})
\end{equation}
summed over all nonnegative integers $l_{1}, \ldots, l_{k}$ such that 
$l_{1} + 2 l_{2} + \cdots + k l_{k} = k,$ it follows that 
\begin{equation*}
\sum_{d \, \in \, \mathscr{P}(\mathfrak{n}, q)} 
(- 1)^{m}a_{m}^{*}(d, q) \cdot \prod_{k = 1}^{r} P_{\scriptscriptstyle C_{d}}( - t_{k})  \;\; = \, \sum_{d \, \in \, \mathscr{P}(\mathfrak{n}, q)} 
a_{m}^{*}(d, q) \cdot \prod_{k = 1}^{r} P_{\scriptscriptstyle C_{d}}(t_{k})  \;\; \qquad \;\; 
(\text{for $m \ge 1$}).
\end{equation*} 
This identity extends to 
\begin{equation*}
\sum_{d \, \in \, \mathscr{P}(\mathfrak{n}, q)} 
H(-a_{1}^{*}(d, q), a_{2}^{*}(d, q), \ldots) \prod_{k = 1}^{r} P_{\scriptscriptstyle C_{d}}( - t_{k})  \;\; = \, 
\sum_{d \, \in \, \mathscr{P}(\mathfrak{n}, q)} 
H(a_{1}^{*}(d, q), a_{2}^{*}(d, q), \ldots) \prod_{k = 1}^{r} P_{\scriptscriptstyle C_{d}}(t_{k})
\end{equation*} 
for arbitrary polynomial expressions $H$ in $a_{1}^{*}(d, q), a_{2}^{*}(d, q), \ldots,$ 
and the lemma follows.
\end{proof}

\vskip10pt
\begin{remark} We can express $A_{\mathfrak{n}, \mathfrak{i}, \mathfrak{j}}(T, q) = \,
^{\scriptscriptstyle \mathcal{D}_{\mathfrak{n}, \mathfrak{i}, \mathfrak{j}, q}}\!A_{\mathfrak{n}}(T, q)\footnote{The differential operator 
$\mathcal{D}_{\mathfrak{n}, \mathfrak{i}, \mathfrak{j}, q}$ 
is obtained as follows. For $l_{1}, \ldots, l_{k} \in \mathbb{N},$ let 
$$
\partial_{l_{1}\!, \ldots, l_{k}} = \Big(\frac{1}{1!^{l_{1}} 2!^{l_{2}} \, \cdots \, k!^{l_{k}}} \Big) \cdot
\frac{\partial^{l_{1} + 2l_{2} + \cdots + k l_{k}}}{\partial t_{r + 1}^{} \cdots \, \partial t_{r + l_{1}}^{}\partial t_{r + l_{1} + 1}^{2} \cdots \, \partial t_{r + l_{1} + l_{2}}^{2} \cdots \, \partial t_{r + l_{1} + \cdots + l_{k - 1} + 1}^{k} \cdots \,  
\partial t_{r + l_{1} + \cdots + l_{k}}^{k}}.
$$ 
Define 
$$
\mathcal{D}_{m}^{*} \;\, = 
\sum_{\substack{k \, \mid \, m \\ (m\slash k)-\text{odd}}} 
\mu(m\slash k)\,  k \sum \frac{(-1)^{l_{1} + \cdots + l_{k} + 1}
(l_{1} + \cdots + l_{k} - 1)!}{l_{1}! \, \cdots \, l_{k}!}\,\partial_{l_{1}\!, \ldots, l_{k}} 
$$
the inner sum being over all $l_{1}, \ldots, l_{k}\in \mathbb{N}$ satisfying 
$l_{1} + 2 l_{2} + \cdots + k l_{k} = k.$ By \eqref{eq: c_{m}(d, q)}, \eqref{eq: a_{m}^{*}} and \eqref{eq: a_{k}}
$$
\sum_{d \, \in \, \mathscr{P}(\mathfrak{n}, q)} 
a_{m}^{*}(d, q) \, \cdot \prod_{k = 1}^{r} P_{\scriptscriptstyle C_{d}}(t_{k}) 
\;= \, 
^{\scriptscriptstyle \mathcal{D}_{m}^{*}}\!A_{\mathfrak{n}}(T, q)
 \,- \, \Big[\rho_{m} (1 + (-1)^{|\mathfrak{n}|})\slash2 \, + \, 
m \; \cdot \sum_{u = 1}^{\nu_{2}(m)} (\mathrm{Irr}_{q}(m \slash 2^{u}) - n_{m \slash 2^{u}})\slash 2^{u}\Big]A_{\mathfrak{n}}(T, q)
$$
with $\rho_{m} \! = 1$ or $0$ according as $m$ is a power of two or not. 
Then 
$
^{\scriptscriptstyle \mathcal{D}_{\mathfrak{n}, \mathfrak{i}, \mathfrak{j}, q}}\!A_{\mathfrak{n}}(T, q)
$ 
is obtained by replacing $a_{m}^{*}(d, q),$ for each $m,$ by  
$$
\mathcal{D}_{m}^{*} 
- \rho_{m}(1 + (-1)^{|\mathfrak{n}|})\slash2 
\, - \, m \; \cdot \sum_{u = 1}^{\nu_{2}(m)}(\mathrm{Irr}_{q}(m \slash 2^{u}) - n_{m \slash 2^{u}})\slash 2^{u}
$$ 
in \eqref{eq: A_{n, i, j}(T, q)}.}\!, 
$ an identity which allows us to define 
$A_{\mathfrak{n}, \mathfrak{i}, \mathfrak{j}}(q\, T, 1\slash q).$ Clearly
$
{^\iota}\!\!A_{\mathfrak{n}, \mathfrak{i}, \mathfrak{j}}(q\, T, 1\slash q)
= A_{\mathfrak{n}, \mathfrak{i}, \mathfrak{j}}(q\, T, 1\slash q) 
$ 
when $|\mathfrak{n}|$ is odd.
\end{remark}

\vskip5pt
Let $X_{1}, \ldots, X_{n}, \ldots, Y_{1}, \ldots, Z_{1}, \ldots$ be algebraically independent variables. Let 
$X = (X_n)_{n \ge 1}, Y = (Y_n)_{n \ge 1},$ $Z = (Z_n)_{n\ge 1},$ and set 
$\underline{X} := (X, Y, Z).$ For $k \ge 1$ and a partition $\mathfrak{n},$ let 
$k \mathfrak{n}$ and 
$X^{k \mathfrak{n}}$ denote, respectively, the partition $(k^{n_{1}}\!, (2k)^{n_{2}}\!, \ldots)$ and the product $X_{k}^{n_{1}} X_{2k}^{n_{2}} \, \cdots.$ 

Now, for $k \ge 1,$ let $\tilde{E}(T^{k}\!, q^{k})  =  E(T^{k}) E(q^{k} T^{k}),$ 
where we recall that 
$
E(T) = \prod_{i = 1}^{r}\, (1 - t_{i})^{-1}.
$  
Define the generating series
\begin{equation*}
C_{k, \text{odd}}(\underline{X}, T, q) \;\, =
\sum_{\substack{\mathfrak{n}, \mathfrak{i}, \mathfrak{j} \\  |\mathfrak{n}|-\text{odd}}}  
\frac{A_{\mathfrak{n}, \mathfrak{i}, \mathfrak{j}}(T^{k}\!, q^{k}) \, X^{k \mathfrak{n}} Y^{k\mathfrak{i}} Z^{k\mathfrak{j}}}{\mathfrak{i}! \mathfrak{j}! E^{^{-}}\!(T^{k}\!, q^{k})^{\mathfrak{i}} E^{^{+}}\!(T^{k}\!, q^{k})^{\mathfrak{j}}}
\end{equation*}  
\begin{equation*}
C_{k, \text{even}}^{+}(\underline{X}, T, q) \;\, =
\sum_{\substack{\mathfrak{n}, \mathfrak{i}, \mathfrak{j}\\  |\mathfrak{n}|-\text{even}}}  \frac{A_{\mathfrak{n}, \mathfrak{i}, \mathfrak{j}}(T^{k}\!, q^{k}) \, X^{k  \mathfrak{n}} 
Y^{k \mathfrak{i}} Z^{k \mathfrak{j}}}{\mathfrak{i}! \mathfrak{j}! E^{^{-}}\!(T^{k}\!, q^{k})^{\mathfrak{i}} E^{^{+}}\!(T^{k}\!, q^{k})^{\mathfrak{j}}}
\end{equation*} 
and 
\begin{equation*}
C_{k, \text{even}}^{-}(\underline{X}, T, q) \;\, = 
\sum_{\substack{\mathfrak{n}, \mathfrak{i}, \mathfrak{j}\\  |\mathfrak{n}|-\text{even}}}  \frac{{^\iota}\!\!A_{\mathfrak{n}, \mathfrak{i}, \mathfrak{j}}(T^{k}\!, q^{k}) \, X^{k  \mathfrak{n}} Y^{k \mathfrak{i}}  
Z^{k \mathfrak{j}}}{\mathfrak{i}! \mathfrak{j}! E^{^{-}}\!(T^{k}\!, q^{k})^{\mathfrak{i}} E^{^{+}}\!(T^{k}\!, q^{k})^{\mathfrak{j}}};
\end{equation*}
here $A_{{\bf 0}, {\mathfrak{i}}, \mathfrak{j}}(T, q)$ and ${^\iota}\!\!A_{{\bf 0}, {\mathfrak{i}}, \mathfrak{j}}(T, q)$ are taken such that: 
\begin{equation} \label{eq: f0}
f_{{\bf 0}}^{}(Y, Z, T, q) \, : = \, 
\sum_{\mathfrak{i}, \mathfrak{j}} A_{{\bf 0}, {\mathfrak{i}}, \mathfrak{j}}(T, q)\frac{Y^{\mathfrak{i}} Z^{\mathfrak{j}}}{\mathfrak{i}! \mathfrak{j}!}  \, = \,  E(T) E(q\,T)
\left[ - 1 + \prod_{m = 1}^{\infty} (1 + Z_{m})^{\mathrm{Irr}_{q}(m)} \right] 
\end{equation} 
and 
\begin{equation} \label{eq: iotaf0}
\sum_{\mathfrak{i}, \mathfrak{j}}\, 
{^\iota}\!\!A_{{\bf 0}, {\mathfrak{i}}, \mathfrak{j}}(T, q)\frac{Y^{\mathfrak{i}} Z^{\mathfrak{j}}}{\mathfrak{i}! \mathfrak{j}!}  \, = \,  E(-T) E(- q\,T)
\left[ - 1 \, + \prod_{m-\text{odd}} (1 + Y_{m})^{\mathrm{Irr}_{q}(m)} \, \cdot 
\prod_{m-\text{even}} (1 + Z_{m})^{\mathrm{Irr}_{q}(m)} \right].
\end{equation} 
Also, $E^{-}(T^{k}\!, q^{k})^{\mathfrak{i}}$ and $E^{+}(T^{k}\!, q^{k})^{\mathfrak{j}}$ 
are defined by 
\begin{equation*}
E^{-}(T^{k}\!, q^{k})^{\mathfrak{i}}  = \tilde{E}( - T^{k}\!, q^{k})^{i_{1}} \cdots \, 
\tilde{E}( - T^{m k}\!, q^{m k})^{i_{m}} \cdots 
\end{equation*}
and 
\begin{equation*}
E^{+}(T^{k}\!, q^{k})^{\mathfrak{j}}  = \tilde{E}(T^{k}\!, q^{k})^{j_{1}} \cdots \, 
\tilde{E}(T^{m k}\!, q^{m k})^{j_{m}} \cdots. 
\end{equation*}
Let $C_{\text{odd}}(\underline{X}, T, q) = (C_{k, \text{odd}}(\underline{X}, T, q))_{k\ge 1},$ 
$C_{\text{even}}^{\pm}(\underline{X}, T, q) = (C_{k, \text{even}}^{\pm}(\underline{X}, T, q))_{k\ge 1},$ and set 
\begin{equation*}
C(\underline{X}, T, q) : = (C_{\text{odd}}(\underline{X}, T, q), C_{\text{even}}^{-}(\underline{X}, T, q), 
C_{\text{even}}^{+}(\underline{X}, T, q)).
\end{equation*} 
Note that by a passage from $q$ to $q^{k}$ we understand a base change to 
$\mathbb{F}_{\! q^{k}}.$ Accordingly, $C_{k, \text{odd}}(\underline{X}, T, q),$ 
for example, is just $C_{1, \text{odd}}(\underline{X}, T, q^{k})$ in which we replace 
$$
t_{1}, \ldots, t_{r}, \ldots, X_{m}, \ldots, Y_{m}, \ldots, Z_{m}, \ldots
$$ 
by 
$$
t_{1}^{k}, \ldots, t_{r}^{k}, \ldots, X_{k m}, \ldots, Y_{k m}, \ldots, Z_{k m}, \ldots,
$$ 
respectively. One defines $C(\underline{X}, q\, T, 1\slash q)$ as follows. For a partition 
$\mathfrak{n} = (1^{n_{1}}\!, 2^{n_{2}}\!, \ldots),$ put 
\begin{equation*}
f_{\mathfrak{n}}(Y, Z, T, q) = \sum_{\mathfrak{i}, \mathfrak{j}} 
A_{\mathfrak{n}, \mathfrak{i}, \mathfrak{j}}(T, q)
\frac{Y^{\mathfrak{i}} Z^{\mathfrak{j}}}{\mathfrak{i}! \mathfrak{j}!}.  
\end{equation*}
Note that, for $\mathfrak{n} \ne {\bf{0}},$ we can write 
\begin{equation*}
f_{\mathfrak{n}}(Y, Z, T, q) \;\; =
\sum_{d \, \in \, \mathscr{P}(\mathfrak{n}, q)} \Bigg[\prod_{m = 1}^{\infty} 
(1 +  Y_{m})^{ \frac{c_{m}^{}\!(d, q) \, - \, a_{m}^{*}\!(d, q)}{2 m}}
(1  +  Z_{m})^{ \frac{c_{m}^{}\!(d, q) \, + \,  a_{m}^{*}\!(d, q)}{2 m}}  
\cdot \; \prod_{k = 1}^{r} P_{\scriptscriptstyle C_{d}}(t_{k})  \Bigg]. 
\end{equation*}
Replacing $c_{m}(d, q)$ by $m (\mathrm{Irr}_{q}(m) - n_{m})$ and $a_{m}^{*}(d, q)$ by
\begin{equation*}
\mathscr{D}_{m}^{*}(\mathfrak{n}, q) = 
\mathcal{D}_{m}^{*} 
- \, \rho_{m}(1 + (-1)^{|\mathfrak{n}|})\slash 2 
\, - \, m \; \cdot \sum_{u = 1}^{\nu_{2}(m)} 
(\mathrm{Irr}_{q}(m \slash 2^{u}) - n_{m \slash 2^{u}})\slash 2^{u}
\end{equation*}
where $\mathcal{D}_{m}^{*}$ is the differential operator defined above, 
we can further write 
\begin{equation*}
f_{\mathfrak{n}}(Y, Z, T, q)
= \,^{\tiny{\tiny{\prod_{m = 1}^{\infty} 
(1  +   Y_{m})^{(\mathrm{Irr}_{q}(m) - n_{m})\slash 2 
- \mathscr{D}_{m}^{*}(\mathfrak{n}, q) \slash 2m}
(1 + Z_{m})^{(\mathrm{Irr}_{q}(m) - n_{m})\slash 2 
+ \mathscr{D}_{m}^{*}(\mathfrak{n}, q) \slash 2m}}}}
\!\!\!A_{\mathfrak{n}}(T, q).
\end{equation*}
Here $(1 +  U)^{\mathscr{D}}\!,$ for a variable $U$ and a (differential) operator $\mathscr{D},$ stands for the formal power series
\begin{equation*}
1 \,+\, \sum_{n = 1}^{\infty} \frac{\mathscr{D}(\mathscr{D} - 1) \, \cdots \, (\mathscr{D} - n  + 1)}{n!} U^{n}.
\end{equation*}
In particular, if $\mathfrak{n} =  {\bf 1}$ (i.e., $\mathfrak{n} = (1^{n_{1}}\!, 2^{n_{2}}\!, \ldots)$ with $n_{1} = 1$ and $n_{m} = 0$ for $m\ge 2$), then
\begin{equation} \label{eq: f1}
f_{{\bf 1}}(Y, Z, T, q) \,=\, 
q \prod_{m = 1}^{\infty} 
(1 + Y_{m})^{c_{m \slash 2^{\nu_{2}(m)}}\big(q^{2^{\nu_{2}(m)}}\big)\slash (2 m)}
(1 + Z_{m})^{\big[2 c_{m}(q) \, - \, c_{m \slash 2^{\nu_{2}(m)}}\big(q^{2^{\nu_{2}(m)}}\big) \big]\slash (2 m)}
\end{equation}
where 
\begin{equation}
\label{eq: c_{m}}
c_{m}(q) 
= \begin{cases}
q - 1 &\mbox{if}\; m = 1 \\
m\, \mathrm{Irr}_{q}(m) & \mbox{if}\; m \ge 2.
\end{cases}  
\end{equation}

We define 
\begin{equation*}
f_{\mathfrak{n}}(Y, Z, q\, T, 1\slash q) \, = \, ^{\tiny{\tiny{\prod_{m = 1}^{\infty} 
(1 + Y_{m})^{(\mathrm{Irr}_{1\slash q}(m) - n_{m})\slash 2 - \mathscr{D}_{m}^{*}(\mathfrak{n}, 1\slash q) \slash 2m}(1 + Z_{m})^{(\mathrm{Irr}_{1\slash q}(m) - n_{m})\slash 2 + \mathscr{D}_{m}^{*}(\mathfrak{n}, 1\slash q) \slash 2m}}}}
\!\!\!A_{\mathfrak{n}}(q \,T, 1\slash q) 
\end{equation*}
and $C_{k, \text{odd}}(\underline{X}, q\, T, 1\slash q),$ for all integers $k \ge 1,$ by  
\begin{equation*}
C_{k, \text{odd}}(\underline{X}, q\, T, 1\slash q) \;\, =
\sum_{\substack{\mathfrak{n} \\  |\mathfrak{n}|-\text{odd}}} 
f_{\mathfrak{n}}((Y_{m  k}\slash \tilde{E}(- T^{m k}\!, q^{m k}))_{m\ge 1}, 
(Z_{m  k}\slash \tilde{E}(T^{m k}\!, q^{m k}))_{m\ge 1}, q^{k} T^{k}\!, q^{-k}) X^{k \mathfrak{n}}.
\end{equation*} 
Define $C_{k, \text{even}}^{\pm}(\underline{X}, q\, T, 1\slash q)$ similarly, and let 
$C_{\text{odd}}(\underline{X}, q\, T, 1\slash q),$ $C_{\text{even}}^{\pm}(\underline{X}, q\, T, 1\slash q)$ and $C(\underline{X}, q\, T, 1\slash q)$ be as defined above.

Using the generating function $C(\underline{X}, T, q),$ we can reinterpret \eqref{eq: tag 4.1} 
and the relations \eqref{eq: MDS-relations} as follows. Let 
\begin{equation*}
Z_{\text{odd}}(T, t_{r + 1}; q) = \left[Z(T, t_{r + 1}; q) - Z(T, - t_{r + 1}; q) \right]\slash 2 \;\,  = \sum_{l-\text{odd}} \Lambda_{l}(T, q) t_{r+1}^{l}
\end{equation*} 
and 
\begin{equation*}
Z_{\text{even}}(T, t_{r + 1}; q) = \left[Z(T, t_{r + 1}; q) + Z(T, - t_{r + 1}; q) \right]\slash 2 \;\,  = \sum_{l-\text{even}} \Lambda_{l}(T, q) t_{r+1}^{l}
\end{equation*} 
denote the odd (respectively even) part of the series $Z(T, t_{r + 1}; q)$ in \eqref{eq: tag 3.3}. Put 
\begin{equation*} 
\underline{Z}_{\text{odd}}(T, t_{r+1}; q) = \big(Z_{\text{odd}}(T^{k}\!, t_{r+1}^{k}; q^{k}) \big)_{k\ge 1} \;\;\; \text{and} \;\;\; 
\underline{Z}_{\text{even}}^{\pm}(T, t_{r+1}; q) 
= \big(E(\pm T^{k}) Z_{\text{even}}(\pm T^{k}\!, t_{r + 1}^{k}; q^{k})\big)_{k\ge 1}.
\end{equation*} 
If we set 
\begin{equation*}
\underline{Z}(T, t_{r + 1}; q) : = \big(\underline{Z}_{\text{odd}}(T, t_{r+1}; q), \underline{Z}_{\text{even}}^{-}(T, t_{r+1}; q), 
\underline{Z}_{\text{even}}^{+}(T, t_{r+1}; q) \big) 
\end{equation*} 
then by \eqref{eq: tag 4.2}, Proposition \ref{Proposition 4.3} and \eqref{eq: tag 4.4}, 
the relation \eqref{eq: tag 4.1} can be expressed as 
\begin{equation} \label{eq: eqn-4.1-rewritten1} 
\underline{Z}(T, t_{r + 1}; q) = C(\underline{Z}(q\, T, q  t_{r + 1}; 1\slash q), T, q).
\end{equation} 
For the purpose of interpreting the relations \eqref{eq: MDS-relations}, let 
$C_{0}(\underline{X}, T, q) = C(\underline{X}, T, q) - q \star \underline{X},$ 
where for $\underline{X}$ as above, 
$ 
q \star \underline{X} = \big((q^{n} X_{n})_{n \ge 1}, (q^{n} Y_{n})_{n \ge 1}, (q^{n} Z_{n})_{n\ge 1} \big)
$ 
is the linear part of $C(\underline{X}, T, q).$ Then \eqref{eq: eqn-4.1-rewritten1} can be rewritten as 
\begin{equation*}
\underline{Z}(T, t_{r + 1}; q) \,-\,  q \star \underline{Z}(q\, T, q  t_{r + 1}; 1 \slash q) =  
C_{0}(\underline{Z}(q\, T, q  t_{r + 1}; 1\slash q), T, q).
\end{equation*} 
Replacing $q \to 1\slash q, T \to q\, T, t_{r + 1} \to q  t_{r + 1},$ and substituting 
$\underline{Z}(T, t_{r + 1}; q)$ by using \eqref{eq: eqn-4.1-rewritten1}, 
we see that $\underline{Z}(q\, T, q t_{r + 1}; 1\slash q)$ has to satisfy 
\begin{equation*}
\underline{Z}(q\, T, q t_{r + 1}; 1\slash q) \,-\,  (q^{-1}) \star 
C(\underline{Z}(q\, T, q  t_{r + 1}; 1\slash q), T, q) =  
C_{0}(C(\underline{Z}(q\, T, q  t_{r + 1}; 1\slash q), T, q), q\, T, 1\slash q).
\end{equation*} 
The last identity prompts us to wonder if we have in fact: 
\begin{equation} \label{eq: MDS-relations-gen-series}
\underline{X}\,- \,  (q^{-1}) \star 
C(\underline{X}, T, q) = C_{0}(C(\underline{X}, T, q), q\, T, 1\slash q)
\end{equation} 
or what amounts to the same, 
\begin{equation} \label{eq: MDS-relations-gen-series-equiv}
C(C(\underline{X}, T, q), q  \, T, 1\slash q) = \underline{X}.
\end{equation} 
Clearly \eqref{eq: MDS-relations-gen-series} and \eqref{eq: MDS-relations-gen-series-equiv} are equivalent to the relations \eqref{eq: MDS-relations}.

Shortly, we will show (see Theorem \ref{princ-theorem}) that, indeed, 
\eqref{eq: MDS-relations-gen-series-equiv} holds. \!Its proof depends chiefly 
upon understanding how \eqref{eq: MDS-relations-gen-series-equiv} is encoded in the combinatorial structure of certain moduli spaces, see Section 
\ref{section: The series BarC(X, T, q)}. 

\vskip5pt
We shall need the following well-known identities:

\vskip10pt
\begin{lemma} 
\label{Identity Irr} --- For $m\in \mathbb{N}^{\times}\!,$ we have 
\begin{equation*}
\sum_{h \mid  m} \mathrm{Irr}_{q_{1}^{}}\!(h) \mathrm{Irr}_{q_{2}^{h}}(m\slash h)
= \mathrm{Irr}_{q_{1} q_{2}}(m)
\end{equation*}
where for an indeterminate $t,$ we put 
\begin{equation*}
\mathrm{Irr}_{t}(m) = m^{-1}\!\sum_{k \mid m} \mu(m\slash k) t^{k}.
\end{equation*}
In particular, if $m \ge 2,$ then
\begin{equation*}
\sum_{h \mid m} \mathrm{Irr}_{q^{-1}}(h) \mathrm{Irr}_{q^{h}}(m\slash h) = 0
\end{equation*} 
and
\begin{equation*} 
\sum_{h \mid m} h\, \mathrm{Irr}_{q}(h) = q^{m} \qquad (\text{for $m \ge 1$}).
\end{equation*}
\end{lemma}

\begin{proof} Indeed, the left-hand side of the identity can be written as  
\begin{equation*}
\frac{1}{m}\sum_{h \mid m}
\bigg(\sum_{k\mid h} \mu(h\slash k) q_{1}^{k} \bigg) 
\bigg(\sum_{l\mid \frac{m}{h}} \mu(m\slash hl) q_{2}^{hl} \bigg) 
\,=\, \frac{1}{m}\sum_{k \mid m} q_{1}^{k} \sum_{l\mid \frac{m}{k}} 
\mu\Big(\frac{m}{l k} \Big) q_{2}^{l k} \sum_{n\mid l} \mu(n).
\end{equation*}
Since $\sum_{n\mid l} \mu(n) = 0$ if $l > 1,$ it follows that 
\begin{equation*}
\sum_{h \mid m} \mathrm{Irr}_{q_{1}^{}}\!(h) \mathrm{Irr}_{q_{2}^{h}}(m\slash h)\,=\, 
\frac{1}{m}\sum_{k \mid m} \mu\Big(\frac{m}{k}\Big) 
(q_{1}q_{2})^{k} = \mathrm{Irr}_{q_{1}q_{2}}(m).
\end{equation*} 
The second identity in the statement of the lemma follows by replacing $q_{1}$ and $q_{2}$ by $q^{-1}$ and $q,$ respectively. By equating the coefficients of $t^{m}\!,$ i.e., the leading coefficients, on both sides of the identity 
\begin{equation*}
\sum_{h \mid m} \mathrm{Irr}_{q}(h) \mathrm{Irr}_{t^{h}}(m\slash h)
= \mathrm{Irr}_{q t}(m)
\end{equation*} 
we deduce that 
$ 
\sum_{h \mid m} h\, \mathrm{Irr}_{q}(h) = q^{m}
$ 
for $m \ge 1.$ \end{proof} 

For $N\ge 1,$ let $\mathcal{D}_{N}^{*}$ be the differential operator defined by 
\begin{equation} \label{eq: D_{N}^{*}}
\mathcal{D}_{N}^{*} \;\; = 
\sum_{\substack{k \mid N \\ (N\slash k)-\text{odd}}} 
\mu(N\slash k) \, k \sum \frac{(-1)^{l_{1} + \cdots + l_{k} + 1} 
(l_{1} + \cdots + l_{k} - 1)!}{l_{1}! \, \cdots \, l_{k}!}\partial_{l_{1}\!, \ldots, l_{k}}
\end{equation}
the inner sum being over all $l_{1}, \ldots, l_{k}\in \mathbb{N}$ satisfying 
$l_{1} + \cdots + k l_{k} = k.$ We shall need some elementary properties 
of the action of this operator on the $\mathbb{C}$-algebra $\mathscr{B}$ 
consisting of linear combinations
\begin{equation*}
\sum_{\alpha} c_{\alpha} F_{\alpha}(T) 
\qquad \;\;\; \text{\bigg($c_{\alpha} \in \mathbb{C}$ and $F_{\alpha}(T) = 
\prod_{k = 1}^{r}f_{\alpha}(t_{k})$ \bigg)}
\end{equation*}
with $f_{\alpha}(t)$ of the form 
\begin{equation*}
f_{\alpha}(t)\,=\, 
\exp  \bigg(\sum_{j = 1}^{\infty} \frac{a_{\alpha}(j)}{j} t^{j} \bigg) 
\qquad \text{($a_{\alpha}(j) \in \mathbb{C}$ for all $j\ge 1$)}
\end{equation*}
(viewed as a composition of formal power-series), for all $\alpha.$ In particular, 
$A_{\mathfrak{n}, \mathfrak{i}, \mathfrak{j}}(T, q)$ is in $\mathscr{B},$ and by using 
the results of Section \ref{section: The series BarC(X, T, q)} (specifically, see 
\eqref{eq: funct-equation-needed-proof-main-thm}), one can express inductively $A_{\mathfrak{n}, \mathfrak{i}, \mathfrak{j}}(q\, T, 1\slash q)$ 
as a linear combination (independent of $r$) of derivatives 
$^{\scriptscriptstyle \partial_{l_{1}\!, \ldots, l_{s}}}
\!A_{\mathfrak{m}, \mathfrak{p}, \mathfrak{l}}(T, q)$ ($l_{1}, \ldots, l_{s}\in \mathbb{N}$), hence $A_{\mathfrak{n}, \mathfrak{i}, \mathfrak{j}}(q\, T, 1\slash q)$ is also in $\mathscr{B}$ for any partitions $\mathfrak{n}, \mathfrak{i}$ and $\mathfrak{j}.$ It is understood that, for $F(T) \in \mathscr{B},$ its formal derivative 
$^{\scriptscriptstyle \mathcal{D}_{N}^{*}}\!F(T)$ is taken as in the right-hand side of 
\eqref{eq: derivative}. Clearly $\mathcal{D}_{N}^{*}$ is $\mathbb{C}$-linear, and it 
satisfies the Leibniz rule 
$^{\scriptscriptstyle \mathcal{D}_{N}^{*}}\!(FG)(T) 
= G(T)\,  ^{\scriptscriptstyle \mathcal{D}_{N}^{*}}\!F(T) 
+ F(T)\,  ^{\scriptscriptstyle \mathcal{D}_{N}^{*}} G(T),$ for $F(T)$ and $G(T)$ in $\mathscr{B}.$ In addition, if $F_{\alpha}(T)$ is as above, then
\begin{equation*}
^{\scriptscriptstyle \mathcal{D}_{N}^{*}}\!F_{\alpha}(T) = F_{\alpha}(T)
\;\,  \cdot \sum_{\substack{k \mid N \\ N\slash k-\text{odd}}} 
\mu(N\slash k)\, a_{\alpha}(k).
\end{equation*} 
In particular, we have

\vskip10pt
\begin{lemma} \label{Identity1} --- For sequences $\{\ga_{l}\}_{l \ge 1}$ and $\{\eta_{l}\}_{l \ge 1}$ of complex numbers, define 
\begin{equation*}
f(t) = \prod_{l = 1}^{\infty} \big(1 - q^{l}t^{2 l}\big)^{\ga_{l} \slash 2}
\; \mathrm{and} \;\; 
g(t)  = \prod_{l \, \ge \, 1} 
\bigg(\frac{1 + t^{l}}{1 - t^{l}}\bigg)^{\eta_{l} \slash  2l}. 
\end{equation*}
Put 
$
F(T) = \prod_{k = 1}^{r} f(t_{k})
$ 
and
$
G(T) = \prod_{k = 1}^{r} g(t_{k}).
$ 
Then, for any $N\in \mathbb{N}^{\times}\!,$ we have 
\begin{equation} \label{eq: 1 D_{N}^{*}}
^{\scriptscriptstyle \mathcal{D}_{N}^{*}}\!F(T) \,=\, 
\Bigg(\sum_{\substack{l m = N \\ m-\mathrm{even}}} l \ga_{l} 
\sum_{\substack{h \mid m \\ (m\slash h)-\mathrm{even}}} \mu\Big(\frac{m}{h}\Big) 
q^{l h}\Bigg)F(T)
\end{equation} 
and 
\begin{equation} \label{eq: 2 D_{N}^{*}}
^{\scriptscriptstyle \mathcal{D}_{N}^{*}}G(T) = \eta_{N}G(T)
\end{equation} 
with the understanding that the double sum in the right-hand side of 
\eqref{eq: 1 D_{N}^{*}} vanishes if $N$ is odd.
\end{lemma}

\begin{proof} Writing 
\begin{equation*}
f(t)  = 
\exp \bigg(\sum_{l = 1}^{\infty} \frac{\ga_{l}}{2} \log\big(1 - q^{l}t^{2 l}\big)\bigg)
\end{equation*} 
and expanding 
$
\log \big(1 - q^{l}t^{2 l}\big) = - \sum_{m = 1}^{\infty} (q^{l m} t^{2 l m}\slash m),
$ 
the identity \eqref{eq: 1 D_{N}^{*}} follows after a simple calculation. 

Similarly, 
\begin{equation*}
^{\scriptscriptstyle \mathcal{D}_{N}^{*}}G(T) \,=\, 
\Bigg(\sum_{\substack{J M = N \\ M-\text{odd}}} \mu(M) 
\sum_{\substack{l k = J \\ k-\text{odd}}} \eta_{l} \Bigg)\, G(T)
\end{equation*}
and \eqref{eq: 2 D_{N}^{*}} follows by applying the M\"obius inversion. This completes the proof. \end{proof}

The remaining of this section is devoted to the proof of our main theorem in the general case, subject to a functional relation, whose proof will be given in Section \ref{section: The series BarC(X, T, q)}.

\vskip10pt
\begin{thm}\label{princ-theorem} --- With the above notation, we have 
\begin{equation*}
C(C(\underline{X}, T, q), q  \, T, 1\slash q) = \underline{X}.
\end{equation*}
In other words, $C(\underline{X}, T, q)$ is the formal compositional inverse of 
$C(\underline{X}, q\, T, 1\slash q).$
\end{thm}

\begin{proof} As in the previous subsection, we shall reduce the assertion of the theorem to the existence of a generating function satisfying a functional relation, see Section 
\ref{section: The series BarC(X, T, q)}.

First, notice that it is enough to show that 
\begin{equation*}
C_{1, \text{odd}}(C(\underline{X}, T, q), q  \, T, 1\slash q) = X_{1} 
\end{equation*}
$
C_{1, \text{even}}^{-}(C(\underline{X}, T, q), q  \, T, 1\slash q) = Y_{1}
$ 
and 
$
C_{1, \text{even}}^{+}(C(\underline{X}, T, q), q  \, T, 1\slash q) = Z_{1}.
$ 
By Lemma \ref{evenness-A}, the expansion of $A_{\mathfrak{n}}(T, q)$ in terms of 
the symplectic Schur functions 
$
s_{\scriptscriptstyle \langle \lambda^{\dag} \rangle}(q^{\pm \frac{1}{2}} T^{\pm 1})
$ 
is only over partitions of even weight whenever $|\mathfrak{n}|$ is odd; therefore, 
$
A_{\mathfrak{n}}( - q\, T, 1\slash q) = A_{\mathfrak{n}}(q\, T, 1\slash q),
$ 
and from the definition of $A_{\mathfrak{n}, \mathfrak{i}, \mathfrak{j}}(q\, T, 1\slash q),$  
\begin{equation*}
{^\iota}\!\!A_{\mathfrak{n}, \mathfrak{i}, \mathfrak{j}}(q \, T, 1\slash q) 
= A_{\mathfrak{n}, \mathfrak{i}, \mathfrak{j}}(q\, T, 1\slash q) 
\end{equation*} 
whenever $|\mathfrak{n}|$ is odd (see also the above remark). Replacing $T$ by $-T$ (i.e., $t_{k}$ by $-t_{k}$ for $1\le k \le r$), and interchanging $Y_{1}$ and $Z_{1},$ $Y_{3}$ and $Z_{3}, \ldots,$ it follows that we have the implication
\begin{equation*}
C_{1, \text{even}}^{+}(C(\underline{X}, T, q), q  \, T, 1\slash q) = Z_{1}
\implies
C_{1, \text{even}}^{-}(C(\underline{X}, T, q), q  \, T, 1\slash  q) = Y_{1}.
\end{equation*}
Now expand 
\begin{equation*}
C_{1, \text{odd}}(C(\underline{X}, T, q), q  \, T, 1\slash q) 
\;\, = \sum_{\substack{\mathfrak{n}, \mathfrak{i}, \mathfrak{j} \\  |\mathfrak{n}|-\text{odd}}} 
C_{\mathfrak{n}, \mathfrak{i}, \mathfrak{j}}(T, q) X^{\mathfrak{n}} Y^{\mathfrak{i}}  
Z^{\mathfrak{j}}
\end{equation*}
and 
\begin{equation*}
C_{1, \text{even}}^{+}(C(\underline{X}, T, q), q  \, T, 1\slash q) 
\;\, = \sum_{\substack{\mathfrak{n}, \mathfrak{i}, \mathfrak{j} \\  |\mathfrak{n}|-\text{even}}} 
C_{\mathfrak{n}, \mathfrak{i}, \mathfrak{j}}(T, q) X^{\mathfrak{n}} Y^{\mathfrak{i}}  
Z^{\mathfrak{j}}.
\end{equation*}
By \eqref{eq: f1}, the subseries of 
$
C_{1, \text{odd}}(C(\underline{X}, T, q), q  \, T, 1\slash q)
$ 
corresponding to $\mathfrak{n} = {\bf 1}$ can be expressed by 
\begin{equation*}
X_{1}\sum_{\mathfrak{i}, \mathfrak{j}} 
C_{{\bf 1}, \mathfrak{i}, \mathfrak{j}}(T, q) Y^{\mathfrak{i}}  
Z^{\mathfrak{j}}  = 
X_{1}\! \prod_{m = 1}^{\infty} 
\big(1  + Y_{m}\slash \tilde{E}( - T^{m}\!, q^{m}) \big)^{\omega_{m}^{-}(q)}
\big(1 + Z_{m}\slash \tilde{E}(T^{m}\!, q^{m}) \big)^{\omega_{m}^{+}(q)}
\end{equation*}
where for $m\ge 1$ with $2^{b} \parallel m,$  
\begin{equation*}
\omega_{m}^{-}(q) = 
\frac{c_{m \slash 2^{b}}(q^{2^b})}{2m}\;\; + 
\sum_{\substack{l \mid m \\ m\slash l - \text{odd}}} 
\frac{\mathrm{Irr}_{q^{l}}(m\slash l) \, c_{l \slash 2^{b}}(q^{- 2^b})}{2 l} 
\end{equation*}
and 
\begin{equation*}
\omega_{m}^{+}(q) = 
\frac{2c_{m}(q) - c_{m \slash 2^{b}}(q^{2^b})}{2m} \;\; + 
\sum_{\substack{k \mid m \\ m\slash k - \text{even}}} \frac{c_{k}(1\slash q)\, \mathrm{Irr}_{q^{k}}(m\slash k)}{k} \;\; + 
\sum_{\substack{l \mid m \\ m\slash l - \text{odd}}} \frac{\mathrm{Irr}_{q^{l}}(m\slash l)
(2c_{l}(1\slash q) - c_{l \slash 2^{b}}(q^{- 2^b}))}{2 l}.
\end{equation*}
Put $m = 2^{b} m'.$ If $m' = 1,$ then, by \eqref{eq: c_{m}}, 
$$
\omega_{m}^{-}(q)  = 2^{- b - 1}\cdot \big(c_{1}(q^{2^b}) \, +\, 
q^{2^{b}}c_{1}(q^{-2^{b}}) \big) \, =\, 0
$$ 
and if $m' > 1,$ we can write 
\begin{equation*}  
\omega_{m}^{-}(q)  = \frac{c_{m'}(q^{2^b})}{2^{b + 1} m'} \, + \,  
\sum_{l' \mid m'} \frac{c_{l'}(q^{-2^b})\, \mathrm{Irr}_{\! q^{2^{b} l'}}(m'\slash l')}{2^{b + 1} l'} \,=\, 2^{- b - 1}\! \sum_{l' \mid m'} \mathrm{Irr}_{q^{-2^{b}}}(l')\, \mathrm{Irr}_{\! q^{2^{b} l'}}(m'\slash l').
\end{equation*} 
The last sum vanishes by Lemma \ref{Identity Irr}. Similarly, $\omega_{1}^{+}(q) = 0,$ and if 
$m  > 1,$ one can reduce $\omega_{m}^{+}(q)$ to 
\begin{equation*}  
\omega_{m}^{+}(q) = -\, \frac{c_{m'}(q^{2^{b}})}{2^{b + 1}m'} \, + \, 
\frac{\mathrm{Irr}_{\! q^{2^{b}}}(m')  - \mathrm{Irr}_{1}(m')}{2^{b + 1}}.
\end{equation*} 
Thus $\omega_{m}^{+}(q) = 0.$

Next, notice that
\begin{equation*}
\sum_{\mathfrak{i}, \mathfrak{j}} C_{{\bf 0}, \mathfrak{i}, \mathfrak{j}}(T, q)  Y^{\mathfrak{i}} Z^{\mathfrak{j}} = \,  E(T) E(q T) \left[\,\prod_{k, l = 1}^{\infty}
\bigg(1 + \frac{Z_{k l}}{E(T^{k l}) E(q^{k l} T^{k l})} \bigg)^{\mathrm{Irr}_{\! q^{-1}}\!(k)\mathrm{Irr}_{\! q^{k}}\!(l)}   - \,  1 \,\right]. 
\end{equation*}
By Lemma \ref{Identity Irr}
\begin{equation*}
\sum_{k l = m} \mathrm{Irr}_{q^{-1}}(k) \mathrm{Irr}_{q^{k}}(l) = 0    \qquad \;\;\;  
(\text{if $m > 1$})
\end{equation*} 
and so 
\begin{equation*} 
\sum_{\mathfrak{i}, \mathfrak{j}} C_{{\bf 0}, \mathfrak{i}, \mathfrak{j}}(T, q)  
Y^{\mathfrak{i}}  Z^{\mathfrak{j}} = \, 
E(T) E(q T) \left[\bigg(1 + \frac{Z_{1}}{E(T) E(q T)} \bigg)^{ \mathrm{Irr}_{\! q^{-1}}(1)\mathrm{Irr}_{q}(1)}   - \,  1 \,\right] 
= Z_{1}.
\end{equation*} 
Our next step is to show that if $C_{\mathfrak{N}, {\bf 0}, {\bf 0}}(T, q) = 0$ for a partition $\mathfrak{N}$ of weight $|\mathfrak{N}|\ge 2,$ then $C_{\mathfrak{N}, \mathfrak{I}, \mathfrak{J}}(T, q) = 0$ for arbitrary partitions $\mathfrak{I}$ 
and $\mathfrak{J}.$ As in the proof of Theorem \ref{split case}, we shall deduce this from the fact that $C_{\mathfrak{N}, \mathfrak{I}, \mathfrak{J}}(T, q) \,= \, ^{*} C_{\mathfrak{N}, {\bf 0}, {\bf 0}}(T, q)$ for a certain differential operator $*;$ the precise definition of this differential operator is given in \eqref{eq: D_{N, I, J}}.

Indeed, as in the proof of Theorem \ref{split case}, we first write 
\begin{equation*}
C_{\mathfrak{N}, {\bf 0}, {\bf 0}}(T, q) = \sum_{S} c_{\textsc{\tiny$S$}} 
M_{\textsc{\tiny$S$}}(T, q) 
\end{equation*} 
summed over tuples 
$
S = (S_{1}^{}, \ldots, S_{1}^{-}, \ldots, S_{1}^{+}, \ldots)
$ 
with each $S_{k}^{},$ or $S_{l}^{\pm},$ itself a tuple of partitions; for $k \ge 1,$ 
$S_{k} = (\mathfrak{n}_{k1}, \mathfrak{n}_{k2}, \ldots, \mathfrak{n}_{k n_{k}})$ with $|\mathfrak{n}_{kh}|$ odd for all $1\le h \le n_{k},$ and for $l\ge 1,$ 
$S_{l}^{\pm} = (\mathfrak{n}_{l1}^{\pm}, \mathfrak{n}_{l2}^{\pm}, \ldots, \mathfrak{n}_{l n_{l}^{\pm}}^{\pm})$ with $|\mathfrak{n}_{l m}^{\pm}|$ even for all $1\le m \le n_{l}^{\pm};$ we also require that
\begin{equation*}
\mathfrak{N} = \mathfrak{n}_{11} \cup \cdots \cup \mathfrak{n}_{1 n_{1}} \cup 2\mathfrak{n}_{21}\cup \cdots \cup 2\mathfrak{n}_{2 n_{2}} \cup \cdots \cup  \mathfrak{n}_{11}^{-} \cup \cdots \cup {\mathfrak{n}^{-}}_{\!\!\!\! 1 n_{1}^{-}} \cup \cdots \cup l\mathfrak{n}_{l1}^{+} \cup \cdots \cup 
{l\mathfrak{n}^{+}}_{\!\!\!\! l n_{l}^{+}} \cup \cdots.
\end{equation*}
If we define $M_{\textsc{\tiny$S_{k}$}}\!(T, q)$ and $M_{\textsc{\tiny$S_{l}^{\pm}$}}(T, q)$ by
\begin{equation*}
M_{\textsc{\tiny$S_{k}$}}\!(T, q) = 
A_{\mathfrak{n}_{k1}}\!(T, q) 
\, \cdots \,
A_{\mathfrak{n}_{k n_{k}}}\!(T, q) \;\;\; \text{and} \;\;\; 
M_{\textsc{\tiny$S_{l}^{\pm}$}}(T, q) = 
A_{\mathfrak{n}_{l1}^{\pm}}\!(\pm T, q) 
\, \cdots \,  
A_{\mathfrak{n}_{l n_{l}^{\pm}}^{\pm}}\!(\pm T, q)
\end{equation*}
then $M_{\textsc{\tiny$S$}}(T, q)$ above is given by 
\begin{equation*}
M_{\textsc{\tiny$S$}}(T, q) = 
\frac{A_{\mathfrak{n}, \mathfrak{i}, \mathfrak{j}}(q\, T, 1\slash q)}{E^{^{-}}\!(T, q)^{ \mathfrak{i}} E^{^{+}}\!(T, q)^{\mathfrak{j}}}
\cdot \prod_{k, l, l'}
M_{\textsc{\tiny$S_{k}^{}$}}\!(T^{k}\!, q^{k}) M_{\textsc{\tiny$S_{l}^{-}$}}(T^{l}\!, q^{l}) 
M_{\textsc{\tiny$S_{l'}^{+}$}}(T^{l'}\!, q^{l'})
\end{equation*}
with $\mathfrak{n} = (1^{n_{1}}\!, 2^{n_{2}}\!, \ldots),$ 
$\mathfrak{i}  = (1^{n_{1}^{-}}\!, 2^{n_{2}^{-}}\!, \ldots)$ 
and $\mathfrak{j}  = (1^{n_{1}^{+}}\!, 2^{n_{2}^{+}}\!, \ldots).$ The 
constant $c_{\textsc{\tiny$S$}}$ is given in \eqref{eq: $c_{S}$}.

Define 
\begin{equation*}
\mathcal{D}^{*}\!(Y, Z) = \prod_{N = 1}^{\infty} 
\bigg(\frac{1 + \tilde{Z}_{N}}{1  + \tilde{Y}_{N}} \bigg)^{\!\!\frac{\mathcal{D}_{N}^{*}}{2N}}
\end{equation*}
where, for notational convenience, we set
$\tilde{Y}_{N} : = Y_{N}\slash \tilde{E}(- T^{N}\!, q^{N})$ 
and $\tilde{Z}_{N} := Z_{N}\slash \tilde{E}(T^{N}\!, q^{N}),$ 
and where $\mathcal{D}_{N}^{*}$ for $N\ge 1$ is the differential 
operator defined by \eqref{eq: D_{N}^{*}}. The idea is to compare 
$^{\scriptscriptstyle \mathcal{D}^{*}\!(Y, Z)}\!M_{\textsc{\tiny$S$}}(T, q)$ with the 
contribution of all the terms {\it derived} from, or corresponding to 
$M_{\textsc{\tiny$S$}}(T, q),$ for fixed $S,$ in the composition 
$
C_{1, \text{odd}}(C(\underline{X}, T, q), q  \, T, 1\slash q).
$ 
We will show that these two agree up to a normalizing factor, 
depending only upon the fixed partition $\mathfrak{N}.$
 
First, the contribution of all the terms derived from 
$M_{\textsc{\tiny$S$}}(T, q)$ in
$
C_{1, \text{odd}}(C(\underline{X}, T, q), q  \, T, 1\slash q)  
$ 
can be easily described as follows. Consider the subseries  
$ 
X^{\mathfrak{n}}\!f_{\mathfrak{n}}((\tilde{Y}_{m})_{m\ge 1}, (\tilde{Z}_{m})_{m\ge 1}, q\, T, 1\slash q)
$ 
of $C_{1, \text{odd}}(\underline{X}, q\, T, 1\slash q),$ where we recall that 
\begin{equation*}
f_{\mathfrak{n}}(Y, Z, q\, T, 1\slash q) \, = \, ^{\tiny{\tiny{\prod_{m = 1}^{\infty} 
(1  +   Y_{m})^{(\mathrm{Irr}_{1\slash q}(m) - n_{m})\slash 2 - \mathscr{D}_{m}^{*}(\mathfrak{n}, 1\slash q) \slash 2m}
(1 + Z_{m})^{(\mathrm{Irr}_{1\slash q}(m) - n_{m})\slash 2 + \mathscr{D}_{m}^{*}(\mathfrak{n}, 1\slash q) \slash 2m}}}}
\!\!\!A_{\mathfrak{n}}(q \,T, 1\slash q); 
\end{equation*} 
for simplicity, let us assume that $|\mathfrak{n}|$ is odd. In 
$ 
X^{\mathfrak{n}}\!f_{\mathfrak{n}}((\tilde{Y}_{m})_{m\ge 1}, (\tilde{Z}_{m})_{m\ge 1}, q\, T, 1\slash q),
$ 
replace $X_{k}, Y_{l}$ and $Z_{l'}$ for all $k, l, l' \ge 1$ by the subseries 
of $C_{k, \text{odd}}(\underline{X}, T, q), C_{l, \text{even}}^{-}(\underline{X}, T, q)$ and 
$C_{l', \text{even}}^{+}(\underline{X}, T, q)$ given by
\begin{equation*}
\sum X^{k \mathfrak{n}_{k h}}\!f_{\mathfrak{n}_{k h}}((\tilde{Y}_{k m})_{m\ge 1}, 
(\tilde{Z}_{k m})_{m\ge 1}, T^{k}\!, q^{k})
\end{equation*}
\begin{equation*}
f_{{\bf 0}}^{-}((\tilde{Y}_{l m})_{m\ge 1}, (\tilde{Z}_{l m})_{m\ge 1}, T^{l}\!, q^{l})  \,+\, 
\sum X^{l \mathfrak{n}_{l h}^{-}}\!f_{\mathfrak{n}_{l h}^{-}}^{-}\!((\tilde{Y}_{l m})_{m\ge 1}, (\tilde{Z}_{l m})_{m\ge 1}, T^{l}\!, q^{l})
\end{equation*} 
and 
\begin{equation*}
f_{{\bf 0}}^{+}((\tilde{Y}_{l' m})_{m\ge 1}, (\tilde{Z}_{l' m})_{m\ge 1}, T^{l'}\!, q^{l'})\,
+\, \sum X^{l' \mathfrak{n}_{l' h}^{+}} 
\!f_{\mathfrak{n}_{l' h}^{+}}^{+}\!((\tilde{Y}_{l' m})_{m\ge 1}, (\tilde{Z}_{l' m})_{m\ge 1}, T^{l'}\!, q^{l'})
\end{equation*} 
respectively, the sums being over all the {\it distinct} components of $S_{k}^{}, S_{l}^{-}$ and $S_{l'}^{+}.$ Here, for a partition $\mathfrak{n}$ of {\it even} weight, we put: 
$f_{\mathfrak{n}}^{+}(Y, Z, T, q) = f_{\mathfrak{n}}(Y, Z, T, q),$ and 
\begin{equation*}
f_{\mathfrak{n}}^{-}(Y, Z, T, q) = 
\sum_{\mathfrak{i}, \mathfrak{j}} 
{^\iota}\!\!A_{\mathfrak{n}, \mathfrak{i}, \mathfrak{j}}(T, q)
\frac{Y^{\mathfrak{i}} Z^{\mathfrak{j}}}{\mathfrak{i}! \mathfrak{j}!}; 
\end{equation*}
the generating series $f_{{\bf 0}}^{+}(Y, Z, T, q) = f_{{\bf 0}}^{}(Y, Z, T, q)$ (resp. 
$f_{{\bf 0}}^{-}(Y, Z, T, q)$) is defined by \eqref{eq: f0} (resp. \eqref{eq: iotaf0}). 
This gives an expression, in which the coefficient of 
$
X^{\mathfrak{N}}: = X^{\mathfrak{n}_{11}} \cdots X^{\mathfrak{n}_{1 n_{1}}} X^{2\mathfrak{n}_{21}} \cdots X^{l\mathfrak{n}_{l1}^{+}} \cdots X^{{l\mathfrak{n}^{+}}_{\!\!\!\!\!^{l n_{l}^{+}}}}  \cdots
$ 
is precisely the contribution of all the terms corresponding to the monomial $M_{\textsc{\tiny$S$}}(T, q).$

To write this contribution explicitly, define 
\begin{equation*}
f_{\! \textsc{\tiny$S_{k}$}}^{}\!(Y, Z, T, q) = 
\binom{n_{k}}{m(\mathfrak{n}_{k1}), \ldots, m(\mathfrak{n}_{k p_{k}})} 
f_{\mathfrak{n}_{k1}}\!(Y, Z, T^{k}\!, q^{k}) 
\, \cdots \, f_{\mathfrak{n}_{k n_{k}}}\!(Y, Z, T^{k}\!, q^{k})
\end{equation*} 
where $\mathfrak{n}_{k1}, \ldots, \mathfrak{n}_{k p_{k}}$ are the distinct components 
of $S_{k},$ and $m(\mathfrak{n}_{k h})$ denotes the multiplicity (i.e., the number of 
occurrences) of $\mathfrak{n}_{k h}$ in $S_{k},$ for all $1 \le h \le  p_{k}.$ Similarly, define
\begin{equation*}
f_{\! \textsc{\tiny$S_{l}^{\pm}$}}(Y, Z, T, q) = 
\binom{\mathscr{E}_{l}^{\pm}(\mathfrak{n}, 1\slash q)}
{m(\mathfrak{n}_{l1}^{\pm}), \ldots, m({\mathfrak{n}^{\pm}}_{\!\!\!\! \scriptscriptstyle l p_{l}^{\pm}}), \mathscr{E}_{l}^{\pm}(\mathfrak{n}, 1\slash q) - \cdots - m({\mathfrak{n}^{\pm}}_{\!\!\!\! \scriptscriptstyle l p_{l}^{\pm}})}
\frac{f_{\! \mathfrak{n}_{l1}^{\pm}}^{\pm}\!(Y, Z, T^{l}\!, q^{l})}
{g_{{\bf 0}}^{\pm}(Y, Z, T^{l}\!, q^{l})} \cdots 
\frac{f_{\!{\mathfrak{n}^{\pm}}_{\!\!\!\! \scriptscriptstyle l n_{l}^{\pm}}}^{\pm}\!(Y, Z, T^{l}\!, q^{l})}{g_{{\bf 0}}^{\pm}(Y, Z, T^{l}\!, q^{l})}
\end{equation*} 
where we put 
$g_{{\bf 0}}^{\pm}(Y, Z, T, q) = 1 + (f_{{\bf 0}}^{\pm}(Y, Z, T, q)
\slash \tilde{E}(\pm T, q)),$ and 
\begin{equation*}
\mathscr{E}_{l}^{\pm}(\mathfrak{n}, 1\slash q) = (\mathrm{Irr}_{1\slash q}(l) - n_{l})\slash 2 \pm \mathscr{D}_{l}^{*}\!(\mathfrak{n}, 1\slash q) \slash 2l
\end{equation*}
are precisely the exponents in the product appearing in the above expression of 
$f_{\mathfrak{n}}(Y, Z, q\, T, 1\slash q).$ If we set
\begin{equation*}
f_{\! \textsc{\tiny$S$}}^{}(Y, Z, T, q) : =
\prod_{k, l, l'} 
f_{\! \textsc{\tiny$S_{k}$}}^{}\!((\tilde{Y}_{k \alpha})_{\alpha}, (\tilde{Z}_{k \alpha})_{\alpha}, T, q) 
f_{\! \textsc{\tiny$S_{l}^{-}$}}((\tilde{Y}_{l \alpha})_{\alpha}, (\tilde{Z}_{l \alpha})_{\alpha}, T, q) 
f_{\! \textsc{\tiny$S_{l'}^{+}$}}((\tilde{Y}_{l' \alpha})_{\alpha}, (\tilde{Z}_{l' \alpha})_{\alpha}, T, q)
\end{equation*}
then the contribution we are interested in is expressed by
\begin{equation} \label{eq: contribution over $M_{S}(T, q)$}
\frac{^{^{f_{\! \scriptscriptstyle \textsc{$S$}}(Y, Z, T, q)
\prod_{l = 1}^{\infty} \prod_{m-\text{odd}} \, (1 + \tilde{Y}_{l m})
^{\mathrm{Irr}_{\! q^{l}}(m)\mathscr{E}_{l}^{-}(\mathfrak{n}, 1\slash q)} 
(1 + \tilde{Z}_{l m})^{\mathrm{Irr}_{\! q^{l}}(m)
\mathscr{E}_{l}^{+}(\mathfrak{n}, 1\slash q)}  
\prod_{m-\text{even}}\, (1 + \tilde{Z}_{l m})^{\mathrm{Irr}_{\! q^{l}}(m) 
(\mathrm{Irr}_{q^{-1}}(l) \,-\, n_{l})}}}
\!\!\!A_{\mathfrak{n}}(q \,T, 1\slash q)}{E^{^{-}}\!(T, q)^{\mathfrak{i}} E^{^{+}}\!(T, q)^{\mathfrak{j}}}.
\end{equation} 
In this expression, the differential operators commute with each other, and we have 
\begin{equation*}
A_{\mathfrak{n}, \mathfrak{i}, \mathfrak{j}}(q\, T, 1\slash q) \, = \;
^{\prod_{l} \mathscr{E}_{l}^{\pm}(\mathfrak{n}, 1\slash q)
(\mathscr{E}_{l}^{\pm}(\mathfrak{n}, 1\slash q) \,-\, 1) \, \cdots \, 
(\mathscr{E}_{l}^{\pm}(\mathfrak{n}, 1\slash q) \,+\, 1 \,-\, m(\mathfrak{n}_{l1}^{\pm}) \,-\, \cdots \,-\, m({\mathfrak{n}^{\pm}}_{\!\!\!\!\! l p_{l}^{\pm}}))}
\!\!\!A_{\mathfrak{n}}(q \,T, 1\slash q).
\end{equation*} 
The remaining parts of the multinomial coefficients make up the constant $c_{\textsc{\tiny$S$}},$ that is, 
\begin{equation} \label{eq: $c_{S}$}
c_{\textsc{\tiny$S$}}  = \prod_{k}\binom{n_{k}}{m(\mathfrak{n}_{k1}), \ldots, m(\mathfrak{n}_{k p_{k}})} \prod_{l} \big(m(\mathfrak{n}_{l1}^{-})! \, \cdots \, m({\mathfrak{n}^{-}}_{\!\!\!\! l p_{l}^{-}})!\, m(\mathfrak{n}_{l1}^{+})! \, \cdots \, m({\mathfrak{n}^{+}}_{\!\!\!\! l p_{l}^{+}})! \big)^{-1}.
\end{equation} 
To simplify things, using Lemma \ref{Identity Irr}, one can write the products in 
\eqref{eq: contribution over $M_{S}(T, q)$} as
\begin{equation} \label{eq: product in $M_{S}(T, q)$}
\small{\prod_{l = 1}^{\infty} \prod_{m-\text{odd}}  
\bigg(\frac{1 + \tilde{Z}_{l m}}{1 + \tilde{Y}_{l m}}\bigg)
^{\!\! \frac{\mathrm{Irr}_{\! q^{l}}(m)\mathscr{D}_{l}^{*}\!(\mathfrak{n}, 1\slash q)}{2l}}
\prod_{m-\text{even}} \bigg(\frac{1 + \tilde{Z}_{l m}}{1 + \tilde{Y}_{l m}}\bigg)
^{\!\! \frac{\mathrm{Irr}_{\! q^{l}}(m)(\mathrm{Irr}_{q^{-1}}(l) \,-\, n_{l})}{2}} 
\prod_{m} [(1 + \tilde{Y}_{l m}) ( 1 + \tilde{Z}_{l m})]
^{\frac{\epsilon_{l m} - n_{l}\mathrm{Irr}_{\! q^{l}}(m)}{2}}}    
\end{equation} 
with $\epsilon_{1} = 1$ and $\epsilon_{N} = 0$ if $N\ge 2.$ 
We combine the factor
$
\prod_{l, m = 1}^{\infty} 
[(1 + \tilde{Y}_{l m}) ( 1 + \tilde{Z}_{l m})]
^{-n_{l}\mathrm{Irr}_{q^{l}}(m)\slash 2} 
$ 
in \eqref{eq: product in $M_{S}(T, q)$} with 
$f_{\textsc{$S$}}(Y, Z, T, q).$ Note that the products in 
\eqref{eq: product in $M_{S}(T, q)$} {\it not} 
involving any of the $\mathscr{D}_{l}^{*}\!(\mathfrak{n}, 1\slash q)$ 
($l = 1, 2, \ldots$) act on $A_{\mathfrak{n}}(q \,T, 1\slash q)$ by 
scalar multiplication. Let 
\begin{equation*}
\mathcal{G} = \small{\prod_{l = 1}^{\infty}\,  \prod_{m-\text{odd}}  
\bigg(\frac{1 + \tilde{Z}_{l m}}{1 + \tilde{Y}_{l m}}\bigg)
^{\!\! \frac{\mathrm{Irr}_{\! q^{l}}(m)\mathscr{D}_{l}^{*}\!(\mathfrak{n}, 1\slash q)}{2l}}
\prod_{m-\text{even}} \bigg(\frac{1 + \tilde{Z}_{l m}}{1 + \tilde{Y}_{l m}}\bigg)
^{\!\! \frac{\mathrm{Irr}_{\! q^{l}}(m)(\mathrm{Irr}_{q^{-1}}(l) \,-\, n_{l})}{2}}}.   
\end{equation*}

\vskip10pt
\begin{lemma} 
\label{Diff-op} --- With notations as above, we have
\begin{equation*}
\mathcal{G} = \prod_{N = 1}^{\infty} 
\bigg(\frac{1 + \tilde{Z}_{N}}{1  + \tilde{Y}_{N}} \bigg)
^{\! (\widetilde{\mathcal{D}}_{N}^{*} \,+\, \mathcal{R}_{N})\slash 2N}
\end{equation*}
with 
\begin{equation*}
\widetilde{\mathcal{D}}_{N}^{*} \;\; = 
\sum_{\substack{k \mid N \\ (N\slash k)-\text{odd}}} 
q^{k}\mu(N\slash k)\, k\sum \frac{(-1)^{l_{1} + \cdots + l_{k} + 1} 
(l_{1} + \cdots + l_{k} - 1)!}{l_{1}! \, \cdots \, l_{k}!} 
\, \partial_{l_{1}\!, \ldots, l_{k}}
\end{equation*} 
the inner sum being over all $l_{1}, \ldots, l_{k}\in \mathbb{N}$ satisfying 
$l_{1} + 2 l_{2} + \cdots + k l_{k} = k,$ and 
\begin{equation*}
\mathcal{R}_{N} \;\, = 
\sum_{\substack{l m = N\\  m-\text{even}}} l (\mathrm{Irr}_{q^{-1}}(l) - n_{l})
\;\cdot  \!\sum_{\substack{h \mid m \\ (m \slash h)-\text{even}}}
\mu\Big(\frac{m}{h}\Big) q^{l h}. 
\end{equation*}
\end{lemma}

\begin{proof} In  
\begin{equation*} 
(2 N)^{-1} \!\sum_{\substack{l m = N \\ m-\text{odd}}} 
m\, \mathrm{Irr}_{q^{l}}(m)\mathscr{D}_{l}^{*}\!(\mathfrak{n}, 1\slash q) 
\end{equation*}
replace $\mathscr{D}_{l}^{*}\!(\mathfrak{n}, 1\slash q)$ by 
\begin{equation*}
\mathscr{D}_{l}^{*}\!(\mathfrak{n}, 1\slash q) = 
\mathcal{D}_{l}^{*} 
\, -\,  l \!\sum_{u = 1}^{\nu_{2}(l)} 
(\mathrm{Irr}_{q^{-1}}(l \slash 2^{u}) - n_{l \slash 2^{u}})\slash 2^{u}
\end{equation*}
(recall that we are assuming that $|\mathfrak{n}|$ is odd) to get
\begin{equation} \label{eq: 1}
(2 N)^{-1} \!\sum_{\substack{l m = N \\ m-\text{odd}}} 
m\, \mathrm{Irr}_{q^{l}}(m)\mathcal{D}_{l}^{*}  
\,-\,  (2 N)^{-1} \!\sum_{\substack{l m = N \\ m-\text{odd}}} 
m \,\mathrm{Irr}_{q^{l}}(m)\, l \!\sum_{u = 1}^{\nu_{2}(l)} 
(\mathrm{Irr}_{q^{-1}}(l \slash 2^{u}) - n_{l \slash 2^{u}})\slash 2^{u}.
\end{equation}
Since 
\begin{equation*}
\mathcal{D}_{l}^{*} \;\; = 
\sum_{\substack{k \mid l \\ (l\slash k)-\text{odd}}} 
\mu(l\slash k) \, k \sum \frac{(-1)^{l_{1} + \cdots + l_{k} + 1} 
(l_{1} + \cdots + l_{k} - 1)!}{l_{1}! \, \cdots \, l_{k}!} 
\, \partial_{l_{1}\!, \ldots, l_{k}}
\end{equation*} 
we can use the M\"obius inversion, in the form 
\begin{equation*}
x(l) \;\, = \sum_{\substack{k \mid l \\ (l\slash k)-\text{odd}}} 
\mu(l\slash k) y(k) \;\;\; \iff  \;\;\;
y(m) \;\, = \sum_{\substack{h \mid m \\ (m\slash h)-\text{odd}}} x(h)
\end{equation*} 
to express the first sum in \eqref{eq: 1} as  
\begin{equation*}
(2 N)^{-1} \!\sum_{\substack{k \mid N \\ (N\slash k)-\text{odd}}} 
q^{k}\mu(N\slash k) \, k\sum \frac{(-1)^{l_{1} + \cdots + l_{k} + 1}
(l_{1} + \cdots + l_{k} - 1)!}{l_{1}! \, \cdots \, l_{k}!} 
\, \partial_{l_{1}\!, \ldots, l_{k}}.
\end{equation*} 
Similarly, if we put 
$
x(l) = l \sum_{u = 1}^{\nu_{2}(l)} 
(\mathrm{Irr}_{q^{-1}}(l \slash 2^{u}) - n_{l \slash 2^{u}})\slash 2^{u},
$ 
then 
\begin{equation*}
y(m) \;\, = \sum_{\substack{h \mid m \\ (m\slash h)-\text{odd}}} x(h) \; = 
\sum_{l \mid (m\slash 2)} l (\mathrm{Irr}_{q^{-1}}(l) - n_{l})
\end{equation*}
and 
\begin{equation} \label{eq: 2}
(2 N)^{-1} \!\sum_{\substack{l m = N \\ m-\text{odd}}} 
m \,\mathrm{Irr}_{q^{l}}(m) x(l) \, = \, (2 N)^{-1} \!\sum_{\substack{l m = N \\ m-\text{odd}}} m \,\mathrm{Irr}_{q^{l}}(m) \!\sum_{\substack{k \mid l \\ (l\slash k)-\text{odd}}} 
\mu(l\slash k) y(k).
\end{equation} 
Summing first over $k$ and using the definition of $\mathrm{Irr}_{q^{l}}(m),$ we can express 
the right-hand side of \eqref{eq: 2} as 
\begin{equation*}
(2 N)^{-1} \!\sum_{\substack{k \mid N \\ (N\slash k)-\text{odd}}} 
y(k) \sum_{m s = (N\slash k)} \mu(s) 
\sum_{h \mid m} \mu(h) q^{\frac{N}{h}}
\;=\;
(2 N)^{-1} \!\sum_{\substack{k \mid N \\ (N\slash k)-\text{odd}}} 
\mu(N\slash k) y(k) q^{k}.
\end{equation*}
This further equals to 
\begin{equation*}
(2 N)^{-1} \!\sum_{\substack{l m = N\\  m-\text{even}}} l (\mathrm{Irr}_{q^{-1}}(l) - n_{l})
\; \cdot  \!\sum_{\substack{h \mid m \\ (m \slash h)-\text{odd}}}
\mu(m \slash h) q^{l h}. 
\end{equation*}
Our assertion follows at once from the fact that  
\begin{equation*}
\small{\prod_{l = 1}^{\infty}\, \prod_{m-\text{even}} 
\bigg(\frac{1 + \tilde{Z}_{l m}}{1 + \tilde{Y}_{l m}}\bigg)
^{\!\! \frac{\mathrm{Irr}_{\! q^{l}}(m)(\mathrm{Irr}_{q^{-1}}(l) \,-\, n_{l})}{2}} 
= \;\, \prod_{N = 1}^{\infty} 
\bigg(\frac{1 + \tilde{Z}_{N}}{1 + \tilde{Y}_{N}}\bigg)
^{\! \mathcal{Q}_{N}\slash 2N}}
\end{equation*}
with 
\begin{equation*} 
\mathcal{Q}_{N} \;\, = 
\sum_{\substack{l m = N\\  m-\text{even}}} l (\mathrm{Irr}_{q^{-1}}(l) - n_{l})
\sum_{h \mid m} \mu\Big(\frac{m}{h}\Big) q^{l h}. 
\end{equation*}
\end{proof} 
Now for any partition $\mathfrak{n} = (1^{n_{1}}\!, 2^{n_{2}}\!, \ldots)$ of odd weight, express  
\begin{equation*}
f_{\mathfrak{n}}(Y, Z, T, q) 
\,= \,^{\tiny{\tiny{\prod_{m = 1}^{\infty} 
(1 + Y_{m})^{(\mathrm{Irr}_{q}(m) - n_{m})\slash 2 - \mathscr{D}_{m}^{*}(\mathfrak{n}, q) \slash 2m}
(1 + Z_{m})^{(\mathrm{Irr}_{q}(m) - n_{m})\slash 2 + \mathscr{D}_{m}^{*}(\mathfrak{n}, q) \slash 2m}}}}
\!\!\!A_{\mathfrak{n}}(T, q).
\end{equation*}
Note that by \eqref{eq: f0} and \eqref{eq: iotaf0}, for any partition $\mathfrak{n} \ne {\bf 0}$ of even 
weight, we can also write 
\begin{equation*}
\frac{f_{\mathfrak{n}}^{+}(Y, Z, T, q)}{g_{{\bf 0}}^{+}(Y, Z, T, q)} 
\,= \,^{\tiny{\tiny{\prod_{m}
(1 + Y_{m})^{(\mathrm{Irr}_{q}(m) - n_{m})\slash 2 - \mathscr{D}_{m}^{*}(\mathfrak{n}, q) \slash 2m}
(1 + Z_{m})^{-(\mathrm{Irr}_{q}(m) + n_{m})\slash 2 +\mathscr{D}_{m}^{*}(\mathfrak{n}, q) \slash 2m}}}}
\!\!\!A_{\mathfrak{n}}(T, q)
\end{equation*}
and 
\begin{equation*}
\frac{f_{\mathfrak{n}}^{-}(Y, Z, T, q)}{g_{{\bf 0}}^{-}(Y, Z, T, q)} 
\,= \,^{\tiny{\tiny{\prod_{m} (1 + Y_{m})^{- \frac{n_{m} \,+\, (- 1)^{m + 1}\mathrm{Irr}_{q}(m)}{2}  \,+\, (-1)^{m + 1}\frac{\mathscr{D}_{m}^{*}(\mathfrak{n}, q)}{2m}}
(1 + Z_{m})^{ - \frac{n_{m} \,+\, (-1)^{m}\mathrm{Irr}_{q}(m)}{2} \,+\,  (-1)^{m}\frac{\mathscr{D}_{m}^{*}(\mathfrak{n}, q)}{2m}}}}}
\!\!({^\iota}\!\!A_{\mathfrak{n}}(T, q)).
\end{equation*} 
One can combine the definition of $f_{\! \textsc{\tiny$S$}}^{}(Y, Z, T, q)$ with the above identities 
to express 
\begin{equation} \label{eq: 1-fS}
f_{\! \textsc{\tiny$S$}}^{}(Y, Z, T, q) \!\prod_{l, m = 1}^{\infty} 
[(1 + \tilde{Y}_{l m}) ( 1 + \tilde{Z}_{l m})]
^{-n_{l}\mathrm{Irr}_{\! q^{l}}(m)\slash 2}. 
\end{equation} 
Indeed, by applying the first identity to  
$
\prod_{k \ge 1} 
f_{\! \textsc{\tiny$S_{k}$}}^{}\!((\tilde{Y}_{k m})_{m}, (\tilde{Z}_{k m})_{m}, T, q),
$ 
notice that the product in \eqref{eq: 1-fS} cancels out. Furthermore, by putting 
$
X^{\mathfrak{N}_{l}}= X^{l\mathfrak{n}_{l1}} \cdots X^{l\mathfrak{n}_{l n_{l}}} 
\cdots X^{l\mathfrak{n}_{l1}^{+}}  \cdots X^{{l\mathfrak{n}^{+}}_{\!\!\!\!\! ^{l n_{l}^{+}}}}\!,
$ 
for $l \ge 1,$ we can also separate out the product
\begin{equation*}
\prod_{l = 1}^{\infty} \prod_{m = 1}^{\infty} 
[(1 + \tilde{Y}_{l m}) (1 + \tilde{Z}_{l m})]
^{-(\text{exponent of $X_{l m}$ in $X^{\mathfrak{N}_{l}}$})\slash 2} 
\end{equation*}
which clearly depends only upon our original fixed partition $\mathfrak{N}.$ 
The remaining part of $f_{\! \textsc{\tiny$S$}}^{}(Y, Z, T, q),$ which acts on 
$A_{\mathfrak{n}}(q \,T, 1\slash q)$ by scalar multiplication, can be expressed using 
\eqref{eq: A_{n}(T, q)}. (Recall that the differential operator $\mathscr{D}_{m}^{*}(\mathfrak{n}, q),$ 
for $m\ge 1,$ acts on a product 
$
\prod_{k = 1}^{r} P_{\scriptscriptstyle C_{d}}(t_{k}), 
$ 
with $d \in \mathscr{P}(\mathfrak{n}, q),$ as scalar multiplication by $a_{m}^{*}(d, q).$)

Consider now $^{\scriptscriptstyle \mathcal{D}^{*}\!(Y, Z)}\!M_{\textsc{\tiny$S$}}(T, q).$ 
By the properties of the differential operator $\mathcal{D}_{N}^{*},$ 
we can write 
\begin{equation*}
^{\scriptscriptstyle \mathcal{D}^{*}\!(Y, Z)}\!M_{\textsc{\tiny$S$}}(T, q) 
\,= \, ^{\scriptscriptstyle \mathcal{D}^{*}\!(Y, Z)}\!A_{\mathfrak{n}, \mathfrak{i}, \mathfrak{j}}(q\, T, 1\slash q) \; \cdot ^{^{^{^{\scriptscriptstyle \mathcal{D}^{*}\!(Y, Z)}}}}\!\!\!{\bigg(\frac{\prod_{k, l, l'}M_{\textsc{\tiny$S_{k}^{}$}}\!(T^{k}\!, q^{k}) M_{\textsc{\tiny$S_{l}^{-}$}}(T^{l}\!, q^{l}) 
M_{\textsc{\tiny$S_{l'}^{+}$}}(T^{l'}\!, q^{l'})}{E^{^{-}}\!(T, q)^{\mathfrak{i}}E^{^{+}}\!(T, q)^{\mathfrak{j}}}\bigg)}
\end{equation*} 
where, for all $N\ge 1,$ $\mathcal{D}_{N}^{*}$ acts on $A_{\mathfrak{n}, \mathfrak{i}, \mathfrak{j}}(q\, T, 1\slash q)$ by 
$\widetilde{\mathcal{D}}_{N}^{*}$ (see Lemma \ref{Diff-op}). The second part 
\begin{equation*}
^{^{^{^{\scriptscriptstyle \mathcal{D}^{*}\!(Y, Z)}}}}\!\!{\bigg(\prod_{k}
M_{\textsc{\tiny$S_{k}^{}$}}\!(T^{k}\!, q^{k})\cdot \frac{\prod_{l, l'} 
M_{\textsc{\tiny$S_{l}^{-}$}}(T^{l}\!, q^{l}) 
M_{\textsc{\tiny$S_{l'}^{+}$}}(T^{l'}\!, q^{l'})}{E^{^{-}}\!(T, q)^{\mathfrak{i}}E^{^{+}}\!(T, q)^{\mathfrak{j}}}\bigg)}
\end{equation*} 
can be easily handled as follows. \!Express first the factors 
$A_{\mathfrak{n}_{k h}}\!(T^{k}\!, q^{k})$ (resp. 
\!$A_{\mathfrak{n}_{l h}^{\pm}}\!(\pm T^{l}\!, q^{l})$)
of $M_{\textsc{\tiny$S_{k}^{}$}}\!(T^{k}\!, q^{k})$ 
(resp. $M_{\textsc{\tiny$S_{l}^{\pm}$}}(T^{l}\!, q^{l})$) by their definition 
\eqref{eq: A_{n}(T, q)}; also, write explicitly 
$
E^{-}\!(T, q)^{\mathfrak{i}}E^{+}\!(T, q)^{\mathfrak{j}}.
$ 
For $\mathfrak{n}_{k h} = (1^{n_{1}(\mathfrak{n}_{k h})}\!, 2^{n_{2}(\mathfrak{n}_{k h})}\!, \ldots),$ 
express each characteristic polynomial $P_{\scriptscriptstyle C_{d}}(t_{\alpha}^{k}),$ for 
$d\in \mathscr{P}(\mathfrak{n}_{k h},  q^{k})$ and $\alpha = 1, \ldots, r,$ by
\begin{equation*}
P_{\scriptscriptstyle C_{d}}(t_{\alpha}^{k}) \,=\, 
\big(1 - q^{k}t_{\alpha}^{2k} \big)^{- \frac{1}{2}} 
\prod_{i = 1}^{\infty} (1 - t_{\alpha}^{2 k i})^{\! \frac{n_{i}(\mathfrak{n}_{k h})}{2}}   
\cdot \, \prod_{j = 1}^{\infty} \bigg(\frac{1 + t_{\alpha}^{k j}}{1 - t_{\alpha}^{k j}} \bigg)^{\!\! \frac{a_{\! j}^{*}\!(d, q^{k})}{2 j}}.
\end{equation*}
For $\mathfrak{n}_{l h}^{\pm} = (1^{n_{1}(\mathfrak{n}_{l h}^{\pm})}\!, 2^{n_{2}(\mathfrak{n}_{l h}^{\pm})}\!, \ldots),$ we express 
\begin{equation*}
\big(1 - \epsilon t_{\alpha}^{l} \big)
\big(1 - \epsilon q^{l}t_{\alpha}^{l} \big) 
P_{\scriptscriptstyle C_{d}}(\epsilon t_{\alpha}^{l}) \,=\, 
\prod_{i = 1}^{\infty} (1 - t_{\alpha}^{2 l i})^{\! \frac{n_{i}(\mathfrak{n}_{l h}^{\pm})}{2}}   
\cdot \, \prod_{j = 1}^{\infty} 
\bigg(\frac{1 + t_{\alpha}^{l j}}{1 - t_{\alpha}^{l j}} \bigg)^{\! \epsilon^{j} \cdot
\big(\!\frac{a_{\! j}^{ *}\!(d, q^{l})}{2 j}  - \frac{\mathrm{Irr}_{\! q^{l}}(j)}{2} \!\big)}
\qquad \text{(for $d\in \mathscr{P}(\mathfrak{n}_{l h}^{\pm}, q^{l})$)}.
\end{equation*}
Here $\epsilon = \pm 1,$ the sign agreeing with that on $\mathfrak{n}_{l h}^{\pm}.$ Notice 
that the product 
\begin{equation*}
^{^{^{^{\mathcal{D}^{*}\!(Y, Z)}}}}\!\!{\bigg(
\prod_{\alpha = 1}^{r} \prod_{k = 1}^{\infty} 
\big(1 - q^{k}t_{\alpha}^{2k} \big)^{ - n_{k} \slash 2} \bigg)} 
\cdot ^{^{^{^{\mathcal{D}^{*}\!(Y, Z)}}}}\!\!\!{\bigg(
\prod_{\alpha = 1}^{r} \prod_{l = 1}^{\infty} \prod_{m = 1}^{\infty} \big(1 - t_{\alpha}^{2 l m} \big)^{(\text{exponent of $X_{l m}$ in $X^{\mathfrak{N}_{l}}$})\slash 2} \bigg)}
\end{equation*}
factors out. Since 
\begin{equation*}
\big(1 - t^{2}\big)^{1\slash 2}  =\, \prod_{l = 1}^{\infty} \big(1 - q^{l}t^{2 l} \big)^{\mathrm{Irr}_{q^{-1}}(l) \slash 2}
\end{equation*}
it follows easily, using Lemma \ref{Identity1} \eqref{eq: 1 D_{N}^{*}}, that 
\begin{equation*}
^{^{^{^{\mathcal{D}^{*}\!(Y, Z)}}}}\!\!{\bigg(
\prod_{\alpha = 1}^{r} \big(1 - t_{\alpha}^{2} \big)^{1\slash 2} \prod_{k = 1}^{\infty} 
\big(1 - q^{k}t_{\alpha}^{2k}\big)^{ - n_{k} \slash 2} \bigg)}
E\big(T^{2}\big)^{1\slash 2}
\end{equation*} 
is precisely the contribution of the remaining part of the product $\mathcal{G}$ in Lemma \ref{Diff-op}, 
i.e., 
$
\prod_{N = 1}^{\infty} 
\Big(\frac{1 + \tilde{Z}_{N}}{1 + \tilde{Y}_{N}} \Big)
^{\!\! \frac{\mathcal{R}_{\! N}}{2N}}
$ 
to \eqref{eq: contribution over $M_{S}(T, q)$}. Then to each of the 
remaining terms, one applies Lemma \ref{Identity1} \eqref{eq: 2 D_{N}^{*}} 
to match each of them with the remaining contributions in \eqref{eq: 1-fS}.

Putting everything together, it follows that we can express 
\eqref{eq: contribution over $M_{S}(T, q)$} as 
\begin{equation*}
c_{\textsc{\tiny$S$}} \mathcal{K}_{\mathfrak{N}}(Y, Z) \, \cdot \, 
^{\scriptstyle \mathcal{D}^{*}\!(Y, Z)}{\big[E\big(T^{2}\big)^{- 1\slash 2} M_{\textsc{\tiny$S$}}(T, q) \big]} E\big(T^{2}\big)^{1\slash 2}
\end{equation*}
where we set 
\begin{equation*}
\mathcal{K}_{\mathfrak{N}}(Y, Z) : =  \frac{\prod_{l = 1}^{\infty} \prod_{m = 1}^{\infty} 
[(1 + \tilde{Y}_{l m}) (1 + \tilde{Z}_{l m})]
^{(\epsilon_{l m} - \, \text{exponent of $X_{l m}$ in $X^{\mathfrak{N}_{l}}$})\slash 2}}
{\prod_{l = 1}^{\infty} \prod_{m = 1}^{\infty} \prod_{\nu = 1}^{\infty} 
\Big(\!\frac{1 + \tilde{Z}_{2^{\nu} \!l m}}{1 + \tilde{Y}_{2^{\nu} \!l m}}\!\Big)
^{\! -(\text{exponent of $X_{l m}$ in $X^{\mathfrak{N}_{l}}$})\slash 2^{\nu + 1}}}. 
\end{equation*} 
Accordingly, if we define 
\begin{equation} \label{eq: D_{N, I, J}}
\mathscr{D}_{\mathfrak{N}, \mathfrak{I}, \mathfrak{J}}(T, q) =
E\big(T^{2}\big)^{1\slash 2}\, \text{Coefficient}_{Y^{\mathfrak{I}} Z^{\mathfrak{J}}}[\mathcal{K}_{\mathfrak{N}}(Y, Z)  \, \cdot \, 
^{\scriptstyle \mathcal{D}^{*}\!(Y, Z)}\!{E\big(T^{2}\big)^{- 1\slash 2}}\mathcal{D}^{*}\!(Y, Z)]
\end{equation}
then it is clear that $C_{\mathfrak{N}, \mathfrak{I}, \mathfrak{J}}(T, q) \,=\, 
^{\scriptscriptstyle \mathscr{D}_{\mathfrak{N}, \mathfrak{I}, \mathfrak{J}}(T, q)} C_{\mathfrak{N}, {\bf 0}, {\bf 0}}(T, q).$ Hence $C_{\mathfrak{N}, \mathfrak{I}, \mathfrak{J}}(T, q) = 0$ if $C_{\mathfrak{N}, {\bf 0}, {\bf 0}}(T, q) = 0.$ Notice that by Lemma \ref{Identity1} \eqref{eq: 1 D_{N}^{*}}, applied 
with $\ga_{1} = 1,$ $\ga_{l} = 0$ for $l\ge 2,$ and $q = 1,$ one could express
\begin{equation*}
^{\scriptstyle \mathcal{D}^{*}\!(Y, Z)}
\!{E\big(T^{2}\big)^{- 1\slash 2}} \cdot E\big(T^{2}\big)^{1\slash 2} =  \, 
\prod_{\nu = 1}^{\infty} 
\bigg(\frac{1 + \tilde{Z}_{2^{\nu}}}{1 + \tilde{Y}_{2^{\nu}}}\bigg)
^{\! - 2^{- \nu - 1}}
\end{equation*} 
which could be absorbed into $\mathcal{K}_{\mathfrak{N}}(Y, Z).$

The argument is completely analogous if $|\mathfrak{n}|$ is even.

In the next section, we shall construct $\bar{C}(X, T, q)$ such that 
\begin{equation} \label{eq: func-rel-section8}
C(C(\bar{C}(X, T, q), T, q), q  \, T, 1\slash q) = \bar{C}(X, T, q).
\end{equation}
As in the proof of Theorem \ref{split case}, one deduces easily that 
$C_{\mathfrak{N}, {\bf 0}, {\bf 0}}(T, q) = 0$ if one assumes that 
$C_{\mathfrak{N}', {\bf 0}, {\bf 0}}(T, q) = 0,$ for {\it all} partitions 
$\mathfrak{N}'$ of weights $2 \le |\mathfrak{N}'| < |\mathfrak{N}|.$ (One checks directly that 
$C_{\mathfrak{N}, {\bf 0}, {\bf 0}}(T, q) = 0$ when $\mathfrak{N} = (1^{2}),$ or $\mathfrak{N} = (2^{1}).$)

This completes the proof of the theorem subject to the construction of $\bar{C}(X, T, q)$ satisfying the required properties. \end{proof}

\section{The series $\bar{C}(X, T, q)$}\label{section: The series BarC(X, T, q)} 
In this section, we shall construct the general series $\bar{C}(X, T, q)$ satisfying the relation 
\eqref{eq: func-rel-section8}, and thus concluding the proof of Theorem \ref{princ-theorem}; the 
series $\bar{c}(x, T, q)$ used in the final step of the proof of Theorem \ref{split case} is just 
a specialization of $\bar{C}(X, T, q).$ Most of what we will say here is independent from the rest of the paper\footnote{We hope that some of the readers will find the material included in this section of independent interest.}\!.

To put things in perspective, we begin by discussing the occurrence of a special case of the series 
$\bar{C}(X, T, q)$\footnote{This series is obtained by specializing $X_{1} = x,$ $X_{n} = 0,$ for $n\ge 2,$ and taking $t_{1} = \cdots = t_{r} = 0.$} in the context of a well-known result of Ezra Getzler.

Let $\mathcal{M}_{0, n}$ ($n \ge 3$) denote the moduli space of irreducible, non-singular, 
projective curves of genus $0$ with $n$ distinct marked points. 
Let $\overline{\mathcal{M}}_{0, n}$ denote its Deligne-Mumford compactification, 
consisting of the stable curves of genus $0$ with $n$ distinct marked points. It is 
well-known (see \cite{DM} and \cite{Knud}) that the corresponding Deligne-Mumford stacks 
$\mathscr{M}_{0, n}$ and $\bar{\mathscr{M}}_{0, n}$ are smooth and defined over 
$\mathbb{Z}.$

We recall that for a Deligne-Mumford stack $\mathscr{X}$ of finite type over $\mathbb{Z},$ 
one defines its number of points over a finite field $\mathbb{F}_{\! q}$ to be 
\begin{equation*}
\selectfont\# \mathscr{X}(\mathbb{F}_{\! q}) \;\,= \sum_{\xi \, \in \, [\mathscr{X}(\mathbb{F}_{\! q})]}\, \frac{1}{\selectfont\# \text{Aut}_{\scriptscriptstyle \mathbb{F}_{q}}\!(\xi)}
\end{equation*}
the sum being taken over the isomorphism classes of the category $\mathscr{X}(\mathbb{F}_{\! q}).$ 
Here, for an isomorphism class $\xi,$ 
$\text{Aut}_{\scriptscriptstyle \mathbb{F}_{q}}\!(\xi)$ denotes the finite group of automorphisms 
of any object representing $\xi.$ In particular, if $\mathscr{X} = \mathscr{M}_{0, n},$ then, 
by counting points, one sees that  
\begin{equation*}
\selectfont\#  \mathscr{M}_{0, n}(\mathbb{F}_{\! q}) 
= \selectfont\#  \mathcal{M}_{0, n}(\mathbb{F}_{\! q}) = 
\prod_{k = 2}^{n - 2} (q - k) \qquad \text{($n \ge 3$)}.
\end{equation*} 
The number of points $\selectfont\#  \bar{\mathscr{M}}_{0, n}(\mathbb{F}_{\! q})$ ($n\ge 3$) can be obtained from the following theorem of Getzler \cite{Getz1}; see also McMullen \cite{McM}.

\vskip10pt
\theoremstyle{Getz}
\newtheorem*{Getzler}{Theorem (Getzler)}
\begin{Getzler} --- The generating functions
\begin{equation*}
f(x) = x - \sum_{n = 2}^{\infty} \selectfont\# \mathscr{M}_{0, n + 1}(\mathbb{F}_{\! q})\frac{x^{n}}{n!} 
= \frac{1 + q^{2} x  - (1 + x)^{q}}{q (q - 1)} \;\;\,  \text{and} \;\;\;
g(x) = x + \sum_{n = 2}^{\infty} \selectfont\# \bar{\mathscr{M}}_{0, n + 1}(\mathbb{F}_{\! q})
\frac{x^{n}}{n!}
\end{equation*}
are formal compositional inverses of one another.
\end{Getzler} 
	
Note that $g(x) = g(x, q)$ satisfies the functional equation 
\begin{equation} \label{eq: feq-g}
q^{2} g(x, q) = q (q - 1) x + g(q x, 1\slash q)
\end{equation} 
which can be verified directly.

	The connection between the function $g(x)$ in Getzler's theorem and 
	$
	\bar{c}(x, T, q)
	$ 
	used in the proof of Theorem \ref{split case} can be easily seen as follows. First, set 
	$t_{1} = \cdots = t_{r} = 0$ and $y = z$ in $c_{\text{odd}}(x, y, z; T, q)$ and 
	$c_{\text{even}}(x, y, z; T, q)$ introduced in \ref{sub-section-7.1-split-case}. Denoting $c_{\text{odd}}(x, z, z; 0, \ldots, 0, q)$ (resp. $c_{\text{even}}(x, z, z; 0, \ldots, 0, q)$) 
	by $c_{\text{odd}}(x, z; q)$ (resp. $c_{\text{even}}(x, z; q)$), one can write, for example,  
	\begin{equation*}
	c_{\text{odd}}(x, z; q) \;\, =
	\sum_{n, k \, \ge \, 0} x^{2n + 1} z^{k}
	\sum_{i + j = k} \frac{A_{2n + 1, i, j}(0, \ldots, 0, q)}{i! j!}.
	\end{equation*}
	For fixed $n\ge 1$ and $k \ge 0,$ we have 
	\begin{equation*}
	\sum_{i + j = k} \frac{A_{n, i, j}(0, \ldots, 0, q)}{i! j!} \; = 
	\sum_{d \, \in \, \mathscr{P}_{\! n}} \sum_{i + j = k} 
	\frac{N_{i, j}(d, q)}{i! j!}\; = \sum_{d \, \in \, \mathscr{P}_{\! n}} \sum_{i + j = k} 
	\binom{\frac{q - n + \epsilon + a_{1}(C_{d})}{2}}{i}
	\binom{\frac{q - n -\epsilon - a_{1}(C_{d})}{2}}{j}.
	\end{equation*} 
	The last double sum is just 
	$
	\binom{q}{n} \binom{q - n}{k}.
	$ 
	Hence 
	\begin{equation*}
	c_{\text{odd}}(x, z; q) =\frac{(1 + x + z)^{q} - (1 - x + z)^{q}}{2}.
	\end{equation*} 
	Similarly,
	\begin{equation*}
	c_{\text{even}}(x, z; q) \, =  - \, 1 \, +\, \frac{(1 + x + z)^{q} + (1 - x + z)^{q}}{2}.
	\end{equation*}
	Next, consider 
	$$
	\bar{c}(x, q) : = \bar{c}(x, T, q)\vert_{\scriptscriptstyle t_{1} = \cdots = t_{r} = 0} = 
	(\bar{A}_{\text{odd}}(x, q), \bar{A}_{\text{even}}(x, q), \bar{A}_{\text{even}}(x, q))
	$$ 
	where $\bar{c}(x, T, q)$ will be given in \ref{princ-theorem}. The components 
	$\bar{A}_{\text{odd}}(x, q)$ and $\bar{A}_{\text{even}}(x, q)$ are formal power series of the form 
	$$
	\bar{A}_{\text{odd}}(x, q) = x \; + \sum_{\substack{m \, \ge \, 3\\ m-\text{odd}}} a_{m}(q) x^{m} \;\;\, 
	\text{and} \;\;\, 
	\bar{A}_{\text{even}}(x, q) \; = \sum_{\substack{m \, \ge \, 2\\ m-\text{even}}} a_{m}(q) x^{m}
	$$ 
	and (as it turns out in this special case) these series can also be obtained as the unique solution of the system 
	\begin{equation*} \label{eq: def-eq1}
	c_{\text{odd}}(\bar{A}_{\text{odd}}(x, q), \bar{A}_{\text{even}}(x, q); q) \, = - \, q (q - 1)x + q^{2} \bar{A}_{\text{odd}}(x, q)
	\end{equation*} 
	and 
	\begin{equation*} \label{eq: def-eq2}
	c_{\text{even}}(\bar{A}_{\text{odd}}(x, q), \bar{A}_{\text{even}}(x, q); q) \,  =\, q^{2}\bar{A}_{\text{even}}(x, q).
	\end{equation*} 
	Setting $A : = \bar{A}_{\text{odd}} + \bar{A}_{\text{even}},$ it follows that 
	\begin{equation*}
	(c_{\text{odd}} \, + \, c_{\text{even}})(\bar{A}_{\text{odd}}(x, q), \bar{A}_{\text{even}}(x, q); q) 
	= -  1  +  (1 + A(x, q))^{q} = 
	- q(q - 1)x + q^{2} A(x, q)
	\end{equation*} 
	and thus 
	\begin{equation*}
	A(x, q) = g(x, q) 
	= x + \sum_{n = 2}^{\infty} \selectfont\# \bar{\mathscr{M}}_{0, n + 1}(\mathbb{F}_{\! q})\frac{x^{n}}{n!}.
	\end{equation*} 
	Note that the functional equation \eqref{eq: feq-g} implies the relation 
	\eqref{eq: functional-rel-special-split-case}.

	As we shall shortly see, Getzler's theorem will be used, in fact, in the initial step of the construction of the more general series $\bar{C}(X, T, q).$

	\subsection{Moduli spaces of admissible double covers} 
	We first need to introduce certain moduli spaces of admissible double covers, 
	and some generating series attached to them.

	For an even positive integer $n$ and a nonnegative integer $m$ such that $m + n\ge 3,$ let 
	$\mathcal Adm_{n, m}$ denote the moduli space of admissible double covers over stable curves of genus $0$ with $n$ branched points and additional $m$ marked points. These moduli spaces are smooth, proper Deligne-Mumford stacks defined over $\mathrm{Spec}\, \mathbb{Z}[\frac{1}{2}];$ the standard references for stacks of admissible covers are \cite{Ha-Mum} and \cite{ACV}.

	Let $\mathbb{F}_{\! q}$ be, as before, a finite field of odd characteristic, and let $\overline{\mathbb{F}}_{\!q}$ be a fixed algebraic closure of it. We shall consider the moduli space $\mathcal Adm_{n, m}$ over $\mathbb{F}_{\! q}.$

	The group $G: = G_{n, m} := S_{n} \times (S_{2} \wr S_{m})$ acts on the 
	$\overline{\mathbb{F}}_{\! q}$-points of the elements of $\mathcal Adm_{n, m}$ by permuting the marked points of the target curves and (eventually) by switching the points in the corresponding fibers. It is mentioned in \cite{Car72} (p. \!2), with reference to \cite{Spech} and \cite{Youn}, that, for the hyperoctahedral group $S_{2} \wr S_{m}$ of order $2^{m} \cdot m!,$ the conjugacy classes are parametrized by pairs of partitions $(\mathfrak{i}, \mathfrak{j})$ such that $|\mathfrak{i}| + |\mathfrak{j}| = m.$ Thus the conjugacy classes of the group $G_{n, m}$ are parametrized by triples of partitions $(\mathfrak{n}, \mathfrak{i}, \mathfrak{j})$ with $|\mathfrak{n}| = n$ and $|\mathfrak{i}| + |\mathfrak{j}| = m.$

	To be more precise, an element $\sigma$ of the hyperoctahedral factor $S_{2} \wr S_{m}$ can be viewed as a signed permutation acting on the set $\mathfrak{I} := \{\pm 1, \ldots, \pm m\}$ with the property that $\sigma(a) = - \sigma(-a)$ for all $a\in \mathfrak{I}.$ The orbit of an element $a\in \mathfrak{I}$ is either even (i.e., if $b$ is in the orbit of $a$ then $-b$ is also in the orbit of $a$) or, alternatively, the orbit of $a$ is disjoint from the orbit of $-a.$ The two partitions $(\mathfrak{i}, \mathfrak{j})$ corresponding to the conjugacy class of $\sigma$ are determined by these orbits. Here $\mathfrak{i}$ corresponds to the even orbits and $\mathfrak{j}$ to the remaining ones.

	Let $D\to C$ be an admissible double cover. Let $P\in C(\overline{\mathbb{F}}_{\! q})$ be a marked point which is not branched, and let $(P_{1}, P_{2})$ be the points above $P.$ Let $\sigma \in S_{2} \wr S_{m},$ and assume that $D\to C$ is invariant under $F\sigma.$ If the orbit of $[(P_{1}, P_{2}) \to P]$ is even, there exists a smallest positive integer $k$ such that $\sigma^{k}([(P_{1}, P_{2}) \to P]) = [(P_{2}, P_{1})\to P].$ It follows that $P\in C(\mathbb{F}_{\! q^{k}})$ and $F^{k}(P_{1}) = P_{2};$ thus $P_{1}$ and $P_{2}$ are {\it not} defined over 
	$\mathbb{F}_{\! q^{k}}.$ If the orbit of $[(P_{1}, P_{2}) \to P]$ is disjoint from the orbit of $[(P_{2},P_{1}) \to P],$ then $F^{k}(P_{1}) = P_{1}$ and $F^{k}(P_{2}) = P_{2}$ for all $k$ such that $\sigma^{k}(P) = P,$ and thus the points above $P$ are defined over the same field as $P.$

	For any $\tau \in G,$ consider the set $\mathcal Adm_{n,m}^{F\tau}$ of fixed points of $F\tau$ on $\mathcal Adm_{n, m}.$ Note that there exists a unique moduli $\mathcal Adm_{n,m}^{\tau}$ such that \begin{equation*}
	\mathcal Adm_{n, m}^{\tau} \otimes_{_{\mathbb{F}_{\! q}}} \! \overline{\mathbb{F}}_{\! q} 
	\simeq \mathcal Adm_{n, m} \otimes_{_{\mathbb{F}_{\! q}}} \! \overline{\mathbb{F}}_{\! q}
	\end{equation*}
	and $F\tau$ becomes the relative Frobenius endomorphism induced by this 
	$\mathbb{F}_{\! q}$-structure on $\mathcal Adm_{n, m}.$

	For partitions 
	$
	\mathfrak{n} = (1^{\! n_{1}}\!, 2^{n_{2}}\!, \ldots), 
	\mathfrak{i} = (1^{\! i_{1}}\!, 2^{i_{2}}\!, \ldots)
	$ 
	and 
	$
	\mathfrak{j} = (1^{\! j_{1}}\!, 2^{j_{2}}\!, \ldots)
	$ 
	with $\mathfrak{n}$ of even weight and 
	$
	|\mathfrak{n}| + |\mathfrak{i}| + |\mathfrak{j}| \ge 3,
	$ 
	let $\overline{\mathcal{D}}_{\mathfrak{n}, \mathfrak{i}, \mathfrak{j}}$ denote the category of admissible double covers over $\mathbb{F}_{\! q}$ of genus $0$ stable curves with branching given by $\mathfrak{n},$ i.e., with $n_{k}$ branched $\mathbb{F}_{\! q^{k}}$-Galois orbits of cardinality $k$ (for $k = 1, 2, \ldots$), and additional marked points of the base given by the partitions $\mathfrak{i}$ and $\mathfrak{j}$ such that: 
	\begin{itemize}
		\item[\textasteriskcentered] each point in the set of marked points corresponding to $\mathfrak{j}$ and the points in its fiber are defined over the same field, 
		
		\item[\textasteriskcentered] the points in the fiber of any marked point corresponding to $\mathfrak{i}$ are defined over an (necessarily quadratic) extension of the field of definition of the base point.
	\end{itemize} 
	By convention, we extend this notation to the case $|\mathfrak{n}| = n = 0,$ where trivial (unramified) admissible double covers can be identified with elements of the moduli space $\overline{M}_{0, |\mathfrak{j}|}.$ Moreover, let $\mathcal{D}_{\mathfrak{n}, \mathfrak{i}, \mathfrak{j}}$ denote the subcategory of $\overline{\mathcal{D}}_{\mathfrak{n}, \mathfrak{i}, \mathfrak{j}}$ corresponding to the smooth locus.

	For independent variables $X = (X_{1}, X_{2}, \ldots),$ $Y = (Y_{1}, Y_{2}, \ldots)$ and 
	$Z=(Z_{1}, Z_{2}, \ldots),$ we introduce the generating series $\mathscr{D}(X, Y, Z; q)$ and $\bar{\mathscr{D}}(X, Y, Z; q)$ defined by 
	\begin{equation} \label{eq: gen-seriesD}
	\mathscr{D}(X, Y, Z; q) : = 
	\sum_{\mathfrak{n},\mathfrak{i},\mathfrak{j}} 
	\selectfont\# \mathcal{D}_{\mathfrak{n}, \mathfrak{i}, \mathfrak{j}}(\mathbb{F}_{\! q}) 
	X^{\mathfrak{n}} Y^{\mathfrak{i}} Z^{\mathfrak{j}}
	\end{equation}
	and 
	\begin{equation} \label{eq: gen-series-barD}
	\bar{\mathscr{D}}(X, Y, Z; q) : = \sum_{\mathfrak{n},\mathfrak{i},\mathfrak{j}} 
	\selectfont\# \overline{\mathcal{D}}_{\mathfrak{n}, \mathfrak{i}, \mathfrak{j}}(\mathbb{F}_{\! q}) 
	X^{\mathfrak{n}} Y^{\mathfrak{i}} Z^{\mathfrak{j}}.
	\end{equation} 
	In what follows, we shall identify $\mathcal Adm_{n, m}^{F\tau}$ with $[\overline{\mathcal{D}}_{\mathfrak{n},\mathfrak{i},\mathfrak{j}}(\mathbb{F}_{\! q})],$ where $\tau$ is any element in the conjugacy class represented by $(\mathfrak{n}, \mathfrak{i}, \mathfrak{j}).$

	By the Grothendieck-Lefschetz trace formula \cite{Beh1, Beh}, 
	\begin{equation*} 
	\selectfont\# \overline{\mathcal{D}}_{\mathfrak{n}, \mathfrak{i}, \mathfrak{j}}(\mathbb{F}_{\! q}) \, = \, q^{\dim \, \mathcal Adm_{n,m}}
	\sum_{k}\,  (-1)^{k}\,\mathrm{Tr}(\Phi_{q}\tau^{-1}\, \vert \, H^{k}(\mathcal Adm_{n,m} \otimes_{_{\mathbb{F}_{\! q}}}\! \overline{\mathbb{F}}_{\! q}, \mathbb{Q}_{\ell}))
	\end{equation*} 
	where $\ell$ is a prime different from the characteristic of $\mathbb{F}_{\! q},$ $\Phi_{q}$ is the arithmetic Frobenius endomorphism (relative to $\mathbb{F}_{\! q}$), and $\Phi_{q}\tau^{-1}$ acts, 
	by {\it transport of structures}, on the $\ell$-adic cohomology. We also define 
	\begin{equation} \label{eq: gen-series-barD-1/q}
	\bar{\mathscr{D}}(X, Y, Z; 1\slash q) : = \sum_{\mathfrak{n},\mathfrak{i},\mathfrak{j}} 
	\Big(\sum_{k}\,  (-1)^{k}\, \mathrm{Tr}(\Phi_{q}\tau^{-1}\, \vert \, H^{k}(\mathcal Adm_{n,m} \otimes_{_{\mathbb{F}_{\! q}}} \! \overline{\mathbb{F}}_{\! q}, \mathbb{Q}_{\ell}))\Big)\,  X^{\mathfrak{n}} Y^{\mathfrak{i}} Z^{\mathfrak{j}}.
	\end{equation}

	\subsection{A homomorphism} 
	To make the transition from 
	$
	\mathscr{D}(X, Y, Z; q)
	$ 
	and 
	$
	\bar{\mathscr{D}}(X, Y, Z; q)
	$ 
	to generating functions of products of characteristic polynomials, closely related to 
	$
	C(\underline{X}, T\!, q),
	$ 
	and to construct the function 
	$
	\bar{C}(X, T\!, q),
	$ 
	which is the main goal of this section, it is convenient to introduce two rings of formal power series equipped with additional structure, and a homomorphism between them compatible with the (additional) structure of the rings. As we shall shortly see, this homomorphism enjoys additional properties. 
	
	Let $\mathscr{F} = \mathbb{Q}(q)$ be the field of rational functions in a variable $q;$ we assume that $q$ takes values in the set of prime powers. Let $\mathscr{R}:=\mathscr{R}_{\mathscr{F}}$ denote the ring of formal power series in infinitely many variables $X: = (X_{n})_{n\ge 1}, Y: = (Y_{n})_{n\ge 1}$ and $Z: = (Z_{n})_{n\ge 1}$ with coefficients in $\mathscr{F}.$ We also consider 
	$\mathscr{S}: = \mathscr{S}_{\mathscr{F}},$ the ring of formal power series in infinitely many variables $X: =(X_{n})_{n\ge 1}$ and $p: = (p_{n})_{n\ge 1}$ with coefficients in $\mathscr{F}.$ 
	
	On each $\mathscr{R}$ and $\mathscr{S},$ we define: 
	\begin{enumerate}
		\item A family of ring homomorphisms ({\it Adams operations}), $\{\psi^{n}\}_{n\in\mathbb{Z}_{>0}}$ such that: 
		\begin{enumerate}
			\item $\psi^{n}(X_{m})=X_{mn}$
			\item $\psi^{n}(Y_{m})=Y_{mn}$
			\item $\psi^{n}(Z_{m})=Z_{mn}$
			\item $\psi^{n}(p_{m})=p_{mn}$
			\item $\psi^{n}\!f(q)=f(q^{n})$ for $f \in \mathscr{F}$ 
		\end{enumerate} 
		
		\item The involution $\iota$ characterized by: 
		\begin{enumerate}
			\item $\iota(Y_{n})=Z_{n}$ if $n$ is odd
			\item $\iota(Z_{n})=Y_{n}$ if $n$ is odd
			\item $\iota(p_{n})=-p_{n}$ if $n$ is odd, respectively 
			\item All the other variables (and constants) are fixed by $\iota$
		\end{enumerate} 
		
		\item The involution $\delta$ given by: 
		\begin{enumerate}
			\item $\delta f(X,Y,Z;q):=f(q\star X, q\star Y, q\star Z; 1\slash q)$ 
			\item $\delta h(X,p;q):=h(q\star X, q\star p; 1\slash q),$ respectively 
			
			Here $q\star (U_{n})_{n\ge 1}:=(q^{n} U_{n})_{n\ge 1}$ for any set of variables $(U_{n})_{n\ge 1}.$
		\end{enumerate}
	\end{enumerate} 
	Using \eqref{eq: gen-seriesD}, \eqref{eq: gen-series-barD} and \eqref{eq: gen-series-barD-1/q}, one can extend $\psi^{n}, \iota$ and $\delta$ 
	to the generating series $\mathscr{D}(X, Y, Z; q)$ and 
	$\bar{\mathscr{D}}(X, Y, Z; q)$ (or similar such series) in an obvious way. 
	For example, 
	\begin{equation*}
	\psi^{m}\bar{\mathscr{D}}(X, Y, Z; q) : = \sum_{\mathfrak{n},\mathfrak{i},\mathfrak{j}} 
	\selectfont\# \overline{\mathcal{D}}_{\mathfrak{n}, \mathfrak{i}, \mathfrak{j}}(\mathbb{F}_{\! q^{m}}) 
	\psi^{m}(X^{\mathfrak{n}}) \psi^{m}(Y^{\mathfrak{i}}) \psi^{m}(Z^{\mathfrak{j}}) \qquad \text{(for $m \ge 1$)}
	\end{equation*} 
	and 
	\begin{equation*}
	\delta \bar{\mathscr{D}}(X,Y,Z;q) 
	: = \bar{\mathscr{D}}(q\star X, q\star Y, q\star Z; 1\slash q)
	= \sum_{\mathfrak{n},\mathfrak{i},\mathfrak{j}} 
	q^{|\mathfrak{n}| + |\mathfrak{i}| + |\mathfrak{j}|}
	\Big(\sum_{k}\,  (-1)^{k}\mathrm{Tr}(\Phi_{q}\tau^{-1}\, \vert \, 
	H^{k}({\mathcal Adm}_{n,m}\otimes_{_{\mathbb{F}_{\! q}}} \! \overline{\mathbb{F}}_{\! q}, \mathbb{Q}_{\ell}))\Big)\,  X^{\mathfrak{n}} Y^{\mathfrak{i}} Z^{\mathfrak{j}}.
	\end{equation*} 
	Since 
	$
	\dim \, \mathcal Adm_{n, m} = n + m - 3,
	$ 
	note that we have the relation
	\begin{equation} \label{eq: func-eq-barD-qto1/q}
	\delta \bar{\mathscr{D}}(X, Y, Z; q) 
	= q^{3} \bar{\mathscr{D}}(X, Y, Z; q).
	\end{equation} 
	We remark that one may choose to work more generally with formal power series over the Grothendieck ring of the abelian category 
	$
	\mathbf{Rep}_{\mathbb{Q}_{\scriptscriptstyle \ell}}\!(\mathrm{Gal}(\overline{\mathbb{Q}}\slash \mathbb{Q}))
	$ 
	of $\ell$-adic Galois representations, where one can define natural operations $\{\psi^{n}\}_{n\in\mathbb{Z}},$ cf. \cite[Chapters~5, 6]{Ser}.

	For two elements $f\in\mathscr{R}$ and $g=(g_{\scriptscriptstyle 1},  g_{\scriptscriptstyle 2}, g_{\scriptscriptstyle 3})\in \mathscr{R}^{3}$ 
	such that $g_{\scriptscriptstyle i}$ ($i = 1, 2, 3$) has no constant term, 
	we define the {\it plethystic} substitution $f\circ g$ by replacing $X_{n}\mapsto \psi^{n}\!(g_{\scriptscriptstyle 1}),$ $Y_{n}\mapsto \psi^{n}\!(g_{\scriptscriptstyle 2})$ and $Z_{n}\mapsto \psi^{n}\!(g_{\scriptscriptstyle 3})$ ($n\ge 1$) in the expression of $f.$

	\subsection{Special elements} 
	In each of the rings $\mathscr{R}$ and $\mathscr{S}$ there exists a special element defined as follows. 
	
	On $\mathscr{R}$ consider the element 
	\begin{equation*}
	\dz \!\mathscr{D}_{0}(Z;q) \, = \, \frac{-1-q Z_{1} + \prod_{n\ge 1}(1 + Z_{n})^{\text{Irr}_{q}(n)}}{q(q-1)}\in
	\mathscr{R}
	\end{equation*}
	where 
	\begin{equation*}
	\mathscr{D}_{0}(Z;q) \, = \, \frac{- 1 + (1 + Z_{1}^{})\prod_{n \ge 1}(1 + Z_{n}^{})^{\text{Irr}_{q}(n)}}{q(q^{2} - 1)} 
	- \frac{Z_{1}^{2}}{2(q - 1)} - \frac{Z_{1}^{}}{q(q - 1)} 
	- \frac{Z_{2}^{}}{2(q + 1)}
	\end{equation*} 
	is the constant term in $X = (X_{n})_{n\ge 1}$ of 
	$
	\mathscr{D}(X,Y,Z;q),
	$ 
	and $\dz$ \!denotes the partial derivative with respect to $Z_{1}.$ The expression of 
	$\mathscr{D}_{0}(Z;q)$ can be obtained by direct computation; see also 
	\cite[p. \!7, Remark]{Di-ch}, where $\mathscr{D}_{0}(Z;q)$ is denoted by 
	$
	\mathrm{ch}_{\scriptscriptstyle 0}^{}(\mathcal{V})(p).
	$ 
	Now, the special element we consider is the unique solution $S(Z;q)$ of the equation   
	\begin{equation} \label{eq: defS}
	\dz \!\mathscr{D}_{0}(Z;q)\circ (Z_{1} + S(Z;q))=S(Z;q)
	\end{equation}
	where ``$\circ$" is the plethystic substitution\footnote{This is defined by replacing 
		$
		Z_{n}\mapsto (Z_{n} + \psi^{n}\!S(Z; q))
		$ 
		($n\ge 1$) in the expression of $\dz \!\mathscr{D}_{0}(Z;q).$}. 
	Letting $\bar{\mathscr{D}}_{0}(Z;q)$ denote the constant term in 
	$X = (X_{n})_{n\ge 1}$ of 
	$
	\bar{\mathscr{D}}(X,Y,Z;q),
	$
	we have 
	\begin{equation*} 
	S(Z;q) = \dz \!\bar{\mathscr{D}}_{0}(Z;q).
	\end{equation*} 
	This identity can be easily deduced from \cite[Theorem~7.17]{GK} and \cite[Theorem~7.15 (c)]{GK}, or \cite[Corollary~4.2]{Di-ch}.

	The special element on $\mathscr{S}$ is 
	\begin{equation*}
	E(p): = \exp{\Big(\sum_{n\ge 1}\frac{p_{n}}{n}\Big)} 
	= 1 + \sum_{k \ge 1} \frac{1}{k!}\Big(\sum_{n\ge 1}\frac{p_{n}}{n}\Big)^{k}.
	\end{equation*}

	With the above notation and terminology, we have:
	
	\vskip10pt
	\begin{proposition}\label{unique-hom} --- There exists a unique homomorphism 
		$\pi:\mathscr{R}\rightarrow \mathscr{S}$ with the following properties:
		\begin{equation*}
		\pi \vert_\mathscr{F} = \mathrm{Id}_{\mathscr{F}},\;\;\, \pi(X_{1})=X_{1},\;\;\,
		\pi(1+Z_{1}+S(Z; q))=E(p),
		\end{equation*}
		$$
		\pi\circ\psi^{n}=\psi^{n}\!\circ \pi\;\;\; \forall n\ge 1\;\;\, \text{and} \;\;\, 
		\pi\circ\iota = \iota\circ \pi.
		$$
		Moreover, we have 
		\begin{equation*}
		\pi(Z_{1})=\frac{1-q^{2}+q^{2} E(p)-\delta E(p)}{q(q-1)}\;\;\, \text{and} \;\;\, \pi\circ\delta = \delta\circ\pi.
		\end{equation*}
	\end{proposition}

	\begin{proof} By setting $\deg(X_{n})=\deg(Y_{n})=\deg(Z_{n})=\deg(p_{n})=n,$ we obtain a grading on each of the rings $\mathscr{R}$ and $\mathscr{S}.$ Then, via Adams operations, each of the rings is freely 
		generated by the degree one elements. From the definition of $\dz \!\mathscr{D}_{0}(Z;q)$ and 
		\eqref{eq: defS}, we see that every monomial in $S(Z;q)$ has (weighted) degree $\ge 2.$ Since 
		\begin{equation*}
		E(p) = 1 + p_{1} + \textrm{higher degree terms}
		\end{equation*} 
		and $Y_{1}=\iota(Z_{1}),$ the existence and uniqueness of $\pi$ follows.
				
		To compute $\pi(Z_{1}),$ we first express \eqref{eq: defS} in the equivalent form 
		\begin{equation*}
		(Z_{1} - \dz \!\mathscr{D}_{0}(Z;q))\circ (Z_{1} + S(Z;q))=Z_{1}.
		\end{equation*} 
		Applying $\pi$ to both sides, and using the fact that $\pi\circ\psi^{n}=\psi^{n}\!\circ \pi,$ one gets: 
		\begin{equation*} 
		\pi(Z_{1}) = (Z_{1} - \dz \!\mathscr{D}_{0}(Z;q))\circ \pi(Z_{1} + S(Z;q)) 
		= E(p) - 1 - \dz \!\mathscr{D}_{0}(Z;q)\circ (E(p) - 1). 
		\end{equation*} 
		Thus 
		\begin{equation*} 
		\pi(Z_{1}) = \frac{1-q^{2}+q^{2}E(p)-\prod_{n\ge 1}(\psi^{n}E(p))^{\text{Irr}_{q}(n)}}{q(q-1)}. 
		\end{equation*} 
		We have the formula
		\begin{equation} \label{eq: formalzeta-Euler-decomp}
		\prod_{n\ge 1}(\psi^{n}\!E(p))^{\text{Irr}_{q}(n)} = \delta E(p)
		\end{equation} 
		which can be easily verified by matching the logarithms of the two sides, and the expression 
		of $\pi(Z_{1})$ stated follows. 
		
		Finally, to prove that $\pi$ commutes with $\delta,$ we note that the formula for $\pi(Z_{1})$ 
		implies that $\pi(\delta(Z_{1}))=\delta(\pi(Z_{1})).$ The remaining relations can be verified 
		by simple calculations.\end{proof}

	\subsection{Further Computations}
	The formulas given in the next two lemmas will play a key role in our argument.

	\vskip10pt
	\begin{lemma}\label{pi-elem1} --- Let $\dz \!S(Z;q)$ denote the partial derivative of $S(Z;q)$ 
		with respect to $Z_{1}.$ Then
		\begin{equation*}
		\pi(1+\dz \!S(Z;q))=\frac{(q-1)E(p)}{q E(p) - \delta E(p)}.
		\end{equation*}  
	\end{lemma}

	\begin{proof} Writing \eqref{eq: defS} explicitly and differentiating with respect to $Z_{1},$ we find easily that 
		\begin{equation*}
		\dz \!S(Z;q) \, = \, (1 + \dz \!S(Z;q))
		\cdot \left\{\frac{-(1 + Z_{1} + S(Z;q)) + \prod_{n\ge 1}\,(1 + Z_{n} + \psi^{n}S(Z;q))^{\text{Irr}_{q}(n)}}
		{(q-1)(1 + Z_{1} + S(Z;q))}\right\}.
		\end{equation*} 
		To this we apply $\pi.$ Then Proposition \ref{unique-hom} and \eqref{eq: formalzeta-Euler-decomp} yield 
		\begin{equation*}
		\pi(1 + \dz \!S(Z;q)) - 1 \, = \, \pi(1 + \dz \!S(Z;q))
		\cdot \left(\frac{-E(p) + \delta E(p)}{(q-1) E(p)}\right)
		\end{equation*} 
		and the lemma follows.\end{proof}

	\vskip10pt
	\begin{lemma}\label{pi-elem2} --- Let $\dx \!\mathscr{D}_{2}(X,Y,Z;q)$ be defined by
		\begin{equation*}
		\dx \!\mathscr{D}_{2}(X,Y,Z;q) = \frac{2 X_{1}\!\left(\prod_{m}(1+Y_{m})^{c_{m \slash 2^{\nu_{2}(m)}}\big(q^{2^{\nu_{2}(m)}}\big)\slash (2m)}
			(1+Z_{m})^{\big[2c_{m}(q)\, - \, c_{m\slash 2^{\nu_{2}(m)}}\big(q^{2^{\nu_{2}(m)}}\big)\big]\slash (2m)}- 1\right)}{q(q-1)}
		\in\mathscr{R}.
		\end{equation*} 
		Define $K(X, Y, Z; q)$ as the unique solution of the equation 
		\begin{equation*}
		\dx \!\mathscr{D}_{2}(X, Y, Z; q)\circ (X_{1} + K(X, Y, Z; q), Y_{1} + \iota(S(Z; q)), 
		Z_{1} + S(Z; q)) = K(X, Y, Z; q)
		\end{equation*} 
		where $\circ$ is the plethystic substitution. Then
		\begin{equation*}
		\pi(K(X, Y, Z; q))=0.
		\end{equation*} 
		In particular, $\pi(\dx \!K(X, Y, Z; q))=0.$ 
	\end{lemma}

	\begin{proof} As before, using Proposition \ref{unique-hom} and \eqref{eq: formalzeta-Euler-decomp}, 
		we write
		\begin{equation*}
		\pi(K(X, Y, Z; q)) =
		\frac{2 (X_{1} + K(X, Y, Z; q))\!\left(\prod_{m}\Big(\frac{\psi^{m}\iota E(p)}{\psi^{m} E(p)}\Big)^{c_{m \slash 2^{\nu_{2}(m)}}\big(q^{2^{\nu_{2}(m)}}\big)\slash (2m)}\cdot \frac{\delta E(p)}{E(p)} - 1\right)}{q(q-1)}.
		\end{equation*} 
		The lemma follows from the identity 
		\begin{equation*}
		\prod_{m \ge 1}\bigg(\frac{\psi^{m}\iota E(p)}{\psi^{m} E(p)}\bigg)^{c_{m \slash 2^{\nu_{2}(m)}}\big(q^{2^{\nu_{2}(m)}}\big)\slash (2m)} = \, \frac{E(p)}{\delta E(p)}.
		\end{equation*} \end{proof}

	\begin{remark} \!It can be verified that $K(X, Y, Z; q)$ in the last lemma is precisely 
		\begin{equation*}
		K(X, Y, Z; q) = \dx \!\bar{\mathscr{D}}_{2}(X,Y,Z;q)
		\end{equation*} 
		where $\bar{\mathscr{D}}_{2}(X, Y, Z; q)$ is the coefficient of $X_{1}^{2}$ in $\bar{\mathscr{D}}(X, Y, Z; q).$	
	\end{remark}

	\subsection{Essential parts of admissible double covers with marked points} 
	For a function $F(X, Y, Z; q),$ define 
	\begin{equation*}
	\partial F(X, Y, Z; q) := (\dx F(X, Y, Z; q), \dy F(X, Y, Z; q), \dz F(X, Y, Z; q))
	\end{equation*} 
	where $\dx,$ $\dy$ \!and $\dz$ denote the partial derivatives of $F$ with respect to $X_{1},$ $Y_{1}$ and $Z_{1},$ respectively. 
	
	The main purpose of this section is to obtain a recursive relation for 
	$
	\pi(\partial \bar{\mathscr{D}}(X, Y, Z; q)). 
	$ 
	To do so, we begin by introducing some useful notation and terminology. 
	
	Let $C$ be an admissible double cover with a distinguished smooth point $P_{0} = P_{0}(C)$ sometimes called the root of $C;$ we shall always assume that the point of the base curve corresponding to $P_{0}$ is defined over the field of definition of $C.$ The irreducible component of $C$ containing $P_{0}$ is called the root component of $C.$ To each point $P$ of $C,$ by which we mean a point of the base curve together with its fiber, we associate a color, or type, $\kappa(P)$ as follows: 
	\begin{enumerate}[label=(\roman*)]
		\item $\kappa(P) = \mathrm{white},$ or $X_{1}$-type, if the fiber of the corresponding point of the base curve consists of one point. 
		
		\item $\kappa(P) = \mathrm{blue},$ or $Z_{1}$-type, if the fiber of the corresponding point of the base curve consists of two points, and both these points are defined over the same field as the corresponding point of the base curve.
		
		\item $\kappa(P) = \mathrm{red},$ or $Y_{1}$-type, if the fiber of the corresponding point of the base curve consists of two points, and these points are {\it not} defined over the same field as the corresponding point of the base curve.
	\end{enumerate} 
	For a type $\heartsuit \in \{X_{1}, Y_{1}, Z_{1} \},$ we shall denote by $E^{\heartsuit}(p)$ the function 
	$$
	E^{\heartsuit}(p) = \begin{cases} 
	1 & \mbox{if } \heartsuit = X_{1} \\
	\iota E(p) & \mbox{if } \heartsuit = Y_{1} \\ 
	E(p) & \mbox{if } \heartsuit = Z_{1}.
	\end{cases} 
	$$ 
	In what follows, we shall consider only {\it rooted} admissible double covers. For 
	$\heartsuit \in \{ X_{1}, Y_{1}, Z_{1} \},$ let 
	$
	\partial_{\heartsuit}\overline{\mathcal{D}}_{\mathfrak{n}, \mathfrak{i}, \mathfrak{j}}
	$ 
	denote the category of rooted admissible double covers $C$ over $\mathbb{F}_{\! q}$ whose distinguished points $P_{0}(C)$ are of type $\heartsuit,$ and the corresponding points of the base curves are marked points, i.e., the point corresponding to $P_{0}(C)$ is one of the $|\mathfrak{n}|$ branch points if $\heartsuit = X_{1},$ or one of the $|\mathfrak{j}|$ (resp. \!$|\mathfrak{i}|$) marked points if $\heartsuit = Z_{1}$ (resp. \!$\heartsuit = Y_{1}$). Note that the point of the base curve corresponding to $P_{0}(C)$ is defined over $\mathbb{F}_{\! q}.$ We also define 
	$$
	\partial_{\heartsuit} \overline{\mathcal{D}}_{\mathfrak{n}} = 
	\begin{cases} 
	\partial_{\scriptscriptstyle X_{1}}
	\!\overline{\mathcal{D}}_{\mathfrak{n}, {\scriptscriptstyle \pmb{0}}, {\scriptscriptstyle \pmb{0}}}
	& \mbox{if } \heartsuit = X_{1} \\
	\partial_{\scriptscriptstyle Y_{1}}
	\!\overline{\mathcal{D}}_{\mathfrak{n}, {\scriptscriptstyle (1)}, {\scriptscriptstyle \pmb{0}}} & \mbox{if } \heartsuit = Y_{1} \\ 
	\partial_{\scriptscriptstyle Z_{1}}
	\!\overline{\mathcal{D}}_{\mathfrak{n}, {\scriptscriptstyle \pmb{0}}, {\scriptscriptstyle (1)}} & \mbox{if } \heartsuit = Z_{1}.
	\end{cases} 
	$$ 
	From now on, a point $P$ of an admissible double cover $C$ will be called {\it ramified} or {\it unramified} according as the corresponding point of the base curve is branched or not. We shall also refer to a point of an admissible double cover as being a node (or marked) if the corresponding point of the base is so. Finally, for each point $P\in C,$ let $n_{P}$ stand for the degree of $P,$ that is, $n_{P} = [k(P) : k],$ where $k$ is the field of definition of $C.$

	Let $C$ ($[C]\in [\overline{\mathcal{D}}_{\mathfrak{n}, \mathfrak{i}, \mathfrak{j}}(\mathbb{F}_{\! q})]$ with $|\mathfrak{n}|\ge 2$ even, and $|\mathfrak{n}| + |\mathfrak{i}| + |\mathfrak{j}| \ge 3$) be an admissible double cover with a distinguished point (i.e., a root) over $\mathbb{F}_{\! q}$ of type $\heartsuit.$ We define its {\it essential} part, denoted by $\mathrm{ess}(C),$ to be the admissible double cover obtained from $C$ by repeating the process of deleting all the unramified markings and stabilizing it. The isomorphism class of the resulting cover is an element of 
	$
	[\partial_{\heartsuit}
	\overline{\mathcal{D}}_{\mathfrak{n}}(\mathbb{F}_{\! q})].
	$ 
	Note that $\partial_{\scriptscriptstyle \heartsuit}\bar{\mathscr{D}}(X, Y, Z; q)$ can be expressed as 
	\begin{equation*}
	\partial_{\scriptscriptstyle \heartsuit}\bar{\mathscr{D}}(X, Y, Z; q) \;\; = \sum_{[C_{0}]\in [\partial_{\heartsuit}
		\overline{\mathcal{D}}_{\mathfrak{n}}(\mathbb{F}_{\! q})]}\;\,
	\sum_{[\mathrm{ess}(C)] \, = \, [C_{0}]} \frac{1}{\selectfont\# \text{Aut}_{\scriptscriptstyle \mathbb{F}_{q}}\!(C)}
	X^{\mathfrak{n}(C)} Y^{\mathfrak{i}(C)} Z^{\mathfrak{j}(C)}
	\end{equation*} 
	where, for $\heartsuit \in \{ X_{1}, Y_{1}, Z_{1} \}$ and $[C]\in [\partial_{\heartsuit} \overline{\mathcal{D}}_{\mathfrak{n}, \mathfrak{i}, \mathfrak{j}}(\mathbb{F}_{\! q})],$ $\mathfrak{n}(C), \mathfrak{i}(C)$ and $\mathfrak{j}(C)$ are such that: 
	$$
	\heartsuit \cdot X^{\mathfrak{n}(C)} Y^{\mathfrak{i}(C)} Z^{\mathfrak{j}(C)} 
	= X^{\mathfrak{n}} Y^{\mathfrak{i}} Z^{\mathfrak{j}}.
	$$ 
	The process of constructing $\mathrm{ess}(C)$ can be reversed, and one has the following explicit reconstruction of all admissible double covers that have the same essential part:
	\begin{enumerate}[label=(\roman*)]
		\item Consider a finite set $S$ of Galois orbits of points that are not nodes, ramified points, or the root. Decompose the set $S = S_{1} \cup S_{2}$ into two disjoint, possibly empty, subsets. We mark the orbits in $S_{1},$ and in every Galois orbit in $S_{2}$ of a point, say, $P,$ we insert a Galois orbit of an unramified admissible double cover defined over $\mathbb{F}_{\! q}(P)$ with marked points and a root of the same type as $P,$ i.e., the point $P$ is identified with the root.

		\item In the Galois orbit of a node $P$ (or the root $P_{0}$), we can either leave it as it is, or insert a Galois orbit of an admissible double cover defined over 
		$\mathbb{F}_{\! q}(P)$ (or $\mathbb{F}_{\! q} = \mathbb{F}_{\! q}(P_{0})$) with marked points, the root $Q_{0},$ and one additional distinguished point\footnote{In the sense that we choose it always to be the last marked point defined over the field of definition of the admissible double cover.} $Q_{1}$ defined over $\mathbb{F}_{\! q}(P)$ (or $\mathbb{F}_{\! q} = \mathbb{F}_{\! q}(P_{0})$), both $Q_{0}$ and $Q_{1}$ of the same type as $P$ (or $P_{0}$), and with no ramified points, other than possibly $Q_{0}$ and $Q_{1}.$ In the case of a node, consider the partial normalization of the base curve at the point $p$ corresponding to $P.$ Since $p$ is a separating node, the corresponding points $\{p_{0}, p_{1}\}$ of the normalization at $p$ (i.e., the normalized double point) can be chosen such that $p_{0}$ and the point corresponding to the root lie in the same connected component of the partial normalization, whereas $p_{1}$ lies in the other connected component. This yields two admissible double covers, with base curves the connected components of the partial normalization at $p,$ each with a point $\tilde{p}_{j}$ ($j = 0,1$) -- namely, that corresponding to $p_{j}.$ We identify $\tilde{p}_{0}$ with $Q_{0}$ and $\tilde{p}_{1}$ with $Q_{1}.$ In the remaining case, we identify the root $P_{0}$ with $Q_{0},$ and $Q_{1}$ becomes the root of the resulting cover.

		\item In the Galois orbit of a ramified point $P,$ different from the root, we can either just leave the orbit marked, or, alternatively, insert a Galois orbit of an admissible double cover defined over $\mathbb{F}_{\! q}(P)$ with marked points and with two ramified points $Q_{0}$ and $Q_{1}$ defined over $\mathbb{F}_{\! q}(P),$ the point $Q_{0}$ 
		being also the root of the cover. In this case, we identify $P$ with $Q_{0}.$
	\end{enumerate} 
	Starting with $C$ ($
	[C]\in [\partial_{\heartsuit} \overline{\mathcal{D}}_{\mathfrak{n}}(\mathbb{F}_{\! q})]
	$), one applies (i), (ii) and (iii) (only once) to obtain an admissible double cover whose essential part is $C.$ 
	
	Finally, for an admissible double cover $C$ defined over $\mathbb{F}_{\! q},$ we define 
	\begin{equation*}
	P_{\!\scriptscriptstyle C}(p) =  
	P_{\!\scriptscriptstyle C}(p, q)\; : = \prod_{P\in C^{+}\slash G_{\scriptscriptstyle \mathbb{F}_{\! q}}} \psi^{n_{P}}\!E(p) 
	\prod_{P\in C^{-}\slash G_{\scriptscriptstyle \mathbb{F}_{\! q}}} 
	\psi^{n_{P}}\!\iota E(p)
	\end{equation*} 
	where $C^{+}$ (resp. $C^{-}$) is the set of $\overline{\mathbb{F}}_{\! q}$-points of $C$ of blue (resp. red) color, and 
	$
	G_{\scriptscriptstyle \mathbb{F}_{\! q}} = 
	\mathrm{Gal}(\overline{\mathbb{F}}_{\! q}\slash \mathbb{F}_{\! q}).
	$ 
	If $C$ is a hyperelliptic curve over $\mathbb{F}_{\! q}$ then, under the specialization 
	$
	p_{n} \mapsto p_{n}(T) = \sum_{i \ge 1} t_{i}^{n}
	$ 
	for all $n \ge 1,$ we have by \eqref{eq: pol-char-in-terms-of-a*} that \begin{equation*}
	P_{\!\scriptscriptstyle C}(p(T)) 
	= E(T)^{\epsilon(C)} \cdot 
	\prod_{m \ge 1} E(T^{m})^{\frac{c_{m}^{}\!(C) \, + \, a_{m}^{*}\!(C)}{2m}} 
	E(-T^{m})^{\frac{c_{m}^{}\!(C) \, - \, a_{m}^{*}\!(C)}{2m}}
	=\, \prod_{i \ge 1} P_{\!\scriptscriptstyle C}(t_{i})
	\end{equation*} 
	where $E(T) = \prod_{i \ge 1} (1 - t_{i})^{-1}.$

	We now prove the following

	\vskip10pt
	\begin{thm} \label{ess-main-thm} --- Consider the series $H(X, Y, Z; p, q)$ defined by 
		\begin{equation*}
		\begin{split}
		H(X, Y, Z; p, q) &= 
		\bigg(X_{1}, \frac{q - 1}{\iota E(p)(q \iota E(p) - \delta \iota E(p))}Y_{1}, 
		\frac{q - 1}{E(p)(q E(p) - \delta E(p))}Z_{1}\bigg)\\
		&\hskip72.6pt \circ \partial\bigg(\sum_{\mathfrak{n}, \mathfrak{i}, \mathfrak{j}}\; \sum_{[C]\in [\mathcal{D}_{\mathfrak{n}, \mathfrak{i}, \mathfrak{j}}(\mathbb{F}_{\! q})]}\, \frac{P_{\scriptscriptstyle C}(p)}{\selectfont\# \text{Aut}_{\scriptscriptstyle \mathbb{F}_{q}}\!(C)}
		\cdot X^{\mathfrak{n}} Y^{\mathfrak{i}} Z^{\mathfrak{j}} \bigg).
		\end{split}
		\end{equation*} 
		Letting 
		\begin{equation*}
		\bar{\mathscr{D}}^{*}(X, Y, Z; q) : = \bar{\mathscr{D}}(X, Y, Z; q) - \bar{\mathscr{D}}_{0}(Z;q) \; = \sum_{\substack{\mathfrak{n}, \mathfrak{i},\mathfrak{j} \\ |\mathfrak{n}| \ge 2 \, \& \, \mathrm{even}}} \selectfont\# \overline{\mathcal{D}}_{\mathfrak{n}, \mathfrak{i}, \mathfrak{j}}(\mathbb{F}_{\! q}) 
		X^{\mathfrak{n}} Y^{\mathfrak{i}} Z^{\mathfrak{j}}
		\end{equation*} 
		we have 
		\begin{equation*}
		\begin{split}
		H(X, Y, Z; p, q) &\circ  \bigg(X_{1} +  \pi(\partial_{\scriptscriptstyle X_{1}} \!\bar{\mathscr{D}}^{*})(X, p; q), \frac{\pi(\partial_{\scriptscriptstyle Y_{1}}\!\bar{\mathscr{D}}^{*})(X, p; q)}{\iota E(p)}, \frac{\pi(\partial_{\scriptscriptstyle Z_{1}}\!\bar{\mathscr{D}}^{*})(X, p; q)}{E(p)}\bigg)\\
		& \hskip21pt = \bigg(\pi(\partial_{\scriptscriptstyle X_{1}} \!\bar{\mathscr{D}}^{*})(X, p; q), \frac{\pi(\partial_{\scriptscriptstyle Y_{1}} \!\bar{\mathscr{D}}^{*})(X, p; q)}{\iota E(p)}, \frac{\pi(\partial_{\scriptscriptstyle Z_{1}} \!\bar{\mathscr{D}}^{*})(X, p; q)}{E(p)}\bigg).
		\end{split}
		\end{equation*}
	\end{thm}

	\begin{proof} For any $\heartsuit \in \{ X_{1}, Y_{1}, Z_{1} \},$ we first write 
		$
		\pi(\partial_{\scriptscriptstyle \heartsuit} \bar{\mathscr{D}}^{*})(X, p; q)
		$ 
		as 
		\begin{equation} \label{eq: transition-identity}
		\begin{split}
		\pi(\partial_{\scriptscriptstyle \heartsuit} \bar{\mathscr{D}}^{*})(X, p; q) \; =  
		\sum_{\substack{|\mathfrak{n}| \, \ge \, 2\\ |\mathfrak{n}|-\mathrm{even}}} \, \sum_{[C]\in [\partial_{\heartsuit}\overline{\mathcal{D}}_{\mathfrak{n}}(\mathbb{F}_{\! q})]} \frac{\alpha^{\kappa(P_{0}(C))}(p) X^{\mathfrak{n}(C)}}{\selectfont\# \text{Aut}_{\scriptscriptstyle \mathbb{F}_{q}}\!(C)}
		\prod_{P\in C^{+}\slash G_{\scriptscriptstyle \mathbb{F}_{\! q}}} \psi^{n_{P}}\!E(p) 
		&\prod_{P\in C^{-}\slash G_{\scriptscriptstyle \mathbb{F}_{\! q}}} \psi^{n_{P}}\!\iota E(p)\\
		&\cdot \prod_{P\in N(C)\slash G_{\scriptscriptstyle \mathbb{F}_{\! q}}} \psi^{n_{P}}\!\alpha^{\kappa(P)}(p)
		\end{split}
		\end{equation} 
		where $N(C)$ denotes the set of nodes of $C,$ 
		$$
		\alpha^{\heartsuit}(p) = \alpha^{\heartsuit}(p, q) 
		= \begin{cases} 
		1 & \mbox{if } \heartsuit = X_{1} \\
		\frac{q \, - \, 1}{q \iota E(p) \, - \, \delta \iota E(p)} & \mbox{if } \heartsuit = Y_{1} \\ 
		\frac{q \, - \, 1}{q E(p) \, - \, \delta E(p)} & \mbox{if } \heartsuit = Z_{1}
		\end{cases} 
		$$ 
		and 
		$
		G_{\scriptscriptstyle \mathbb{F}_{\! q}} 
		= \mathrm{Gal}(\overline{\mathbb{F}}_{\! q}\slash \mathbb{F}_{\! q}).
		$
		
		Indeed, one follows the recipe to reconstruct all the admissible double covers from a fixed $C$ with 
		$
		[C]\in [\partial_{\heartsuit}\overline{\mathcal{D}}_{\mathfrak{n}}(\mathbb{F}_{\! q})].
		$ 
		Counting the number of curves that can be inserted by the process, it follows from Proposition \ref{unique-hom}, Lemma \ref{pi-elem1} and Lemma \ref{pi-elem2} that, after applying $\pi,$ the contribution corresponding to the Galois orbit of a node (or the root) $P$ is $\psi^{n_{P}}\!(\alpha^{\kappa(P)}(p)\cdot E^{\kappa(P)}(p)),$ while the contribution corresponding to the Galois orbit of a point $P$ which is neither a node nor the root is $\psi^{n_{P}}\!E^{\kappa(P)}(p).$

		Next, one considers all the curves $C$ with the same root component $C_{0}.$ Note that $C_{0}$ can be represented by a hyperelliptic curve with marked points given by the nodes of $C$ belonging to $C_{0},$ and with a distinguished point given by the root. In particular, $[C_{0}]$ is an element of 
		$
		[\mathcal{D}_{\mathfrak{n}, \mathfrak{i}, \mathfrak{j}}(\mathbb{F}_{\! q})]
		$ 
		for some partitions $\mathfrak{n}, \mathfrak{i}$ and $\mathfrak{j}.$ 
		Now the contribution to the first two products in \eqref{eq: transition-identity} coming from the points of $C_{0}$ is 
		$
		P_{\scriptscriptstyle C_{0}}(p).
		$ 
		Finally, notice that the subcurves of $C$ emerging from $C_{0}$ in a Galois orbit 
		of a node $P$ contributes the same as in 
		$
		\psi^{n_{P}}\!\Big(\frac{\pi(\partial_{\kappa(P)} \bar{\mathscr{D}}^{*})(X, p; q)}{E^{\kappa(P)}(p)}\Big)
		$ 
		if $\kappa(P) = Y_{1}$ or $Z_{1},$ and as in 
		$
		\psi^{n_{P}}\!(X_{1} + \pi(\partial_{\scriptscriptstyle X_{1}} \!\bar{\mathscr{D}}^{*})(X, p; q))
		$ 
		if $\kappa(P) = X_{1},$ and the theorem follows.\end{proof}

	\vskip10pt
	\begin{corollary} \label{cor-rec-relation-B} --- Define $B(X, Y, Z; p, q)$ by 
		\begin{equation*}
		B(X, Y, Z; p, q) = (q X_{1}, q Y_{1}, q Z_{1}) + 
		q(q - 1)\cdot  \bigg(X_{1}, \frac{q \iota E(p) - \delta \iota E(p)}{(q - 1)\cdot \delta \iota E(p)}Y_{1}, \frac{q E(p) - \delta E(p)}
		{(q - 1)\cdot \delta E(p)}Z_{1}\bigg) \circ H(X, Y, Z; p, q).
		\end{equation*} 
		Then  
		\begin{equation*}
		\begin{split}
		B(X, Y, Z; p, q) &\circ \bigg(X_{1} + \pi(\partial_{\scriptscriptstyle X_{1}}\!\bar{\mathscr{D}}^{*})(X, p; q), \frac{\pi(\partial_{\scriptscriptstyle Y_{1}}\!\bar{\mathscr{D}}^{*})(X, p; q)}{\iota E(p)}, \frac{\pi(\partial_{\scriptscriptstyle Z_{1}}\!\bar{\mathscr{D}}^{*})(X, p; q)}{E(p)}\bigg)\\
		& \hskip-7pt = \delta \bigg(X_{1} + \pi(\partial_{\scriptscriptstyle X_{1}} \!\bar{\mathscr{D}}^{*})(X, p; q), \frac{\pi(\partial_{\scriptscriptstyle Y_{1}}\!\bar{\mathscr{D}}^{*})(X, p; q)}{\iota E(p)}, \frac{\pi(\partial_{\scriptscriptstyle Z_{1}}\!\bar{\mathscr{D}}^{*})(X, p; q)}{E(p)} \bigg).
		\end{split}
		\end{equation*} 
	\end{corollary}
	
	\begin{proof} A simple calculation using the relation in Theorem \ref{ess-main-thm} implies that the left-hand side of the identity in the corollary is 
		\begin{equation*}
		\bigg(q X_{1} + q^{2}\pi(\partial_{\scriptscriptstyle X_{1}}\!\bar{\mathscr{D}}^{*})(X, p; q), 
		\frac{q^{2} \pi(\partial_{\scriptscriptstyle Y_{1}}\!\bar{\mathscr{D}}^{*})(X, p; q)}
		{\delta\iota E(p)}, \frac{q^{2} \pi(\partial_{\scriptscriptstyle Z_{1}}\!\bar{\mathscr{D}}^{*})(X, p; q)}{\delta E(p)}\bigg).
		\end{equation*} 
		From the definition of $\bar{\mathscr{D}}_{0}(Z;q),$ we have that  
		\begin{equation} \label{eq: func-eq-barD0-qto1/q} 
		\delta \bar{\mathscr{D}}_{0}(Z; q)
		= q^{3} \bar{\mathscr{D}}_{0}(Z;q).
		\end{equation} 
		Differentiating \eqref{eq: func-eq-barD-qto1/q} and \eqref{eq: func-eq-barD0-qto1/q}, one finds the relations 
		\begin{equation*}
		\begin{split}
		&\delta(\partial_{\scriptscriptstyle X_{1}}\!\bar{\mathscr{D}}^{*})(X, Y, Z; q) 
		= q^{2} \partial_{\scriptscriptstyle X_{1}}\!\bar{\mathscr{D}}^{*}(X, Y, Z; q)\\
		&\delta(\partial_{\scriptscriptstyle Y_{1}}\!\bar{\mathscr{D}}^{*})(X, Y, Z; q) 
		= q^{2} \partial_{\scriptscriptstyle Y_{1}}\!\bar{\mathscr{D}}^{*}(X, Y, Z; q)\\
		&\delta(\partial_{\scriptscriptstyle Z_{1}}\!\bar{\mathscr{D}}^{*})(X, Y, Z; q) 
		= q^{2} \partial_{\scriptscriptstyle Z_{1}}\!\bar{\mathscr{D}}^{*}(X, Y, Z; q)
		\end{split}
		\end{equation*} 
		and thus we can write 
		\begin{equation*}
		q^{2}\pi(\partial_{\scriptscriptstyle X_{1}}\!\bar{\mathscr{D}}^{*})(X, p; q) 
		= \pi(q^{2}\partial_{\scriptscriptstyle X_{1}}\!\bar{\mathscr{D}}^{*}(X, Y, Z; q))
		= \pi(\delta(\partial_{\scriptscriptstyle X_{1}}\!\bar{\mathscr{D}}^{*})(X, Y, Z; q))
		\end{equation*} 
		(and similarly for the other partial derivatives of $\bar{\mathscr{D}}^{*}(X, Y, Z; q)$).

		The corollary follows now from the commutativity of $\pi$ and $\delta,$ see Proposition \ref{unique-hom}.\end{proof}

	\subsection{Finishing the proof of Theorem \ref{princ-theorem}} 
	From now on we specialize $p = (p_{n})_{n\ge 1}$ to be the power sums, 
	\begin{equation*}
	p_{n} = p_{n}(T) = \sum_{i \ge 1}\, t_{i}^{n}
	\end{equation*}
	for independent variables $T = (t_{i})_{i\ge 1}.$ For compatibility, put $\psi^{n}(t_{i}) = t_{i}^{n},$ $\iota(t_{i}) = - t_{i}$ and $\delta(t_{i}) = q t_{i},$ for all $i \ge 1.$ We define the plethystic substitution for functions in the variables $X, Y, Z, T$ and $q$ in the obvious way. 
	
	We first establish the connection between the function 
	\begin{equation*}
	C_{1}(\underline{X}, T\!, q) : = (C_{1, \text{odd}}(\underline{X}, T\!, q), C_{1, \text{even}}^{-}(\underline{X}, T\!, q), 
	C_{1, \text{even}}^{+}(\underline{X}, T\!, q))
	\end{equation*} 
	(see \ref{Gen-Case}) and the function 
	$
	B(\underline{X}, T\!, q) = B(X, Y, Z; p, q)\vert_{(p_{n}\rightarrow p_{n}(T),\, 
		n\ge 1)}.
	$ 
	To do so, consider temporarily the function $C_{1}^{*}(\underline{X}, T\!, q)$ 
	defined by 
	\begin{equation*}
	C_{1}^{*}(\underline{X}, T\!, q) = - \,(q X_{1}, q \tilde{E}(- T\!, q)Y_{1}, q \tilde{E}(T\!, q)Z_{1})
	+ C_{1}(\underline{X}, T\!, q) \circ 
	(X_{1}, \tilde{E}(- T\!, q)Y_{1}, \tilde{E}(T\!, q)Z_{1})
	\end{equation*} 
	where, as before, $\tilde{E}(T, q) = E(T) E(q\, T).$ Note that we have removed the 
	linear parts of the components of $C_{1}(\underline{X}, T\!, q).$ Let $M(\underline{X}, T\!, q)$ denote the part of the function $H$ introduced in Theorem \ref{ess-main-thm} (normalized by $q(q - 1)$) given by 
	\begin{equation*}
	M(\underline{X}, T\!, q) = q (q - 1)\cdot \partial\bigg(\sum_{\mathfrak{n}, \mathfrak{i}, \mathfrak{j}}\; \sum_{[C]\in [\mathcal{D}_{\mathfrak{n}, \mathfrak{i}, \mathfrak{j}}(\mathbb{F}_{\! q})]}\, \frac{\prod_{i \ge 1}P_{\scriptscriptstyle C}(t_{i})}{\selectfont\# \text{Aut}_{\scriptscriptstyle \mathbb{F}_{q}}\!(C)}
	\cdot X^{\mathfrak{n}} Y^{\mathfrak{i}} Z^{\mathfrak{j}}\bigg).
	\end{equation*} 
	Then, by a simple counting argument, we see that 
	\begin{equation*}
	C_{1}^{*}(\underline{X}, T\!, q) = M(\underline{X}, T\!, q).
	\end{equation*} 
	From the definition of $B(\underline{X}, T\!, q),$ it follows easily that 
	\begin{equation*}
	C_{1}(\underline{X}, T\!, q) = 
	(X_{1}, \tilde{E}(- T\!, q)Y_{1}, \tilde{E}(T\!, q) Z_{1}) \circ 
	B(\underline{X}, T\!, q) \circ \bigg(X_{1}, \frac{Y_{1}}{\tilde{E}(- T\!, q)}, \frac{Z_{1}}{\tilde{E}(T\!, q)}\bigg)
	\end{equation*} 
	that is, $C_{1}(\underline{X}, T\!, q)$ and $B(\underline{X}, T\!, q)$ 
	are conjugated under a linear plethystic transformation. Thus we can rewrite the relation in Corollary \ref{cor-rec-relation-B} as 
	\begin{equation} \label{eq: funct-equation-needed-proof-main-thm}
	C_{1}(\underline{X}, T\!, q) \circ \bar{C}_{1}(X, T\!, q)
	= \delta(\bar{C}_{1}(X, T\!, q)) 
	= \bar{C}_{1}(q\star X, q\, T\!, 1\slash q)
	\end{equation} 
	where 
	\begin{equation*} 
	\bar{C}_{1}(X, T\!, q) : = \bigg(X_{1} + \pi(\partial_{\scriptscriptstyle X_{1}}\!\bar{\mathscr{D}}^{*})(X, T\!, q), E(- q\, T)\cdot \pi(\partial_{\scriptscriptstyle Y_{1}}\!\bar{\mathscr{D}}^{*})(X, T\!, q), E(q\, T)\cdot \pi(\partial_{\scriptscriptstyle Z_{1}}\!\bar{\mathscr{D}}^{*})(X, T\!, q)\bigg).
	\end{equation*} 
	Notice that 
	\begin{equation*}
	X_{1} + \pi(\partial_{\scriptscriptstyle X_{1}}\!\bar{\mathscr{D}}^{*})(X, T\!, q) = 
	X_{1} + \textrm{higher degree terms}
	\end{equation*} 
	while  
	\begin{equation*}
	\pi(\partial_{\scriptscriptstyle Z_{1}}\!\bar{\mathscr{D}}^{*})(X, T\!, q)
	\;\; = \sum_{\substack{\mathfrak{n},\mathfrak{i},\mathfrak{j}\\
			|\mathfrak{n}| \ge 2 \, \& \, \mathrm{even} \\ |\mathfrak{n}|+|\mathfrak{i}|+|\mathfrak{j}|
			\,\ge \, 3}} \selectfont\# \overline{\mathcal{D}}_{\mathfrak{n}, \mathfrak{i}, \mathfrak{j}}(\mathbb{F}_{\! q}) X^{\mathfrak{n}}\cdot \pi(Y^{\mathfrak{i}} \partial_{\scriptscriptstyle Z_{1}}\!Z^{\mathfrak{j}})(T\!, q).
	\end{equation*} 
	Applying $C_{1}(\underline{X}, q\, T\!, 1\slash q) \circ - $ to both sides, 
	we obtain 
	\begin{equation*}
	C_{1}(\underline{X}, q\, T\!, 1\slash q) \circ
	C_{1}(\underline{X}, T\!, q) \circ \bar{C}_{1}(X, T\!, q) 
	= C_{1}(\underline{X}, q\, T\!, 1\slash q) \circ 
	\bar{C}_{1}(q\star X, q\, T\!, 1\slash q) 
	= \bar{C}_{1}(X, T\!, q)
	\end{equation*} 
	which implies the functional relation \eqref{eq: func-rel-section8} used to finish the proof of Theorem \ref{princ-theorem}.

	Finally, the function $\bar{c}(x, T, q)$ used to complete the induction process in 
	the proof of Theorem \ref{split case} is obtained by specializing $X_{1} = x$ and $X_{n} = 0,$ for $n\ge 2,$ in $\bar{C}_{1}(X, T\!, q).$

\section{Constructing the generating series $\Lambda_{l}(T, q)$}\label{section 9} 
To determine the coefficients $\lambda(\kappa, l; q)$ of $\Lambda_{l}(T, q),$ for $l \ge 5$ and 
$\kappa = (k_{1}, \ldots, k_{r}) \in \mathbb{N}^{r},$ one proceeds as in the case $l = 4.$ We 
show that every coefficient $\lambda(\kappa, l; q)$ is a finite sum 
\begin{equation} \label{eq: final-formula} 
\lambda(\kappa, l; q) = \sum_{j}  P_{j}^{(\kappa,  l)}(q) \alpha_{j}
\end{equation} 
where $P_{j}^{(\kappa,  l)}(x) \in  \mathbb{Q}[x]$ is independent of $q$ for all $j,$ and 
$\alpha_{j}$ are distinct $q$-Weil algebraic integers of weights $m_{j} \in \mathbb{N};$ every $\alpha_{j}$ occurs in the sum together with all its complex conjugates, and 
$
P_{j}^{(\kappa,  l)}(x) = P_{j'}^{(\kappa,  l)}(x)
$ 
if the algebraic numbers $\alpha_{j}$ and $\alpha_{j'}$ are conjugates over $\mathbb{Q}.$ Moreover, 
for each $j,$ we have 
\begin{equation} \label{eq: degree P + mj} 
\deg P_{j}^{(\kappa,  l)}  +  m_{j} \le |\kappa| + l 
\;\;\; \text{and} \;\;\; 
P_{j}^{(\kappa,  l)}(x)^{2} \equiv \,  0 \pmod{x^{|\kappa| + l - m_{j} + 2}} 
\end{equation} 
with $|\kappa| = k_{1} + \cdots + k_{r}.$ As we shall see, the numbers $\alpha_{j}$ 
are the (suitably) normalized eigenvalues of Frobenius acting on the components 
$H_{c,  \mu}^{\bullet}(\mathscr{H}_{g}[2] \otimes \overline{\mathbb{F}}_{\! q}, \mathbb{V}'(\lambda))$ of the $\mathbb{S}_{2g + 2}$-isotypic decomposition of 
$H_{c}^{\bullet}(\mathscr{H}_{g}[2] \otimes \overline{\mathbb{F}}_{\! q}, \mathbb{V}'(\lambda)).$

Note that 
\begin{equation} \label{eq: final-formula q->1/q} 
\lambda(\kappa, l; 1\slash q) = \sum_{j} P_{j}^{(\kappa,  l)}(1\slash q) q^{- m_{j}} \alpha_{j};
\end{equation}
compare also with \cite[($*$), p. 6]{Ser}. Accordingly, we must have that 
\begin{equation} \label{eq: l-Recurrence1} 
\lambda(\kappa, l; q)  -  q^{|\kappa| +  l + 1}\lambda(\kappa, l; 1\slash q)   
\,=\, \sum_{j}  \left(P_{j}^{(\kappa,  l)}(q) 
-  q^{|\kappa| +  l - m_{j}  +  1}P_{j}^{(\kappa,  l)}(1\slash q) \right)\alpha_{j}; 
\end{equation} 
note that 
$
Q_{j}^{(\kappa,  l)}(x) : = x^{|\kappa| +  l - m_{j}}P_{j}^{(\kappa,  l)}(1\slash x) \in \mathbb{Q}[x],
$ 
and 
$
x^{\deg Q_{j}^{(\kappa,  l)} + 2} \mid P_{j}^{(\kappa,  l)}(x) 
$  
for all $j,$ which says that (iii) in Section \ref{section 3} is satisfied. As $\mathbb{F}_{\! q}$ is an arbitrary finite field of odd characteristic, the condition (ii) is also satisfied.

Now let us justify our assertions. From the initial conditions (i) in Section \ref{section 3}, we have 
$\lambda(0, \ldots, 0, l; q) = q^{l}.$ We argue by induction on $\kappa$ and $l.$ Assume that 
\eqref{eq: final-formula} (together with the corresponding conditions on the polynomials and $q$-Weil integers involved) holds for all coefficients $\lambda(\ast, l'; \ast)$ with $l' < l,$ and all coefficients $\lambda(\kappa', l; q)$ with 
$\kappa' < \kappa.$ Consider \eqref{eq: tag 4.1} in the symmetric form 
\begin{equation*}
E(T)^{\epsilon_{l}}\Lambda_{l}(T, q)  =  E(T)^{\epsilon_{l}}A_{l}(T, q)  
+  q^{l + 1} E(q \, T)^{\epsilon_{l}}\Lambda_{l}(q \, T, 1\slash q). 
\end{equation*} 
By equating the coefficients of $T^{\kappa} = t_{1}^{k_{1}} \cdots \, t_{r}^{k_{r}}$ on both sides, 
we find that the left-hand side of \eqref{eq: l-Recurrence1} can be expressed as
\begin{equation} \label{eq: prel-final-formula}
\sum_{\kappa' \, \le \, \kappa} a(\kappa', l; q) \;
- \sum_{\kappa' \, < \, \kappa} \big(\lambda(\kappa', l; q) \, - \,  
q^{|\kappa| +  l + 1}\lambda(\kappa', l; 1\slash q) \big)
\end{equation} 
if $l$ is even, and $a(\kappa, l; q)$ if $l$ is odd. 

We recall that the generating series $A_{l}(T, q)$ satisfies the functional equation \eqref{eq: func-eq A_{l}(T, q)}; as explained at the beginning of Section \ref{section 7}, \eqref{eq: func-eq A_{l}(T, q)} is implied by the identities \eqref{eq: MDS-relations}, or equivalently \eqref{eq: MDS-relations-gen-series-equiv} (proved in Theorem \ref{princ-theorem}). Accordingly, if $l$ is odd, $a(\kappa, l; q)$ satisfies the functional equation 
\begin{equation} \label{eq: functional-equation-a(k, l)} 
a(\kappa, l; q) = -\, q^{|\kappa| + l + 1} a(\kappa, l; 1\slash q)
\end{equation} 
for each $\kappa = (k_{1}, \ldots, k_{r}) \in \mathbb{N}^{r}.$ It is clear that both sums in \eqref{eq: prel-final-formula} also satisfy the functional equation \eqref{eq: functional-equation-a(k, l)}; the first sum in \eqref{eq: prel-final-formula} is the coefficient of $T^{\kappa}$ in $E(T)A_{l}(T, q).$

Next, we shall apply Deligne's theory of weights to express the first sum in 
\eqref{eq: prel-final-formula} (respectively $a(\kappa, l; q)$ when $l$ is odd) 
in terms of $q$-Weil numbers. The coefficient $\lambda(\kappa, l; q)$ 
will then be easily determined from the corresponding expression of 
the sums \eqref{eq: prel-final-formula}. To this end, let us first recall 
some facts from Deligne's theory that we shall need. We follow the standard reference \cite{Del2}.

Let $\mathbb{F}_{\! q}$ be a finite field with $q$ elements. Fix $\overline{\mathbb{F}}_{\! q}$ an algebraic closure of $\mathbb{F}_{\! q}.$ Let $\mathscr{F}$ be a constructible $\overline{\mathbb{Q}}_{\ell}$-sheaf 
($\ell \nmid q$) on a scheme $X$ of finite type over $\mathbb{F}_{\! q}.$ Set 
$\overline{X} : = X \otimes_{_{\mathbb{F}_{\! q}}}  \! \overline{\mathbb{F}}_{\! q},$ and denote 
the pullback of $\mathscr{F}$ to $\overline{X}$ by $\overline{\mathscr{F}}.$ The Frobenius morphism 
$F: \overline{X} \to \overline{X}$ induces a natural isomorphism 
$F^{*} : F^{*} \overline{\mathscr{F}} \stackrel{\sim}{\rightarrow} \overline{\mathscr{F}}.$ The finite 
morphisms $F$ and $F^{*}$ define an endomorphism 
\begin{equation*}
F^{*} : H_{c}^{i}(\overline{X}, \overline{\mathscr{F}}) \to H_{c}^{i}(\overline{X}, F^{*} \overline{\mathscr{F}}) \to 
H_{c}^{i}(\overline{X}, \overline{\mathscr{F}})
\end{equation*}
of cohomology groups with compact support. For a closed point 
$y \in  \left|\overline{X} \right|,$ $F^{*}$ defines a morphism 
\begin{equation*}
F_{y}^{*}: \overline{\mathscr{F}}_{\! F(y)} \to \overline{\mathscr{F}}_{\! y}.
\end{equation*} 
Let $x \in |X|$ be a closed point, and let $\bar{x}$ be a geometric point of $X$ over $x.$ (The closed 
point of $\overline{X}$ corresponding to $\bar{x}$ is also denoted by $\bar{x}.$) Let $F_{\! x}^{*}$ denote the endomorphism $F^{* \deg x}_{\! \bar{x}}$ of 
$\overline{\mathscr{F}}_{\! \bar{x}},$ where 
$\deg x := [\mathbb{F}_{\! q}(x) : \mathbb{F}_{\! q}].$ Up to isomorphism, 
$\left(\overline{\mathscr{F}}_{\! \bar{x}}, F_{\! x}^{*} \right)$ does not depend on the choice of $\bar{x}.$ If we put 
\begin{equation*}
\det(I - F_{\! x}^{*} t \,\vert \, \mathscr{F}) = \det(I - F_{\! x}^{*} t \, \vert \, \overline{\mathscr{F}}_{\! \bar{x}})
\end{equation*}
then the Grothendieck-Lefschetz trace formula (see \cite{Groth1}, \cite{Groth2} and \cite{Del3}) 
gives the following identity\footnote{The Grothendieck-Lefschetz trace formula is commonly given as 
the identity obtained by equating coefficients on both sides of the logarithmic derivative of $L(X, \mathscr{F}, t).$} of formal power series:
\begin{equation*}
L(X, \mathscr{F}, t) \; : = \prod_{x \, \in  \, |X|} \det(I - F_{\! x}^{*} t^{\deg x} \,\vert \, \mathscr{F})^{-1} = \, \prod_{i} \det \left(I - F^{*} t \,\vert \, 
H_{c}^{i}(\overline{X}, \overline{\mathscr{F}}) \right)^{(-1)^{i + 1}}.
\end{equation*}

\subsection{Deligne's purity theorem} 
Let notations be as above.

\vskip10pt
\begin{defn} Let $K$ be a field of characteristic $0,$ and let $n \in \mathbb{Z}.$ 
An element $\lambda \in K$ is called a $q$-Weil number of weight $n$ if $\lambda$ 
is algebraic over $\mathbb{Q},$ and all its complex conjugates have absolute value 
$q^{n\slash 2}.$ 
\end{defn} 

\vskip10pt
\begin{defn} Let $\mathscr{F}$ be a constructible $\overline{\mathbb{Q}}_{\ell}$-sheaf 
on the $\mathbb{F}_{\! q}$-scheme $X.$ The sheaf $\mathscr{F}$ is said to be 
{\it pure} of weight $n$ if, for every closed point $x \in |X|,$ the 
$\overline{\mathbb{Q}}_{\ell}$-eigenvalues of the endomorphism $F_{\! x}^{*}$ 
are $N(x) :=  |\mathbb{F}_{\! q}(x)|$-Weil numbers of weight $n.$ 
The sheaf $\mathscr{F}$ is said to be {\it mixed} if it admits a finite increasing filtration 
by constructible $\overline{\mathbb{Q}}_{\ell}$-subsheaves with successive quotients that 
are pure. The weights of the successive quotients are the {\it weights} of $\mathscr{F}.$  
\end{defn}

\vskip10pt
\theoremstyle{Del}
\newtheorem*{Deligne}{Theorem (Deligne)}
\begin{Deligne} --- Let $f : X \to S$ be a separated map between finite-type $\mathbb{F}_{\! q}$-schemes, 
and let $\mathscr{F}$ be a constructible $\overline{\mathbb{Q}}_{\ell}$-sheaf mixed of weights $\le n$ 
on $X.$ Then the sheaf $R^{i}f_{!}\mathscr{F}$ on $S$ is mixed of weights $\le n + i,$ for all $i.$ 
In particular, every eigenvalue $\alpha$ of Frobenius on $H_{c}^{i}(\overline{X}, \overline{\mathscr{F}})$ 
is a $q$-Weil number of weight $m,$ for some integer $m\le n + i.$ 
\end{Deligne} 

\vskip5pt
Now let $H_{g}[2]$ ($g \ge 2$) denote the coarse moduli space of 
$\mathscr{H}_{g}[2] \otimes_{_{\mathbb{F}_{\! q}}} \! \overline{\mathbb{F}}_{\! q},$ and let 
$p : \mathscr{H}_{g}[2] \otimes_{_{\mathbb{F}_{\! q}}} \! \overline{\mathbb{F}}_{\! q} 
\to H_{g}[2]$ be the natural map. The direct image 
$\mathbb{V}_{\lambda}' : = p_{*}\mathbb{V}'(\lambda)$ on $H_{g}[2]$ is a constructible sheaf, pure of weight $|\lambda|.$ We recall from \ref{Loc-sys Ag} that $\mathbb{V}'(\lambda)$ is the $\ell$-adic 
local system (for a prime $\ell \nmid q$) on $\mathscr{H}_{g}[2]$ corresponding to the irreducible 
symplectic representation $V_{\lambda};$ 
as in loc. cit., the pullback of this local system to 
$\mathscr{H}_{g}[2] \otimes_{_{\mathbb{F}_{\! q}}} \! \overline{\mathbb{F}}_{\! q}$ 
will still be denoted by $\mathbb{V}'(\lambda).$ If we put 
\begin{equation*}
e_{c}(H_{g}[2], \mathbb{V}_{\lambda}')  \, =
\sum_{i \, \ge \, 0} \,  (-1)^{i}  [H_{c}^{i}(H_{g}[2], \mathbb{V}_{\lambda}')]
\end{equation*} 
then the trace of Frobenius on 
$
e_{c}(\mathscr{H}_{g}[2] \otimes_{_{\mathbb{F}_{\! q}}} \! \overline{\mathbb{F}}_{\! q}, \mathbb{V}'(\lambda))
$ 
is the same as the trace of the corresponding Frobenius on 
$
e_{c}(H_{g}[2], \mathbb{V}_{\lambda}')
$ 
(cf. Section 2 in \cite{BvdG}). By Deligne's theorem, every eigenvalue of Frobenius on 
$
H_{c}^{i}(H_{g}[2], \mathbb{V}_{\lambda}')
$ 
is a $q$-Weil number of weight at most $|\lambda| + i.$ Since $\mathbb{V}_{\lambda}'$ is in particular 
integral, every eigenvalue is in fact a $q$-Weil algebraic integer (cf. Corollaire {\bf (3.3.4)} in \cite{Del2}). Of course, the cohomology group 
$
H_{c}^{i}(H_{g}[2], \mathbb{V}_{\lambda}')
$ 
vanishes for $i >  2 \dim (H_{g}[2]) = 4 g - 2.$ The natural action of the symmetric group 
$\mathbb{S}_{2g + 2}$ on $\mathscr{H}_{g}[2]$ induces an isotypic decomposition of the Galois 
representation 
$
H_{c}^{\bullet}(\mathscr{H}_{g}[2] \otimes_{_{\mathbb{F}_{\! q}}} \! \overline{\mathbb{F}}_{\! q}, \mathbb{V}'(\lambda))
$ 
as 
\begin{equation*}
H_{c}^{\bullet}(\mathscr{H}_{g}[2] \otimes_{_{\mathbb{F}_{\! q}}} \! \overline{\mathbb{F}}_{\! q}, \mathbb{V}'(\lambda))
\;\, =  \bigoplus_{\mu \,\vdash\, 2g + 2} 
H_{c,  \mu}^{\bullet}(\mathscr{H}_{g}[2] \otimes_{_{\mathbb{F}_{\! q}}} \! \overline{\mathbb{F}}_{\! q}, \mathbb{V}'(\lambda)) 
\end{equation*} 
where for an irreducible representation $R_{\mu}$ of $\mathbb{S}_{2g + 2}$ indexed by the partition 
$\mu$ of $2g + 2,$  
\begin{equation*} 
H_{c,  \mu}^{\bullet}(\mathscr{H}_{g}[2] \otimes_{_{\mathbb{F}_{\! q}}} \! \overline{\mathbb{F}}_{\! q}, \mathbb{V}'(\lambda))
= R_{\mu} \otimes \text{Hom}_{_{\mathbb{S}_{2g + 2}}}\!(R_{\mu}, H_{c}^{\bullet}(\mathscr{H}_{g}[2] \otimes_{_{\mathbb{F}_{\! q}}} \! \overline{\mathbb{F}}_{\! q}, \mathbb{V}'(\lambda))).
\end{equation*} 
By Theorem \ref{q->1/q} and \eqref{eq: basic-moment-id}, for each 
$\nu = (1^{\nu_{1}}\!, 2^{\nu_{2}}\!, \ldots, (2g + 2)^{\nu_{2g + 2}}),$ the sum  
\begin{equation} \label{eq: LHS-moment identity} 
\frac{1}{|\mathrm{GL}_{2}(\mathbb{F}_{\! q})|}
\sum_{d \, \in \, \mathscr{P}_{\! g}(\nu, \mathbb{F}_{\! q})}
\left(\prod_{k = 1}^{r}  P_{\scriptscriptstyle C_{d}}(t_{k})  \right)
\end{equation} 
can be expressed in terms of traces of $F^{*}$ on 
$
H_{c,  \mu}^{\bullet}(\mathscr{H}_{g}[2] \otimes_{_{\mathbb{F}_{\! q}}} 
\! \overline{\mathbb{F}}_{\! q}, \mathbb{V}'(\lambda))
$ 
with $\lambda \subseteq (r^{g})$ and $|\lambda|$ even; the $q$-Weil 
algebraic integers appearing in 
$ 
\Tr(F^{*} \vert \, \tilde{e}_{c,  \nu}(\mathscr{H}_{g}[2] \otimes_{_{\mathbb{F}_{\! q}}} 
\! \overline{\mathbb{F}}_{\! q}, \mathbb{V}'(\lambda)))
$ 
are of weights at most $|\lambda| + 4 g  - 2.$ Moreover, for every $\lambda,$ 
\begin{equation} \label{eq: Monomials-Schur-9.1}
q^{- \frac{|\lambda|}{2}}\Big(q^{\frac{1}{2}}t_{1} \, \cdots \, q^{\frac{1}{2}}t_{r}\Big)^{g}
\!s_{\scriptscriptstyle \langle \lambda^{\dag} \rangle}\big(q^{\pm \frac{1}{2}} t_{1}^{\pm 1}\!, \ldots, 
q^{\pm \frac{1}{2}}t_{r}^{\pm 1} \big)
\end{equation} 
in \eqref{eq: basic-moment-id} is easily seen to be a polynomial in $t_{1}, \ldots, t_{r};$ the weight $|\kappa|$ of any monomial 
$
t_{1}^{k_{1}} \cdots \, t_{r}^{k_{r}}
$
occurring in this polynomial is an even integer in the set  
$
\{|\lambda|, \ldots, 2 r g - |\lambda| \}.
$ 
Hence the powers of $q$ occurring in this polynomial are all nonnegative and $\le r g - |\lambda|.$

We express \eqref{eq: LHS-moment identity} in terms of $q$-Weil numbers as follows. Fix 
$
\lambda  \subseteq  (r^{g})
$ 
a partition with $|\lambda|$ even. Let $\beta$ be an eigenvalue of $F^{*}$ on 
$
H_{c,  \mu}^{\bullet}(\mathscr{H}_{g}[2] \otimes_{_{\mathbb{F}_{\! q}}} 
\! \overline{\mathbb{F}}_{\! q}, \mathbb{V}'(\lambda))
$ 
of weight $m_{\beta}  \le  |\lambda| + 2d,$ where 
$d = 2 g - 1 = \dim \, \mathscr{H}_{g}[2].$ By duality, 
$q^{|\lambda| + d}\beta^{ - 1} = q^{|\lambda| + d - m_{\beta}}\bar{\beta}$ 
is an eigenvalue of Frobenius on the ordinary cohomology. We normalize 
$\beta$ by setting 
\begin{equation*}
\alpha : =
 \begin{cases} 
 \beta 
 & \text{if $m_{\beta}  \le  |\lambda| + d$}\\ 
q^{|\lambda| + d - m_{\beta}}\beta
 & \text{if $m_{\beta}  >  |\lambda| + d$}  
\end{cases}
\end{equation*} 
and denote by $m_{\alpha}$ the weight of $\alpha.$ Notice that 
$
m_{\alpha}  + m_{\beta}  \le  2(|\lambda| + d),
$ 
with equality if $m_{\beta}  \ge  |\lambda| + d.$

Let $\kappa = (k_{1}, \ldots, k_{r}) \in \mathbb{N}^{r}$ be such that $|\kappa|$ is even and 
$
|\lambda| \le |\kappa| \le 2 r g - |\lambda|, 
$
and consider the coefficient of the monomial
$
t_{1}^{k_{1}} \cdots \, t_{r}^{k_{r}}
$ 
in the sum \eqref{eq: LHS-moment identity} corresponding to $\lambda.$ Note that 
$
t_{1}^{k_{1}} \cdots \, t_{r}^{k_{r}}
$ 
comes together with a $q^{(|\kappa| - |\lambda|)\slash 2}$ in the polynomial \eqref{eq: Monomials-Schur-9.1}. By taking this into account, we see that the relevant part of the contribution of an eigenvalue $\beta$ to the coefficient we are interested in is $q^{(|\kappa| - |\lambda|)\slash 2} \beta.$ Notice that 
\begin{equation*} 
q^{(|\kappa| - |\lambda|)\slash 2} \beta = 
q^{m_{\beta} - d + (|\kappa| - 3 |\lambda|)\slash 2}\alpha \;\; \qquad\;\;  \text{(if $m_{\beta}  >  |\lambda| + d$)}.
\end{equation*} 
Referring to the right-hand side of this identity, consider the sum between the degree of the coefficient of $\alpha$ and the weight $m_{\alpha}$ of $\alpha.$ We shall need the following simple estimate of this quantity: 
\begin{equation} \label{eq: basic-estimate_k+d}
m_{\alpha}  + m_{\beta} - d + (|\kappa| - 3 |\lambda|)\slash 2 = 
2(|\lambda| + d) - d + (|\kappa| - 3 |\lambda|)\slash 2 = 
d + (|\kappa| +  |\lambda|)\slash 2 \le |\kappa| + d
\end{equation} 
as $|\kappa|  \ge  |\lambda|.$ 

The fact that the normalized eigenvalues $\alpha$ of $F^{*}$ on the compactly supported 
cohomology are algebraic integers follows from \cite[\S 5, Appendice, Th\'eor\`eme 5.2.2]{Del4}; see also \cite[Appendix, Theorem 0.2]{E}.

\subsection{Decomposing $A_{l}(T, q)$ in terms of $q$-Weil numbers} 
\label{Decomp Al(T, q) in q-Weil numbers} 
We shall now combine our induction assumptions and the above considerations 
with \eqref{eq: tag 4.2}, Proposition \ref{Proposition 4.3}, \eqref{eq: tag 4.4} and 
\eqref{eq: tag 4.5} to express the coefficient of $T^{\kappa} = t_{1}^{k_{1}} \cdots \, t_{r}^{k_{r}}$ 
in $E(T)^{\epsilon_{l}}A_{l}(T, q)$ in terms of traces of Frobenius on the cohomology of the local systems 
$\mathbb{V}'(\lambda)$ on $\mathscr{H}_{g}[2] \otimes_{_{\mathbb{F}_{\! q}}} \! \overline{\mathbb{F}}_{\! q}.$

Consider first the sum $A_{\mathfrak{n}}(T, q)$ ($\mathfrak{n} = (1^{n_{1}}\!, 2^{n_{2}}\!, \ldots)$) introduced at the beginning of \ref{Gen-Case}. We can express it in terms of the sum \eqref{eq: LHS-moment identity} as follows. For $d \in \mathscr{P}(\mathfrak{n}, q),$ we write 
\begin{equation*} 
\prod_{k = 1}^{r} P_{\scriptscriptstyle C_{d}}(t_{k}) \,=\, 
\exp \bigg( - \sum_{j = 1}^{\infty} \frac{a_{j}(C_{d})}{j}
\cdot\sum_{k = 1}^{r}t_{k}^{j} \bigg)
\end{equation*} 
with 
\begin{equation*}
a_{j}(C_{d}) = \Tr(F^{*j} \, \vert \, H_{\text{\'et}}^{1}(\bar{C}_{d}, \mathbb{Q}_{\ell}))
= - \sum_{\theta \, \in \,  {\bf{P^{1}}}(\mathbb{F}_{\! q^{j}})} 
\chi_{j}(d(\theta)).
\end{equation*} 
Using this identity, express $A_{\mathfrak{n}}(T, q)$ as    
\begin{equation*}
A_{\mathfrak{n}}(T, q) \; = \sum_{d \, \in \, \mathscr{P}(\mathfrak{n}, q)}
\exp \bigg( - \sum_{j = 1}^{\infty} \frac{a_{j}(C_{d})}{j}
\cdot\sum_{k = 1}^{r} t_{k}^{j} \bigg).
\end{equation*} 
We distinguish two cases according as $|\mathfrak{n}|$ is even or odd. If 
$|\mathfrak{n}|$ is odd, then, by Proposition \ref{Prop1-appendixB}, Appendix \ref{B}, 
we can further express 
\begin{equation*} 
A_{\mathfrak{n}}(T, q) = 
\frac{(n_{1} + 1)|\mathrm{GL}_{2}(\mathbb{F}_{\! q})|}{q^{2} - 1}\Bigg[\frac{1}{|\mathrm{GL}_{2}(\mathbb{F}_{\! q})|}
\sum_{d \, \in \, \mathscr{P}_{(|\mathfrak{n}'|\slash 2) - 1}(\mathfrak{n}',\, \mathbb{F}_{\! q})}
\left(\prod_{k = 1}^{r} P_{\scriptscriptstyle C_{d}}(t_{k}) \right)\Bigg]
\end{equation*} 
where $\mathfrak{n}' : = (1^{n_{1} + 1}\!, 2^{n_{2}}\!, \ldots).$ If $|\mathfrak{n}|$ is even, we write 
\begin{equation*}
A_{\mathfrak{n}}(T, q) = \frac{1}{2}\sum_{\varepsilon = 0, 1} \, \sum_{d \, \in \, \mathscr{P}(\mathfrak{n}, q)}\, \left[\exp\bigg( - \sum_{j = 1}^{\infty}\frac{a_{j}(C_{d})}{j}
\cdot\sum_{k = 1}^{r}t_{k}^{j}\bigg) \, + \, (- 1)^{\varepsilon} 
\exp\bigg( - \sum_{j = 1}^{\infty}\frac{a_{j}(C_{d})}{j}
\cdot\sum_{k = 1}^{r} ( - t_{k})^{j}\bigg)\right].
\end{equation*} 
Again, by Proposition \ref{Prop1-appendixB}, the even ($\varepsilon = 0$) part is just  
\begin{equation*} 
\frac{(q + 1 - n_{1})|\mathrm{GL}_{2}(\mathbb{F}_{\! q})|}{q^{2} - 1}\Bigg[\frac{1}{|\mathrm{GL}_{2}(\mathbb{F}_{\! q})|}
\sum_{d \, \in \, \mathscr{P}_{(|\mathfrak{n}|\slash 2) - 1}(\mathfrak{n},\, \mathbb{F}_{\! q})}
\left(\prod_{k = 1}^{r} P_{\scriptscriptstyle C_{d}}(t_{k}) \right)\Bigg].
\end{equation*} 
Note that this holds even when $n_{1} = 0.$ The remaining part can be expressed using Proposition 
\ref{Prop2-appendixB}, Appendix \ref{B} as 
\begin{equation*} 
\frac{|\mathrm{GL}_{2}(\mathbb{F}_{\! q})|}{q^{2} - 1}\Bigg[\frac{1}{|\mathrm{GL}_{2}(\mathbb{F}_{\! q})|}
\sum_{d \, \in \, \mathscr{P}_{(|\mathfrak{n}|\slash 2) - 1}(\mathfrak{n},\, \mathbb{F}_{\! q})}
\left(\prod_{k = 1}^{r} P_{\scriptscriptstyle C_{d}}(t_{k})\right)\frac{d P_{\scriptscriptstyle C_{d}}}{d t_{r + 1}}(0) \Bigg].
\end{equation*} 
Accordingly, we obtain the desired expression for the coefficients of $A_{\mathfrak{n}}(T, q),$ 
\begin{equation} \label{eq: final_decomp_eigen_A_n(T, q)}
\text{Coefficient}_{\scriptscriptstyle T^{\kappa}} A_{\mathfrak{n}}(T, q)
\,=\, \sum_{\alpha} p_{\alpha}^{(\kappa, \mathfrak{n})}(q) \alpha  \;\; \qquad \;\; 
\text{(for $\kappa  \in  \mathbb{N}^{r}$)}
\end{equation} 
in terms of normalized eigenvalues of Frobenius on the cohomology of the local systems 
$\mathbb{V}'(\lambda)$ ($|\lambda| \le |\kappa| + 1$) on 
$\mathscr{H}_{g}[2] \otimes_{_{\mathbb{F}_{\! q}}} \! \overline{\mathbb{F}}_{\! q}.$ Here $p_{\alpha}^{(\kappa, \mathfrak{n})}(q)$ are polynomials in $q$ with rational coefficients, 
and since $|\mathrm{GL}_{2}(\mathbb{F}_{\! q})| =  q (q + 1) (q - 1)^{2},$ 
$p_{\alpha}^{(\kappa, \mathfrak{n})}(0) = 0,$ for every $\alpha;$ if $\alpha$ and $\alpha'$ 
are conjugates, $p_{\alpha}^{(\kappa, \mathfrak{n})} \! = p_{\alpha'}^{(\kappa, \mathfrak{n})}.$ 
Furthermore, by \eqref{eq: basic-estimate_k+d} (applied also with $r$ replaced by $r + 1$), 
we see that 
\begin{equation*}
\deg \, p_{\alpha}^{(\kappa, \mathfrak{n})} +  m_{\alpha} \le |\kappa| + |\mathfrak{n}|. 
\end{equation*}

\begin{remark} \hskip-2ptBy Poincar\'e duality and Definition \ref{Definition q->1/q} we find that 
\begin{equation*}
\text{Coefficient}_{\scriptscriptstyle T^{\kappa}} A_{\mathfrak{n}}(T, 1\slash q)
\,=\, \sum_{\alpha} p_{\alpha}^{(\kappa, \mathfrak{n})}(1\slash q) \alpha^{- 1}.
\end{equation*} 
\end{remark} 
As in the proof of Lemma \ref{evenness-A}, by increasing $r$ to $r + r'$ in 
\eqref{eq: final_decomp_eigen_A_n(T, q)} with $r'$ sufficiently large, 
differentiating accordingly, and then setting $t_{r + 1}  = \cdots = t_{r + r'}  = 0,$ we see that \eqref{eq: final_decomp_eigen_A_n(T, q)} yields a similar expression for the coefficients of a sum of the form 
\begin{equation*}
\sum_{d \, \in \, \mathscr{P}(\mathfrak{n}, q)} 
\left(H(a_{1}^{*}(d, q), a_{2}^{*}(d, q), \ldots) \cdot \prod_{k = 1}^{r} 
P_{\scriptscriptstyle C_{d}}(t_{k})  \right).
\end{equation*} 
Here $H(\cdots)$ is any given polynomial expression in $a_{1}^{*}(d, q), a_{2}^{*}(d, q), \ldots .$ This applies in particular to \eqref{eq: A_{n, i, j}(T, q)}, and for partitions $\mathfrak{n}, \mathfrak{i}$ and $\mathfrak{j}$ ($\mathfrak{n} \ne {\bf{0}}$), note that the degree of the coefficient of any eigenvalue $\alpha$ occurring in the expression of 
$
\text{Coefficient}_{\scriptscriptstyle T^{\kappa}} A_{\mathfrak{n}, \mathfrak{i}, \mathfrak{j}}(T, q)
$ 
is at most 
$
|\kappa| + |\mathfrak{n}| + |\mathfrak{i}| + |\mathfrak{j}| - m_{\alpha}. 
$

\vskip5pt 
We can now prove the following 

\vskip10pt
\begin{thm} --- For $\kappa = (k_{1}, \ldots, k_{r}) \in \mathbb{N}^{r}\!,$ 
the coefficient of $t_{1}^{k_{1}} \cdots \, t_{r}^{k_{r}}$ in $E(T)^{\epsilon_{l}}A_{l}(T, q)$ 
can be expressed as 
\begin{equation*}
\sum_{\alpha} R_{\alpha}^{(\kappa,  l)}(q) \alpha
\end{equation*} 
the sum being over the distinct (normalized) eigenvalues of Frobenius acting on 
$\left(H_{c,  \mu}^{\bullet}(\mathscr{H}_{g}[2] \otimes \overline{\mathbb{F}}_{\! q}, 
\mathbb{V}'(\lambda))\right)_{g, \lambda, \mu}$ with $g \le [(l - 1) \slash 2],$ and $R_{\alpha}^{(\kappa,  l)}(x) \in x\mathbb{Q}[x].$ Moreover, we have 
$
\deg R_{\alpha}^{(\kappa,  l)} +  m_{\alpha} \le |\kappa| + l.
$
\end{thm}

\begin{proof} \hskip-0.3ptBy \eqref{eq: tag 4.2}, it suffices to investigate the 
coefficients of $E(T)^{\epsilon_{l}}A_{\mu, \nu}(T, q)$ for 
$\mu \ne (1)$ and $\nu \ne (l).$ We recall that 
$\mu = (\mu_{1} \ge \cdots \ge \mu_{n} \ge 1)$ and $\nu = (\nu_{1}, \ldots, \nu_{n})$ 
are such that $\nu_{1} \mu_{1} + \cdots + \nu_{n} \mu_{n} =  l.$
By the induction assumptions (see \eqref{eq: final-formula}, 
\eqref{eq: degree P + mj} and \eqref{eq: final-formula q->1/q}), 
the coefficients of 
$
q^{\nu_{j} \mu_{j}} \Lambda_{\nu_{j}}\!(q^{\mu_{j}}T^{\mu_{j}}\!, 1\slash q^{\mu_{j}}) 
$ 
can be expressed as 
\begin{equation*} 
q^{\mu_{j}(|\kappa| + \nu_{j})}\lambda(\kappa, \nu_{j}; 1\slash q^{\mu_{j}}) 
= \sum_{\alpha} q^{\mu_{j}(|\kappa| + \nu_{j} - m_{\alpha})}
P_{\alpha}^{(\kappa, \nu_{j})}(1\slash q^{\mu_{j}})\alpha^{\mu_{j}}
\qquad \text{($\kappa  \in \mathbb{N}^{r}$)}
\end{equation*} 
where 
$
Q_{\alpha}^{(\kappa, \nu_{j})}(x) : 
= x^{|\kappa| + \nu_{j} - m_{\alpha}} P_{\alpha}^{(\kappa, \nu_{j})}(1\slash x) \in \mathbb{Q}[x]
$ 
with 
$
x^{\deg Q_{\alpha}^{(\kappa, \nu_{j})}  + \, 2} \mid P_{\alpha}^{(\kappa, \nu_{j})}(x) 
$ 
for all $\alpha$ and $1\le j \le n;$ then, clearly, 
$\deg \, Q_{\alpha}^{(\kappa, \nu_{j})} \le (|\kappa| + \nu_{j} - m_{\alpha} - 2)\slash 2,$ 
and so 
\begin{equation*}
\deg \, Q_{\alpha}^{(\kappa, \nu_{j})} + m_{\alpha} \le (|\kappa| + \nu_{j} + m_{\alpha} - 2)\slash 2
\le |\kappa| + \nu_{j} - 1.
\end{equation*} 
It follows that for $\kappa_{1}  \in \mathbb{N}^{r}\!,$ the coefficients of the $q$-Weil numbers 
$\alpha_{1}$ occurring in 
$
\text{Coefficient}_{\scriptscriptstyle T^{\kappa_{1}}} \Lambda_{\delta, \mu, \nu}(T, q), 
$ 
with $\Lambda_{\delta, \mu, \nu}(T, q)$ defined by \eqref{eq: tag 4.5}, 
are polynomials of degree at most 
$ 
|\kappa_{1}| + l - |\mu| - m_{\alpha_{1}}
$ 
with rational coefficients.

To handle the other part of $E(T)^{\epsilon_{l}}A_{\mu, \nu}(T, q)$ in \eqref{eq: tag 4.4}, 
note that the sum 
\begin{equation*}
\sum_{d_{0} \, \in \, \mathscr{P}_{\! \scriptscriptstyle \nu}(\mu)}
\bigg(N_{\mu, \nu}(d_{0}, \delta) \prod_{i = 1}^{r} P_{\scriptscriptstyle C_{d_{0}}}\!(t_{i}) \bigg) 
\end{equation*} 
is (up to a constant) just $A_{\mathfrak{n}, \mathfrak{i}, \mathfrak{j}}(T, q)$ with $\mathfrak{n} = (\mu_{j})_{j\in J_{0}},$ and $\mathfrak{i}, \mathfrak{j}$ determined by splitting $(\mu_{j})_{j\in J_{1}}$ (depending on $\delta$) using Proposition \ref{Proposition 4.9}. As 
$
|\mathfrak{n}| + |\mathfrak{i}| + |\mathfrak{j}| = |\mu|, 
$ 
the coefficient of any eigenvalue $\alpha_{2}$ in 
$
\text{Coefficient}_{\scriptscriptstyle T^{\kappa_{2}}} A_{\mathfrak{n}, \mathfrak{i}, \mathfrak{j}}(T, q)
$ 
($\kappa_{2} \in \mathbb{N}^{r}$) is a polynomial of degree at most 
$
|\kappa_{2}| + |\mu| -  m_{\alpha_{2}}
$ 
with rational coefficients. Our assertions follow now for contributions to 
$E(T)^{\epsilon_{l}}A_{l}(T, q)$ given by \eqref{eq: tag 4.4}.

To complete the proof, note that, for $A_{\mu, \nu}(T, q)$ given by Proposition \ref{Proposition 4.3}, 
the same estimates apply. 
\end{proof}

We can further assume that the sum in the theorem is {\it reduced} in the sense that 
$\alpha' = q^{n}\alpha$ for some $n\in \mathbb{N}$ if and only if $\alpha' = \alpha.$ Since \begin{equation*}
\sum_{\alpha} R_{\alpha}^{(\kappa, l)}(q^{n}) \alpha^{n} 
= - \sum_{\alpha} q^{n(|\kappa| + l - m_{\alpha} + 1)}R_{\alpha}^{(\kappa, l)}(1\slash q^{n}) \alpha^{n}
\;\; \qquad \;\; \text{(for all $n \ge 1$)}
\end{equation*} 
(obtained from \eqref{eq: func-eq A_{l}(T, q)} applied over a finite extension $\mathbb{F}_{\! q^{n}}$ 
of $\mathbb{F}_{\! q}$), it follows  
that 
\begin{equation*}
R_{\alpha}^{(\kappa, l)}(x) = -\, x^{|\kappa| + l - m_{\alpha} + 1}R_{\alpha}^{(\kappa, l)}(1\slash x)
\;\; \qquad \;\; \text{(for all $\alpha$)}.
\end{equation*} 
Using this combined with \eqref{eq: prel-final-formula}, one defines $\lambda(\kappa, l; q)$ by cutting off the coefficient of each $q$-Weil algebraic integer in the expression of \eqref{eq: prel-final-formula} 
accordingly.

\section{An application}\label{section 10}
For $0 \le i \le 6,$ let $\mathscr{A}_{2}(w^{i})$ denote the quotient stack 
$\mathscr{A}_{2, 2}\slash \mathbb{S}_{6 - i},$ where $\mathbb{S}_{6 - i}$ is the subgroup of 
$\mathbb{S}_{6}$ fixing $\{1, \ldots, i \}$ pointwise. Here we identify 
$\mathrm{GSp}_{4}(\mathbb{Z}\slash 2)$ with $\mathbb{S}_{6},$ 
as in \cite[Section 2]{BFvdG1}, under the isomorphism defined by an 
embedding of the stack $\mathscr{H}_{2}[2]$ (denoted by $\mathscr{M}_{2}(w^{6})$ in \cite{BFvdG1}) 
into $\mathscr{A}_{2, 2}.$ Take\footnote{We are taking $i = 1$ for simplicity.} $i = 1,$ and consider the 
Euler characteristic 
\begin{equation*}
e_{\scriptscriptstyle \rm{Eis}}(\mathscr{A}_{2}(w^{1}) \otimes \overline{\mathbb{Q}}, \mathbb{V}(\lambda)) = \sum_{j = 0}^{6}\, (-1)^{j} [H_{\rm{Eis}}^{j}(\mathscr{A}_{2}(w^{1}) \otimes \overline{\mathbb{Q}}, \mathbb{V}(\lambda))]
\end{equation*} 
where $\lambda = (\lambda_{1} \ge \lambda_{2} \ge 0)$ with $|\lambda| = \lambda_{1} + \lambda_{2} \equiv 0 
\pmod 2.$

In what follows, we shall be interested in computing the trace of Frobenius 
\begin{equation*}
\Tr(F^{*} \vert \, e_{\scriptscriptstyle \rm{Eis}}(\mathscr{A}_{2}(w^{1}) \otimes \overline{\mathbb{Q}}, \mathbb{V}(\lambda))) = 
\Tr(F^{*} \vert \, e_{\scriptscriptstyle \rm{Eis}}(\mathscr{A}_{2}(w^{1}) 
\otimes \overline{\mathbb{F}}_{\! q}, \mathbb{V}(\lambda)))
\end{equation*}
see also \cite{vdG} and \cite{BFvdG1}. To do so, we first recall (see Theorem \ref{mom-for-odd_poly}) that 
\begin{equation*} 
\frac{1}{q (q - 1)} \!\!\sum_{\substack{\deg d_{0} = 5 \\ d_{0}-\text{monic \& square-free}}}
\left(\prod_{k = 1}^{r} P_{\scriptscriptstyle C_{d_{0}}}\!(t_{k})  \right)  = \,
(t_{1} \, \cdots \, t_{r})^{2}
\!\!\!\sum_{\substack{\lambda \subseteq (r^{2}) \\ \text{$|\lambda|$-\text{even}}}}
\Tr(F^{*} \vert \,
e_{c}^{\lambda}(\mathscr{H}_{2}(w^{1})  \otimes  \overline{\mathbb{F}}_{\! q}))\,
s_{\scriptscriptstyle \langle \lambda^{\dag} \rangle}(q^{\pm \frac{1}{2}} T^{\pm 1})\,
q^{r - \frac{|\lambda|}{2}}.
\end{equation*} 
Put 
\begin{equation*}
A_{5}^{(0)}(T, q) \;\; : =
\sum_{\substack{\deg d_{0} = 5 \\ d_{0}-\text{monic \& square-free}}}
\left(\prod_{k = 1}^{r} P_{\scriptscriptstyle C_{d_{0}}}\!(t_{k}) \right)
\end{equation*}
and note that we can express
\begin{equation*} 
A_{5}^{(0)}(T, q)
= \frac{q (q - 1)}{|\mathrm{GL}_{2}(\mathbb{F}_{\! q})|} 
\sum_{\nu = (1^{\nu_{1}}\!, \ldots, 6^{\nu_{6}})}
\nu_{1} \; \cdot \sum_{d_{0} \, \in \, \mathscr{P}_{2}(\nu, \mathbb{F}_{\! q})}\,
\left(\prod_{k = 1}^{r} P_{\scriptscriptstyle C_{d_{0}}}\!(t_{k}) \right).
\end{equation*}
For each partition $\nu = (1^{\nu_{1}}\!, \ldots, 6^{\nu_{6}}),$ the (normalized) inner sum
\begin{equation}\label{eq: Jacobians-alg-curves}
\frac{1}{|\mathrm{GL}_{2}(\mathbb{F}_{\! q})|} 
\sum_{d_{0} \, \in \, \mathscr{P}_{2}(\nu, \mathbb{F}_{\! q})}
\left(\prod_{k = 1}^{r} P_{\scriptscriptstyle C_{d_{0}}}\!(t_{k}) \right)
\end{equation}
can be thought of as the contribution corresponding to Jacobians of smooth projective irreducible 
algebraic curves of genus two in the moduli of principally polarized abelian surfaces. As we are interested in the Euler characteristics of $\mathscr{A}_{2}(w^{1}),$ to  
\eqref{eq: Jacobians-alg-curves} we add the remaining contribution  
\begin{equation} \label{eq: prod-ell-curves}
\frac{1}{2} \frac{1}{|\mathrm{G}(\mathbb{F}_{\! q})|^{2}}\sum_{\nu = \sigma \, \cup \,\tau} \,
\sum_{\substack{f \, \in \, \mathscr{P}_{1}(\sigma, \mathbb{F}_{\! q}) \\ 
h \, \in \, \mathscr{P}_{1}(\tau, \mathbb{F}_{\! q})}} 
\left(\prod_{k = 1}^{r} P_{\scriptscriptstyle E_{f}}(t_{k}) P_{\scriptscriptstyle E_{h}}(t_{k})\right)  +\, \frac{1}{2} \frac{1}{|\mathrm{G}(\mathbb{F}_{\! q^{2}})|}
\sum_{f \, \in \, \mathscr{P}_{1}(\nu^{1\slash 2},\, \mathbb{F}_{\! q^{2}})} 
\left(\prod_{k = 1}^{r} P_{\scriptscriptstyle E_{f}}(t_{k}^{2})\right)
\end{equation}
corresponding to pairs of elliptic curves joined at the origin. Here the first sum is over all ordered choices of partitions $\sigma$ and $\tau$ of 3 such that $\nu = \sigma \, \cup \,\tau,$ and 
$\mathscr{P}_{1}(\mu, \mathbb{F}_{\! q})$ (for a partition $\mu$ of 3 and a finite field 
$\mathbb{F}_{\! q}$ of odd characteristic) denotes the subset of $\mathbb{F}_{\! q}[x]$ 
consisting of all square-free cubic polynomials with factorization type $\mu;$ each element $f \in \mathscr{P}_{1}(\mu, \mathbb{F}_{\! q})$ defines an elliptic curve $E_{f}$ given by $y^2 = f(x)$ with 
$x = \infty$ as origin. The normalizing constant 
$$
|\mathrm{G}(\mathbb{F}_{\! q})| = q (q - 1)^{2}
$$ 
represents the number of $\mathbb{F}_{\! q}$-isomorphisms among the elliptic curves 
$E_{f}$ ($f \in \mathscr{P}_{1}(\mu, \mathbb{F}_{\! q})$). The second 
contribution in \eqref{eq: prod-ell-curves}, corresponding to elliptic curves defined 
over $\mathbb{F}_{\! q^{2}},$ each curve being joined at the origin with its Frobenius conjugate, 
occurs only if $\nu$ is of the form $(2^{\nu_{2}}\!, 4^{\nu_{4}}\!, 6^{\nu_{6}}),$ in which case we set 
$\nu^{1\slash 2} : = (1^{\nu_{2}}\!, 2^{\nu_{4}}\!, 3^{\nu_{6}})$ (see \cite[Section5]{BFvdG1}). Accordingly, if $A_{\scriptscriptstyle (2, 1)}(T, q)$ and $A_{\scriptscriptstyle (1^{3})}(T, q)$ are defined by \eqref{eq: A_{n}(T, q)}, the generating function 
\begin{equation}\label{eq: gen-series-cptcohA2(w1)}
(t_{1} \, \cdots \, t_{r})^{2}
\!\!\!\sum_{\substack{\lambda \subseteq (r^{2}) \\ \text{$|\lambda|$-\text{even}}}}
\Tr(F^{*} \vert \, e_{c}(\mathscr{A}_{2}(w^{1}) \otimes \overline{\mathbb{Q}}, \mathbb{V}(\lambda)))\,
s_{\scriptscriptstyle \langle \lambda^{\dag} \rangle}(q^{\pm \frac{1}{2}} T^{\pm 1})\,
q^{r - \frac{|\lambda|}{2}}
\end{equation} 
is just 
\begin{equation*}
\frac{A_{5}^{(0)}(T, q)}{q (q - 1)}
+ \frac{A_{3}^{(0)}(T, q)\big(A_{\scriptscriptstyle (2, 1)}(T, q) + 3 A_{\scriptscriptstyle (1^{3})}(T, q)\big)}{q^{2} (q - 1)^{2}}
\end{equation*}
with 
\begin{equation*}
A_{3}^{(0)}(T, q) \;\;\;  =
\sum_{\substack{\deg f = 3 \\ f-\text{monic \& square-free}}}
\left(\prod_{k = 1}^{r} P_{\scriptscriptstyle E_{f}}(t_{k}) \right).
\end{equation*} 
We recall now that the {\it full} Eisenstein cohomology is the difference 
between the compactly supported and the usual cohomology, and the corresponding 
Euler characteristic is 
\begin{equation*}
e_{\scriptscriptstyle \rm{full\, Eis}}(\mathscr{A}_{2}(w^{1}), \mathbb{V}(\lambda)) \,=\, 
e_{c}(\mathscr{A}_{2}(w^{1}), \mathbb{V}(\lambda)) \,-\, e(\mathscr{A}_{2}(w^{1}), \mathbb{V}(\lambda)).
\end{equation*} 
Letting 
\begin{equation*}
A_{3}^{(2)}(T, q) = A_{\scriptscriptstyle (2, 1)}(T, q) + 3 A_{\scriptscriptstyle (1^{3})}(T, q)
\end{equation*}
it follows that the generating function 
\eqref{eq: gen-series-cptcohA2(w1)} corresponding to 
$
e_{\scriptscriptstyle \rm{full\, Eis}}(\mathscr{A}_{2}(w^{1}), \mathbb{V}(\lambda))
$ 
equals 
\begin{equation}\label{eq: fullEis-coh}
\begin{split}
&\frac{A_{5}^{(0)}(T, q)}{q (q - 1)}
\,+ \,
\frac{A_{3}^{(0)}(T, q)A_{3}^{(2)}(T, q)}{q^{2} (q - 1)^{2}}\\
& - \frac{q^{4} A_{5}^{(0)}(q\, T, 1\slash q)}{q^{-1} - 1}
\,- \,
\frac{q^{5} A_{3}^{(0)}(q\, T, 1\slash q) A_{3}^{(2)}(q\, T, 1\slash q)}
{(q^{-1} - 1)^{2}}.
\end{split}
\end{equation} 
On the other hand, the generating series 
$
A_{5}(T, q)
$ 
introduced at the beginning of Section \ref{section 4} can be decomposed as
\begin{equation*}
A_{5}(T, q) = A_{5}^{(0)}(T, q) \,+ \, \text{degenerate part}
\end{equation*}
where the degenerate part consists of the following contributions: 
\begin{equation}\label{eq: degenerate1}
\begin{split}
&\frac{q^{2}  \Big(q A_{3}^{(0)}(T, q) - A_{3}^{(2)}(T, q) + A_{3}^{(1)}(T, q) \Big)}{2}
\frac{\Lambda_{2}(q \, T, 1\slash q)}{E(T)}\\
&+\, \frac{q^{2}  \Big(q A_{3}^{(0)}(- T, q) -   A_{3}^{(2)}( - T, q) + A_{3}^{(1)}(- T, q) \Big)}{2}
\frac{\Lambda_{2}(- q \, T, 1\slash q)}{E(-T)}
\end{split}
\end{equation}
corresponding to the partitions 
$1 \cdot 1 + 1\cdot 1 + 1\cdot 1 + 2\cdot 1,$ $1 \cdot 2 + 1\cdot 1 + 2\cdot 1$ and 
$1 \cdot 3 + 2\cdot 1,$ 
\begin{equation}\label{eq: degenerate2}
q^{3} A_{3}^{(2)}(T, q) \Lambda_{3}(q\, T, 1\slash q)
\end{equation} 
corresponding to $1 \cdot 1 + 1\cdot 1 + 3\cdot 1$ and $1 \cdot 2 + 3\cdot 1,$ 
\begin{equation}\label{eq: degenerate3}  
\frac{q^{5}(q - 1)^{2}}{4}\frac{\Lambda_{2}(q^{2} T^{2}\!, 1\slash q^{2})}{E(T^{2})}\,
+\,\frac{q^{5}(q^{2} - 1)}{4}\frac{\Lambda_{2}(- q^{2} T^{2}\!, 1\slash q^{2})}{E(- T^{2})} 
\end{equation}
corresponding to $2 \cdot 2 + 1\cdot 1,$ 
\begin{equation}\label{eq: degenerate4}
\frac{q^{5}(q - 1)^{2}}{4}\frac{\Lambda_{2}(q \, T, 1\slash q) \Lambda_{2}(- q \, T, 1\slash q)}{E(T)E(-T)}\, + \, \frac{q^{5}(q - 1)(q - 3)}{8}
\Big(\frac{\Lambda_{2}(q \, T, 1\slash q)^{2}}{E(T)^{2}}  +  
\frac{\Lambda_{2}(- q \, T, 1\slash q)^{2}}{E(- T)^{2}}\Big)
\end{equation}
corresponding to $2 \cdot 1 + 2\cdot 1 + 1\cdot 1,$ 
\begin{equation}\label{eq: degenerate5}
\frac{q^{6}(q - 1)}{2}\Lambda_{3}(q \, T, 1\slash q)
\Big(\frac{\Lambda_{2}(q \, T, 1\slash q)}{E(T)}  + 
\frac{\Lambda_{2}(- q \, T, 1\slash q)}{E(- T)} \Big)
\end{equation}
corresponding to $3 \cdot 1 + 2\cdot 1,$ and finally 
\begin{equation}\label{eq: degenerate6}
\frac{q^{5}(q - 1)}{2}\Big(\frac{\Lambda_{4}(q\, T, 1\slash q)}{E(T)} +  
\frac{\Lambda_{4}(- q\,  T, 1\slash q)}{E(- T)} \Big)
\end{equation}
corresponding to $4 \cdot 1 + 1\cdot 1.$ We recall that 
\begin{equation*}
A_{5}(T, q)  = - \, q^{6}A_{5}(q \, T, 1\slash q).
\end{equation*} 
Dividing both sides by $q^{2} - q,$ and using the trivial identity 
$$
q^{2} - q = - \, q^{3}\bigg(\frac{1}{q^{2}} - \frac{1}{q} \bigg)
$$
we see that 
\begin{equation}\label{eq: func-eq-A5(T, q)-normalized}
\frac{A_{5}(T, q)}{q^{2} - q} - q^{3}\frac{A_{5}(q \, T, 1\slash q)}{q^{-2} - q^{-1}} = 0.
\end{equation} 
Notice that the results of Section \ref{section: The series BarC(X, T, q)} (specifically, see 
\eqref{eq: funct-equation-needed-proof-main-thm}) give, in particular, a formula for the trace of Frobenius on 
$
e_{\scriptscriptstyle \rm{full\, Eis}}^{\lambda}
(\mathscr{H}_{g}(w^{1})\otimes \overline{\mathbb{F}}_{\! q})
$ 
($g \ge 2$); when $g = 2,$ the corresponding formula can be made explicit. \!In what follows, 
we shall illustrate this in the case of $\mathscr{A}_{2}(w^{1}),$ thus recovering the formula in \cite[Corollary 4.6]{BFvdG1}, and proved in \cite[Theorem 9.1]{vdG}.

\vskip10pt 
\theoremstyle{l=5 vs vdG-eis-coh}
\newtheorem*{l=5 vs vdG-eis-coh}{Theorem (Bergstr\"om-Faber-van der Geer)}
\begin{l=5 vs vdG-eis-coh} --- For every $\lambda = (\lambda_{1} \ge \lambda_{2} \ge 0),$ the 
trace of Frobenius $F^{*}$ on 
$e_{\scriptscriptstyle \rm{Eis}}(\mathscr{A}_{2}(w^{1}) \otimes \overline{\mathbb{F}}_{\! q}, \mathbb{V}(\lambda))$ is given by 
\begin{equation*}
\left[\frac{\lambda_{1} - \lambda_{2} - 1}{4} \right] - \left[\frac{\lambda_{1} + \lambda_{2} + 1}{4} \right] q^{\lambda_{2} + 1}
+ \,  \begin{cases} 
2 + 2 T_{\lambda_{2} + 2}(q) & \text{if $\lambda_{2}$ is even}\\ 
- 2T_{\lambda_{1} + 3}(q) & \text{if $\lambda_{2}$ is odd.} 
\end{cases}
\end{equation*}
\end{l=5 vs vdG-eis-coh}

\begin{proof} First, express the left-hand side of \eqref{eq: func-eq-A5(T, q)-normalized} using 
\eqref{eq: degenerate1}--\eqref{eq: degenerate6}, and then apply \eqref{eq: tag 4.1} to transform each 
$\Lambda_{l}(\pm T^{m}\!, q^{m}),$ for $2\le l \le 4,$ into $\Lambda_{l}(\pm q^{m}T^{m}\!, q^{-m});$ 
we also write the corresponding contribution $A_{l}(\pm T^{m}\!, q^{m})$ explicitly by recalling that 
$E(T)A_{2}(T, q) = q^{2} - q,$ that is, 
\begin{equation*}
A_{3}(T, q) = A_{3}^{(0)}(T, q) +  
\frac{q^{3}(q - 1)}{2}\Big(\frac{\Lambda_{2}(q\, T, 1\slash q)}{E(T)} + 
\frac{\Lambda_{2}(- q\, T, 1\slash q)}{E(- T)}\Big)
\end{equation*}
and, by \eqref{eq: tag 5.4}--\eqref{eq: tag 5.8},
\begin{equation*}
\begin{split}
& E(T)A_{4}(T, q) = q A_{3}^{(0)}(T, q)  +  A_{3}^{(1)}(T, q)+\frac{q^{5}(q - 1)}{2} \frac{E(q\, T)\Lambda_{2}(q^{2} T^{2}\!, 1\slash q^2)}{E(- T)}\\
& + q^{3}(q - 1)\left[\frac{q^{2} E(q\, T)\Lambda_{2}(q \, T, 1\slash q)^{2}}{2 E(T)} 
+  \frac{(q - 2)\Lambda_{2}(q \, T, 1\slash q)}{2 E(T)} + 
\frac{q \Lambda_{2}(- q \, T, 1\slash q)}{2 E(- T)}
+  q \Lambda_{3}(q \, T, 1\slash q)\right] 
\end{split}
\end{equation*} 
with $A_{3}^{(1)}(T, q)$ defined as in the proof of Proposition \ref{functional equation A4}. 
To simplify things somewhat, we can use the identities: 
\begin{equation*}
\Lambda_{3}(- q \, T, 1\slash q) = \Lambda_{3}(q \, T, 1\slash q) \qquad
A_{3}^{(0)}(-T, q) = A_{3}^{(0)}(T, q) \;\;\; \text{and} \;\;\; 
A_{3}^{(2)}(-T, q) = A_{3}^{(2)}(T, q).
\end{equation*}
In the expression obtained, the coefficient of every monomial 
\begin{equation*}
\Lambda_{l_{1}}\!(\pm q^{m_{1}}T^{m_{1}}\!, q^{-m_{1}})^{k_{1}}
\cdots \, 
\Lambda_{l_{n}}\!(\pm q^{m_{n}}T^{m_{n}}\!, q^{-m_{n}})^{k_{n}}
\qquad (k_{1}, \ldots, k_{n} \ge 0)
\end{equation*}
has to vanish. By taking the coefficient of 
$
\Lambda_{3}(q \, T, 1\slash q)
$ 
we find, in particular, that
\begin{equation*}
A_{3}^{(2)}(T, q) + q^{4}A_{3}^{(2)}(q\, T, 1\slash q)  
= q (q - 1)^{2}\Big(\frac{1}{E(T) E(q\, T)}  + 
\frac{1}{E(-T) E(-q\, T)} \Big). 
\end{equation*} 
We also need \eqref{eq: tag 5.10} together with the identity
\begin{equation*}
A_{3}^{(0)}(T, q) + q^{4}A_{3}^{(0)}(q\, T, 1\slash q)  
= \frac{q (q - 1)^{2}}{2}\Big(\frac{1}{E(T) E(q\, T)}  + 
\frac{1}{E(-T) E(-q\, T)} \Big)
\end{equation*}
which follows easily by combining the above expression of $A_{3}(T, q)$ with the 
functional equations \eqref{eq: tag 4.1'} and \eqref{eq: tag 4.1}, applied for $l = 3$ and 
$l = 2,$ respectively.

Now subtract the constant term of the expression obtained for the left-hand side of 
\eqref{eq: func-eq-A5(T, q)-normalized} (i.e., all the terms free of factors 
$\Lambda_{l}(\pm q^{m}T^{m}\!, q^{-m})$ with $2\le l \le 4$) from \eqref{eq: fullEis-coh}. 
By using the last two identities and \eqref{eq: tag 5.10}, we can express each 
$A_{3}^{(j)}(q\, T, 1\slash q)$ ($j = 0,1,2$) in terms of $A_{3}^{(j)}(T, q),$ giving 
the following alternative expression for the generating function of the trace of Frobenius 
on $e_{\scriptscriptstyle \rm{full\, Eis}}(\mathscr{A}_{2}(w^{1}) \otimes \overline{\mathbb{F}}_{\! q},\mathbb{V}(\lambda)):$ 
\begin{equation*}
\begin{split}
& \frac{(q + 1)A_{3}^{(0)}(T, q)  + A_{3}^{(1)}(T, q)}{q E(T) E(q\, T)} \, +\, 
\frac{(q + 1)A_{3}^{(0)}(-T, q)  + A_{3}^{(1)}(-T, q)}{q E(-T) E(- q\, T)} \\
& -\, \frac{(q  - 1)^{3}}{8}\bigg( \frac{1}{E(T)^{2} E(q\, T)^{2}} \,+\, \frac{1}{E(- T)^{2} E(- q\, T)^{2}} \bigg)\, 
+ \, \frac{1}{4} \frac{(q - 1)(q + 1)^{2}}{E(- T^{2}) E(- q^{2}  T^{2})}\, -\, 
\frac{(q + 1)(q  - 1)^{2}}{E(T^{2}) E(q^{2}  T^{2})}.
\end{split}
\end{equation*} 
Note that 
\begin{equation*}
(q + 1) A_{3}^{(0)}(T, q)  +     A_{3}^{(1)}(T, q)  \;\;\,  = 
\!\sum_{\substack{\deg f = 3 \\ f-\text{monic \& square-free}}}
\left[\left(q + 1 - a(E_{f}) \right) \cdot \prod_{k = 1}^{r} 
P_{\scriptscriptstyle E_{f}} (t_{k}) \right] .
\end{equation*}
Expressing $P_{\scriptscriptstyle E_{f}}(t_{k})$ for all $f$ as 
\begin{equation*}
P_{\scriptscriptstyle E_{f}}(t_{k}) =  1 - a(E_{f})t_{k}  + \, q t_{k}^{2} 
\, =\,  (1 - \sqrt{q}\, \alpha_{\scriptscriptstyle E_{f}} t_{k})(1 - \sqrt{q}\, \bar{\alpha}_{\scriptscriptstyle E_{f}} t_{k})
\qquad (\text{with $|\alpha_{\scriptscriptstyle E_{f}}| = 1$})
\end{equation*}
it follows from Lemma \ref{BG} that we can write 
\begin{equation} \label{eq: prelim1}
\begin{split}
& \frac{(q + 1)A_{3}^{(0)}(T, q) + A_{3}^{(1)}(T, q)}{E(T) E(q\, T)} \\
&= \sum_{f} \, \left[\left(q + 1 - a(E_{f}) \right)\prod_{k = 1}^{r} 
(1 - \sqrt{q}\, \alpha_{\scriptscriptstyle E_{f}} t_{k})
(1 - \sqrt{q}\, \bar{\alpha}_{\scriptscriptstyle E_{f}} t_{k})
(1 - t_{k}) (1 - q t_{k}) \right]\\
&= q^{r} (t_{1} \, \cdots \, t_{r})^{2} \sum_{f}\, 
\left[\left(q + 1 - a(E_{f}) \right) \, \cdot  \sum_{\lambda \subseteq (r^{2})}  (-1)^{|\lambda|}
s_{\scriptscriptstyle \langle \lambda \rangle}(\alpha_{\scriptscriptstyle E_{f}}^{\pm 1}, q^{\pm \frac{1}{2}}) s_{\scriptscriptstyle \langle \lambda^{\dag} \rangle} 
(q^{\pm \frac{1}{2}} t_{1}^{\pm 1}, \ldots, q^{\pm \frac{1}{2}} t_{r}^{\pm 1}) \right].
\end{split}
\end{equation} 
By applying the Weyl character formula, for each $\lambda \subseteq (r^{2}),$
\begin{equation*}
s_{\scriptscriptstyle \langle \lambda \rangle}(\alpha_{\scriptscriptstyle E_{f}}^{\pm 1}, q^{\pm \frac{1}{2}}) \; = \;
\frac{\begin{vmatrix}
\alpha_{\scriptscriptstyle E_{f}}^{\lambda_{1} + 2} - 
\alpha_{\scriptscriptstyle E_{f}}^{-\lambda_{1} -  2}& 
\alpha_{\scriptscriptstyle E_{f}}^{\lambda_{2} + 1} - 
\alpha_{\scriptscriptstyle E_{f}}^{-\lambda_{2} - 1} \cr 
q^{\frac{\lambda_{1} + 2}{2}} - q^{-\frac{\lambda_{1} + 2}{2}} & 
q^{\frac{\lambda_{2} + 1}{2}} - q^{-\frac{\lambda_{2} + 1}{2}} 
\end{vmatrix}}{\begin{vmatrix}
\alpha_{\scriptscriptstyle E_{f}}^{2} - \,  
\alpha_{\scriptscriptstyle E_{f}}^{-2} & 
\alpha_{\scriptscriptstyle E_{f}}^{} - \, 
\alpha_{\scriptscriptstyle E_{f}}^{- 1} \cr 
q - q^{-1} & q^{\frac{1}{2}} - q^{-\frac{1}{2}} 
\end{vmatrix}}
\end{equation*}
\vskip-7pt
\begin{equation*}
\hskip141pt =
-\, \sqrt{q}\, \frac{s_{\scriptscriptstyle \langle \lambda_{2} \rangle}(q^{\pm \frac{1}{2}})  
\, s_{\scriptscriptstyle \langle \lambda_{1} + 1 \rangle}(\alpha_{\scriptscriptstyle E_{f}}^{\pm 1}) \,-\,  s_{\scriptscriptstyle \langle \lambda_{1} + 1 \rangle}(q^{\pm \frac{1}{2}}) 
\, s_{\scriptscriptstyle \langle \lambda_{2} \rangle}(\alpha_{\scriptscriptstyle E_{f}}^{\pm 1})}
{q + 1 - a(E_{f})}.
\end{equation*}
Substituting this into \eqref{eq: prelim1}, summing over $f$ and applying Birch-Ihara identity, one sees that
\begin{equation*}
\hskip-24pt \frac{(q + 1)A_{3}^{(0)}(T, q)  + A_{3}^{(1)}(T, q)}{q E(T) E(q\, T)} \, +\, 
\frac{(q + 1)A_{3}^{(0)}(-T, q)  + A_{3}^{(1)}(-T, q)}{q E(-T) E(- q\, T)} 
\end{equation*}
\vskip-10pt
\begin{equation}\label{eq: BergFabvdG1}
\begin{split}
&= \, 2 (t_{1} \, \cdots \, t_{r})^{2}  
\!\!\!\sum_{\substack{\lambda \subseteq (r^{2}) \\ \text{$\lambda_{1}, \lambda_{2}$-odd}}}  
(1 + T_{\lambda_{1} + 3}(q))\, q^{r +  \frac{1 - \lambda_{1}}{2}} \Big(q^{\frac{\lambda_{2} + 1}{2}} - q^{-\frac{\lambda_{2} + 1}{2}} \Big)
s_{\scriptscriptstyle \langle \lambda^{\dag} \rangle}(q^{\pm \frac{1}{2}} T^{\pm 1})\\
& + \,  2 (t_{1} \, \cdots \, t_{r})^{2}  
\!\!\!\sum_{\substack{\lambda \subseteq (r^{2}) \\ \text{$\lambda_{1}, \lambda_{2}$-even}}} 
(- 1  - T_{\lambda_{2} + 2}(q))\, q^{r + 1 - \frac{\lambda_{2}}{2}} 
\Big(q^{1 + \frac{\lambda_{1}}{2}} - q^{- 1 - \frac{\lambda_{1}}{2}} \Big) 
s_{\scriptscriptstyle \langle \lambda^{\dag} \rangle}(q^{\pm \frac{1}{2}} T^{\pm 1}).
\end{split}
\end{equation}

On the other hand, we can write 
\begin{equation}\label{eq: BergFabvdG2}
-\, \frac{(q  - 1)^{3}}{8}\bigg( \frac{1}{E(T)^{2} E(q\, T)^{2}} \,+\, \frac{1}{E(- T)^{2} E(- q\, T)^{2}} \bigg)\, + \, \frac{1}{4} \frac{(q - 1)(q + 1)^{2}}{E(- T^{2}) E(- q^{2}  T^{2})}\, -\, 
\frac{(q + 1)(q  - 1)^{2}}{E(T^{2}) E(q^{2}  T^{2})}
\end{equation}
\vskip-10pt
\begin{equation*}
\begin{split}
= \,  q^{r} (t_{1} \, \cdots \, t_{r})^{2}  
\!\!\sum_{\substack{\lambda \subseteq (r^{2}) \\ \text{$|\lambda|$-even}}} 
&\Big[- \tfrac{1}{4} (q - 1)^{3} s_{\scriptscriptstyle \langle \lambda \rangle}(q^{\pm \frac{1}{2}}, q^{\pm \frac{1}{2}}) \, +\,  \tfrac{1}{4}(q - 1) (q + 1)^{2}s_{\scriptscriptstyle \langle \lambda \rangle}(i q^{\pm \frac{1}{2}}, - i q^{\pm \frac{1}{2}})  \\
& \hskip50pt - \, (q + 1)(q  - 1)^{2} s_{\scriptscriptstyle \langle \lambda \rangle}(q^{\pm \frac{1}{2}}, -  q^{\pm \frac{1}{2}}) \Big] 
s_{\scriptscriptstyle \langle \lambda^{\dag} \rangle}(q^{\pm \frac{1}{2}}  T^{\pm 1}). 
\end{split}
\end{equation*} 
Using again the Weyl character formula, we have 
\begin{equation*}
s_{\scriptscriptstyle \langle \lambda \rangle}(q^{\pm \frac{1}{2}}, q^{\pm \frac{1}{2}}) \, =\,  
\frac{(\lambda_{1} + 2)(q^{2 + \lambda_{1}} + 1)(q^{1 + \lambda_{2}} - 1) \,-\, 
(\lambda_{2} + 1) (q^{2 + \lambda_{1}} - 1)(q^{1 + \lambda_{2}} + 1)}{(q - 1)^3}
\, q^{-\frac{|\lambda|}{2}}
\end{equation*}
\begin{equation*}
\hskip-70pt s_{\scriptscriptstyle \langle \lambda \rangle}(i q^{\pm \frac{1}{2}}, 
- i q^{\pm \frac{1}{2}}) \, =\, 
(-1)^{(\lambda_{1} - \lambda_{2})\slash 2}\, \frac{(q^{2 + \lambda_{1}} - (-1)^{\lambda_{1}})
(q^{1 + \lambda_{2}}  + (-1)^{\lambda_{2}})}{(q^{2} - 1) (q + 1)} \, q^{- \frac{|\lambda|}{2}}
\end{equation*}
and 
\begin{equation*}
\hskip-145pt s_{\scriptscriptstyle \langle \lambda \rangle}(q^{\pm \frac{1}{2}}, - q^{\pm \frac{1}{2}}) \, =\, (-1)^{\lambda_{1}} \frac{(q^{2 + \lambda_{1}} - 1)(q^{1 + \lambda_{2}} - 1)}{(q^{2} - 1) (q - 1)} \, q^{-\frac{|\lambda|}{2}}.
\end{equation*} 
Fix $\lambda \subseteq (r^{2})$ of even weight. By adding the coefficients of 
$q^{r - \frac{|\lambda|}{2}}(t_{1} \, \cdots \, t_{r})^{2} s_{\scriptscriptstyle \langle \lambda^{\dag} \rangle}(q^{\pm \frac{1}{2}} T^{\pm 1})$ in \eqref{eq: BergFabvdG1} and \eqref{eq: BergFabvdG2}, we obtain the expression
\begin{equation*}
\hskip-23pt \left[\frac{\lambda_{1} - \lambda_{2} - 1}{4} \right] \, - \, \left[\frac{\lambda_{1} + \lambda_{2} + 1}{4} \right] q^{\lambda_{2} + 1}
+\, \begin{cases} 
2 +  2 T_{\lambda_{2} + 2}(q) & \text{if $\lambda_{2}$ is even}\\ 
- 2T_{\lambda_{1} + 3}(q) & \text{if $\lambda_{2}$ is odd} 
\end{cases}
\end{equation*}
\vskip-5pt
\begin{equation*}
\hskip10pt - \,  \left(\left[\frac{\lambda_{1} - \lambda_{2} - 1}{4} \right]\, - \,  \left[\frac{\lambda_{1} + \lambda_{2} + 1}{4} \right] q^{-\lambda_{2} - 1}
+\, \begin{cases} 
2 + 2 T_{\lambda_{2} + 2}(1\slash q) & \text{if $\lambda_{2}$ is even}\\ 
- 2T_{\lambda_{1} + 3}(1\slash q) & \text{if $\lambda_{2}$ is odd} 
\end{cases}\right) q^{3 + |\lambda|}.
\end{equation*}
We recall that, by convention, we take $T_{2}(q) = -\, q - 1.$

The full Eisenstein cohomology is anti-invariant under Poincar\'e duality, and is determined by 
the Eisenstein cohomology by antisymmetrizing. By \cite[Theorem 3.5]{Peters}, the terms 
in the above expression that come from the compactly supported Eisenstein cohomology are 
precisely those of weights $< |\lambda| + 3$ which completes the proof of the theorem.  \end{proof}

\appendix

\section{Appendix}\label{A} 
In this appendix, \!we give, for each $r\le 3,$ the explicit expression of the power series $Z(T, t_{r + 1}; q)$ in Section \ref{section 2} as a rational function. \!One way to obtain these expressions is by using the averaging technique in \cite{CG1}. In fact, the expression of \eqref{eq: tag 2.6} as a rational function is given in \cite[Example 3.6]{CG1} when $r = 1,$ and \cite[Example 3.7]{CG1} when $r = 2.$ Accordingly, we have 
\begin{equation*}
Z(t_{1}, t_{2}; q) = \frac{1 -  q^2 t_{1} t_{2}}{\left(1 - q t_{1} \right) \left(1 - q t_{2} \right)
\left(1 - q^3 t_{1}^2 t_{2}^2 \right)}
\end{equation*} 
when $r = 1,$ and 
\begin{equation*}
Z(t_{1}, t_{2}, t_{3}; q)  = \frac{1 -  q^2 t_{1} t_{3} - q^2 t_{2} t_{3} +  q^3 t_{1} t_{2} t_{3} 
+ q^3 t_{1} t_{2} t_{3}^2 - q^4 t_{1}^2 t_{2} t_{3}^2 - q^4 t_{1} t_{2}^2 t_{3}^2 
+ q^6 t_{1}^2 t_{2}^2 t_{3}^3}{\left(1 - q t_{1} \right) \left(1 - q  t_{2} \right) \left(1 - q t_{3} \right) 
\left(1 - q^3 t_{1}^2 t_{3}^2 \right) \left(1 - q^3 t_{2}^2 t_{3}^2 \right) 
\left(1 - q^4 t_{1}^2 t_{2}^2 t_{3}^2 \right)}
\end{equation*} 
when $r = 2.$ For $r = 3,$ we used instead the recurrence relations in the proof of Theorem 3.7 in \cite{BD}. 
If we put 
$$
Z(t_{1}, t_{2}, t_{3}, t_{4}; q) = \frac{N(t_{1}, t_{2}, t_{3}, t_{4}; q)}{D(t_{1}, t_{2}, t_{3}, t_{4}; q)}
$$ 
then by using Mathematica, one finds: 
\begin{equation*}
\begin{split}
& N(t_{1}, t_{2}, t_{3}, t_{4}; q) = 
1 -  q^2 t_{1} t_{4} - q^2 t_{2} t_{4} +  q^3 t_{1} t_{2} t_{4} - q^2 t_{3} t_{4} + q^3 t_{1} t_{3} t_{4} 
+ q^3 t_{2} t_{3} t_{4}  -  q^4 t_{1} t_{2} t_{3} t_{4} \\
& + q^3 t_{1} t_{2} t_{4}^2 - q^4 t_{1}^2 t_{2} t_{4}^2 - q^4 t_{1} t_{2}^2 t_{4}^2 
+ q^3 t_{1} t_{3} t_{4}^2 - q^4 t_{1}^2 t_{3} t_{4}^2 + q^3 t_{2} t_{3} t_{4}^2 - 2 q^4 t_{1} t_{2} t_{3} t_{4}^2 
+ q^4 t_{1}^2 t_{2} t_{3} t_{4}^2 \\ 
&+ q^5 t_{1}^2 t_{2} t_{3} t_{4}^2 - q^4 t_{2}^2 t_{3} t_{4}^2 + q^4 t_{1} t_{2}^2 t_{3} t_{4}^2 
+ q^5 t_{1} t_{2}^2 t_{3} t_{4}^2 - q^5 t_{1}^2 t_{2}^2 t_{3} t_{4}^2 - q^4 t_{1} t_{3}^2 t_{4}^2 
- q^4 t_{2} t_{3}^2 t_{4}^2 + q^4 t_{1} t_{2} t_{3}^2 t_{4}^2 + q^5 t_{1} t_{2} t_{3}^2 t_{4}^2 \\ 
&- q^5 t_{1}^2 t_{2} t_{3}^2 t_{4}^2 - q^5 t_{1} t_{2}^2 t_{3}^2 t_{4}^2 + q^6 t_{1}^2 t_{2}^2 t_{4}^3 
- q^5 t_{1} t_{2} t_{3} t_{4}^3 + q^6 t_{1}^2 t_{2} t_{3} t_{4}^3 - q^6 t_{1}^3 t_{2} t_{3} t_{4}^3 
+ q^6 t_{1} t_{2}^2 t_{3} t_{4}^3 - q^6 t_{1}^2 t_{2}^2 t_{3} t_{4}^3 \\ 
& - q^7 t_{1}^2 t_{2}^2 t_{3} t_{4}^3 + q^7 t_{1}^3 t_{2}^2 t_{3} t_{4}^3 - q^6 t_{1} t_{2}^3 t_{3} t_{4}^3 
+ q^7 t_{1}^2 t_{2}^3 t_{3} t_{4}^3 + q^6 t_{1}^2 t_{3}^2 t_{4}^3 + q^6 t_{1} t_{2} t_{3}^2 t_{4}^3 
- q^6 t_{1}^2 t_{2} t_{3}^2 t_{4}^3 - q^7 t_{1}^2 t_{2} t_{3}^2 t_{4}^3 \\ 
& + q^7 t_{1}^3 t_{2} t_{3}^2 t_{4}^3 + q^6 t_{2}^2 t_{3}^2 t_{4}^3 - q^6 t_{1} t_{2}^2 t_{3}^2 t_{4}^3 
- q^7 t_{1} t_{2}^2 t_{3}^2 t_{4}^3 + 3 q^7 t_{1}^2 t_{2}^2 t_{3}^2 t_{4}^3 - q^8 t_{1}^3 t_{2}^2 t_{3}^2 t_{4}^3 
+ q^7 t_{1} t_{2}^3 t_{3}^2 t_{4}^3 - q^8 t_{1}^2 t_{2}^3 t_{3}^2 t_{4}^3 \\ 
& - q^6 t_{1} t_{2} t_{3}^3 t_{4}^3 + q^7 t_{1}^2 t_{2} t_{3}^3 t_{4}^3 + q^7 t_{1} t_{2}^2 t_{3}^3 t_{4}^3
 -  q^8 t_{1}^2 t_{2}^2 t_{3}^3 t_{4}^3 - q^8 t_{1}^3 t_{2}^3 t_{3} t_{4}^4 + q^8 t_{1}^2 t_{2}^2 t_{3}^2 t_{4}^4
- q^8 t_{1}^3 t_{2}^2 t_{3}^2 t_{4}^4 + q^9 t_{1}^4 t_{2}^2 t_{3}^2 t_{4}^4 \\ 
& - q^8 t_{1}^2 t_{2}^3 t_{3}^2 t_{4}^4 + q^9 t_{1}^3 t_{2}^3 t_{3}^2 t_{4}^4 + q^9 t_{1}^2 t_{2}^4 t_{3}^2 t_{4}^4 
- q^8 t_{1}^3 t_{2} t_{3}^3 t_{4}^4 - q^8 t_{1}^2 t_{2}^2 t_{3}^3 t_{4}^4 + q^9 t_{1}^3 t_{2}^2 t_{3}^3 t_{4}^4 
- q^8 t_{1} t_{2}^3 t_{3}^3 t_{4}^4 + q^9 t_{1}^2 t_{2}^3 t_{3}^3 t_{4}^4 \\ 
& - q^9 t_{1}^3 t_{2}^3 t_{3}^3 t_{4}^4 + q^9 t_{1}^2 t_{2}^2 t_{3}^4 t_{4}^4 + q^9 t_{1}^3 t_{2}^3 t_{3} t_{4}^5 
- q^9 t_{1}^2 t_{2}^2 t_{3}^2 t_{4}^5 + q^9 t_{1}^3 t_{2}^2 t_{3}^2 t_{4}^5 - q^{10} t_{1}^4 t_{2}^2 t_{3}^2 t_{4}^5 
+ q^9 t_{1}^2 t_{2}^3 t_{3}^2 t_{4}^5 - q^{10} t_{1}^3 t_{2}^3 t_{3}^2 t_{4}^5 \\ 
& - q^{10} t_{1}^2 t_{2}^4 t_{3}^2 t_{4}^5 + q^9 t_{1}^3 t_{2} t_{3}^3 t_{4}^5 + q^9 t_{1}^2 t_{2}^2 t_{3}^3 t_{4}^5 
- q^{10} t_{1}^3 t_{2}^2 t_{3}^3 t_{4}^5 + q^9 t_{1} t_{2}^3 t_{3}^3 t_{4}^5 - q^{10} t_{1}^2 t_{2}^3 t_{3}^3 t_{4}^5 
+ q^{10} t_{1}^3 t_{2}^3 t_{3}^3 t_{4}^5 - q^{10} t_{1}^2 t_{2}^2 t_{3}^4 t_{4}^5 \\ 
& - q^{10} t_{1}^3 t_{2}^3 t_{3}^2 t_{4}^6 + q^{11} t_{1}^4 t_{2}^3 t_{3}^2 t_{4}^6 
+ q^{11} t_{1}^3 t_{2}^4 t_{3}^2 t_{4}^6 - q^{12} t_{1}^4 t_{2}^4 t_{3}^2 t_{4}^6 - q^{10} t_{1}^3 t_{2}^2 t_{3}^3 t_{4}^6
+ q^{11} t_{1}^4 t_{2}^2 t_{3}^3 t_{4}^6 - q^{10} t_{1}^2 t_{2}^3 t_{3}^3 t_{4}^6 + 
3 q^{11} t_{1}^3 t_{2}^3 t_{3}^3 t_{4}^6 \\ 
& - q^{11} t_{1}^4 t_{2}^3 t_{3}^3 t_{4}^6 - q^{12} t_{1}^4 t_{2}^3 t_{3}^3 t_{4}^6 + 
q^{12} t_{1}^5 t_{2}^3 t_{3}^3 t_{4}^6 + q^{11} t_{1}^2 t_{2}^4 t_{3}^3 t_{4}^6 - 
q^{11} t_{1}^3 t_{2}^4 t_{3}^3 t_{4}^6 - q^{12} t_{1}^3 t_{2}^4 t_{3}^3 t_{4}^6 + 
q^{12} t_{1}^4 t_{2}^4 t_{3}^3 t_{4}^6 + q^{12} t_{1}^3 t_{2}^5 t_{3}^3 t_{4}^6 \\ 
& + q^{11} t_{1}^3 t_{2}^2 t_{3}^4 t_{4}^6 - q^{12} t_{1}^4 t_{2}^2 t_{3}^4 t_{4}^6 
+ q^{11} t_{1}^2 t_{2}^3 t_{3}^4 t_{4}^6 - q^{11} t_{1}^3 t_{2}^3 t_{3}^4 t_{4}^6 
- q^{12} t_{1}^3 t_{2}^3 t_{3}^4 t_{4}^6 + q^{12} t_{1}^4 t_{2}^3 t_{3}^4 t_{4}^6 - q^{12} t_{1}^2 t_{2}^4 t_{3}^4 t_{4}^6 
+ q^{12} t_{1}^3 t_{2}^4 t_{3}^4 t_{4}^6 \\ 
& - q^{13} t_{1}^4 t_{2}^4 t_{3}^4 t_{4}^6 + q^{12} t_{1}^3 t_{2}^3 t_{3}^5 t_{4}^6 
- q^{13} t_{1}^4 t_{2}^3 t_{3}^3 t_{4}^7 - q^{13} t_{1}^3 t_{2}^4 t_{3}^3 t_{4}^7 
+ q^{13} t_{1}^4 t_{2}^4 t_{3}^3 t_{4}^7 + q^{14} t_{1}^4 t_{2}^4 t_{3}^3 t_{4}^7 
- q^{14} t_{1}^5 t_{2}^4 t_{3}^3 t_{4}^7 - q^{14} t_{1}^4 t_{2}^5 t_{3}^3 t_{4}^7 \\
& - q^{13} t_{1}^3 t_{2}^3 t_{3}^4 t_{4}^7 + q^{13} t_{1}^4 t_{2}^3 t_{3}^4 t_{4}^7 
+ q^{14} t_{1}^4 t_{2}^3 t_{3}^4 t_{4}^7 - q^{14} t_{1}^5 t_{2}^3 t_{3}^4 t_{4}^7 
+ q^{13} t_{1}^3 t_{2}^4 t_{3}^4 t_{4}^7 + q^{14} t_{1}^3 t_{2}^4 t_{3}^4 t_{4}^7 
- 2 q^{14} t_{1}^4 t_{2}^4 t_{3}^4 t_{4}^7 + q^{15} t_{1}^5 t_{2}^4 t_{3}^4 t_{4}^7 \\ 
& - q^{14} t_{1}^3 t_{2}^5 t_{3}^4 t_{4}^7 + q^{15} t_{1}^4 t_{2}^5 t_{3}^4 t_{4}^7 
- q^{14} t_{1}^4 t_{2}^3 t_{3}^5 t_{4}^7 - q^{14} t_{1}^3 t_{2}^4 t_{3}^5 t_{4}^7 
+ q^{15} t_{1}^4 t_{2}^4 t_{3}^5 t_{4}^7 - q^{14} t_{1}^4 t_{2}^4 t_{3}^4 t_{4}^8 
+ q^{15} t_{1}^5 t_{2}^4 t_{3}^4 t_{4}^8 + q^{15} t_{1}^4 t_{2}^5 t_{3}^4 t_{4}^8 \\ 
& - q^{16} t_{1}^5 t_{2}^5 t_{3}^4 t_{4}^8+q^{15} t_{1}^4 t_{2}^4 t_{3}^5 t_{4}^8-q^{16} t_{1}^5 t_{2}^4
t_{3}^5 t_{4}^8-q^{16} t_{1}^4 t_{2}^5 t_{3}^5 t_{4}^8+q^{18} t_{1}^5 t_{2}^5 t_{3}^5 t_{4}^9. 
\end{split}
\end{equation*} 
The corresponding denominator is easily seen to be
 \begin{equation*}
\begin{split}
D(t_{1}, t_{2}, t_{3}, t_{4}; q) = &\left(1 - q t_{1} \right) \left(1 - q  t_{2} \right)  \left(1 - q t_{3} \right) 
\left(1 - q t_{4} \right) \left(1 - q^3 t_{1}^2 t_{4}^2 \right) \left(1 - q^3 t_{2}^2 t_{4}^2 \right) 
\left(1 - q^3 t_{3}^2 t_{4}^2 \right) \\
& \cdot \left(1 - q^4 t_{1}^2 t_{2}^2 t_{4}^2 \right) \left(1 - q^4 t_{1}^2 t_{3}^2 t_{4}^2 \right) 
\left(1 - q^4 t_{2}^2 t_{3}^2 t_{4}^2 \right) 
\left(1 - q^5 t_{1}^2 t_{2}^2 t_{3}^2 t_{4}^2 \right) \left(1 - q^6 t_{1}^2 t_{2}^2 t_{3}^2 t_{4}^4 \right).
\end{split}
\end{equation*} 
A simple comparison with the corresponding (transformed) rational function \eqref{eq: tag 2.6} 
shows in each case that our conditions (i) -- (iii) (see Section \ref{section 3}) are satisfied.

We conclude this appendix by remarking that Ian Whitehead established in his doctoral thesis 
that the unique multiple Dirichlet series associated to the fourth moment of quadratic Dirichlet
L-series over $\mathbb{F}_{\! q}(x)$ satisfying the conditions (i) -- (iii) is given by 
\begin{equation}\label{eq: appendA1-1}
Z_{4}(t_{1}, \ldots, t_{5}; q) \,=\, 
\prod_{l = 0}^{\infty} \left(1 - q^{6l + 6} \left(t_{1}  t_{2} t_{3} t_{4} t_{5}^2 \right)^{2l + 2} \right)^{\! -1} \!\left(1 - q^{6l + 7} \left(t_{1}  t_{2} t_{3} t_{4} t_{5}^2 \right)^{2l + 2} \right)^{\! -1}  \cdot \, Z(t_{1}, \ldots, t_{5}; q) 
\end{equation}
where $Z(t_{1}, \ldots, t_{5}; q)$ is such that its residue at $t_{5} = 1 \slash q$ 
corresponds to \cite[eqn. (4.29)]{BD} by setting $t_{j} = q^{-s_{j}}\!,$ for $j = 1, \ldots, 4;$ the series $Z_{4}(t_{1}, \ldots, t_{5}; q)$ can be made quite explicit by expressing it in terms of the average \cite[eqn. (4.16)]{BD}. There is also a number field version of this multiple 
Dirichlet series (obtained from \eqref{eq: appendA1-1} as explained 
in the introduction). The series \eqref{eq: appendA1-1} (or its number field counterpart) is precisely 
the multiple Dirichlet series that we expect to occur in a Fourier-Whittaker coefficient of an Eisenstein series on the double cover of a loop group (with the corresponding affine Weyl group isomorphic to $W_{4}$).

\section{Appendix}\label{B} 
Throughout this appendix, $\mathbb{F}$ will be a fixed finite field of odd 
characteristic. For partitions 
$
\mu = (1^{\mu_{1}}\!, 2^{\mu_{2}}\!, \ldots, n^{\mu_{n}})
$ 
and 
$
\gamma = (1^{ \gamma_{1}}\!, 2^{\gamma_{2}}\!, \ldots, m^{\gamma_{m}}),
$ 
we shall denote 
$
\mu' = (1^{\mu_{1} + 1}\!, 2^{\mu_{2}}\!, \ldots, n^{\mu_{n}})
$ 
and 
$
\gamma' = (1^{ \gamma_{1} + 1}\!, 2^{\gamma_{2}}\!, \ldots, m^{\gamma_{m}}).
$ 
Let $M_{\mu, \gamma}(q)$ (for $\mu$ and $\gamma$ as above) be defined as in Section \ref{section 3}.

We have the following

\vskip10pt 
\begin{proposition} \label{Prop1-appendixB} --- For $|\mu| = 2g + 1,$ 
we have
\begin{equation*}
\frac{\mu_{1} + 1}{q + 1}\left(M_{\mu, \gamma}(q) + M_{\mu', \gamma}(q) \right) 
= M_{\mu, \gamma}(q).
\end{equation*}
\end{proposition}

\begin{proof} The identity we need to prove is equivalent to
\begin{equation}\label{eq: appendB1}
(\mu_{1} + 1)M_{\mu', \gamma}(q) = (q - \mu_{1}) M_{\mu, \gamma}(q).
\end{equation}
For any partition 
$
\nu = (1^{\nu_{1}}\!, 2^{\nu_{2}}\!, \ldots),
$ 
let $\mathscr{P}(\nu)  \subset  \mathbb{F}[x]$ denote the set of all monic square-free polynomials 
with factorization type $\nu.$ For each $d\in \mathscr{P}(\mu)$ and $\alpha \in \mathbb{F}$ such that 
$d(\alpha) \ne 0,$ let $d_{\alpha}\in \mathscr{P}(\mu')$ be defined by
\begin{equation*}
d_{\alpha}(x) = \frac{(x - \alpha)^{2g + 2}}{d(\alpha)} d\Big(\alpha  +  \frac{1}{x - \alpha} \Big).
\end{equation*}
Notice that for each $d$ there are exactly $q - \mu_{1}$ choices of $\alpha,$ and for each such choice, 
we must have $d_{\alpha}(\alpha) = 0.$ Conversely, let $\tilde{d}\in \mathscr{P}(\mu'),$ and let  
$\alpha \in \mathbb{F}$ be a root of $\tilde{d}.$ (The polynomial $\tilde{d}$ has exactly $\mu_{1} + 1$ distinct roots in $\mathbb{F}.$) If we define 
\begin{equation*} 
d(x) = \frac{(x - \alpha)^{2g + 2}}{\tilde{d}^{^{\, '}}\!(\alpha)} \tilde{d}\Big(\alpha  +  \frac{1}{x - \alpha}\Big)
\end{equation*}
then $d\in \mathscr{P}(\mu),$ $d(\alpha) \ne 0$ and $\tilde{d} = d_{\alpha},$ where $d_{\alpha}$ is defined as above. 

Since the hyperelliptic curves $C_{d}$ and $C_{d_{\alpha}}$ corresponding to $d \in \mathscr{P}(\mu)$ and 
$d_{\alpha}$ are isomorphic, the identity \eqref{eq: appendB1} follows by a simple counting argument. 
\end{proof}

We remark that both sides of the identity in the above proposition vanish if the partition $\gamma$ 
has odd weight. \!When $\gamma = (1^{\gamma_{1}}\!, \ldots, g^{\gamma_{g}})$ ($g\ge 2$) 
has even weight, we can express the identity in the equivalent form 
\begin{equation*}
(\mu_{1} + \, 1)\tilde{M}_{\mu', \gamma}(q) = \frac{1}{q  (q - 1)}M_{\mu, \gamma}(q)
\end{equation*}
where $\tilde{M}_{\mu', \gamma}(q)$ is the moment-sum introduced at the beginning of Section \ref{section 6}.

\vskip10pt 
\begin{proposition} \label{Prop2-appendixB} --- For $|\mu| = 2g + 1$ and $|\gamma|$ odd, 
we have the identity
\begin{equation*}
M_{\mu', \gamma'}(q) + M_{\mu, \gamma'}(q) = 
- \, (q + 1)M_{\mu', \gamma}(q).
\end{equation*}
Moreover, if 
$
\mu = (2^{\mu_{2}}\!, 3^{\mu_{3}}\!, \ldots)
$ 
is a partition of weight $2g + 2,$ and $|\gamma|$ is odd, then   
\begin{equation*}
M_{\mu, \gamma'}(q) = - \, (q + 1)M_{\mu, \gamma}(q).
\end{equation*}
\end{proposition} 

\begin{proof} We first write 
\begin{equation*}
M_{\mu', \gamma'}(q) = -\, M_{\mu', \gamma}(q) - \sum_{\theta \, \in \, \mathbb{F}} \,
\sum_{d \, \in \, \mathscr{P}(\mu')} \chi(d(\theta))
\bigg(\prod_{j = 1}^{m} a_{j}(C_{d})^{\gamma_{j}}\bigg)
\end{equation*}
and let $\mathcal{P}_{\scriptscriptstyle 0}(\mu')$ (respectively 
$\mathcal{P}^{\scriptscriptstyle 0}(\mu')$) 
denote the set of polynomials in $\mathscr{P}(\mu')$ which do not vanish (respectively vanish) at $0.$ 
Replacing $d(x)$ by $d(x - \theta)$ (for each $\theta \in \mathbb{F}$), we see that 
\begin{equation*}
M_{\mu', \gamma'}(q) = -\, M_{\mu', \gamma}(q) \,- \, 
q \!\sum_{d \, \in \, \mathcal{P}_{\scriptscriptstyle 0}(\mu')} \chi(d(0))
\bigg(\prod_{j = 1}^{m} a_{j}(C_{d})^{\gamma_{j}} \bigg).
\end{equation*} 
Since $|\gamma|$ is odd, we can write  
\begin{equation*}
\chi(d(0)) = \chi(d(0))^{|\gamma|} = \prod_{j = 1}^{m} \chi_{j}(d(0))^{\gamma_{j}}
\qquad \text{(for $d \in \mathcal{P}_{\scriptscriptstyle 0}(\mu'))$}.
\end{equation*}
Moreover, one can express $\chi_{j}(d(0)) a_{j}(C_{d}),$ for 
$d \in \mathcal{P}_{\scriptscriptstyle 0}(\mu')$ and $j = 1, \ldots, m,$ as
\begin{equation*}
- 1  - \chi_{j}\bigg(\frac{1}{d(0)}\bigg)\; - \sum_{\theta \, \in \, \mathbb{F}_{\! j}^{\times}} 
\chi_{j}\bigg(\theta^{2g + 2}\, \frac{d(1 \slash \theta)}{d(0)} \bigg).
\end{equation*}
By using the transformation $d(x)\mapsto \frac{x^{2g + 2}}{d(0)}d(1\slash x)$ on 
$\mathcal{P}_{\scriptscriptstyle 0}(\mu'),$ it follows that
\begin{equation*} 
\sum_{d \, \in \, \mathcal{P}_{\scriptscriptstyle 0}(\mu')} \chi(d(0))
\bigg(\prod_{j = 1}^{m} a_{j}(C_{d})^{\gamma_{j}}\bigg)
\;\, = \sum_{d \, \in \, \mathcal{P}_{\scriptscriptstyle 0}(\mu')}\, 
\prod_{j = 1}^{m} a_{j}(C_{d})^{\gamma_{j}}
\end{equation*}
and hence 
\begin{equation}\label{eq: appendB2-1}
M_{\mu', \gamma'}(q) = -\, (q + 1)M_{\mu', \gamma}(q) \,+ \, 
q \!\sum_{d \, \in \, \mathcal{P}^{\scriptscriptstyle 0}(\mu')}\, 
\prod_{j = 1}^{m} a_{j}(C_{d})^{\gamma_{j}}.
\end{equation} 
Similarly, 
\begin{equation*}
M_{\mu, \gamma'}(q) \,=\,  - \sum_{\theta \, \in \, \mathbb{F}}\, 
\sum_{d \, \in \, \mathscr{P}(\mu)} \chi(d(\theta))
\bigg(\prod_{j = 1}^{m} a_{j}(C_{d})^{\gamma_{j}}\bigg) \,= \, 
- \, q \!\sum_{d \, \in \, \mathcal{P}_{\scriptscriptstyle 0}(\mu)} \chi(d(0)) 
\bigg(\prod_{j = 1}^{m} a_{j}(C_{d})^{\gamma_{j}}\bigg)
\end{equation*}
from which we deduce, as before, that 
\begin{equation}\label{eq: appendB2-2}
M_{\mu, \gamma'}(q)  = 
- \, q \!\sum_{d \, \in \, \mathcal{P}^{\scriptscriptstyle 0}(\mu')}\,
\prod_{j = 1}^{m} a_{j}(C_{d})^{\gamma_{j}}.
\end{equation} 
Now from \eqref{eq: appendB2-1} and \eqref{eq: appendB2-2}, it follows that 
\begin{equation*}
M_{\mu', \gamma'}(q) + M_{\mu, \gamma'}(q) = -\, (q + 1)M_{\mu', \gamma}(q)
\end{equation*}
which proves our first assertion. 

The second assertion is proved by a similar argument. \end{proof}

\section{Appendix}\label{C} 
Let us consider \eqref{eq: tag 3.16}, written in the equivalent form 
\begin{equation*}
\lambda(\kappa, 2; q) \, - \, q^{|\kappa| + 3}\lambda(\kappa, 2; 1\slash q) 
\, =\, a(\kappa, 2; q)\, + 
\sum_{\substack{\kappa' < \, \kappa}} 
q^{|\kappa| - \mathrm{r}(\kappa - \kappa') + 3} 
(q - 1)^{\mathrm{r}(\kappa - \kappa')}\lambda(\kappa', 2; 1\slash q)
\end{equation*} 
for $\kappa = (k_{1}, \ldots, k_{r}) \in \mathbb{N}^{r}$ ($r \ge 1$), where 
\begin{equation*}
a(\kappa, 2; q)
\,=\, 
 \begin{cases} 
(- 1)^{|\kappa|}\, q (q - 1) & \text{if $k_{i} = 0$ or $1$ for all $i = 1, \ldots, r$}\\ 
0 & \text{otherwise}.  
\end{cases}
\end{equation*}
This is equivalent to \eqref{eq: tag 4.1} (with $l = 2$), which can be written explicitly as    
\begin{equation}\label{eq: appendC3-1}
\Lambda_{2}(T, q)  =  \frac{q (q - 1)}{E(T)} 
+ q^{3} \frac{E(q \, T)}{E(T)}\Lambda_{2}(q \, T, 1\slash q). 
\end{equation} 
We recall that $T = (t_{1}, \ldots, t_{r}),$ $q\, T = (q t_{1}, \ldots, q t_{r}),$ and 
$$
E(T) = \prod_{i = 1}^{r} (1 - t_{i})^{-1}.
$$ 
As explained in Section \ref{section 5}, in the case $l = 4,$ starting from 
$\lambda(0, \ldots, 0, 2; q) = q^{2},$ we can find recursively the coefficients 
$\lambda(\kappa, 2; q) = P(\kappa, q)$ in polynomial form; the polynomials 
$P(\kappa, q)$ have integer coefficients, are independent of $q,$
$\deg P(\kappa, q) \le |\kappa| + 2$ (condition ensuring that $q^{|\kappa| + 3}\lambda(\kappa, 2; 1\slash q)$ is a polynomial with no constant term), and $q^{[(|\kappa| + 1)\slash 2] + 2} \mid  P(\kappa, q)$ 
(implying that $\deg(q^{|\kappa| + 3}\lambda(\kappa, 2; 1\slash q)) < [(|\kappa| + 1)\slash 2] + 2$).

\vskip5pt
{\it The generating series of $P(\kappa, q).$} Let notations be as above, and take 
\begin{equation*}
\Lambda_{2}(T, q) \; = \sum_{\kappa \, \in \, \mathbb{N}^{r}}P(\kappa, q)\, T^{\kappa}
\end{equation*}
where, for $\kappa = (k_{1}, \ldots, k_{r}),$ we put (as before) 
$T^{\kappa} : = t_{1}^{k_{1}} \cdots \, t_{r}^{k_{r}}.$

\vskip10pt 
\begin{proposition} \label{Prop1-appendixC} --- We have 
\begin{equation*}
\Lambda_{2}(T, q) = q E(q\,  T)\,
\underset{z = 0}{\rm{Res}}\,\bigg[\frac{z^{2} - q}{z(z - 1)} \cdot \frac{1}{E(z\, T)E(q\, T \slash z)}\bigg].
\end{equation*}
\end{proposition} 

\begin{proof} Put
$
f(z, T, q) =\frac{q z^{2} - 1}{z(z - 1)} \cdot \frac{1}{E(q  z  T)E(T \slash z)}. 
$ 
Then \eqref{eq: appendC3-1} amounts to: 
\begin{equation}\label{eq: appendC3-2}
\underset{z = 1}{\rm{Res}}\,f(z, T, q)
\, + \,  \underset{z = 0}{\rm{Res}}\,f(z, T, q)
= \, \underset{z = 0}{\rm{Res}}\,\big[f(1 \slash z, T, q)\slash z \big].
\end{equation} 
Since $f(\cdot, T, q),$ as a function of $z$ is rational, the left-hand side of \eqref{eq: appendC3-2} 
equals 
\begin{equation*}
- \, \underset{z = \infty}{\rm{Res}}\,f(z, T, q) =\, \underset{z = 0}{\rm{Res}}\,\big[f(1 \slash z, T, q)\slash z^{2} \big] 
\end{equation*}
and thus we are left to show that 
\begin{equation}\label{eq: appendC3-3}
\underset{z = 0}{\rm{Res}}\,\big[f(1 \slash z, T, q)\slash z^{2} \big]
= \, \underset{z = 0}{\rm{Res}}\,\big[f(1 \slash z, T, q)\slash z \big].
\end{equation} 
Now 
\begin{equation*}
\Big(\frac{z - 1}{z^{2}}\Big)f(1 \slash z, T, q) \,=\,
\frac{1}{E(z\, T)E(q\, T \slash z)} \,-\, \frac{q}{z^{2}} \frac{1}{E(z\, T)E(q\, T \slash z)} 
\end{equation*}
and note that 
\begin{equation*}
\underset{z = 0}{\rm{Res}}\,\bigg[\frac{1}{E(z\, T)E(q\, T \slash z)} \bigg] =  
- \, \underset{z = \infty}{\rm{Res}}\,\bigg[\frac{1}{E(z\, T)E(q\, T \slash z)} \bigg] =\, 
\underset{z = 0}{\rm{Res}}\,\bigg[\frac{1}{z^{2}} \frac{1}{E(T\slash z)E(q z T)} \bigg].
\end{equation*}
Replacing $z$ by $z\slash q$ in the expression of the last residue, we have 
\begin{equation*}
q^{2}\, \underset{z = 0}{\rm{Res}}\,\bigg[\frac{1}{z^{2}} \frac{1}{E(q\, T\slash z)E(z T)} \bigg] = \,  
q \, \underset{z = 0}{\rm{Res}}\,\bigg[\frac{1}{z^{2}} \frac{1}{E(T\slash z)E(q z T)} \bigg] 
\end{equation*}
and \eqref{eq: appendC3-3} follows. As 
\begin{equation*}
q^{2}\Lambda_{2}(q\, T, 1\slash q) \; =
\sum_{\kappa \, \in \, \mathbb{N}^{r}} q^{|\kappa| + 2}P(\kappa, 1\slash q)\, T^{\kappa}
= \,  E(T)\, \underset{z = 0}{\rm{Res}}\,\bigg[\frac{q z^{2} - 1}{z(z - 1)} \cdot \frac{1}{E(q z\, T)E(T \slash z)}\bigg]
\end{equation*}
clearly $q^{|\kappa| + 2}P(\kappa, 1\slash q)$ is a polynomial in $q,$ or, what amounts to the same, 
$\deg P(\kappa, q) \le |\kappa| + 2.$

Finally, let us fix $\kappa \in \mathbb{N}^{r}.$ Then  
\begin{equation}\label{eq: appendC3-4}
P(\kappa, q) \;\;  = 
\sum_{\substack{\kappa' \le \, \kappa \\ k_{1}'\!, \ldots, k_{r}' \, \in \, \{0, 1, 2\}}} 
(- 1)^{r_{1}(\kappa')}q^{|\kappa - \kappa'| + r_{2}(\kappa') + 1} 
\, \underset{z = 0}{\rm{Res}}\,\bigg[\frac{z^{2} - q}{z(z - 1)}
\Big(\frac{q}{z} + z \Big)^{r_{1}(\kappa')}\bigg]
\end{equation}
where, as before, $r_{j}(\kappa')$ ($j = 1, 2$) is the number of components of 
$\kappa'$ equal to $j.$ For each $\kappa',$ the expression  
\begin{equation*} 
q^{|\kappa - \kappa'| + r_{2}(\kappa') + 1} 
\, \underset{z = 0}{\rm{Res}}\,\bigg[\frac{z^{2} - q}{z(z - 1)}
\Big(\frac{q}{z} + z \Big)^{r_{1}(\kappa')}\bigg]
\end{equation*} 
is trivially a polynomial in $q,$ and it is easy to see that the exponent of the largest power of $q$ dividing this polynomial is at least 
\begin{equation*}
|\kappa - \kappa'| + r_{2}(\kappa') + \frac{r_{1}(\kappa')}{2} + 2 = 
\frac{|\kappa - \kappa'|}{2} + \frac{|\kappa|}{2} + 2.
\end{equation*} 
Then clearly $q^{[(|\kappa| + 1)\slash 2] + 2} \mid  P(\kappa, q),$ and our assertion follows now from 
the uniqueness of a polynomial solution $\Lambda_{2}(T,q)$ to \eqref{eq: appendC3-1} satisfying the 
above properties.
\end{proof} 

Notice that by taking $\kappa = (1, 1, \ldots, 1)\in \mathbb{N}^{r}$ in \eqref{eq: appendC3-4}, we 
can write 
\begin{equation*}
\lambda_{2}(r; q)  := \lambda(1,\ldots,1, 2; q)\;\;  = 
\sum_{\substack{\kappa' \le \, \kappa \\ k_{1}'\!, \ldots, k_{r}' \, \in \, \{0, 1\}}} 
(- 1)^{|\kappa'|}\, q^{r - |\kappa'|  +  1} 
\,\underset{z = 0}{\rm{Res}}\,\bigg[\frac{z^{2} - q}{z(z - 1)}
\Big(\frac{q}{z} + z \Big)^{|\kappa'|}\bigg]
\end{equation*} 
from which it follows that 
\begin{equation*}
\lambda_{2}(r; q) \,=\, \sum_{j = 1}^{[\frac{r}{2}]}\, 
\frac{1}{(r - j + 1)(r - j)}\,
\frac{r!}{j!(j-1)!(r - 2j)!}\,q^{r + 2 -  j} 
\end{equation*}
as claimed in Section \ref{section 3}, Example \ref{exmpl=2}.

The following identity was used in the proof of Theorem \ref{Theorem 5.2}.

\vskip10pt 
\begin{proposition} \label{Prop2-appendixC} --- If $\kappa = (k_{1},\ldots, k_{r})\in \mathbb{N}^{r}$ 
is such that $r_{1}(\kappa) = 2R$ is even, and $|\kappa| = r_{1}(\kappa) +  2 r_{2}(\kappa)$ 
then 
\begin{equation*}
\sum_{\substack{\kappa' \le \, \kappa_{_{\phantom{X}}} \\ k_{i}^{} - k_{i}' = 0, 1}} 
(- 1)^{|\kappa - \kappa'|}\, q^{|\kappa'| + 2} \lambda(\kappa', 2; 1\slash q) 
\, = \sum_{j = 0}^{R} \, (2j + 1)\, \frac{(2R)! \,q^{R + r_{2}(\kappa) - j}}{(R - j)!(R + j + 1)!}
\end{equation*}
\end{proposition}

\begin{proof} The left-hand side of the identity is the coefficient of 
$T^{\kappa} = t_{1}^{k_{1}} \cdots \, t_{r}^{k_{r}}$ in   
\begin{equation*}
q^{2}\frac{\Lambda_{2}(q\, T, 1\slash q)}{E(T)} =  
\underset{z = 0}{\rm{Res}}\,\bigg[\frac{q z^{2} - 1}{z(z - 1)} 
\cdot \frac{1}{E(q z\, T)E(T \slash z)}\bigg];
\end{equation*} 
it vanishes, unless $|\kappa| = r_{1}(\kappa) +  2 r_{2}(\kappa).$ If $|\kappa| = r_{1}(\kappa) +  2 r_{2}(\kappa),$ and $r_{1}(\kappa) = 2R,$ then we can express this coefficient by 
\begin{equation*}
q^{r_{2}(\kappa)}\cdot \underset{z = 0}{\rm{Res}}\,
\bigg[\frac{q z^{2} - 1}{z(z - 1)}\bigg(q  z  + \frac{1}{z} \bigg)^{\! 2 R}\, \bigg] 
\end{equation*} 
from which our assertion follows. 
\end{proof}

\section{Appendix}\label{D} 
As before, let $\mathbb{F} = \mathbb{F}_{\! q}$ be a fixed finite field of odd characteristic, and let $\chi$ denote the non-trivial real character of $\mathbb{F}^{\times}\!,$ extended to $\mathbb{F}$ 
by $\chi(0) = 0.$ For $r \in \mathbb{N},$ let 
\begin{equation*} 
\mathcal{M}_{3}(r; q)\;\,= \hskip-4pt \sum_{\substack{\deg  d = 3 \\ d-\text{monic \& square-free}}}
\hskip-6pt \left(- \sum_{\theta \, \in \, \mathbb{F}} \chi(d(\theta)) \right)^{\! r} 
\;\;\,\text{and} \;\;\; \mathcal{M}_{4}(r; q) 
\;\, = \hskip-4pt\sum_{\substack{\deg d = 4 \\ d-\text{monic \& square-free}}}
\hskip-6pt \left(-\, 1\, - \, \sum_{\theta \, \in \, \mathbb{F}} \chi(d(\theta)) \right)^{\! r}\!.
\end{equation*} 
be the moment-sums introduced in Section \ref{section 5}.

The purpose of this appendix is to discuss the identity 
\begin{equation*}
\mathcal{M}_{4}(r; q)
\,=\, 
\begin{cases} 
 -\, \mathcal{M}_{3}(r + 1; q) & \text{if $r$ is odd}\\ 
q \mathcal{M}_{3}(r; q) & \text{if $r$ is even}
\end{cases}
\end{equation*}
used in the proof of Proposition \ref{functional equation A4}. Notice that Proposition \ref{Prop2-appendixB}, applied for all partitions of 4, implies easily the equivalence:
\begin{equation*}
\mathcal{M}_{4}(r; q) = - \mathcal{M}_{3}(r + 1; q) \;\, \text{if $r$ is odd}
\iff
\mathcal{M}_{4}(r; q) = q \mathcal{M}_{3}(r; q) \;\, \text{if $r$ is even}.
\end{equation*}

Accordingly, it suffices to prove the following

\vskip10pt 
\begin{proposition} \label{Prop-appendixD} --- For $r \in \mathbb{N}$ even, we have
$
\mathcal{M}_{4}(r; q) = q \mathcal{M}_{3}(r; q).
$
\end{proposition}

\begin{proof} Recall that every complete, non-singular curve of genus one over 
$\mathbb{F}$ has an $\mathbb{F}$-rational point, and so, it is isomorphic to its Jacobian; 
the Jacobian is an elliptic curve defined over $\mathbb{F}.$ Let 
\begin{equation*}
\mathscr{E}(r; q) \,=\, \sum_{[E]} 
\frac{a(E)^{r}}{\# \text{Aut}_{_{\mathbb{F}}}(E)}
\end{equation*}
the sum being over all $\mathbb{F}$-isomorphism classes 
of elliptic curves over $\mathbb{F},$ and where 
$a(E) = q + 1 - \# E(\mathbb{F}).$ We shall compare 
$\mathcal{M}_{4}(r; q)$ and $\mathcal{M}_{3}(r; q)$ with 
$\mathscr{E}(r; q).$ To do so, let $\mathscr{P}$ denote the set 
of square-free polynomials of degree 3 or 4 with coefficients in 
$\mathbb{F}.$ Let $C_{d}$ denote the curve corresponding to $d \in \mathscr{P}.$ 
The group $G  = \left(\mathrm{GL}_{2}(\mathbb{F}) \times \mathbb{F}^{\times}\right)\slash D,$ where 
$$
D = \left\{\left(\begin{pmatrix}
\alpha &  \\ 
   & \alpha 
\end{pmatrix}, \alpha^{2} \right) : \alpha \in 
\mathbb{F}^{\times} \right\},
$$ 
acts on $\mathscr{P}$ by 
\begin{equation*}
^{g}d(x) = \frac{(\gamma x + \delta)^{4}}{\eta^{2}}
d\left(\frac{\alpha x + \beta}{\gamma x + \delta}\right) 
\;\;\;\qquad \;\;\; 
\left(\text{for $g = \left[\left(\begin{pmatrix}
\alpha &  \beta \\ 
\gamma   & \delta 
\end{pmatrix}, \eta \right)\right] \in G$ and $d \in \mathscr{P}$} \right);
\end{equation*} 
the corresponding curves $C_{d}$ and $C_{^{^{g}}\!\! d}$ are clearly $\mathbb{F}$-isomorphic. \!As explained in \cite[pp. 268-269]{Berg}, we can write 
\begin{equation*}
\mathscr{E}(r; q) \;\, = 
\sum_{[d] \, \in \, \mathscr{P} \slash G}\, 
\frac{a(C_{d})^{r}}{\# \mathrm{Stab}_{\scriptscriptstyle G}(d)} 
\,=\, \frac{1}{\# G}\!\sum_{d \, \in \, \mathscr{P}} a(C_{d})^{r}. 
\end{equation*} 
Since $\# G = |\mathrm{GL}_{2}(\mathbb{F})| = q (q + 1) (q - 1)^{2}$ 
and $r$ is even, this identity can be written in the form
\begin{equation}\label{eq: appendD4-1}
\mathscr{E}(r; q)\, =\, \frac{1}{q^{3} - q}
\big(\mathcal{M}_{3}(r; q) + \mathcal{M}_{4}(r; q)\big). 
\end{equation} 
Now, let $E\slash \mathbb{F}$ be an elliptic curve. It is well-known that 
the Riemann-Roch theorem gives an isomorphism of $E$ onto a curve given by a 
Weierstrass equation of the form $y^{2} = d(x),$ for some monic square-free cubic 
polynomial $d \in \mathbb{F}[x].$ Moreover, any two such Weierstrass equations for $E$ are related by a linear change of variables of the form 
\begin{equation*}
x = \alpha^{2} x' + \beta \;\; \text{and} \;\; y = \alpha^{3} y' 
\;\;\qquad\;\; \text{(with $\alpha \in \mathbb{F}^{\times}$ and 
$\beta \in \mathbb{F}$)}.
\end{equation*}
(Recall that we are assuming that $\mathrm{char}(\mathbb{F}) > 2.$) Identifying the group of these transformations with the corresponding subgroup (say $G_{0}$) of $G,$ it follows that 
\begin{equation*}
\mathscr{E}(r; q) \;= 
\sum_{\left\{d \in \mathbb{F}[x] \, : \, 
\text{monic and square-free, $\deg d = 3$} \right\} \slash G_{0}}\, 
\frac{a(C_{d})^{r}}{\# \mathrm{Stab}_{\scriptscriptstyle G_{0}}\!(d)} 
\,=\, \frac{1}{\# G_{0}}\mathcal{M}_{3}(r; q). 
\end{equation*}
Thus 
\begin{equation*}
\mathscr{E}(r; q)\, =\, \frac{1}{q^{2} - q}\mathcal{M}_{3}(r; q) 
\end{equation*}
which, combined with \eqref{eq: appendD4-1}, gives the identity in the proposition. 
\end{proof}

\end{document}